\documentclass[11pt]{amsart}
\usepackage[utf8]{inputenc}
\usepackage{graphicx} % Required for inserting images

\usepackage{fullpage, amsmath, amssymb, latexsym, amsfonts, amscd, amsthm, extarrows, mathtools, tikz-cd}

\usepackage{stackengine}
\usepackage{upgreek}
\usepackage{comment}
\usepackage{enumitem}

\usepackage{hyperref}
\hypersetup{pdftoolbar=true, pdftitle={Template}, pdffitwindow=true, colorlinks=true, citecolor=green, filecolor=black, linkcolor=blue, urlcolor=red, hypertexnames=false}

\newtheoremstyle{thmlike}% name
{8pt}% Space above
{3pt}% Space below
{\slshape}% Body font
{}% Indent amount
{\bfseries}% Theorem head font
{.}% Punctuation after theorem head
{1em}% Space after theorem head
{}% Theorem head spec (can be left empty, meaning `normal')

\newtheoremstyle{deflike}% name
{8pt}% Space above
{3pt}% Space below
{}% Body font
{}% Indent amount
{\bfseries}% Theorem head font
{.}% Punctuation after theorem head
{1em}% Space after theorem head
{}% Theorem head spec (can be left empty, meaning `normal')

\theoremstyle{thmlike}
\newtheorem*{theorem*}{Theorem}
\newtheorem{theorem}{Theorem}[section]
\newtheorem{proposition}[theorem]{Proposition}
\newtheorem{lemma}[theorem]{Lemma}
\newtheorem{corollary}[theorem]{Corollary}

\newtheorem{conjecture}[theorem]{Conjecture}

\theoremstyle{deflike}
\newtheorem{definition}[theorem]{Definition}
\newtheorem{remark}[theorem]{Remark}

\newtheorem{hypothesis}[theorem]{HYPOTHESIS}

\newcommand{\can}{\mathrm{can}}
\newcommand{\cusp}{\mathrm{cusp}}
\newcommand{\hol}{\mathrm{hol}}
\newcommand{\inv}{\mathrm{inv}}
\newcommand{\nil}{\mathrm{nil}}
\newcommand{\norm}{\mathrm{norm}}
\newcommand{\opp}{\mathrm{op}}
\newcommand{\ord}{\mathrm{ord}}

\newcommand{\Pord}{\text{$P$-ord}}
\newcommand{\Pword}{\text{$P_w$-ord}}
\newcommand{\Paord}{\text{$P$-a.ord}}
\newcommand{\Pwaord}{\text{$P_w$-a.ord}}
\newcommand{\Ser}{\mathrm{Ser}}
\newcommand{\sgl}{\mathrm{Sgl}}
\newcommand{\sub}{\mathrm{sub}}
\newcommand{\tor}{\mathrm{tor}}
\newcommand{\tp}[1]{\prescript{t}{}{#1}} 
\newcommand{\level}[2]{\prescript{}{#1}{#2}} %transpose prescript

\newcommand{\diag}{\operatorname{diag}}

\newcommand{\ehom}{\operatorname{End}}
\newcommand{\End}{\operatorname{End}}

\renewcommand{\ker}{\operatorname{ker}}
\renewcommand{\hom}{\operatorname{Hom}}
\newcommand{\id}{\operatorname{id}}

\newcommand{\isom}{\operatorname{Isom}}
\newcommand{\Span}{\operatorname{span}}

\newcommand{\Sh}{\operatorname{Sh}}
\newcommand{\spec}{\operatorname{Spec}}

\newcommand{\supp}{\operatorname{supp}}

\newcommand{\Gm}{\operatorname{\mathbb{G}_{m}}}
\newcommand{\Ga}{\operatorname{\mathbb{G}_{a}}}
\newcommand{\Gal}{\operatorname{Gal}}
\newcommand{\GL}{\operatorname{GL}}
\newcommand{\Her}{\operatorname{Her}}
\newcommand{\Lie}{\operatorname{Lie}}
\newcommand{\SL}{\operatorname{SL}}

\newcommand{\GU}{\operatorname{GU}}
\newcommand{\SU}{\operatorname{SU}}
\newcommand{\Mat}{\operatorname{Mat}}

\newcommand{\Char}{\operatorname{char}}
\newcommand{\cond}{\operatorname{cond}}
\renewcommand{\det}{\operatorname{det}}
\newcommand{\nm}{\operatorname{Nm}}
\newcommand{\rk}{\operatorname{rank}}
\newcommand{\tr}{\operatorname{tr}}
\renewcommand{\v}{\operatorname{\nu}} %Valuation and shorthand for greek letter \nu
\newcommand{\vol}{\operatorname{Vol}}

\newcommand{\ad}{\operatorname{ad}}

\newcommand{\ind}[2]{\operatorname{\iota}_{#1}^{#2}}
\newcommand{\Ind}{\operatorname{Ind}}
\newcommand{\Rep}{\operatorname{Rep}}
\newcommand{\res}{\operatorname{Res}}
\renewcommand{\r}[2]{\operatorname{r}_{#1}^{#2}}
\newcommand{\cg}[1]{\widetilde{#1}} %For contragradient representations

\renewcommand{\AA}{\mathbb{A}}
\newcommand{\Ab}{\mathcal{A}}
\newcommand{\CC}{\mathbb{C}}

\newcommand{\DD}{\mathbb{D}}
\newcommand{\EE}{\mathcal{E}}
\newcommand{\Eis}{\mathrm{Eis}}

\newcommand{\GG}{\mathfrak{G}}

\newcommand{\II}{\mathcal{I}}

\newcommand{\Ig}{\mathrm{Ig}}

\newcommand{\KK}{\mathcal{K}}
\newcommand{\LL}{\mathcal{L}}
\newcommand{\MM}{\mathrm{M}}
\newcommand{\MMM}{\mathcal{M}}
\newcommand{\OO}{\mathcal{O}}
\newcommand{\PP}{\mathcal{P}}
\newcommand{\QQ}{\mathbb{Q}}
\newcommand{\bQQ}{\overline{\QQ}}
\newcommand{\RR}{\mathbb{R}}

\newcommand{\SSS}{\mathcal{S}}
\newcommand{\TT}{\mathbb{T}} %Localized Hecke algebra
\newcommand{\bTT}{\mathbf{T}} %Global Hecke algebra
\newcommand{\TTT}{\mathcal{T}}
\newcommand{\VV}{\mathcal{V}}
\newcommand{\WWW}{\mathcal{W}}
\newcommand{\ZZ}{\mathbb{Z}}
\newcommand{\bZZ}{\overline{\ZZ}}

\renewcommand{\i}{\iota}

\newcommand{\w}{\omega}

\newcommand{\bb}{\mathfrak{b}} 
\newcommand{\g}{\mathfrak{g}} 
\newcommand{\m}{\mathfrak{m}}
\newcommand{\p}{\mathfrak{p}} 
\newcommand{\s}{\mathfrak{s}}
\newcommand{\X}{\mathfrak{X}}
\renewcommand{\P}{\mathfrak{P}}
\renewcommand{\k}{\mathfrak{k}}

\renewcommand{\d}{\bold{d}} %Notation for partitions of a number
\newcommand{\Ss}{S_\square}
\newcommand{\dEis}{d\mathrm{Eis}}

\newcommand{\la}{\langle} %For inner products
\newcommand{\ra}{\rangle}

\newcommand{\incl}{\mathrm{incl}}
\newcommand{\pr}{\mathrm{pr}} %shorthand for projection maps

\newcommand{\brkt}[2]{\la #1, #2 \ra}
\newcommand{\brktdotdot}{\brkt{\cdot}{\cdot}}
\newcommand{\absv}[1]{\left|#1\right|}

\newcommand {\ol}[1] {\overline{#1}}
\newcommand {\ul}[1] {\underline{#1}}
\newcommand {\wt}[1] {\widetilde{#1}}
\newcommand {\wh}[1] {\widehat{#1}}

\newcommand{\at}{\makeatletter @\makeatother}

\usepackage[textwidth=14cm]{geometry}
\usepackage{fancyhdr}
\begin{document}
\pagestyle{fancy}
\fancyhf{}
\fancyhead[RO,LE]{\footnotesize\thepage}
\fancyhead[CE]{\footnotesize\leftmark}
\fancyhead[CO]{\footnotesize $p$-ADIC $L$-FUNCTIONS FOR $P$-ORDINARY FAMILIES}
\renewcommand{\headrulewidth}{0pt}

\title{$p$-adic $L$-functions for $P$-ordinary Hida families on unitary groups}
\author[D. Marcil]{David Marcil}
\date{\today}
\address{David Marcil, Department of Mathematics, Columbia University, 2990 Broadway, New York, NY 10027, USA}
\email{dmarcil\at math.columbia.edu}


\begin{abstract}
    We construct a $p$-adic $L$-function for $P$-ordinary Hida families of cuspidal automorphic representations on a unitary group $G$. The main new idea of our work is to incorporate the theory of Schneider-Zink types for the Levi quotient of $P$, to allow for the possibility of higher ramification at primes dividing $p$, into the study of ($p$-adic) modular forms and automorphic representations on $G$. For instance, we describe the local structure of such a $P$-ordinary automorphic representation $\pi$ at $p$ using these types, allowing us to analyze the geometry of $P$-ordinary Hida families. Furthermore, these types play a crucial role in the construction of certain Siegel Eisenstein series designed to be compatible with such Hida families in two specific ways : Their Fourier coefficients can be $p$-adically interpolated into a $p$-adic Eisenstein measure on $d+1$ variables and, via the doubling method of Garrett and Piatetski--Shapiro-Rallis, the corresponding zeta integrals yield special values of standard $L$-functions. Here, $d$ is the rank of the Levi quotient of $P$. Lastly, the doubling method is reinterpreted algebraically as a pairing between modular forms on $G$, whose nebentype are types, and viewed as the evaluation of our $p$-adic $L$-function at classical points of a $P$-ordinary Hida family.
\end{abstract}

\maketitle

\setcounter{tocdepth}{2}
\tableofcontents

\section*{Introduction}
In \cite{HLS}, Michael Harris, Jian-Shu Li and Christopher Skinner initiated a project to construct a $p$-adic $L$-function for ordinary Hida families on a unitary group of arbitrary signature. In \cite{EHLS}, together with Ellen Eischen, they completed this project $p$-adic $L$-function for ordinary families on unitary groups. This required the development of several technical results on $p$-adic differential operators, accomplished in great part by Eischen in \cite{Eis12}, to obtain a more general Eisenstein measure \cite{Eis15} than the one originally constructed in \cite{HLS}. Fundamental properties of their $p$-adic $L$-function for families are obtained by carefully computing local zeta integrals related to the doubling method \cite{GPSR87} as well as local coefficients of Siegel Eisenstein series \cite{Eis15}. The most technical calculations are for local factors at places above the fixed prime $p$. Moreover, a theorem of Hida in \cite{Hid98} establishing the uniqueness (up to scalar) of ordinary vectors plays a crucial role in their analysis.

In this article, we generalize the many steps involved in their construction to construct a $p$-adic $L$-function for a \emph{$P$-ordinary} Hida family on $G$. Here, $P$ is a parabolic subgroup of a product of general linear groups related to $G$. When $P$ corresponds to (a product of) upper triangular Borel subgroups, the notion of $\pi$ being ``$P$-ordinary'' coincides with the usual notion of being ``ordinary'' studied in \cite{EHLS}.

One advantage of working with $P$-ordinary families is that they are substantially more general than ordinary families. In particular, every cuspidal automorphic representation lies in a $P$-ordinary family for some choice of $P$. However, the dimension of families is inverse proportional to the size of the chosen parabolic $P$. Namely, if $d$ is the rank of the center of the Levi of $P$, then a $P$-ordinary family is $d$-dimensional.

In this paper, we therefore construct a $(d+1)$-variable $p$-adic $L$-function on a $P$-ordinary Hida family associated to a $P$-ordinary automorphic representation. Here the extra variable corresponds to the cyclotomic variable.

Note that our work actually considers \emph{(anti-holomorphic) $P$
-anti-ordinary} automorphic representations, however we discuss the technicalities of this notion in details later in this paper. 

\subsubsection*{Structure of this paper.} 

In Part \ref{part : Pord theory on unitary groups}, we introduce the necessary setup to discuss the geometry, representation theory and the complex analysis related to ($P$-ordinary) automorphic representations. 

The fundamental difference between working with a representation $\pi$ that is $P$-ordinary (at $p$) instead of ordinary representation is the necessity to consider local finite-dimensional representations, i.e. \emph{types}, instead of characters to study the structure at $p$ of $\pi$. Therefore, in Section \ref{sec:notation and conventions}, we discuss some of the work of Bushnell-Kutzko \cite{BusKut93, BusKut98, BusKut99} and Schneider-Zink \cite{SchZin99} to study smooth representations of local $p$-adic groups via types. 
we generalize this theorem of Hida to construct a canonical finite-dimensional subspace in the space of $P$-ordinary vectors for a $P$-ordinary representation $\pi$ on a unitary group $G$. 

In Section \ref{sec:PIwa subgps of unitary gps}, we study the geometry of the Shimura varieties associated to general unitary groups and vector bundles associated to a weight and a type. The automorphic representations of interest in this paper have a non-trivial contribution in the cohomology groups of such bundles. In particular, we introduce Iwahori subgroups $I_{P,r}^0$ and pro-$p$-Iwahori subgroup $I_{P,r}$, that depend on $P$, and whose quotient is a product of general linear group (and not a torus). The center of this quotient plays an important role to parameterize $P$-ordinary Hida families. Furthermore, a type (or a \emph{$P$-nebentypus}) is a smooth representation of $I_{P,r}^0$ that factors through $I_{P,r}$.

In Section \ref{sec:P(anti)ord aut representations}, we introduce the notion of a \emph{$P$-ordinary} representation $\pi$, namely a representation that is ordinary with respect to some parabolic subgroup $P$ of $G$. The main goal is to explain how, for our purpose, their theory is encapsulated by the information of a weight $\kappa$, a level $K_r$ and a $P$-nebentypus $\tau$. The weight holds archimedean information on $\pi_\infty$, the level holds information about ramified places and $p$, and $\tau$ holds information about $\pi_p$. We refer to the datum $(\kappa, K_r, \tau)$ has the \emph{$P$-weight-level-type} of $\pi$.

In Section \ref{sec:comp and comp bw PEL data}, we briefly recall the functors necessary to compare automorphic representations between the various unitary groups involved in the doubling method. Most of this work is well-established in \cite{EHLS}, however one needs to make minor modifications to consider the $P$-ordinary setting when comparing level subgroups on different unitary groups.

In Section \ref{sec:Pord padic mod forms}, we introduce $p$-adic modular forms with respect to the choice of parabolic $P$. To the best of the author's knowledge, this material has yet to be discussed in the literature when working with unitary groups.

We introduce two variations of such $p$-adic forms: scalar-valued ones and vector-valued ones. The first case, i.e. the scalar-valued case, is relatively similar to the usual notion of $p$-adic modular forms, as discussed in \cite{Hid04} and \cite{EHLS}. However, the global sections on the Igusa tower considered are not fixed by a the unipotent radical subgroup of a Borel but rather by the unipotent radical of $P$. This allows us to study the smooth action of the Levi of $P$ on this space of $p$-adic $L$-functions and decompose it as a direct sum over types. 

The second case, i.e. the vector-valued case, considers global sections over the Igusa tower of a non-trivial vector bundle. The vector bundles involved are closely related to the $P$-nebentypus introduced in the previous sections. After introducing the relevant notation, we then discuss some of the ``standard'' results of Hida theory, i.e. density, classicality and the vertical control theorem, in the $P$-ordinary setting. Note that some of these conjectures are later stated in Section \ref{sec:P-(anti-)ord Hida families}. 

As the goal of this paper is not to establish ``$P$-ordinary Hida theory'', we simply leave the necessary results as conjectures and the author plans to revisit each of these conjectures in a subsequent paper to complete the theory developped in this section.

In Part \ref{part:families of P(a)ord aut reps}, we dedicate Section \ref{sec:structure at p of P(a-)ord aut reps} to study of the local representation theory at $p$ of $P$-ordinary (and $P$-anti-ordinary) representations. The goal is to generalize the theorem of Hida mentioned above in the introduction to the $P$-ordinary setting and interpret it in various settings necessary for applications with the doubling method in later sections.

We obtain the uniqueness (up to scalar) of an embedding of certain ``Schneider-Zink types'' (discussed in Section \ref{sec:notation and conventions}) inside the space of $P$-ordinary vectors of a $P$-ordinary representations. In other words, although $P$-ordinary vectors are not unique (up to scalar) as in the ordinary case, we can construct a canonical subspace associated to a type whose dimension equals the dimension of this type. This represents the first main accomplishment of this paper.

In later section, it becomes clear that the construction of our $p$-adic $L$-function does not depend on the choice of a $P$-ordinary vector in this subspace associated to a type but only on the unique (up to scalar) embedding of that type. 

Note that to obtain this result, we impose a certain hypothesis on the supercuspidal support of the $P$-ordinary representations involved. The author expects that this hypothesis is completely superficial and can be removed. We use it to simplify the analysis of the filtration obtained by the Bernstein-Zelevinsky geometric lemma of local representations involved. However, the results and the remainder of the paper are phrased with as little dependency as possible to this hypothesis. The author plans to revisit this issue in a later paper to explain how the results of Section \ref{sec:structure at p of P(a-)ord aut reps} should still holds without this hypothesis.

We discuss in Sections \ref{sec:explicit choice of P-(anti-)ord vectors}-\ref{sec:P-(anti-)ord Hida families} the relation between such subspaces of $P$-ordinary vectors and analogous subspaces of $P$-ordinary modular forms and how these subspaces vary nicely over a $d$-dimensional $P$-ordinary family of representations over a weight space associated to $P$. Here, $d$ is the rank of the center of the Levi of $P$, as mentioned in the introduction above. This requires the study of certain Hecke algebras of infinite level at $p$ associated to a weight and $P$-nebentypus. In particular, this analysis demonstrates that the $P$-nebentypus cuts out a branch of an infinite dimensional $p$-adic space containing such a $P$-ordinary family. The $p$-adic $L$-function constructed in this paper is a function on such a branch.

In Part \ref{part:Pord family of Eis series}, we present all the necessary computations construct our $p$-adic $L$-function using the doubling method of Garrett and Piatetski--Shapiro-Rallis. This first requires the construction in Section \ref{sec:Sgl Eis series for dbl method} of certain Siegel Eisenstein series depending on various inputs, most importantly on the $P$-weight-level-type of a $P$-ordinary representation as well as a Hecke character. 

To accomplish this, we build the Eisenstein series from local Siegel-Weil sections, one for every place of $\QQ$. In the literature, such local sections have been well studied, especially at archimedean places and at unramified places. However our construction of local sections at $p$ in Section \ref{subsec:Local SW section at p} is considerably more involved and represents the core of the computations that follow.

The main goal for this task consists of finding a local section (at $p$), given a fixed type $\tau$, whose contribution to local zeta integrals yields the right Euler factors at $p$ of $L$-functions, see Theorem \ref{thm:main thm - comp of Zw above p}, and whose contribution to the Fourier coefficients of the corresponding Eisenstein series fit in a $p$-adic measure, see Proposition \ref{prop:global Fourier coeff formula}. Our construction generalizes the already complicated machinery developed in \cite[Section 4.3.1]{EHLS} (where $\tau$ is only allowed to be a character). However, the author hopes that the systematic use of types helps to simplify the exposition to some extent.

In Section \ref{sec:doubling method integral}, we then compute the local zeta integrals associated to the factorization of doubling method integral between the Siegel Eisenstein series constructed in the previous section and the test vectors constructed in Section \ref{sec:explicit choice of P-(anti-)ord vectors}.

This yields the second main accomplishment of this paper, briefly mentioned above, namely the calculations of local zeta integrals at $p$ in the $P$-ordinary setting. Our approach owes a great deal to the precise details explained in \cite[Section 4.3.6]{EHLS}. Nonetheless, our analysis requires to resolve many issues related to the dimension of the types. In particular, the ``ordinary characters'' involved in \emph{loc.cit} are replaced by $P$-nebentypes. The main novelty of our work is to use matrix coefficients of these $P$-nebentypus to compute the necessary integrals explicitly and relate them to special values of standard $L$-functions.

Then, in Section \ref{sec:Pord Eisenstein measure}, we $p$-adically interpolate the Siegel Eisenstein series previously constructed to obtain an \emph{Eisenstein measure} which generalizes an analogous construction in \cite{EFMV18}. The latter is also a generalization of analogous Eisenstein measure in various papers of Eischen and ultimately finds its roots in the seminal work of Katz \cite{Kat78}. Our approach is to $p$-adically interpolate the Fourier coefficient of Eisenstein series. Inspired by the computations presented in \cite{Eis15}, the third main accomplishment of this paper is the explicit computation of local Fourier coefficients at $p$ of our Siegel Eisenstein series, generalizing one of the main results in \emph{loc.cit.} Once more, the use of matrix coefficients leads to simple formulae for local Fourier coefficients.

Lastly, in Part \ref{part:padic Lfct on Pord families}, we reinterpret the Eisenstein measure constructed in previous sections as an element $\mathcal{L}$ of a certain Hecke algebra tensored with an Iwasawa algebra (related to the cyclotomic variable). This follows an approach, adapted to the $P$-ordinary setting, parallel to the one discussed in \cite[Section 7.4]{EHLS}. We interpret the results of previous sections algebraically to interpolate special values of standard $L$-functions as the evaluation of $\mathcal{L}$ at classical points of a $P$-ordinary Hida family.

\subsubsection*{Main result.} 
The main result of this paper is Theorem \ref{thm:main thm} which can be summarized as follows.

\begin{theorem*}
    For a general unitary group $G$ associated to a Hermitian vector space over a CM field $\KK$, fix a parabolic subgroup $P$ of $G(\ZZ_p)$ as in Section \ref{subsubsec:para subgp of G over Zp}. 
    
    Let $\pi$ be an (anti-)holomorphic, $P$-(anti-)ordinary cuspidal automorphic form on a general unitary group $G(\AA)$. Let $(\kappa, K_r, \tau)$ be its $P$-(anti-)weight-level-type, where $\tau$ is a certain (fixed) Schneider-Zink type of $\pi$. Assume that the ``standard conjectures of $P$-ordinary Hida theory'' hold and that $\pi$ satisfy various other standard hypotheses discussed in Sections \ref{sec:structure at p of P(a-)ord aut reps} and \ref{sec:P-(anti-)ord Hida families}.
    
    Let $\TT = \TT_{\pi, [\kappa, \tau]}$ be the $P$-ordinary Hecke algebra associated to $\pi$ as in Section \ref{subsubsec:indep of weights}, which only depends on $\kappa$ and $\tau$ up to ``$P$-parallel shifts'' as discussed in Sections \ref{subsubsec:P parallel weights} and \ref{subsubsec:P nebentypus} respectively. 

    Let $\Lambda_{X_p}$ denote the $\ZZ_p$-Iwasawa algebra of the ray class group $X_p$ of conductor $p^\infty$ over $\KK$.
    
    Given test vectors $\varphi \in \wh{I}_\pi$, $\varphi^\flat \in \wh{I}_{\pi^\flat}$ as in Section \ref{subsubsec:I pi and test vectors}, there exists a unique element
    \[
        L(\Eis^{[\kappa, \tau]}, \Pord; \varphi \otimes \varphi^\flat) 
            \in 
        \Lambda_{X_p, R} \wh{\otimes} \TT_{\pi}
    \]
    satisfying the following property :

    Let $\chi = ||\cdot||^{\frac{n-k}{2}}\chi_u : X_p \to R^\times$ be the $p$-adic shift of a Hecke character as in Section \ref{subsubsec:padic shifts of Hecke char}. Let $\pi' \in \SSS(K^p, \pi)$ be a classical point of the $P$-ordinary Hida family $\TT_\pi$ as in Section \ref{subsubsec:classical points of P-anti-ordinary families}. Let $\lambda_{\pi'}$ be the Hecke character of $\TT$ associated to $\pi'$ as in Sections \ref{subsec:lattices of Pord holo forms}--\ref{subsec:lattices of aholo Paord forms}.

    Then, $L(\Eis^{[\kappa, \tau]}, \Pord; \varphi \otimes \varphi^\flat)$ is mapped under the character $\chi \otimes \lambda_{\pi'}$ to
    \begin{align*}
        c(\pi',\chi)
        \Omega_{\pi', \chi}&(\varphi, \varphi^\flat)
        L_p
        \left(
            \frac{k-n+1}{2}, 
            \Pord,
            \pi',
            \chi_u
        \right) 
            \\ \times\,\,\,
        &L_\infty\left(
            \frac{k-n+1}{2};
            \chi_u, \kappa'
        \right)
        I_S
        \frac{
            L^S(
                \frac{k-n+1}{2}, 
                \pi', 
                \chi_u
            )
        }{
            P_{\pi', \chi}
        }\,,
    \end{align*}
    where $P_{\pi', \chi} = Q_{\pi', \chi}^{-1}$. Here, $c(\pi', \chi)$, $\Omega_{\pi', \chi}$ and $Q_{\pi', \chi}$ are algebraic numbers related to periods and congruence ideals associated to $\pi$ discussed in Section \ref{subsec:periods and cong ideals}. Furthermore, $L_p$, $L_\infty$, and $L^S$ are various Euler factors associated to standard $L$-functions discussed in Section \ref{sec:doubling method integral}. An explicit formula for $L_p$, one of the main accomplishment of this paper, is given in Theorems \ref{thm:main thm - comp of Zw above p}--\ref{thm:main local thm}. On the other hand, $I_S$ is a constant volume factor fixed to ``simplify'' the theory at ramified places.
\end{theorem*}

\subsubsection*{Additional comments.}

As mentioned above, the author plans to establish the necessary results of $P$-ordinary theory on unitary groups in a subsequent paper. Similar results have been obtained and used when working with symplectic groups, for instance see \cite{Pil12} and \cite{LiuRos20}. However, in both cases, their work only considers a version of $P$-ordinary representations where all types involved are 1-dimensional. More precisely, the pro-$p$-Iwahori subgroup considered are larger than the one involved in this paper. Nonetheless, the geometry of the Igusa tower (and the relevant vector bundles) is relatively unaffected by the dimension of the types involved. Therefore, the author plans to adapt the proofs of \cite{Pil12} to unitary groups to prove the conjectures discussed in Section \ref{sec:Pord padic mod forms}.

\subsubsection*{Acknowledgments} 
I thank my advisor Michael Harris who suggested that I look at the work of \cite{EHLS} and adapt it to the $P$-ordinary setting for my doctoral thesis. His countless insights and comments greatly helped me to obtain the results of this article. His encouragements and endless support helped me tremendously to complete this project. I also thank Ellen Eischen, Zheng Liu, Christopher Skinner and Eric Urban for many helpful conversations about their work and the various complexities related to Hida theory and $p$-adic $L$-functions.

\part{$P$-(anti-)ordinary theory on unitary groups.} \label{part : Pord theory on unitary groups}
\section{Notation and conventions.} \label{sec:notation and conventions}
Let $\bQQ \subset \CC$ be the algebraic closure of $\QQ$ in $\CC$. For any number field $F \subset \bQQ$, let $\Sigma_F$ denote its set of complex embeddings $\hom(F, \CC) = \hom(F, \bQQ)$.

Throughout this article, we fix a CM field $\KK \subset \bQQ$ with ring of integers $\OO = \OO_{\KK}$. Let $\KK^+$ be the maximal real subfield of $\KK$ and denote its ring of integers as $\OO^+ = \OO_{\KK^+}$. Let $c \in \Gal(\KK/\KK^+)$ denote complex conjugation, the unique nontrivial automorphism. Given a place $w$ of $\KK$, we usually denote $c(w)$ as $\bar{w}$.

Given a representation $\rho$ of some group $G$, we always denote its contragredient representation by $\rho^\vee$. We write $\brktdotdot_\rho$ for the tautological pairing between $\rho$ and $\rho^\vee$.

We denote the kernel $\ZZ.(2\pi i) \subset \CC$ of the exponential map $\exp : \CC \to \CC^\times$ by $\ZZ(1)$. Given any commutative ring $R$, we set $R(1) := R \otimes \ZZ(1)$.

Given a map $R \to R'$ of commutative rings and an $R$-module $M$, we write $M_{R'}$ for the base change $M \otimes_R R'$ of $M$ to $R'$. 

Let $M$ be an $R$-module endowed with an action of some group $G$. For any representation $\tau$ of $G$, we denote the $\tau$-isotypic component of $M$ by $M[\tau]$. We say that $\tau$ occurs in $M$ if $M[\tau] \neq 0$.

\subsection{CM types and local places.} \label{subsec:CM types and local places}
Fix an integer prime $p$ that is unramified in $\KK$. Throughout this paper, we assume the following :

\begin{hypothesis}\label{hyp:above p split}
	Each place $v^+$ of $\KK^+$ above $p$ totally splits as $v^+ = w \bar{w}$ in $\KK$, for some place $w$ of $\KK$.
\end{hypothesis}

Fix an algebraic closure $\bQQ_p$ of $\QQ_p$ and an embedding $\incl_p : \bQQ \hookrightarrow \bQQ_p$. Define
\[
	\bZZ_{(p)} = \{z \in \bQQ : \v_p(\incl_p(z)) \geq 0 \} \ ,
\]
where $\v_p$ is the canonical extension to $\bQQ_p$ of the normalized $p$-adic valuation on $\QQ_p$. 

Let $\CC_p$ be the completion of $\bQQ_p$. The map $\incl_p$ yields an isomorphism between its valuation ring $\OO_{\CC_p}$ and the completion of $\bZZ_{(p)}$ which extends to an isomorphism $\i : \CC \xrightarrow{\sim} \CC_p$. In particular, $\CC$ is viewed as an algebra over $\ZZ_p$ (or even $\bZZ_{(p)}$) via $\iota$.

Fix an embedding $\i_\infty : \bQQ \hookrightarrow \CC$ such that $\incl_p = \i \circ \i_\infty$. Identify $\bQQ$ with its images $\i_\infty(\bQQ) \subset \CC$ and $\incl_p(\bQQ) \subset \CC_p$.

Given $\sigma \in \Sigma_{\KK}$, the embedding $\incl_p \circ \sigma$ determines a prime ideal $\p_\sigma$ of $\Sigma_\KK$. There may be several embeddings inducing the same prime ideal. Similarly, given a place $w$ of $\KK$, let $\p_w$ denote the corresponding prime ideal of $\OO$.

Under Hypotesis \ref{hyp:above p split}, for each place $v^+$ of $\KK^+$ above $p$, there are exactly two primes of $\OO$ above $v^+$. Fix a set $\Sigma_p$ containing exactly one of these prime ideals for each place $v^+ \mid p$. Moreover, let 
\begin{equation} \label{eq:def of Sigma wrt Sigma p}
    \Sigma = \{\sigma \in \Sigma_\KK \mid \p_\sigma \in \Sigma_p\}\,,
\end{equation}
a CM type of $\KK$, see \cite[p.202]{Kat78}.

\subsection{Local theory of types for smooth representations.} \label{subsec:local theory of types for smooth reps}
Let $F$ be a non-archimedean local field. Denote its ring of integers by $\OO_F$, and set $G = \GL_n(F)$ and $\mathcal{G} =\GL_n(\OO_F)$.

\subsubsection{Parabolic inductions.} \label{subsubsec:Parabolic inductions}
For any parabolic subgroup $P$ of $G$, let $P^u$ be its unipotent radical and $L = P/P^u$, its Levi factor. Let $\delta_P : P \to \CC^\times$ denote its modulus character.

Recall that $\delta_P$ factors through $L$. Moreover, if $P$ is the standard parabolic subgroup associated to the partition $n = n_1 + \ldots + n_s$, one has
\begin{equation}
	\delta_P(l) = \prod_{k=1, \ldots, s} \absv{\det(l_k)}^{-\sum_{i < k} n_i + \sum_{j > k} n_j}
\end{equation}
for any $l = (l_1, \ldots, l_s)$ in $L = \prod_{k=1}^s \GL_{n_k}(F)$.

Given a smooth representation $\sigma$ of $L$, we often consider $\sigma$ as a representation of $P$ without comments. Let $\Ind_P^{G} \sigma $ denote the classical parabolic induction functor from $P$ to $G$. Similarly, we let
\[
    \ind{P}{G} \sigma = \Ind_P^G (\sigma \otimes \delta_P^{1/2})
\]
denote the \emph{normalized} parabolic induction functor.

In our work (especially Sections \ref{sec:structure at p of P(a-)ord aut reps} and \ref{subsec:Local test vectors at p}), we prefer to work with the normalized version but the main calculations of Section \ref{subsec:Local zeta integral at p} can entirely be done with unnormalized parabolic induction as well.

\subsubsection{Supercuspidal support} \label{subsubsec:supercuspidal support}
A theorem of Jacquet (see \cite[Theorem 5.1.2]{Cas95}) implies that given any irreducible representation $\pi$ of $G$, one may find a parabolic subgroup $P$ of $G$ with Levi subgroup $L$ and a supercuspidal representation $\sigma$ of $L$ such that $\pi \subset \ind{P}{G} \sigma$.

The pair $(L, \sigma)$ is uniquely determined by $\pi$, up to $G$-conjugacy and one refers to this conjugacy class as the \emph{supercuspidal support} of $\pi$.

Consider two pairs $(L, \sigma)$ and $(L', \sigma')$ consisting of a Levi subgroup of $G$ and one of its supercuspidal representation. One says that they are \emph{$G$-inertially equivalent} if there exists some $g \in G$ such that $L' = g^{-1}Lg$ and some unramified character $\psi$ of $L'$ such that $\prescript{g}{}{\sigma} \cong \sigma' \otimes \psi$, where $\prescript{g}{}{\sigma}(x) = \sigma(gxg^{-1})$. We write $[L, \sigma]_G$ for the $G$-inertial equivalence class of $(L, \sigma)$.

For such an equivalence class $\s$, let $\Rep^\s(G)$ denote the full subcategory of $\Rep(G)$ whose objects are the representations such that all their irreducible subquotients have inertial equivalence class $\s$. The Bernstein-Zelevinsky geometric lemma, see \cite[subsection VI.5.1]{Ren10}, implies that $\ind{P}{G} \sigma \in \Rep^\s(G)$, where $\s = [L, \sigma]_G$.

\begin{definition}[\cite{BusKut98}] \label{def:s types}
    Let $J$ be a compact open subgroup of $G$ and $\tau$ be an irreducible represention of $J$. Let $\Rep_\tau(G)$ denote the full subcategory of $\Rep(G)$ whose objects are the representations generated over $G$ by their $\tau$-isotypic subspace. We say that $(J, \tau)$ is an $\s$-type if $\Rep_\tau(G) = \Rep^\s(G)$. 
\end{definition}

The work of Bushnell-Kutzko in \cite{BusKut99} constructs a type for every supercuspidal support. In fact, \cite{BusKut98} and \cite{BusKut99} establish the core theory of using \emph{types} to study the category of smooth complex representations of $G$. However, the fact that the compact group $J$ acting on a given type need not be maximal is inconvenient for the calculus in Section \ref{subsec:Local zeta integral at p}. Therefore, we prefer to work with Schneider-Zink types, which are refinements of Bushnell-Kutzko types, see \cite{SchZin99}.

\subsubsection{Schneider-Zink types} \label{subsubsec:SZ types}
Using the local Langlands correspondence, the types introduced by Schneider and Zink refines the ones of Bushnell-Kutzko by also studying the monodromy and the associated Weil-Deligne representations of a given smooth representation of $G$. 

Although we do not need the full depth of this point of view for our purposes, we use \cite[Theorem 6.5.3]{BC09} which imply that for each admissible irreducible representation $\sigma$ of $G$, there exists a smooth irreducible representation $\tau$ of $\mathcal{G}$ such that $\tau$ has multiplicity one in $\sigma|_{\mathcal{G}}$. The other properties of $\tau$ provided by \cite[Theorem 6.5.3]{BC09} (see also \cite[Theorem 2.5.4]{HLLM23}) play no role in our work and we omit them.

\begin{remark}
    We later use types (or more precisely, their inertial equivalence class) to construct ``branches $P$-ordinary Hida families'' associated to some particular automorphic representations, see Definition \ref{def:P-anti-ord Hida family of pi}. As mentioned above, we strictly use their multiplicity one property. However, it could be interesting to see how the additional properties of these types can be used to study these Hida families.
\end{remark}

\begin{remark} \label{rmk:SZ types as irred comp of BK types}
    These Schneider-Zink types are essentially constructed by studying irreducible components of $\Ind_J^{\mathcal{G}} \tau'$, where $(J, \tau')$ is some Bushnell-Kutzko type for the supercuspidal support of $\sigma$.
\end{remark}

Note that the above does not mention anything about the uniqueness of such a representation $\tau$ of $\mathcal{G}$. Therefore, for later purposes, we fix a choice of such a representation $\tau = \tau_\sigma$ for each $\sigma$ and refer to it as our \emph{fixed choice of Schneider-Zink type for $\sigma$}. We also say that $\tau$ is the (chosen) \emph{SZ-type} of $\sigma$.

\begin{remark} \label{rmk:SZ types of twists and contragredient}
    We choose them compatibly so that for given an unramified character $\psi$ of $G$, the SZ-types of $\sigma$ and $\sigma \otimes \psi$ satisfy $\tau_{\sigma \otimes \psi} = \tau_{\sigma} \otimes \psi$. We also choose them so that $\tau_{\sigma^\vee} = (\tau_\sigma)^\vee$.
\end{remark}

%%%%%%%%%%%%%%%%%%%%%%%%%%%%%%%%%%%%%%%%%%%%%%%%%%%%%%%%%%%%%%%%%%

\section{Modular forms on unitary groups with $P$-Iwahoric level at $p$.}  \label{sec:PIwa subgps of unitary gps}
Let $V$ be an $n$-dimensional $\KK$-vector space, equipped with a non-degenerate Hermitian pairing $\brktdotdot_{V}$ with respect to the quadratic imaginary extension $\KK/\KK^+$ fixed in the previous section.

\subsection{Unitary PEL datum.} \label{subsec:unitary pel datum}
Let $\delta \in \OO$ be totally imaginary and prime to $p$. Define $\brktdotdot = \tr_{\KK / \QQ}(\delta \brktdotdot_{V})$. This choice of $\delta$ and our Hypothesis \ref{hyp:above p split} ensure the existence of an $\OO$-lattice $L \subset V$ such that the restriction of $\brktdotdot$ to $L$ is integral and yields a perfect pairing on $L \otimes \ZZ_p$.

For each $\sigma \in \Sigma_{\KK}$, let $V_{\sigma}$ denote $V \otimes_{\KK, \sigma} \CC$. Fix a $\CC$-basis diagonalizing the pairing $\brktdotdot$. We assume that the basis is chosen so that the corresponding diagonal matrix is $\diag(1, \ldots, 1, -1, \ldots, -1)$ with $a_\sigma$ entries equal to $1$ and $b_\sigma = n - a_\sigma$ entries equal to $-1$. Fixing such a basis, let $h_{\sigma} : \CC \to \End_{\RR}(V_{\sigma})$ be $h_{\sigma} = \diag(z 1_{a_{\sigma}}, \bar{z} 1_{b_{\sigma}})$.

Let $h = \prod_{\sigma \in \Sigma} h_{\sigma} : \CC \to \End_{\KK^+ \otimes \RR}(V \otimes \RR)$, using the canonical identification
\[
    \prod_{\sigma \in \Sigma}
        \End_{\RR}(V_{\sigma}) 
    = 
        \End_{\KK^+ \otimes \RR}(V \otimes \RR)
\]
provided by our fixed choice of CM type $\Sigma$ of $\KK$. The \emph{signature} of $h$ is defined as the collection of pairs $\{(a_\sigma, b_\sigma)\}_{\sigma \in \Sigma_K}$.

The signature of $h$ is naturally related to the pure Hodge structure of weight $-1$ on $V_\CC = L \otimes \CC$ determined by $h$. Namely, we have $V_\CC = V^{-1, 0} \oplus V^{0,-1}$ where $h(z)$ acts as $z$ on $V^{-1,0}$ and as $\bar{z}$ on $V^{0,-1}$. By definition, the complex dimension of $V^{-1,0} \otimes_{\OO \otimes \CC, \sigma} \CC$ is equal to $a_\sigma$ if $\sigma \in \Sigma$ and $b_{\sigma}$ if $\sigma \in \Sigma_K \setminus \Sigma$. 

Throughout this paper, we assume the following two hypothesis : 
\begin{hypothesis}[Standard hypothesis] \label{hyp:standard hypothesis}
    We assume that $h$ is standard, as defined in \cite[Section 2.3.2]{EHLS}. Namely, there is a $\KK$-basis of $V$ that simultaneously diagonalizes the matrix associated to $\brktdotdot_V$ as well as the image of $h_\sigma$ (with respect to the induced basis of $V \otimes_{\KK, \sigma} \CC$), for each $\sigma \in \Sigma$ .
\end{hypothesis}

\begin{hypothesis}[Ordinary hypothesis] \label{hyp:ordinary hypothesis}
	For all embeddings $\sigma, \sigma' \in \Sigma_{\KK}$, if $\p_{\sigma} = \p_{\sigma'}$, then $a_{\sigma} = a_{\sigma'}$.
\end{hypothesis}

Using the second hypothesis, given a place $w$ of $\KK$ above $p$, we can define $(a_{w}, b_{w}) := (a_{\sigma}, b_{\sigma})$, where $\sigma \in \Sigma_\KK$ is any embedding such that $\p_\sigma = \p_w$.

The tuple 
\[
    \PP = 
        (
            \KK, c, \OO, L, 
            2\pi\sqrt{-1}\brktdotdot, h
        )
\]
is a PEL datum of unitary type, as defined in \cite[Section 2.1-2.2]{EHLS}. One can associate a group scheme $G = G_\PP$ over $\ZZ$ to $\PP$ whose $R$-points are
\begin{equation} \label{eq:def sim uni gp}
	G(R) = \{ (g, \nu) \in \GL_{\OO \otimes R}(L \otimes R) \times R^\times \mid \brkt{gx}{gy} = \nu \brkt{x}{y}, \forall x, y \in L \otimes R \},
\end{equation}
for any commutative ring $R$. We define the signature of $G$ as the signature of the underlying homomorphism $h$.

Note that $G_{/\QQ}$ is a reductive group. Moreover, our assumptions on $p$ imply that $G_{/\ZZ_p}$ is smooth and $G(\ZZ_p)$ is a hyperspecial maximal compact of $G(\QQ_p)$. 

\begin{remark} \label{rmk:PEL datum for many Hermitian spaces}
    Here and in what follows, we only introduce the relevant theory for a PEL datum $\PP$ as above associated to a single Hermitian vector space. In later sections, we also need to consider more general PEL data (and the associated objects) obtained from a pair of Hermitian vector spaces, see $\PP_3$ in Section \ref{subsec:G1 G2 G3 G4}. The necessary modifications to construct the relevant objects for such PEL data are obvious, hence we do not address them explicitly to lighten our notation. See \cite[Section 2]{EHLS} for precise details on the theory of unitary PEL data associated to any (finite) number of Hermitian vectors spaces over $\KK$.
\end{remark}

\subsubsection{Unitary moduli spaces} \label{subsubsec:unitary moduli space}
Let $F = F_\PP$ be the reflex field of the PEL datum $\PP$ introduced above, as defined in \cite[1.2.5.4]{Lan13}. Let $\OO_F$ be its ring of integers and let $S_p = \OO_F \otimes \ZZ_{(p)}$.

\begin{remark} \label{rmk:integral away from p models}
    In later sections, we work with models of Shimura varieties that are integral away-from-$p$. The existence of such integral models over $S_p$ is obtained by restricting our attention to level subgroups at $p$ satisfying certain conditions, see Sections \ref{subsec:Shi var of P Iwa level at p} and \ref{subsec:P nbtp thy of mod forms}. 
    
    One may remove these conditions and consider more general level subgroups by working over $F$ instead. For more details about the differences (and similarities) between working over $S_p$ or $F$, see \cite[Section 2]{EHLS}. We prefer (and need) to work with models over $S_p$ as we only work with with level subgroups at $p$ satisfying the conditions briefly mentioned above.
\end{remark}

Let $K^p \subset G(\AA_f^p)$ be any open compact subgroup and set $K = G(\ZZ_p)K^p$. Define the moduli problem $\MM_{K} = \MM_{K}(\PP)$ as the functor that assigns, to any locally noetherian $S_p$-scheme $T$, the set of equivalence classes of quadruples $\ul{A} = (A, \lambda, \iota, \alpha K^p)$, where
\begin{enumerate}
    \item $A$ is an abelian scheme over $T$;
    \item $\lambda : A \to A^\vee$ is a prime-to-$p$ polarization;
    \item $\iota : S_p \hookrightarrow \End_T A \otimes \ZZ_{(p)}$ such that $\iota(b)^\vee \circ \lambda = \lambda^\vee \circ \iota(\ol{b})$;
    \item $\alpha K^p$ is a $K^p$-level structure, in the sense of \cite[Section 2.1]{EHLS}. Namely, $\alpha$ is a rule that assigns, to each connected component $T^\circ$ of $T$, an isomorphism
    \[
        \alpha_t 
            : 
        L \otimes \AA_f^p 
            \xrightarrow{\sim}
        H^1(A_t, \AA_f^p)
    \]
    over $\OO_\KK \otimes \AA_f^p$ such that the $K^p$-orbit $\alpha K^p$ is $\pi_1(T, t)$-stable (where $t$ is an arbitrary geometric point of $T^\circ$). Furthermore, it identifies the pairing $\brktdotdot$ with a $\AA_f^{p, \times}$-multiple of the symplectic pairing on $H^1(A_t, \AA_f^p)$ induced by the Weil pairing and the polarization $\lambda$;
    \item $\Lie_T A$ satisfies the Kottwitz determinant condition defined by $(L \otimes R, \brkt{\cdot}{\cdot}, h)$, see \cite[Definition 1.3.4.1]{Lan13};
\end{enumerate}
and two quadruples $(A, \lambda, \iota, \alpha)$ and $(A', \lambda', \iota', \alpha')$ are equivalent if there exists some prime-to-$p$ isogeny $f : A \to A'$ such that
\begin{enumerate}
    \item $\lambda$ and $f^\vee \circ \lambda' \circ f$ are equal, up to multiplication by some positive element in $\ZZ_{(p)}^\times$;
    \item $\iota'(b) \circ f = f \circ \iota(b)$, for all $b \in \OO_\KK$;
    \item $\alpha' K^p = f \circ \alpha K^p$.
\end{enumerate}

When $K^p$ is clear from context, we often denote the orbit $\alpha K^p$ simply by $\alpha$. Furthermore, for a generalization (over $F$ instead of $S_p$) of this moduli problem for all open compact subgroups $K \subset G(\AA_f)$, see \cite[Section 2.1]{EHLS}.

In this article, we always assume that $K$ is \emph{neat}, in the sense of \cite[Definition 1.4.1.8.]{Lan13}. Then, \cite[Corollary 7.2.3.10]{Lan13} implies that there is a smooth, quasi-projective $S_p$-scheme that represents this moduli problem $\MM_K$. By abuse of notation, we denote this scheme by $\MM_K$ again. If $K' = G(\ZZ^p)K'^{,p} \subset K$, there is a natural homomorphism $\MM_{K'} \to \MM_K$ induced by the ``forgetful map'' $\alpha K'^{,p} \mapsto \alpha K^p$. Similarly, given $g \in G(\AA_f^p)$, there is a canonical map $[g] : \MM_{gKg^{-1}} \to \MM_K$ induced by the functor $(A, \lambda, \iota, \alpha) \mapsto (A, \lambda, \iota, \alpha g)$.

\subsubsection{Toroidal compactifications.} \label{subsubsec:toroidal compactification}
We now briefly recall the existence of toroidal compactifications of the moduli spaces above constructed in \cite{Lan13}. These are associated to \emph{smooth projective polyhedral cone decompostions}, a notion whose exact definition plays no role later in this article. Hence, we do not introduce this notion precisely. 

The only properties relevant for this paper are that given such a polyhedral cone decomposition $\Omega$, there exists a smooth toroidal compactification $\MM^\tor_{K, \Omega}$ of $\MM_K$ over $S_p$, and that there exists a partial ordering on the set of such $\Omega$'s by \emph{refinements}. 

Given two polyhedral cone decompositions $\Omega$ and $\Omega'$, if $\Omega'$ refines $\Omega$, then there is a canonical proper surjective map $\pi_{\Omega', \Omega} : \MM_{K, \Omega'}^\tor \to \MM_{K, \Omega}^\tor$ which restricts to the identity on $\MM_K$. We denote the tower $\{\MM_{K, \Omega}^\tor\}_\Omega$ by $\MM_K^\tor$. 

\begin{remark}
    We often refer to the tower as if it were a single scheme and do not emphasize the specific compatible choices of $\Omega$ in some constructions. This is essentially justified by the K\"oecher's principle in many cases, see Remark \ref{rmk:Koecher principle}. See \cite[Section 2.4]{EHLS} for more details.
\end{remark}

Furthermore, if $K' \subset K$, the map $\MM_K \to \MM_{K'}$ extends canonically to maps $\MM_{K, \Omega}^\tor \to \MM_{K', \Omega}^\tor$, for each $\Omega$, and hence to a map $\MM_K^\tor \to \MM_{K'}^\tor$. Similarly, the maps $[g] : \MM_{gKg^{-1}} \to \MM_K$ also extend canonically to maps $[g] : \MM_{gKg^{-1}}^\tor \to \MM_K^\tor$, for all $g \in G(\AA_f^p)$. Hence, $G(\AA_f^p)$ acts on the tower (of towers) $\{\MM_{G(\ZZ_p)K^p}^\tor\}_{K^p \subset G(\AA_f^p)}$.

\subsection{Structure of $G$ over $\ZZ_p$.} \label{subsec:structure of G over Zp}
\subsubsection{Comparison to general linear groups.} \label{subsubsec:comp to gen linear groups}
For each prime $w \mid p$ of $\KK$, denote the localization of $\KK$ at $w$ by $\KK_w$ and its ring of integers by $\OO_w$.

The factorization $\OO \otimes \ZZ_p = \prod_{w \mid p} \OO_{w}$, over primes $w \mid p$, yields a decomposition $L \otimes \ZZ_p = \prod_{w \mid p} L_w$. Using Hypothesis \ref{hyp:above p split}, we fix identifications $\KK_w = \KK_{\bar{w}}$ and $\OO_w = \OO_{\bar{w}}$. We consider both $L_w$ and $L_{\bar{w}}$ as $\OO_w$-lattices. 

The above factorization of $L \otimes \ZZ_p$ corresponds to 
\begin{equation} \label{eq:prod GL(L otimes Zp)}
    \GL_{
        \OO \otimes \ZZ_p
    }(L \otimes \ZZ_p) 
        \xrightarrow{\sim} 
    \prod_{w \mid p} 
        \GL_{\OO_w}(L_w), 
            \ \ \ \ g \mapsto (g_{w})_{w \mid p} \,,
\end{equation}
a canonical $\ZZ_p$-isomorphism. From the above, one obtains the identification
\begin{equation} \label{eq:prod G over Zp}
    G_{/\ZZ_p} 
        \xrightarrow{\sim} 
    \Gm \times 
    \prod_{w \in \Sigma_p} 
        \GL_{\OO_w}(L_w), 
            \ \ \ \ (g, \nu) \mapsto (\nu, (g_{w})_{w \in \Sigma_p}) \,.
\end{equation}

Furthermore, our assumption above on the pairing $\brkt{\cdot}{\cdot}$ implies that for each $w \mid p$, there is an $\OO_w$-decomposition of $L_w = L_w^+ \oplus L_w^-$ such that 
\begin{enumerate}
	\item $\rk_{\OO_w}{L_w^+} = a_{w}$ and $\rk_{\OO_w}{L_w^-} = b_{w}$;
	\item Upon restricting $\brkt{\cdot}{\cdot}$ to $L_w \times L_{\ol{w}}$, the annihilator of $L_w^\pm$ is $L_{\ol{w}}^{\pm}$. Hence, one has a perfect pairing $L_w^+ \oplus L_{\ol{w}}^- \to \ZZ_p(1)$, again denoted $\brktdotdot$.
\end{enumerate}

Fix dual $\OO_w$-bases (with respect to the perfect pairing above) for $L_w^+$ and $L_{\ol{w}}^-$. They yield isomorphisms
\begin{equation} \label{eq:GL(Lw pm) basis}
    \begin{tikzcd}
            \GL_{a_w}(\OO_w) 
                \arrow[r, "\sim"] 
        &
            \GL_{\OO_w}(L_w^+) 
                \arrow[r, "dual"] 
        &
            \GL_{\OO_w}(L_{\ol{w}}^-) 
                \arrow[r, "\sim"] 
        &
            \GL_{b_{\ol{w}}}(\OO_{w})
    \end{tikzcd}
\end{equation}
such that the composition is the adjoint map $A \mapsto A^* = \tp{\ol{A}}$ on $\GL_{a_w}(\OO_w) = \GL_{b_{\ol{w}}}(\OO_w)$. Furthermore, this induces an identification $\GL_{\OO_w}(L_w) = \GL_{n}(\OO_w)$ such that the obvious map
\begin{equation}\label{eq:GL(Lw) basis}
    \GL_{\OO_w}(L_w^+) 
        \times 
    \GL_{\OO_w}(L_w^-) 
        \hookrightarrow 
    \GL_{\OO_w}(L_w)
\end{equation}
is simply the diagonal embedding of block matrices.

Let $L^\pm = \prod_{w \mid p} L_w^\pm$ and let $H := \GL_{\OO \otimes \ZZ_p}(L^+)$. The identification \eqref{eq:GL(Lw pm) basis} above induces a canonical isomorphism
\begin{equation} \label{eq:def H}
    H \cong
    \prod_{w \mid p}
        \GL_{a_{w}}(\OO_w) = 
    \prod_{w \in \Sigma_p}
        \GL_{a_{w}}(\OO_w) \times 
        \GL_{b_{w}}(\OO_w)
\end{equation}

\begin{remark} \label{rmk:notation conv for algebraic groups}
    Here, we view $H$ as an algebraic group over $\OO \otimes \ZZ_p$. Namely, for any algebra $S$ over $\OO \otimes \ZZ_p$, we have $H(S) = \GL_S(L^+ \otimes_{\OO \otimes \ZZ_p} S)$. This technically leads to the confusion in notation since $H(\OO \otimes \ZZ_p)$ is equal to the set $\GL_{\OO \otimes \ZZ_p}(L^+)$ (also denoted $H$ above). However, we keep this convention of denoting an algebraic group by its set of points over its base ring, ignoring this minor abuse in notation.

    For instance, the algebraic group denoted $\GL_{a_w}(\OO_w)$ above technically stands for $\GL(a_w)_{/\OO_w}$. We use such a convention in many instance in what follows without comments. The only exception is for $\Gm$ which we refrain from denoting $\GL_1(\OO_w)$ or $\OO_w^\times$.
\end{remark}

\subsubsection{Parabolic subgroups of $G$ over $\ZZ_p$.} \label{subsubsec:para subgp of G over Zp}
For $w \mid p$, let 
\begin{equation} \label{eq:def partitions dw}
	\d_{w} = \left( n_{w,1}, \ldots, n_{w, t_{w}} \right)
\end{equation}
be a partition of $a_{w} = b_{\ol{w}}$. Let $P_{\d_{w}} \subset \GL_{a_{w}}(\OO_w)$ denote the standard parabolic subgroup corresponding to $\d_{w}$. Define $P_H \subset H$ as the $\ZZ_p$-parabolic that corresponds to the products of all the $P_{\d_{w}}$ via the isomorphism \eqref{eq:def H}. We denote the unipotent radical of $P_H$ by $P_H^u$ and its maximal subtorus by $T_H$.

We identify the elements of the Levi factor $L_H = P_H / P_H^u$ of $P_H$ with collections of block-diagonal matrices, with respect to the partitions $\d_{w}$, via \eqref{eq:def H}. In other words, we embed $L_{\d_w} := \GL_{n_{w,1}}(\OO_w) \times \ldots \GL_{n_{w,t}}(\OO_w)$ in $\GL_{a_w}(\OO_w)$ diagonally and identify $L_H$ with $\prod_{w \mid p} L_{\d_w}$. 

Define $\det_{\d_w} : L_{\d_w} \to (\Gm)^{t_w}$ as the homomorphism taking determinant of each $\GL$-block of $L_{\d_w}$ individually (in the obvious order). Let $SL_{\d_w} \subset L_{\d_w}$ denote the kernel of $\det_{\d_w}$ and identify $SL_H = \prod_{w \mid p} \SL_{\d_w}$ as a subgroup of $H$ via \eqref{eq:def H}. We define $SP_H$ as the product $SL_H \cdot P_H^u$ in $P_H$ and identify $P_H / SP_H$ with $\prod_{w \mid p} (\Gm)^{t_w}$. 

Note that the center $Z_{L_H}$ of $L_H$ is also canonically isomorphic to $\prod_{w \mid p} (\Gm)^{t_w}$. The identity map between these two copies of $\prod_{w \mid p} (\Gm)^{t_w}$ yields an identification that sends an element $g = (g_w)_{w \mid p} \in P_H/SP_H$ such that $\det_{\d_w}(g_w) = (g_{w,1}, \ldots, g_{w,t_w})$ with
\[
    (
        \diag(
            g_{w,1}, \ldots, g_{w,1};
            g_{w,2}, \ldots, g_{w,2};
            \ldots;
            g_{w,t_w}, \ldots, g_{w,t_w}
        )
    )_{w \in \mid p}
        \in
    Z_{L_H}\,,
\]
where the entry $g_{w,i}$ appears $n_{w,i}$-times.

\begin{remark} \label{rmk:char of Z and P/SP}
    We use this identification later to view a character $\chi$ of $Z_{L_H}$ as a character of $P_H$ that factors through $\prod_{w \mid p} \det_{\d_w}$. 
    
    We can write such a character $\chi$ as a product $\prod_{w \mid p} \chi_w$ via the canonical identification $Z_{L_H} = \prod_{w \mid p} (\Gm)^{t_w}$. Then, the corresponding character $\chi'$ of $P_H/SP_H = \prod_{w \mid p} (\Gm)^{t_w}$ is
    \[
        \chi' = \prod_{w \mid p} \chi_w \circ \det_{\d_w}\,.
    \]

    In particular, the reader should keep in mind that the restriction of $\chi'$ to $Z_{L_H}$ is \emph{not} $\chi$. Nonetheless, by abuse of notation, we often denote $\chi'$ as $\chi$ again. We remind the reader of this convention when necessary to avoid confusion.
\end{remark}

Let $P^+ \subset G_{/\ZZ_p}$ be the parabolic subgroup that stabilizes $L^+$ and such that
\begin{equation} \label{eq:def P+}
    P^+ 
        \twoheadrightarrow 
    \Gm \times P_H \subset \Gm \times H
\end{equation}
is surjective, where the map to the first factor is the similitude character $\nu$ and the map to the second factor is projection to $H$. 

For $w \in \Sigma_p$, let $P_{w}$ be the parabolic subgroup of $\GL_{\OO_w}(L_w)$ given by
\begin{equation} \label{eq:def of local parabolic Pw}
    P_w = 
    \left\{
        \begin{pmatrix}
		A & B \\
		0 & D
        \end{pmatrix} 
            \in \GL_{n}(\OO_w) \mid
                A \in P_{\d_w}, 
                D \in P^{\opp}_{\d_{\ol{w}}}
    \right\} \ ,
\end{equation}
via the isomorphisms \eqref{eq:GL(Lw pm) basis} and \eqref{eq:GL(Lw) basis}.

We identify $P = \prod_{w \in \Sigma_p} P_{w}$ as a subgroup of $G_{/\ZZ_p}$ via \eqref{eq:prod G over Zp}. Our choices of bases above imply that under the isomorphisms \eqref{eq:prod G over Zp} and \eqref{eq:GL(Lw pm) basis}, $P^+$ corresponds to
\begin{equation} \label{eq:iso P+ with Gm x P}
	P^+ \xlongrightarrow{\sim} \Gm \times P\,.
\end{equation}

This induces an isomorphism $L_H \cong L_P := P/P^u$, where $P^u$ is the unipotent radical of $P$. We again identify $L_P$ as the subgroup of $P$ consisting of collections of block-diagonal matrices (the sizes of the blocks are determined by the partitions $\d_w$). 

Let $SL_P \subset L_P$ be the subgroup corresponding to $SL_H$ via this isomorphism $L_H \cong L_P$ and let $SP = SL_P \cdot P$. Proceeding as above, we obtain a natural identification between the center $Z_{L_P}$ of $L_P$ and the quotient $P/SP$.

\begin{remark} \label{rmk:trivial partition}
    The trivial partition of $a_{w}$ is $(1, \ldots, 1)$ (of length $t_{w} = a_{w}$). If the partitions fixed above are all trivial, we write $B_w$, $B$ and $B^+$ instead of $P_w$, $P$ and $P^+$. In this case, $L_B = B/B^u$ is equal to $Z_{L_B}$ and identified with the maximal torus subgroup of $\prod_{w \in \Sigma_p} \GL_n(\OO_w)$
\end{remark}

\begin{definition} \label{def:def of PIwahori of level r}
    We define the $P$-Iwahori subgroup of $G$ of level $r \geq 0$ as
    \[
        I_r^0 = I_{P,r}^0 :=
        \left\{
            g \in G(\ZZ_p) \mid 
                g \text{ mod } p^r 
                \in P^+(\ZZ_p/p^r \ZZ_p)
        \right\}
    \]
    and the pro-$p$ $P$-Iwahori subgroup $I_r = I_{P,r}$ of $G$ of level $r$ as 
    \[
        I_r = I_{P,r} :=
        \left\{
            g \in G(\ZZ_p) \mid 
                g \text{ mod } p^r 
                \in (\ZZ_p/p^r\ZZ_p)^\times 
                \times P^u(\ZZ_p/p^r \ZZ_p)
        \right\}.
    \]
\end{definition}

\begin{remark} \label{rmk:not parahoric}
    We refrain from referring to $I_r^0$ as a \emph{parahoric} subgroup of $G$. This terminology is usually reserved for stabilizers of points in Bruhat-Tits building. We make no attempt here to introduce our construction from the point of view of these combinatorial and geometric structures.
\end{remark}

The inclusion of $L_P(\ZZ_p)$ in $I_r^0$ yields a canonical isomorphism 
\begin{equation} \label{eq:iso LP mod pr with Ir0 mod Ir}
    L_P(\ZZ_p/p^r\ZZ_p) \xrightarrow{\sim} I_r^0/I_r \ .
\end{equation}

For each $w \in \Sigma_p$, one similarly defines $I_{w, r}^0$ and $I_{w, r}$ by replacing $P^+$ by $P_{w}$ and working in $\GL_{n}(\OO_w)$ instead of $G(\ZZ_p)$. Let
\begin{equation} \label{eq:facto PIwahori over Sigma p}
    I_r^{\GL} = 
    \prod_{w \in \Sigma_p}
        I_{w, r} 
        \ \ \ \text{and} \ \ \ 
    I_r^{0,\GL} = 
    \prod_{w \in \Sigma_p} 
        I_{w, r}^0 \,,
\end{equation}
so that $I_{r}$ and $I_{r}^0$ correspond to $\ZZ_p^\times \times I_{P,r}^{\GL}$ and $\ZZ_p^\times \times I_{P,r}^{0, \GL}$ respectively, via the isomorphisms \eqref{eq:prod G over Zp} and \eqref{eq:GL(Lw pm) basis}.

In subsequent sections, we study certain $P$-ordinary Hecke operators at $p$ associated to the parabolic subgroups introduced above. Therefore, for later purposes, let us define the following matrices :

Given $w \in \Sigma_p$ and $1 \leq j \leq n$, let $t_{w,j} \in \GL_{n}(\OO_w)$ denote the diagonal matrix
\begin{equation} \label{eq:def t w j}
    t_{w,j} = 
    \begin{cases}
        \diag(p1_j, 1_{n - j}), &
        \text{if } j \leq a_{w} \\
        \diag(p1_{a_{w}}, 1_{n - j}, p1_{j - a_{w}}), &
        \text{if } j > a_{w}
    \end{cases}
\end{equation}

It corresponds to an element of $G(\QQ_p)$ under \eqref{eq:prod G over Zp} and \eqref{eq:GL(Lw) basis}, which we denote $t^+_{w,j}$ (namely, all its other components are equal to 1). We set $t^-_{w,j} = (t^+_{w,j})^{-1}$.

Furthermore, let $r_w = t_w + t_{\ol{w}}$ and consider
\[
    \widetilde{\d}_w = 
    \left( 
        \widetilde{\d}_{w,1}, 
        \ldots,
        \widetilde{\d}_{w, t_w};
        \widetilde{\d}_{w,t_w+1}, 
        \ldots,
        \widetilde{\d}_{w, r_w}
    \right) :=
    \left( 
        n_{w, 1}, 
        \ldots, 
        n_{w, t_w} ; 
        n_{\ol{w}, t_{\ol{w}}}, 
        \ldots, 
        n_{\ol{w}, 1} 
    \right)\,,
\]
a partition of $n = a_w + b_w$. For $j=1, \ldots, r_w$, let $D_w(j)$ be the partial sum $\sum_{i=1}^j \widetilde{\d}_{w,i}$. We define
\begin{equation} \label{eq:definition of tP pm}
    t^\pm_{P, p} =
    \prod_{w \in \Sigma_p}
    \prod_{j = 1}^{r_w}
        t^\pm_{w,D_w(j)}
\end{equation}

By construction, $t_{P,p}^{\pm}$ lies in the center $Z_{L_P}(\QQ_p)$ of $L_P(\QQ_p)$.

\begin{remark} \label{rmk:t w vs t w i}
    The reader should not confuse $t_w$ and $t_{w,i}$ (or $t_{w, D_w(j)}$). The former is only ever used to denote an integer while the latter denotes an $n \times n$-matrix over $\OO_w$.
\end{remark}

\subsection{Structure of $G$ over $\CC$.} \label{subsec:structure of G over C}
Consider the pure Hodge decomposition $V_\CC = L \otimes \CC = V^{-1, 0} \oplus V^{0, -1}$ of weight $-1$, as in Section \ref{subsec:unitary pel datum}, for the $\OO$-lattice $L$ associated to $\PP$. By definition of the reflex field of $\PP$, the graded piece $W = V/V^{0,-1}$ of the corresponding Hodge filtration is defined over $F$. 

Fix an $S_p$-submodule $\Lambda_0$ of $W$ that is stable under the $\OO$-action and such that $\Lambda_0 \otimes_{S_p} \CC = W$. The module $\Lambda_0^\vee = \hom_{\ZZ_{(p)}}(\Lambda_0, \ZZ_{(p)}(1))$ has a natural $\OO \otimes S_p$-action via
\[
    (b \otimes s)f(x) = f(\ol{b}sx)\,,
\]
for all $b \in \OO$ and $s \in S_p$.

Define $\Lambda = \Lambda_0 \oplus \Lambda_0^\vee$ and
\begin{align*}
    \brkt{\cdot}{\cdot}_{can} : \Lambda \times \Lambda &\to \ZZ_{(p)}(1) \\
    \brkt{(f_1,x_1)}{(f_2,x_2)}_{can} &= f_2(x_1) - f_1(x_2)
\end{align*}
so that both $\Lambda_0$ and $\Lambda_0^\vee$ are isotropic submodules of $\Lambda$. One has $\brkt{bx}{y}_{can} = \brkt{x}{\ol{b}y}_{can}$, for $b \in \OO$.

The pair $(\Lambda, \brkt{\cdot}{\cdot}_{can})$ induces an $S_p$-group scheme $G_0$ whose $R$-points are given by
\[
    G_0(R) = 
    \left\{ 
        (g, \nu) \in \GL_{R}(\Lambda \otimes_{\Ss} R) \times R^\times \mid \brkt{gx}{gy}_{can} = \nu \brkt{x}{y}_{can}, x, y \in \Lambda \otimes R
    \right\} \ ,
\]
for any $S_p$-algebra $R$. Let $P_0 \subset G_0$ denote the parabolic subgroup that stabilizes $\Lambda_0$.

One readily checks that there is an isomorphism $V \cong \Lambda \otimes_{S_p} \CC$ of $\CC$-vector spaces that identifies $V^{-1,0}$ (resp. $V^{0, -1}$) with $\Lambda_0 \otimes_{S_p} \CC$ (resp. $\Lambda^\vee_0 \otimes_{S_p} \CC$) and the pairing $\brktdotdot$ with $\brktdotdot_{can}$. In other words, it yields an identification between $G_{/\CC}$ and $G_{0/\CC}$. Clearly, it identifies $P_0(\CC)$ with $P_h(\CC)$, where $P_h$ is the stabilizer of the Hodge filtration on $L \otimes \RR$.

\begin{remark} \label{rmk:Lambda0 vs V -1 0}
    The advantage to introduce $\Lambda_0$ is that it is well-defined over $S_p$, as opposed to $V^{-1,0}$. This is necessary to later view classical algebraic weights $p$-adically, see Section \ref{subsubsec:padic weights}.
\end{remark}

Let $H_0 \subset G_0$ be the stabilizer of the polarization $\Lambda = \Lambda_0 \oplus \Lambda_0^\vee$. The natural projection
\begin{equation} \label{eq:iso H0 with Gm x GL(Lambda0vee)}
    H_0 \to \Gm \times \GL_{S_p} (\Lambda_0^\vee)
\end{equation}
is an isomorphism, and the isomorphism between $G_{/\CC}$ and $G_{0/\CC}$ above identifies $H_0(\CC)$ with $C(\CC)$, where $C_{/\RR}$ is the centralizer of $h$ under the conjugation action of $G_{/\RR}$. We recall the classification of the algebraic representations of $H_0$ in the next section to later describe cohomological weights of automorphic representations. 

\subsubsection{Algebraic weights.} \label{subsubsec:alg weights}
Let $\KK'$ be the Galois closure of $\KK$ and $\p' \subset \OO_{\KK'}$ be the prime above $p$ determined by $\incl_p$. From \cite[Corollary 1.2.5.6]{Lan13}, $\KK'$ contains $F$. Therefore, we can view $S_0 := \OO_{\KK', (\p')}$ as an algebra over $S_p = \OO_{F, (p)}$.

By definition of $\KK'$, we have $\OO \otimes S_0 = \prod_{\sigma \in \Sigma_\KK} S_0$. This naturally induces decompositions $\Lambda_0 \otimes S_0 = \prod_{\sigma \in \Sigma_\KK} \Lambda_{0, \sigma}$ and $\Lambda_0^\vee \otimes S_0 = \prod_{\sigma \in \Sigma_\KK} \Lambda^\vee_{0, \sigma}$. Moreover, the identification \eqref{eq:iso H0 with Gm x GL(Lambda0vee)} yields an isomorphism
\begin{equation} \label{eq:prod H0 over S0}
    H_{0 /S_0} 
        \xrightarrow{\sim} 
    \Gm
        \times 
    \prod_{\sigma \in \Sigma_\KK}
        \GL_{\OO \otimes_{\OO, \sigma} S_0}(
            \Lambda_{0,\sigma}^\vee
        )
    \,.
\end{equation}

Since $S_0$ is a PID, one readily sees that $\Lambda_{0, \sigma}$ (resp. $\Lambda_{0,\sigma}^\vee$) is a free $S_0$-module of rank $a_\sigma$ (resp. $b_\sigma$). Furthermore, for each $\sigma \in \Sigma_\KK$, the pairing $\brktdotdot_{can}$ identifies $\Lambda^\vee_{0, \sigma c}$ with $\hom_{\ZZ_{(p)}}(\Lambda_{0, \sigma} ,\ZZ_{(p)}(1))$. Fix dual bases for $\Lambda_{0, \sigma}$ and $\Lambda^\vee_{0, \sigma c}$, so that \eqref{eq:prod H0 over S0} induces an identification
\begin{equation} \label{eq:H0 basis}
    H_{0\,/S_0} \xrightarrow{\sim} 
    \Gm \times 
    \prod_{\sigma \in \Sigma_\KK}
        \GL_{b_\sigma}(S_0) \ .
\end{equation}

Let $B_{H_0} \subset H_{0\,/S_0}$ be the Borel subgroup that corresponds to the product of the lower-triangular Borel subgroups via the isomorphism \eqref{eq:H0 basis}. Let $T_{H_0} \subset B_{H_0}$ denote its maximal subtorus and let $B^u_{H_0}$ denote its unipotent radical subgroup.

Given an $S_0$-algebra $R$, a character $\kappa$ of $T_{H_0}$ over $R$ is identified via the isomorphism \eqref{eq:H0 basis} with a tuple
\[
    \kappa = 
        (
            \kappa_0, 
            (
                \kappa_{\sigma}
            )_{
                \sigma \in \Sigma_\KK 
            }
        ) \,,
\]
where $\kappa_0 \in \ZZ$ and $\kappa_{\sigma} = (\kappa_{\sigma, j}) \in \ZZ^{b_{\sigma}}$. Namely, for 
\[
    t = 
        (
            t_0, 
            ( 
                \diag(
                    t_{\sigma, i, 1}, 
                    \ldots, 
                    t_{\sigma, i, b_{\sigma, i}}
                )
            )_{
                \sigma \in \Sigma_\KK
            }
        ) \in T_{H_0} \,,
\]
one has
\begin{equation} \label{eq:def kappa(t) formula}
    \kappa(t) = 
    t_0^{\kappa_0} 
    \prod_{\sigma \in \Sigma_\KK} 
    \prod_{j=1}^{b_{\sigma}}
        t_{\sigma, j}^{\kappa_{\sigma, j}} \,.
\end{equation}

We refer to $\kappa$ as a \emph{weight}. We say that $\kappa$ is \emph{dominant} if $\kappa_{\sigma, j-1} \geq \kappa_{\sigma, j}$ for all $\sigma \in \Sigma_\KK$, $1 < j \leq b_{\sigma}$, or equivalently if it is dominant with respect to the opposite Borel $B_{H_0}^\opp$ (of upper-triangular matrices).

We say that $\kappa$ is \emph{regular} if $\kappa_{\sigma, j-1} > \kappa_{\sigma, j}$ for all $\sigma \in \Sigma_\KK$, $1 < j \leq b_{\sigma}$. Furthermore, we say that $\kappa$ is \emph{very regular} if $\kappa$ is regular and $\kappa_{\sigma, b_\sigma} \gg 0$ for each $\sigma$. 

\begin{remark}
    Note that we do not include an explicit lower bound in the definition of \emph{very regular} weights above. This is because we only use this notion in conjectures, see Conjecture \ref{conj:classicality scalar padic forms}. The author plans to study this notion in more details in the future.
\end{remark}

Given a dominant character $\kappa$ of $T_{H_0}$ over an $S_0$-algebra $R$, extend it trivially to $B_{H_0}$. Define
\[
    W_\kappa = W_\kappa(R) = \Ind_{B_{H_0}}^{H_0} \kappa 
    = 
        \{
            \phi : H_{0_{/R}} \to \Ga \mid 
            \phi(bh) = \kappa(b)\phi(h), 
            \forall b \in B_{H_0} 
        \} \,.
\]
with its natural structure as a left $H_0$-module via multiplication on the right.

As explained in \cite[Part II. Chapter 2]{Jan03} and \cite[Section 8.1.2]{Hid04}, if $R$ is flat over $S_0$, this is an $R$-model for the highest weight representation of $H_0$ with respect to $(T_{H_0}, B_{H_0}^\opp)$ of weight $\kappa$. 

\subsubsection{$P$-parallel weights.} \label{subsubsec:P parallel weights}
Given $\sigma \in \Sigma_\KK$, let $w$ be the place of $\KK$ above $p$ such that $\p_\sigma = \p_w$. In this section, we write $\d_\sigma$ for the partition $\d_w = (n_{w, 1}, \ldots, n_{w, t_w})$ of $a_\sigma = a_w$ introduced in Section \ref{subsubsec:para subgp of G over Zp}, $t_\sigma$ for $t_w$ and $n_{\sigma, j}$ for $n_{w,j}$.

We denote the standard lower-triangular parabolic subgroup of $\GL_{b_\sigma}(S_0)$ corresponding to $\d_{\sigma c}$ by $P_{0, \d_{\sigma c}}$. Define $P_{H_0} \subset H_0$ as the $S_0$-parabolic subgroup corresponding to the product $\prod_{\sigma \in \Sigma_\KK} P_{0, \d_{\sigma c}}$ via \eqref{eq:H0 basis}. We denote its unipotent radical by $P_{H_0}^u$ and its Levi factor by $L_{H_0}$.

We identify $L_{H_0}$ with
\begin{equation} \label{eq:prod L H0 over S0}
    \Gm 
        \times 
    \prod_{\sigma \in \Sigma_\KK} 
    \prod_{i=1}^{t_\sigma} 
        \GL_{S_0}(n_{w,i})
\end{equation}
via \eqref{eq:prod H0 over S0} and the obvious block-diagonal embeddings (for each $\sigma \in \Sigma_\KK$). Let $SL_{H_0} \subset L_{H_0}$ be the kernel the \emph{block-by-block} determinant map (analogous to the definition of $SL_H \subset L_H$ in Section \ref{subsubsec:para subgp of G over Zp}).

We say that a weight $\kappa = (\kappa_0, (\kappa_\sigma)_{\sigma \in \Sigma_\KK})$ of $T_{H_0}$ is \emph{$P$-parallel} if $\kappa$ extends to a character of $L_{H_0}$ that factors through $L_{H_0}/SL_{H_0}$. Using the conventions set in Remark \ref{rmk:trivial partition}, we see that every weight is $B$-parallel.

For $k = 1, \ldots, t_{\sigma c}$, let $N_{\sigma, k}$ denote the partial sum $\sum_{j=1}^k n_{\sigma c, j}$ and define $N_{\sigma, 0} = 0$. By identifying each $\kappa_\sigma$ with a tuple in $\ZZ^{b_\sigma}$ as above, $\kappa$ is \emph{$P$-parallel} if and only if
\begin{equation}\label{eq:P parallel weights equalities}
    \kappa_{\sigma, 1+N_{\sigma, k}} = \kappa_{\sigma, 2+N_{\sigma, k}} = \ldots = \kappa_{\sigma, N_{\sigma, k+1}}\,,
\end{equation}
for all $\sigma \in \Sigma_\KK$ and $0 \leq k < t_{\sigma c}$.

The tuple $\kappa_{Z_0} = (\kappa_0, (\kappa_{N_{\sigma, 1}}, \kappa_{N_{\sigma, 2}}, \ldots, \kappa_{N_{\sigma, t_{\sigma_c}}})_{\sigma \in \Sigma_\KK})$ naturally corresponds to a character of the center $Z_0$ of $L_{H_0}$. However, note that $\kappa_{Z_0}$ is not the restriction of $\kappa$ from $T_{H_0}$ to $Z_0$ (see Remark \ref{rmk:char of Z and P/SP}).

For later purposes, let $B(L_{H_0})$ denote the $S_0$-group given by the intersection of $B_{H_0} \cap L_{H_0}$. Equivalently, $B(L_{H_0})$ is the Borel of $L_{H_0}$ corresponding to the product (over $\sigma \in \Sigma_\KK$, $1 \leq i \leq t_\sigma$) of standard lower-triangular Borel subgroups via \eqref{eq:prod L H0 over S0}. 

Let $\rho_\kappa$ denote the $L_{H_0}$-representation $\Ind_{B(L_{H_0})}^{L_{H_0}} \kappa$ and write $V_\kappa$ for the associated algebraic vector space. In particular, we have $W_\kappa = \Ind_{L_{H_0}}^{H_0} \rho_\kappa$. 

Now, let $\beta$ be some $P$-parallel weight of $T_{H_0}$ and denote its extension to a character of $L_{H_0}$ by $\beta$ again. Note that $\rho_{\kappa + \beta}$ is canonically isomorphic to $\rho_\kappa \otimes \beta$. 

Therefore, we view $V_\kappa$ as the vector space associated to the representation $\rho_{\kappa'}$ for every algebraic weight $\kappa'$ in the ``$P$-parallel lattice'' 
\begin{equation} \label{eq:def [kappa]}
    [\kappa] := \{\kappa + \theta \mid \theta \text{ is $P$-parallel}\}
\end{equation}
of algebraic weights containing $\kappa$. We sometimes write $V_\kappa$ as $V_{[\kappa]}$ to emphasize this fact.

\subsubsection{$p$-adic weights} \label{subsubsec:padic weights}
Let $\OO'$ be the ring of integers of the smallest field $\LL' \subset \bQQ_p$ containing the image of all embeddings $\KK \hookrightarrow \bQQ_p$. In particular, $\LL'$ contains $\incl_p(\KK')$, hence $\incl_p$ identifies $\OO'$ as an $S_0$-algebra (and as an $S_p$-algebra).

Consider the factorization $\OO_{(p)} = \prod_{w \mid p} \OO_w$. Then, we have
\[
    \OO_{(p)} \otimes \OO' 
        = 
    \prod_{w \mid p} \OO_w \otimes \OO' 
        \xrightarrow{\sim}
    \prod_{w \mid p} 
        \prod_{
            \substack{
                \sigma \in \Sigma_\KK \\ 
                \p_\sigma = \p_w
            }
        } 
    \OO' 
        = 
    \prod_{\sigma \in \Sigma_\KK} \OO'\,,
\]
by definition of $\OO'$. This identification, together with the choice of basis for $L^+$ in Section \ref{subsubsec:comp to gen linear groups}, yields a decomposition
\[
    L^+ \otimes \OO' 
        = 
    \prod_{w \mid p} L_w \otimes \OO' 
        = 
    \prod_{\sigma \in \Sigma_\KK} (\OO')^{a_w}\,.
\]

Similarly, the choice of basis for $\Lambda_0$ in Section \ref{subsubsec:alg weights} induces
\[
    \Lambda_0 \otimes_{S_p} \OO' 
        = 
    \prod_{\sigma \in \Sigma_\KK} \Lambda_{0, \sigma} \otimes_{S_0} \OO'
        =
    \prod_{\sigma \in \Sigma_\KK} (\OO')^{a_\sigma}\,.
\]

From the above, we obtain an identification $L^+ \otimes \OO' = \Lambda_0 \otimes_{S_p} \OO'$ over $\OO \otimes \OO' = \OO_{(p)} \otimes \OO'$. Therefore, using the duality between $\Lambda_0$ and $\Lambda_0^\vee$, we have an isomorphism $H_{0 /\OO'} \xrightarrow{\sim} \Gm \times H_{/\OO'}$ given by
\begin{equation} \label{eq:iso H0 with Gm x H}
    (
        \nu, 
        (g_\sigma)_{\sigma \in \Sigma_\KK}
    ) 
        \mapsto
    \left(
        \nu, 
        (
            \prod_{
                \substack{
                    \sigma \in \Sigma_\KK \\
                    \p_\sigma = \p_w
                }
            }
                \nu \cdot
                \tp{g}_{\sigma c}^{-1}
        )_{w \mid p}
    \right)
\end{equation}
using the isomorphisms \eqref{eq:GL(Lw pm) basis}, \eqref{eq:def H}, \eqref{eq:prod H0 over S0} and \eqref{eq:H0 basis}.

In particular, it induces a natural inclusion $L_H(\ZZ_p) \hookrightarrow L_{H_0}(\OO')$ and allows us to view $V_\kappa$ as a representation of $L_H(\ZZ_p)$. We write $\rho_{\kappa_p}$ instead of $\rho_\kappa$ when referring to $V_\kappa$ as an $L_H(\ZZ_p)$-module. For instance, given $l \in L_H(\ZZ_p)$, we write
\begin{equation}\label{eq:relation kappa vs kappa p on L H ZZp}
    \rho_\kappa(\tp{l}^{-1}) = \rho_{\kappa_p}(l)\,,
\end{equation}
where we abuse notation to denote the element of $L_{H_0}(\OO')$ corresponding to $l$ under the isomorphism \eqref{eq:iso H0 with Gm x H} by $\tp{l}^{-1}$.

Similarly, the identification \eqref{eq:iso H0 with Gm x H} induces an embedding $T_H(\ZZ_p) \hookrightarrow T_{H_0}(\OO')$. Given $t = (\diag(t_{w,1}, \ldots, t_{w, a_w})_{w \mid p}) \in T_H(\ZZ_p)$, its image in $T_{H_0}(\OO')$ is naturally identified with $x = (1, t^{-1})$ and we have
\begin{equation} \label{eq:relation kappa kappa p}
    \kappa(x) = \kappa_p(t)\,,
\end{equation}
where
\[
    \kappa_p(t) 
        = 
    \prod_{w \mid p}
    \prod_{
        \substack{
            \sigma \in \Sigma_\KK \\
            \p_\sigma = \p_w
        }
    }
    \prod_{j=1}^{a_\sigma}
        \sigma(t_{w,j})^{\kappa_{\sigma c, j}}\,.   
\]

We sometimes write
\begin{equation} \label{eq:tuple definition of kappa p}
    \kappa_p 
        = 
    (\kappa_{\sigma c})_{\sigma \in \Sigma_\KK} 
        \in 
    \prod_{\sigma \in \Sigma_\KK} \ZZ^{a_\sigma}\,,
\end{equation}
for convenience and refer to $\kappa_p$ as a \emph{$p$-adic weight}. 

We say that $\kappa_p$ is a \emph{$P$-parallel} if $\kappa$ is $P$-parallel. Clearly, $P$-parallel $p$-adic weights extend to characters of $L_H(\ZZ_p)$ (that factor through $L_H(\ZZ_p)/SL_H(\ZZ_p)$). 

If $\beta$ is an algebraic $P$-parallel weight and $\kappa$ is any algebraic weight, then $\rho_{\kappa_p + \beta_p}$ is canonically isomorphic to $\rho_{\kappa_p} \otimes \beta_p$. Thus, we again view $V_\kappa$ as the space on which $\rho_{\kappa_p + \beta_p}$ acts for all $P$-parallel $p$-adic weights $\beta_p$.

\begin{remark} \label{rmk:rho kappa p is kappa p if P is B}
    The representation $\rho_{\kappa_p}$ and the character $\kappa_p$ coincide with one another when $P = B$ as in Remark \ref{rmk:trivial partition}. This is what occurs in \cite[Section 2.9.4]{EHLS}. 
\end{remark}

\subsection{Shimura varieties of $P$-Iwahoric level at $p$.} \label{subsec:Shi var of P Iwa level at p}
We first recall the familiar theory of integral away-from-$p$ models of Shimura variety for the unitary group $G$. We use them to define holomorphic and anti-holomorphic automorphic representations of $G$.

\begin{remark}
In Section \ref{subsec:P nbtp thy of mod forms}, we introduce more general level structures at $p$ related to the $P$-Iwahori subgroups constructed in Section \ref{subsubsec:para subgp of G over Zp} and define the notion of \emph{$P$-nebentypus} for both modular forms and automorphic representations.
\end{remark}

Let $X = X_\PP$ denote the conjugacy class of $h$ via the natural action of $G(\RR)$ on $\End_{\KK^+ \otimes \RR}(V \otimes \RR)$. It is well-known that the pair $(G, X)$ defines a \emph{Shimura datum} in the usual sense whose reflex field is again $F$.

Let $K = G(\ZZ_p)K^p$ as in the beginning of Section \ref{subsubsec:unitary moduli space}. Let $Sh_K(G,X)$ be the canonical model of the Shimura variety of level $K$ over $F$ associated to $(G,X)$. Then, the moduli space $\MM_{K\,/F}$ is the union of finitely many copies of $Sh_K(G, X)$, see \cite[Section 8]{Kot92} for details.

More precisely, let $V^{(1)}, \ldots, V^{(k)}$ be representatives for the isomorphism classes of all hermitian vector spaces that are locally isomorphic to $V$ at every place of $\QQ$. As explained in \cite[Section 2.3.2]{CEFMV}, there are finitely many such classes, in fact $k = \absv{\ker^1(\QQ, G)}$, where 
\[
    \ker^1(\QQ, G) = \ker\left( H^1(\QQ, G) \to \prod_v H^1(\QQ_v, G) \right).
\]

The base change of $\MM_{K}$ over $F$ is the disjoint union of $F$-schemes $\MM_{K, V^{(j)}}$, naturally indexed by the $V^{(j)}$ and all isomorphic to $Sh_K(G, X)$. 

Assume that $V^{(1)} = V$. To work integrally (away-from-$p$), denote the scheme-theoretic closure of $\MM_{K, V}$ in $\MM_{K}$ by $\level{K}{\Sh}(V)$. When $V$ is clear from context, we simply write $\level{K}{\Sh}$.

It is well-known that $\level{K}{\Sh}$ is a smooth, quasi-projective $S_p$-scheme. We refer to $\level{K}{\Sh}$ as a \emph{Shimura variety} of level $K$ (associated to $\PP$) and $\MM_K$ as a \emph{moduli space}. Denote the natural inclusion $\level{K}{\Sh} \hookrightarrow \MM_K$ over $S_p$ by $s_K$.

Furthermore, to work with compactified Shimura varieties, let $\Omega$ be a polyhedral cone decomposition, as in Section \ref{subsubsec:toroidal compactification}, and denote the scheme-theoretic closure of $\level{K}{\Sh}$ in $\MM_{K, \Omega}^{\tor}$ by $\level{K}{\Sh}_{\Omega}^{\tor}$. This is the natural smooth toroidal compactification of $\level{K}{\Sh}$ discussed in \cite[Sections 3-4]{Lan12} and, over $F$, it recovers the usual toroidal compactification of $Sh_K(G, X)$. 

We often treat the tower $\level{K}{\Sh}^{\tor} := \{\level{K}{\Sh}_{\Omega}^{\tor}\}_\Omega$ as a single scheme. We denote the natural inclusions $\level{K}{\Sh}_{\Omega}^{\tor} \hookrightarrow \MM_{K, \Omega}^{\tor}$ and $\level{K}{\Sh}^{\tor} \hookrightarrow \MM_{K}^{\tor}$ by $s_{K, \Omega}$ and $s_K$ respectively.

Given a neat compact open subgroup $K'^{,p} \subset K^p$, let $K' = G(\ZZ_p)K'^{,p}$. The map $\MM_{K'} \to \MM_K$ is compatible with the inclusions $s_K$ and $s_{K'}$, hence induces an analogous homomorphism $\level{K'}{\Sh} \to \level{K}{\Sh}$. The latter extends canonically to a map (of towers) $\level{K'}{\Sh}^\tor \to \level{K}{\Sh}^\tor$ on toroidal compactifications. 

A similar statement holds true for $[g] : \level{gKg^{-1}}{\Sh} \to \level{K}{\Sh}$, given any $g \in G(\AA_f^p)$. This induces a natural action of $G(\AA_f^p)$ on the towers $\{\level{K}{\Sh}\}_{K^p}$ and $\{\level{K}{\Sh}^\tor\}_{K^p}$.

Lastly, we set $\Sh(V) := \varprojlim_K \level{K}{\Sh(V)}$ and $\Sh(V)^{\tor} := \varprojlim_K \level{K}{\Sh(V)}^{\tor}$ when working with the profinite Shimura variety of infinite level and its compactification.

\subsubsection{The canonical bundle} \label{subsubsec:canonical bundle}
The following section recalls some of the material of \cite[Sections 2.6 and 6.1]{EHLS}.

Let $\w$ be the $\OO_{\MM_K^{\tor}}$-dual of $\Lie_{\MM_K^{\tor}} \Ab^\vee$ over $S_p$. The Kottwitz determinant condition mentioned in the definition of the moduli problem $\MM_K(\PP)$ implies that $\w$ is locally isomorphic to $\Lambda_0^\vee \otimes_{S_p} \OO_{\MM_K^\tor}$ over $\OO \otimes \OO_{\MM_K^\tor}$. Define the canonical bundle $\EE$ as the scheme
\[
    \isom_{\OO_\KK \otimes \OO_{\MM_K^\tor}}
    ( 
        \OO_{\MM_K^\tor}(1),
        \OO_{\MM_K^\tor}(1)
    )
        \times
    \isom_{\OO_\KK \otimes \OO_{\MM_K^\tor}}
    ( 
        \w,
        \Lambda_0^\vee \otimes_{S_p} \OO_{\MM_K^\tor}
    )\,,
\]
over $\MM_K^\tor$.

The natural structure map $\pi : \EE \to \MM_K^\tor$ is an $H_0$-torsor and is defined over $S_p$ when $K$ is a neat open compact subgroup of $G(\AA_f)$ of the form $G(\ZZ_p)K^p$. Note that the first factor in the definition of $\EE$ is included to keep track of the action of the (similitude) $\Gm$-factor of $H_0$, however it does not play a significant role in the rest of the paper.

\subsubsection{Modular forms of weight $\kappa$} \label{subsubsec:mod forms of weight kappa}
Let $R$ be an algebra over $S_0 = \OO_{\KK', (\p')}$, and let $\kappa$ be a dominant character of $T_{H_0}$ over $R$ as in Section \ref{subsubsec:alg weights}. Consider the vector bundle
\[
    \w_\kappa = \w_{\kappa, \Omega} = s_{K, \Omega}^* \pi_* (\OO_{\EE}[\kappa]) \ ,
\]
above $\level{K}{\Sh}_\Omega$ defined over $S_0$. Here, we extend $\kappa$ to an algebraic character of $B_{H_0}$ trivially and $\OO_\EE[\kappa]$ denotes the $\kappa$-isotypic part of $\OO_\EE$. By taking limits over $K$ and $\Omega$, we often view $\w_\kappa$ over $\Sh(V)^\tor$ without comment.

\begin{remark} \label{rmk:aut vec bundles from reps}
    Recall that given an irreducible representation of $P_0$ over $\CC$ that factors through $H_0$, one can view it as a $G$-equivariant vector bundle on the compact dual $\wh{X}$ of $X$ and thus define an automorphic vector bundle $\w_W$ on $\Sh(V)_{/\CC}$ using the usual $\otimes$-functor
    \[
        \text{$G$--Bun}(\wh{X}) \to \text{Bun}(\Sh(V))\,,
    \]
    see \cite[Section 6.1.1]{EHLS} for further details.

    It is well-known that each such $\w_W$ has a canonical model over a number field $F(W)/F$ such that $F(W) \subset \KK'$. For instance, the base change of $\w_{W_\kappa}$ from $F(W_\kappa)$ to $\KK'$ is actually canonically isomorphic to the restriction from $\Sh(V)^{\tor}$ to $\Sh(V)$ of $\w_\kappa$.
\end{remark}

For each polyhedral cone decomposition $\Omega$, let $D_\Omega$ be the Cartier divisor $\level{K}{\Sh_\Omega^\tor} - \level{K}{\Sh}$ equipped with its structure of a reduced closed subscheme. Implicitly, we restrict our attention to choices of $\Omega$ for which this complement $D_\Omega$ is a divisor with normal crossing. Let $\w_\kappa(-D_\Omega)$ be the twist of $\w_\kappa$ by the ideal sheaf of the boundaries corresponding to $D_\Omega$.

Then, cuspidal cohomology (of degree $i$) of level $K$ with respect to $\Omega$ is defined as
\[
    H_!^i(\level{K}{\Sh(V)}_\Omega^{\tor}, \w_\kappa)
        :=
    \mathrm{Im}
    \left(
        H^i(\level{K}{\Sh(V)}_\Omega^{\tor}, \w_\kappa(-D_\Omega)) \to H^i(\level{K}{\Sh(V)}_\Omega^{\tor}, \w_\kappa)
    \right)\,,
\]
and we mainly work with
\[
    H_!^i(\Sh(V), \w_\kappa)
        =
    H_!^i(\Sh(V)^{\tor}, \w_\kappa)
        :=
    \varinjlim _{K, \Omega}
        H_!^i(\level{K}{\Sh(V)}_\Omega^{\tor}, \w_\kappa)\,,
\]
where the limit is restricted to subgroups $K$ of the form $G(\ZZ_p)K^p$, so that the above is defined over $S_0$. We first review the theory of degree $i = 0$ in what follows and discuss the middle degree cohomology in Section \ref{subsec:duality and integrality}.

\begin{remark} \label{rmk:Koecher principle}
    If the reflex field $F$ is different from $\QQ$ or the derived group $G^{\mathrm{der}}$ of $G$ (over $\QQ$) has no irreducible factor isomorphic to $\SU(1,1)$, then we can invoke the K\"ocher principle, namely 
    \[
        H^0(\level{K}{\Sh(V)}_\Omega^{\tor}, \w_\kappa)
            =
        H^0(\level{K}{\Sh(V)}, \w_\kappa)\,,    
    \]
    see \cite{Lan16}. Therefore, in that case, we can ignore the toroidal compactification and omit the limit over $\Omega$ in the definitions above.
    
    Otherwise, the toroidal compactifications are canonical; they are simply the minimal compactification. We ignore the details needed to treat this case and implicitly view the tower $\level{K}{\Sh(V)}^{\tor}$ as a single scheme, see the remarks at the end of Section \ref{subsubsec:toroidal compactification} (or \cite[Section 2.6.5]{EHLS} for more details and a similar treatment).
\end{remark}

The action of $G(\AA_f^p)$ on 
\[
    H^0(\Sh(V), \w_\kappa) 
        := 
    \varinjlim_{K, \Omega} 
        H^0(\level{K}{\Sh(V)}_\Omega^\tor, \w_\kappa)
\]
induced by its action on the tower $\{\level{K}{\Sh}^\tor\}_{K^p}$ stabilizes $H^0_!(\Sh(V), \w_\kappa)$.

The $R$-modules of $M_\kappa(K; R)$ and $S_\kappa(K; R)$ of modular forms and cusp forms of weight $\kappa$ and level $K = G(\ZZ_p)K^p$ are defined by taking $K^p$-fixed points of this action, namely
\[
    M_\kappa(K; R) 
        := 
    H^0(\Sh(V)_{/R}, \w_\kappa)^{K^p} 
        = 
    H^0(\level{K}{\Sh(V)}, \w_\kappa)
\]
and
\[
    S_\kappa(K; R) 
        := 
    H_!^0(\Sh(V)_{/R}, \w_\kappa)^{K^p} 
        = 
    H_!^0(\level{K}{\Sh(V)}, \w_\kappa)
    \,,
\]
respectively.

Via the moduli interpretation of $\MM_K$, we view a modular form $f \in M_\kappa(K; R)$ as a rule on the set of pairs $(\ul{A}, \varepsilon) \in \EE(S)$, for any $R$-algebra $S$, such that $f(\ul{A}, \varepsilon) \in S$, the rule is functorial in $S$, and
\[
    f(\ul{A}, b\varepsilon)
        =
    \kappa(b)
    f(\ul{A}, \varepsilon)\,,
\]
for all $b \in B_{H_0}(S)$.

\subsubsection{Hecke operators away from $p$.} \label{subsubsec:Hecke ops away from p}
Given $K = G(\ZZ_p)K^p$ as above, $g \in G(\AA_f^p)$ and an $S_0$-algebra $R$, the double coset $KgK$ naturally defines an operator
\[
    [KgK] : M_\kappa(K; R) \to M_\kappa(K; R)
\]
induced by viewing $KgK$ as a correspondence on $\MM_K$. More precisely, given $f \in M_\kappa(K; R)$ and writing $K^pgK^p$ as finite disjoint union $\bigsqcup_i g_i K^p$ of right cosets, we have
\begin{equation} \label{eq:def Hecke op Tg level Kp}
    ([KgK] f)(A, \lambda, \iota, \alpha, \varepsilon)
        =
    \sum_i
        f(A, \lambda, \iota, \alpha \circ g_i, \varepsilon)\,,
\end{equation}
which is obviously independent of the choice of representatives $g_i$. When the level $K$ is clear from context, we simply write $T(g)$ instead of $[KgK]$. One readily checks that $T(g)$ stabilizes $S_\kappa(K; R)$.

\subsection{$P$-nebentypus theory of modular forms.} \label{subsec:P nbtp thy of mod forms}
We now introduce a more general level structure at $p$ via covers of $\MM_K$ and $\MM_K^{\tor}$.

\subsubsection{Level subgroup $K_{P,r}$.} \label{subsubsec:lvl subgp KPr}
Let $\ul{\Ab} = (\Ab, \lambda, \iota, \alpha)$ be the universal abelian scheme over $\MM_K$. Using \cite[Theorem 6.4.1.1]{Lan13}, $\Ab$ can be extended to a semiabelian scheme over $\MM_K^\tor$ that is part of a degenerating family and which we still denote $\Ab$. 

By \cite[Theorem 3.4.3.2]{Lan13}, there exists a dual semiabelian scheme $\Ab^\vee$ together with homomorphisms $\Ab \to \Ab^\vee$, $S_p \to \End_{\MM_K^{\tor}} \Ab$ and a $K^{(p)}$-level structure on $\Ab$ that extend $\lambda$, $\iota$ and $\alpha$ respectively.

Define an $S_p$-scheme $\overline{\MM}_{K_r}$ over $\MM_K^{\tor}$ whose $S$-points is the set of $P^u_H(\ZZ_p)$-orbits of injections $\phi : L^+ \otimes \mu_{p^r} \hookrightarrow \Ab^\vee[p^r]_{/S}$ of group schemes over $\OO \otimes \ZZ_p$ such that the image of $\phi$ is an isotropic subgroup scheme. The natural action of $\LL_r = L_H(\ZZ_p/p^r\ZZ_p)$ on $L^+ \otimes \mu_{p^r}$ induces a structure of $\LL_r$-torsor on $\overline{\MM}_{K_r} \to \MM_K^{\tor}$.

Let $\MM_{K_r}$ denote the pullback of $\overline{\MM}_{K_r}$ over $\MM_K$, i.e. we have the Cartesian commutative diagram 
\begin{equation} \label{eq:def pullback M Kr}
    \begin{tikzcd}
        \MM_{K_r}  
            \arrow[r, hook] \arrow[d] &
        \ol{\MM}_{K_r} 
            \arrow[d] \\
        \MM_K
            \arrow[r, hook] &
        \MM^{\tor}_K
    \end{tikzcd}
\end{equation}
and the vertical arrows are $\LL_r$-torsors. We set $K_{P, r} := I_{P,r}K^p \subset G(\AA_f)$

\begin{remark} \label{rmk:moduli problem of level Kr over F}
    Recall that one can define an $F$-rational moduli problem  generalizing the one in Section \ref{subsubsec:unitary moduli space} for each neat open compact subgroup $K \subset G(\AA_f)$ (by essentially dropping all ``prime-to-$p$'' conditions). We again denote the corresponding moduli space by $\MM_K$. The $G(\AA_f^p)$-action on $\{\MM_K\}_{K^p}$ extends to an action of $G(\AA_f)$ on $\{\MM_K\}_K$. We do not include the exact details needed to modify the prior theory to $\MM_{K\,/F}$ and instead refer the reader to \cite[Section 2.1]{EHLS} or \cite[Corollary 7.2.3.10]{Lan13}.
    
    A choice of basis of $\ZZ_p(1)$ induces a natural isomorphism between the scheme $\MM_{K_r\, /F}$, defined as a pullback in \eqref{eq:def pullback M Kr}, and the moduli space $\MM_{I_r K^p\, /F}$ representing the $F$-rational moduli problem mentioned in the previous paragraph. Furthermore, this same choice identifies $\ol{\MM}_{K_r\, /F}$ with the normalization of $\MM_{K\, /F}^{\tor}$ in $\MM_{K_r\, /F}$. Therefore, we can write $K_{P,r} = K_r$ without risk of confusion when $P$ is clear from context.
\end{remark}

Over $S_p$, define $\level{K_r}{\Sh}$ (resp. $\level{K_r}{\ol{\Sh}}$) as the pullback of $\MM_{K_r}$ (resp. $\ol{\MM}_{K_r}$) via $s_K$. Hence, we have the commutative diagrams
\[
    \begin{tikzcd}
        \level{K_r}{\Sh}
            \arrow[r, hook] \arrow[d] &
        \MM_{K_r} 
            \arrow[d] & &
        \level{K_r}{\ol{\Sh}} 
            \arrow[r, hook] \arrow[d] &
        \ol{\MM}_{K_r} 
            \arrow[d] \\
        \level{K}{\Sh}
            \arrow[r, hook] &
        \MM_K & &
        \level{K}{\Sh^\tor} 
            \arrow[r, hook] &
        \MM^\tor_K
    \end{tikzcd}\,,
\]
and by abusing notation, we denote all four horizontal inclusions by $s_K$. All four vertical arrows are again $\LL_r$-torsors.

\begin{remark} \label{rmk:shi var of level Kr over F}
    As in Remark \ref{rmk:moduli problem of level Kr over F}, a choice of basis of $\ZZ_p(1)$ identifies $\Sh_{K_r\, /F}$ with $Sh_{I_rK^p}(G, X)_{/F}$ (the analogue of $Sh_K(G, X)$ introduced in Section \ref{subsec:Shi var of P Iwa level at p} for $K = I_rK^p$), and identifies $\level{K_r}{\ol{\Sh}}_{/F}$ with the normalization of $\level{K}{\Sh}_{/F}^{\tor}$ in $\level{K_r}{\Sh}_{/F}$.
\end{remark}

The action of $G(\AA_f^p)$ on the tower $\{\level{K}{\Sh}\}_{K^p}$ naturally induces an action on $\{\level{K_r}{\Sh}\}_{K^p}$. Analogous statements hold true for $\{\MM_{K_r}\}_{K^p}$, $\{\level{K_r}{\ol{\Sh}}\}_{K^p}$, and $\{\ol{\MM}_{K_r}\}_{K^p}$.

Furthermore, let $\EE_r = \EE \times_{\MM_K^\tor} \ol{\MM}_{K_r}$, so
\[
    \begin{tikzcd} 
        \EE_r 
            \arrow[r, "H_0"] \arrow[d, "\LL_r"] &
        \ol{\MM}_{K_r}
            \arrow[d, "\LL_r"] \\
        \EE
            \arrow[r, "H_0"] &
        \MM_K^\tor
    \end{tikzcd}
\]
and denote the structure map $\EE_r \to \ol{\MM}_{K_r}$ by $\pi_r$. 

Given a dominant weight $\kappa$ of $T_{H_0}$ over some $S_0$-algebra $R$, we define
\[
    \w_{\kappa, r} := s_{K_r}^*(\pi_r)_*(\OO_{\EE_r}[\kappa])
\]
as a sheaf over $\level{K_r}{\ol{\Sh}}_{/R}$. 

We define the space of modular forms on $G$ over $R$ of level $K_r$ and weight $\kappa$ as
\begin{equation} \label{eq:def of M kappa Kr R}
    M_\kappa(K_r; R) 
        := 
    H^0(
        \level{K_r}{\ol{\Sh}}, 
        \w_{\kappa, r}
    )
\end{equation}
and its subspace of cusp forms as
\begin{equation} \label{eq:def of S kappa Kr R}
    S_\kappa(K_r; R) 
        := 
    H^0_!(
        \level{K_r}{\ol{\Sh}},
        \w_{\kappa, r}
    )\,,
\end{equation}
where $H^0_!$ denotes cuspidal cohomology as in Section \ref{subsubsec:canonical bundle}.

It follows from Remarks \ref{rmk:moduli problem of level Kr over F} and \ref{rmk:shi var of level Kr over F} that $M_\kappa(K_r; S_p)$ (resp. $S_\kappa(K_r; S_p)$) is an $S_p$-integral structure of the usual space of modular (resp. cusp) forms over $F$ on $G$ of level $I_rK^p$ and weight $\kappa$.

We view a modular form $f \in M_\kappa(K_r; R)$ as a rule on the set of pairs $(\ul{A}, \phi, \varepsilon) \in \EE_r(S)$, for any $R$-algebra $S$, such that $f(\ul{A}, \phi, \varepsilon) \in S$, the rule is functorial in $S$, and
\[
    f(\ul{A}, \phi, b\varepsilon)
        =
    \kappa(b)
    f(\ul{A}, \phi, \varepsilon)\,,
\]
for all $b \in B_{H_0}(S)$. 

Given a $\bQQ_p$-valued multiplicative character $\psi_B$ of the maximal torus $T_H(\ZZ_p)$ of $H(\ZZ_p)$ that factors through $T_H(\ZZ_p/p^r\ZZ_p)$, let $S_p[\psi_B]$ denote the smallest ring extension of $S_p$ containing the values of $\psi_B$. Given an $S_p[\psi_B]$-algebra $R$, we define the $R$-module of modular forms over $R$, weight $\kappa$, level $K_r$, and \emph{(classical) nebentypus} $\psi_B$ as
\[
    M_\kappa(K_r, \psi_B; R)
        :=
    \{
        f \in H^0(
            \level{K_r}{\ol{\Sh}}, \w_{\kappa, r}
        )
            :
        t \cdot f
            =
        \psi_B(t) f\,, \forall t \in T_H(\ZZ_p)
    \}\,,
\]
and we define the analogous $R$-module of cusp forms $S_\kappa(K_r, \psi_B; R)$ similarly.

Given $g \in G(\AA_f)$, the formula \eqref{eq:def Hecke op Tg level Kp} can similarly be adapted to define on operator $M_\kappa(K_r; R)$ via
\begin{equation} \label{eq:def Hecke op Tg level Kr}
    ([K_rgK_r]f)(A, \lambda, \iota, \alpha, \phi, \varepsilon)
        =
    \sum_i
        f(A, \lambda, \iota, \alpha \circ g_i, \phi, \varepsilon)\,,
\end{equation}
using the same notation as in Section \ref{subsubsec:Hecke ops away from p}. By abuse of notation, we again denote this operator by $T(g)$ when $K_r$ is clear from context.

Furthermore, if $R$ contains the reflex field $F$, then $M_\kappa(K_r; R)$ is also obtained as the $K_r$-fixed points of the $R$-module
\[
    \varinjlim_{K \subset G(\AA_f)}
        H^0(\level{K}{\ol{\Sh}}, \w_{\kappa})\,,
\]
and the same holds true for $S_\kappa(K_r; R)$ upon replacing $H^0(-)$ by $H^0_!(-)$.

\subsubsection{$P$-nebentypus of modular forms.} \label{subsubsec:P nebentypus}
Let $\tau$ be a smooth irreducible representation of $L_H(\ZZ_p)$ acting on a module $\MMM_\tau$ over some $S_p$-algebra $S_p[\tau] \subset \CC$.

\begin{definition}
    We say that $\tau$ is a \emph{$P$-nebentypus of level $r$} if it factors through $\LL_r = L_H(\ZZ_p/p^r\ZZ_p)$. In this case, we can always assume $S_p[\tau]$ is finite over $S_p$, hence contained in $\bQQ$.
\end{definition}

We do not assume that $r$ is minimal for this property. In fact, if $\tau$ is of level $r$, it is obviously of level $r'$ for every $r' \geq r$, and therefore we sometimes write that $\tau$ is a ``$P$-nebentypus of level $r \gg 0$''.

Define $\EE_{r, \tau}$ as the $S_p[\tau]$-scheme over $\EE_r$ whose $R$-points are given by
\[
    \EE_{r, \tau}(R) 
        = 
    \EE_r(R) \times^{\tau} \MMM_{\tau, R}
        := 
    (
        \EE_r(R) \times \MMM_{\tau, R}
    )/{\sim^\tau}
\]
for any $S_p[\tau]$-algebra $R$, where the equivalence relation $\sim^\tau$ is
\[
    ( (\varepsilon, \phi), m) 
        \sim^\tau
    ( (\varepsilon, \phi \circ l), \tau(l)m) \ ,
\]
for all $(\varepsilon, \phi) \in \EE_r$, $m \in \MMM_{\tau, R}$ and $l \in L_H(\ZZ_p)$. We denote the structure map $\EE_{r, \tau} \to \ol{\MM}_{K_r}$ by $\pi_{r, \tau}$.

Let $S_0[\tau] \subset \bQQ$ be the compositum of $S_p[\tau]$ and $S_0 = \OO_{\KK', (\p')}$. Given a dominant weight $\kappa$ of $T_{H_0}$ over an $S_0[\tau]$-algebra $R$, we define
\[
    \w_{\kappa, r, \tau} = 
        s_{K}^* (\pi_{r, \tau})_*
        (
            \OO_{\EE_{r, \tau}}[\kappa]
        )
\]
as a sheaf on $\level{K_r}{\ol{\Sh}}$ over $R$. We denote its restriction to $\level{K_r}{\Sh}$ by $\w_{\kappa, r, \tau}$ as well.

\begin{definition} \label{def:M kappa(Kr, tau, R)}
    We define the space of modular forms on $G$ over $R$, level $K_r$, weight $\kappa$ and \emph{$P$-nebentypus} $\tau$ as
    \[
        M_\kappa(K_r, \tau; R) 
            := 
        H^0(
            \level{K_r}{\ol{\Sh}}, 
            \w_{\kappa, r, \tau}
        )
    \]
    and its subspace of cusp forms as
    \[
        S_\kappa(K_r, \tau; R) 
            := 
        H^0_!(
            \level{K_r}{\ol{\Sh}},
            \w_{\kappa, r, \tau}
        )\,,
    \]
    where $H^0_!$ again denotes cuspidal cohomology as in Section \ref{subsubsec:canonical bundle}.
\end{definition}

\begin{remark} \label{rmk:B-nebentypus is nebentypus}
    Classically, the nebentypus of a modular form is a finite-order character of the maximal torus $T_H(\ZZ_p)$ of $H$. In our terminology, see Remark \ref{rmk:trivial partition}, this is equivalent to a $B$-nebentypus.
\end{remark}

\begin{remark}
    The reader should note that in this notation $\tau$ is always a \emph{$P$-nebentypus}, i.e. a smooth finite-dimensional representation of the Levi subgroup of $P_H(\ZZ_p)$. On the other hand, when writing $M_\kappa(K_r, \psi_B; R)$, we always use the symbol $\psi_B$ for a character of the maximal torus $T_H(\ZZ_p)$ of the Borel subgroup $B_H(\ZZ_p)$. The subscript $B$ is to remind the reader of the relation between $\psi_B$ and $B_H$ and help distinguish between the \emph{similar yet different} spaces $M_\kappa(K_r, \psi_B; R)$ and $M_\kappa(K_r, \tau; R)$. The two notions overlap exactly when $P=B$, as in Remark \ref{rmk:trivial partition}, in which case the notation is not ambiguous.
\end{remark}

A modular form $f \in M_{\kappa}(K_r, \tau; R)$ can be interpreted as a functorial rule that assigns to a tuple $(\ul{A}, \phi, \varepsilon) \in \EE_r(S)$, over an $R$-algebra $S$, an element 
\[
    f(\ul{A}, \phi, \epsilon) 
        \in 
    \hom_{S}(
        \MMM_{\tau, S},
        S
    )
        =
    \MMM_{\tau, S}^\vee
\]
    such that
\[
    f(\ul{A}, \phi \circ l^{-1}, b \epsilon)
        = 
    \kappa(b)
    \tau^\vee(l)
        f(\ul{A}, \phi, \epsilon)
\]
for all $b \in B_{H_0}(S)$ and $l \in L_H(\ZZ_p)$.

Equivalently, using Frobenius reciprocity, $f$ can be interpreted as a functorial rule such that
\[
    f(\ul{A}, \phi, \epsilon) 
        \in 
    V_{\kappa, S} \otimes \MMM_{\tau, S}^\vee
\]
and
\[
    f(\ul{A}, \phi \circ l^{-1}, l_0 \epsilon)
        = 
    f(\ul{A}, \phi \circ l^{-1}, l_0 \epsilon)(v)
        = 
    (
        \rho_\kappa(l_0)
            \otimes
        \tau^\vee(l)
    )
    f(\ul{A}, \phi, \epsilon)
\]
for all $l_0 \in L_{H_0}(S)$ and $l \in L_H(\ZZ_p)$.

Given $g \in G(\AA_f^p)$, one can again define a Hecke operator $T(g)$ on $M_\kappa(K_r, \tau; R)$ which stabilizes the subspace of cusp forms via \eqref{eq:def Hecke op Tg level Kr}.

More generally, view $\MMM = \MMM_\tau$ simply as a module over $S_p[\MMM] = S_p[\tau]$, forgetting the representation $\tau$ momentarily. 

We define $\EE_{r, \MMM}$ as the $S_p[\MMM]$-scheme over $\EE_r$ whose $R$-points are given by
\[
    \EE_{r, \MMM}(R) 
        = 
    \EE_r(R) \times \MMM_{R}
\]
for any $S_p[\MMM]$-algebra $R$, without any equivalence relation. We denote the structure map $\EE_{r, \MMM} \to \ol{\MM}_{K_r}$ by $\pi_{r, \MMM}$.

Let $S_0[\MMM] \subset \bQQ$ be the compositum of $S_p[\tau]$ and $S_0 = \OO_{\KK', (\p')}$. Given a dominant weight $\kappa$ of $T_{H_0}$ over an $S_0[\MMM]$-algebra $R$, we define
\[
    \w_{\kappa, r, \MMM} = 
        s_{K}^* (\pi_{r, \MMM})_*
        (
            \OO_{\EE_{r, \MMM}}[\kappa]
        )
\]
as a sheaf on $\level{K_r}{\ol{\Sh}}$ over $R$. We denote its restriction to $\level{K_r}{\Sh}$ by $\w_{\kappa, r, \MMM}$ as well.

\begin{definition} \label{def:M kappa(Kr, [tau],R)}
    For any $S_0[\MMM]$-algebra $R$, we define the space of modular forms over $R$ on $G$ of weight $\kappa$, level $K_r$ and \emph{$P$-type $\MMM$} as 
    \[
        M_\kappa(K_r, \MMM; R) 
            := 
        H^0(
            \level{K_r}{\ol{\Sh}}, 
            \w_{\kappa, r, \MMM}
        )
    \]
    and its subspace of cusp forms as
    \[
        S_\kappa(K_r, \MMM; R) 
            := 
        H^0_!(
            \level{K_r}{\ol{\Sh}},
            \w_{\kappa, r, \MMM}
        )\,.
    \]
    
    In particular, $f \in M_\kappa(K_r, \MMM; R)$ can be viewed as a functorial rule on the set of tuples $(\ul{A}, \phi, \epsilon) \in \EE_r(S)$, for any $R$-algebra $S$, such that
    \[
        f(\ul{A}, \phi, \epsilon) \in V_{\kappa, S} \otimes \MMM^\vee_{S}
    \]
    and
    \[
        f(\ul{A}, \phi, l_0\epsilon) 
            = 
        \rho_\kappa(l_0)f(\ul{A}, \phi, \epsilon)\,.
    \]
\end{definition}

\begin{remark} \label{rmk:class of B-nbtp is usual theory}
    When working with $P = B$, as in Remark \ref{rmk:trivial partition}, then $M_\kappa(K_r, [\tau]; R) = M_\kappa(K_r; R)$ and $S_\kappa(K_r, [\tau]; R) = S_\kappa(K_r; R)$.
\end{remark}

Going back to the representation $\tau$ on $\MMM = \MMM_\tau$, consider an algebra $R$ over $S_0[\tau] := S_0[\MMM_\tau]$. Naturally, $M_\kappa(K_r, \MMM; R)$ contains $M_\kappa(K_r, \tau; R)$ but it also contains $M_\kappa(K_r, \tau'; R)$ for any representation $\tau'$ on $\MMM_{\tau, R}$.

In this work, we are mostly concern with twists of $\tau$ by finite-order characters of $\LL_r$,  all viewed as acting on the same module $\MMM$ (over a sufficiently large ring). This leads to the following definition.

\begin{definition} \label{def:equiv class of P nbtp of lvl r}
    We say that two $P$-nebentype $\tau$ and $\tau'$ of level $r$ are \emph{equivalent}, and write $\tau \sim_r \tau'$, if $\tau = \tau' \otimes \psi$ for some finite-order character $\psi$ of $\LL_r$. We let $[\tau]_r$ denote the (finite) equivalence class of $\tau$ as a $P$-nebentypus of level $r$. This notion obviously depends on $r$ but we sometimes write $[\tau]$ when $r$ is clear from the context.
\end{definition}

For each $r \gg 0$, fix a ring $S_r[\tau]$ large enough to contain $S_0[\tau']$ for all $\tau' \sim_r \tau$. After base change, if necessary, we view $\MMM_\tau$ as the $S_r[\tau]$-module on which $\tau'$ acts, for all $\tau' \sim_r \tau$. To emphasize this convention, we now refer to $\MMM_\tau$ as $\MMM_{[\tau]}$. Similarly, given any $S_r[\tau]$-algebra $R$, we set $\MMM_{[\tau], R} = \MMM_{\tau, R}:= \MMM_\tau \otimes_{S_r[\tau]} R$. Note that the contragredient module $\MMM_{[\tau]}^\vee = \MMM_\tau^\vee$ and the tautological pairing $(\cdot, \cdot)_\tau = (\cdot, \cdot)_{[\tau]}$ on $\MMM_\tau \otimes \MMM_\tau^\vee$ are both well-defined up to equivalence of $P$-nebentype.

Therefore, one readily sees that
\begin{equation} \label{eq:direct sum of twists of a type}
    M_\kappa(K_r, [\tau]; R) := \bigoplus_{\tau' \in [\tau]_r} M_\kappa(K_r, \tau'; R)\,.
\end{equation}
is a subspace of $M_\kappa(K_r, \MMM; R)$.

\begin{remark}
    One similarly defines $S_\kappa(K_r, [\tau]; R)$ and $\w_{\kappa, r, [\tau]}$. We refer to $f \in M_\kappa(K_r, [\tau]; R)$ (resp. $S_\kappa(K_r, [\tau]; R)$) as a modular (resp. cusp) form over $R$ on $G$ of weight $\kappa$, level $K_r$ and \emph{$P$-type class} $[\tau]$.
\end{remark}

\begin{remark}
    In general, $M_\kappa(K_r, \MMM; R)$ is strictly larger than $M_\kappa(K_r, [\tau]; R)$. Indeed, if $\psi$ and $\psi'$ are two characters of $\LL_r$ that are congruent modulo $p$, and $f \in M_\kappa(K_r, \tau; R)$, then
    \begin{equation} \label{eq:detecting congruences}
        \frac{1}{p}(f \otimes \psi - f \otimes \psi')
    \end{equation}
    lies in $M_\kappa(K_r, [\tau]; R)$ but not in the direct sum of \eqref{eq:direct sum of twists of a type}.
\end{remark}

\begin{remark}
    In all that follows, we almost exclusively work with $M_\kappa(K_r, [\tau]; R)$. Effectively, in Section \ref{sec:P-(anti-)ord Hida families}, this leads us to consider \emph{$P$-ordinary Hida families}, viewed as closed subschemes of the spectrum of certain \emph{$P$-ordinary Hecke algebras}, containing a dense set of classical points. This set of classical points corresponds to \emph{$P$-ordinary automorphic representations} whose \emph{$P$-nebentypus at $p$} are members $\tau' \in [\tau]$ that all congruent modulo $p$. Although there are additional details omitted in this comment, the notions above are all defined properly later in the text. See \eqref{eq:def SSS Kp pi} for a concrete description of this set of classical points.

    It is certainly interesting to work with $M_\kappa(K_r, \MMM; R)$ instead. In this case, one obtains larger Hida family whose dense set of classical points corresponds to all \emph{$P$-ordinary automorphic representations} whose \emph{$P$-nebentypus at $p$} are all representations $\tau'$ on $\MMM$ that are congruent modulo $p$. These families are sensitive to the existence of congruences as in \eqref{eq:detecting congruences}. 
    
    However, our computations in this paper are only worked out when the types are all in the same $P$-class, i.e. twists of each other by finite-order characters. The author hopes to generalize the necessary computation in later work to consider these larger families.
\end{remark}

\subsection{Complex Uniformization.} \label{subsec:complex unif}
The coherent cohomology group defining the various spaces of algebraic modular forms introduced in the previous sections can be computed with Lie algebra cohomology groups, at least over $\CC$. 

\subsubsection{Complex structure.} \label{subsubsec:complex structure}
Recall that $X$ denotes the $G(\RR)$-conjugacy class of $h$. Let $C \subset G_{/\RR}$ denote the centralizer of $h$, so that there is a natural identification $G(\RR)/C(\RR) \xrightarrow{\sim} X$. In particular, this induces a structure of a real manifold on $X$. In what follows, we set $U_\infty := C(\RR)$.

Furthermore, recall that under the identification of $G_\CC$ with $G_{0\,/\CC}$ from Section \ref{subsec:structure of G over C}, $P_h(\CC) \subset G(\CC)$ corresponds to $P_0(\CC)$, and $C(\CC)$ corresponds to $H_0(\CC)$. It is well-known that $X$ then corresponds to an open subspace of $G_0(\CC)/P_0(\CC)$ and hence also admits the structure of a complex manifold.

Let $r \geq 0$ and $K^p \subset G(\AA_f^p)$ be a neat open compact subgroup. Let $\Sh = \Sh(V)$ be the pro-finite tower of Shimura varieties associated to $G$.

Given $(h', g) \in X \times G(\AA_f)$, with $g_p \in G(\ZZ_p)$, one can naturally define a tuple 
\[
    X_{h', g} = (A_{h'}, \lambda_{h'}, \iota_{h'}, \alpha_{g}, \phi_g) 
        \in 
    \level{K_r}{\Sh}(\CC)
\]
as well as an $\OO \otimes \CC$-isomorphism $\varepsilon_{h'} : \w_{A_{h'}^\vee} \xrightarrow{\sim} \Lambda_0 \otimes_{S_0} \CC$. The precise descriptions of $X_{h', g}$ and $\varepsilon_{h'}$ plays no role in what follows, see \cite[Sections 2.7.1-2.7.2]{EHLS} for details.

In fact, the map $(h', g) \to X_{h', g}$ provides a bijection
\begin{equation} \label{complex points of Sh(V)}
    G(\QQ) 
        \backslash 
    G(\RR) \times G(\AA_f) 
        / 
    U_\infty K_r
        =
    G(\QQ) 
        \backslash 
    X \times G(\AA_f)
        / 
    K_r
        \xrightarrow{\sim}
    \level{K_r}{\Sh}(\CC) 
\end{equation}
which identifies the complex analytic structures on both side. In particular, the dimension $d$ of $\Sh(V)$ is just the $\CC$-dimension of $X$, i.e.
\[
    d 
        = 
    \sum_{\sigma \in \Sigma_\KK} 
        a_\sigma b_\sigma\,.
\]

\subsubsection{Complex modular forms.} \label{subsubsec:complex modular forms}
Similarly, there is an identification
\begin{equation} \label{complex points of EE r}
    G(\QQ) 
        \backslash 
    G(\RR) \times H_0(\CC) \times G(\AA_f) 
        / 
    U_\infty K_r
        \xrightarrow{\sim}
    \EE_r(\CC)
\end{equation}
given by sending $(h', h_0, g) \in X \times H_0 \times G(\AA_f)$ to $(X_{h', g}, (h_0 \cdot \varepsilon_{h'}, \nu(h_0)))$.

Hence, according to \eqref{eq:def of M kappa Kr R}, given an dominant character $\kappa$ of $T_{H_0}(\CC)$, a modular form $\varphi \in M_\kappa(K_r; \CC)$ is a smooth holomorphic $\CC$-valued function on $G(\RR) \times H_0(\CC) \times G(\AA_f) = G(\AA) \times H_0(\CC)$ such that
\[
    \varphi(\gamma g u k, b h_0 u) = \kappa(b)\varphi(g, h_0)\,,
\]
for all $\gamma \in G(\QQ)$, $g \in G(\AA)$, $u \in U_\infty$, $k \in K_r$, $b \in B_{H_0}(\CC)$ and $h_0 \in H_0(\CC)$. 

Similarly, let $(\tau, \MMM_\tau)$ be a $P$-nebentypus of level $r$ over $\CC$ and view it as a representation of $K_r^0$ that factors through $K_r$. A modular form $\varphi \in M_\kappa(K_r, \tau; \CC)$ can be viewed as a smooth holomorphic function $\varphi : G(\AA) \times H_0(\CC) \to \MMM_\tau^\vee$ such that
\[
    \varphi(\gamma g u k, b h_0 u) = \kappa(b) \tau^\vee(k) \varphi(g, h_0)\,,
\]
for all $k \in K_r^0$, $\gamma \in G(\QQ)$, $g \in G(\AA)$, $u \in U_\infty$, $k \in K_r$, $b \in B_{H_0}(\CC)$ and $h_0 \in H_0(\CC)$. 

\subsubsection{Lie algebra cohomology} \label{subsubsec:Lie algebra cohomology}
We now reinterpret the above using the algebraic representations $W_\kappa$ of $H_0$ associated to $\kappa$ as in Section \ref{subsubsec:alg weights}. 

Let $\g = \Lie(G(\RR))_{\CC} = \Lie(G(\CC))$ and write
\[
    \g = \p_h^- \oplus \k_h \oplus \p_h^+
\]
for the Harish-Chandra decomposition corresponding to the $-1$, $0$ and $1$ eigenspaces of the involution $\ad h(\sqrt{-1})$ respectively. 

Then, $\k_h = \Lie(U_\infty)$ and $\P_h = \p_h^- \oplus \k_h = \Lie(P_h(\RR))_{\CC} = \Lie(P_h(\CC))$. Note that a function $\varphi$ as above is holomorphic (with respect to the complex structure on $G(\RR)/U_\infty$) if and only if it vanishes under the action of $\p_h^-$.

In what follows, when considering Lie algebra cohomology, we write $K_h$ for $U_\infty = C(\RR)$ so that $\k_h = \Lie(K_h)$. 

The Borel-Weil theorem states that the set of $\CC$-points of $W_\kappa$ is
\[
    W_\kappa(\CC)
        =
    \{
        \phi : H_0(\CC) \to \CC
            \mid
        \phi \text{ is holomorphic and } \phi(bx) = \kappa(b) \phi(x),\,\forall\,x \in B_{H_0}(\CC)
    \}\,,
\]
which we view as a $(\P_h, K_h)$-module, under the identification of $P_0(\CC)$ with $P_h(\CC)$ and $H_0(\CC)$ with $C(\CC)$.

It is well-known that over $\CC$, one has a natural $G(\AA_f)$-equivariant isomorphism
\begin{equation} \label{eq:Lie alg coh for H!i Sh V kappa}
    H^i(\P_h, K_h; \Ab_0(G) \otimes W_\kappa) 
        = 
    H_!^i(\Sh(V), \w_\kappa)\,, \text{ for } i = 0 \text{ or } d\,,
\end{equation}
where $\Ab_0(G)$ denotes the space of complex-valued cusp forms on $G(\AA)$.

Taking $i = 0$ and $K_r$-equivariance on both sides of \eqref{eq:Lie alg coh for H!i Sh V kappa}, one obtains 
\begin{equation} \label{eq:Lie alg coh for S kappa Kr CC}
    H^0(\P_h, K_h; \Ab_0(G) \otimes W_\kappa)^{K_r} = S_\kappa(K_r; \CC)\,,
\end{equation}
which identifies $\varphi \in S_\kappa(K_r; \CC)$ as a above with a function $f : G(\AA) \to W_\kappa(\CC)$ such that $f(\gamma g u k) = u^{-1}f(g)$, for all $\gamma \in G(\QQ)$, $g \in G(\AA)$, $u \in K_h$ and $k \in K_r$. The correspondence is given by $f(g)(x) = \varphi(g, x)$, for all $(g, x) \in G(\AA) \times H_0(\CC)$.

Similarly, taking tensor with $\MMM_\tau^\vee$ over $\CC$ (momentarily forgetting the action of $\LL_r$), we obtain an isomorphism
\begin{equation} \label{eq:Lie alg coh for H!i Sh V kappa r [tau]}
    \hom_{\CC}(
        \MMM_{[\tau]},
        H^i(\P_h, K_h; \Ab_0(G) \otimes W_\kappa) 
    )
        = 
    H_!^i(\Sh(V)_{/\CC}, \w_{\kappa, [\tau]})\,,
\end{equation}
and taking tensor over $(\LL_r, \tau^\vee)$ instead, we obtain
\begin{equation} \label{eq:Lie alg coh for H!i Sh V kappa r tau}
    \hom_{\LL_r}(
        \MMM_\tau,
        H^i(\P_h, K_h; \Ab_0(G) \otimes W_\kappa) 
    )
        = 
    H_!^i(\Sh(V)_{/\CC}, \w_{\kappa, \tau})\,,
\end{equation}

Since $L_H(\ZZ_p)$ normalizes $K_r$, taking $i=0$ as well as $K_r$-equivariance, we obtain
\begin{equation} \label{eq:Lie alg coh for S kappa Kr [tau] CC}
    \hom_{\CC}(
        \MMM_{[\tau]},
        H^0(\P_h, K_h; \Ab_0(G) \otimes W_\kappa)
    )
        = 
    S_\kappa(K_r, [\tau]; \CC)\,,
\end{equation}
and
\begin{equation} \label{eq:Lie alg coh for S kappa Kr tau CC}
    \hom_{K_r^0}(
        \MMM_\tau,
        H^0(\P_h, K_h; \Ab_0(G) \otimes W_\kappa)
    )
        = 
    S_\kappa(K_r, \tau; \CC)\,.
\end{equation}

Then as above, $\varphi_{[\tau]} \in S_\kappa(K_r, [\tau]; \CC)$ and $\varphi_\tau \in S_\kappa(K_r, \tau; \CC)$ corresponds via \eqref{eq:Lie alg coh for S kappa Kr tau CC} to functions $f_{[\tau]}, f_\tau : G(\AA) \to W_\kappa(\CC) \otimes_{\CC} \MMM_\tau^\vee$, respectively, such that
\[
    f_{[\tau]}(\gamma g u k) = u^{-1}f(g)
        \;\;\;\text{and}\;\;\;
    f_\tau(\gamma g u k_0) = \tau^\vee(k)(u^{-1}f(g))\,,
\]
for all $\gamma \in G(\QQ)$, $g \in G(\AA)$, $u \in K_h$, $k \in K_r$, and $k_0 \in K^0_r$.

\subsection{Duality and integrality.} \label{subsec:duality and integrality}
In the previous sections, we mostly dealt with holomorphic modular forms, i.e. degree $0$ cohomology. For our purposes however, following the approach of \cite[Sections 6-7]{EHLS}, it is necessary to deal with anti-holomorphic modular forms as well, i.e. degree $d$ cohomology where $d = \sum_\sigma a_\sigma b_\sigma$.

\subsubsection{Convention on measures} \label{subsubsec:conv on measures}
In the following sections, we introduce pairings between modular forms via integration over $G(\AA)$. We first need to set various conventions.

We fix a Haar measure $dg = \prod_{l \leq \infty} dg_l$ on $G(\AA)$, where the product runs over all places of $\QQ$, such that the following properties hold :
\begin{enumerate}
    \item Given a finite prime $l$ such that $G$ is unramified at $l$, $dg_l$ is the normalized Haar measure on $G(\QQ_l)$ assigning volume 1 to any hyperspecial maximal compact subgroup.
    
    \item Given a finite prime $l$ such that $G$ splits over $l$ (e.g. when $l = p$), i.e. $G(\QQ_l) = \prod_{i = 1}^k \GL_{n_i}(F_{v_i})$ where $F_{v_i}$ is a finite extension of $\QQ_l$ with ring of integer $\OO_{v_i}$, then $dg_l$ is normalized so that $\prod_{i = 1}^k \GL_{n_i}(\OO_{v_i})$ has volume 1. In this case, we further write $dg_l = \prod_{i=1}^k dg_{v_i}$, where $dg_{v_i}$ is the standard Haar measure on $\GL_{n_i}(F_{v_i})$ with the obvious normalization.
    
    \item At all finite primes $l$, the volume of any compact open subgroup of $G(\QQ_l)$ with respect to $dg_l$ is rational.
    
    \item For $l = \infty$, $dg_\infty$ is Tamagawa measure on $G(\RR)$. We write 
    \[
        dg_\infty = dk_h \times dx \times dt/t\,,
    \]
    where $d k_h$ is the unique measure on $K_h$ with total mass equal to 1, $dx$ is a differential form on $\p_h$ and $dt/t$ is the Lebesgue measure on the center $Z_G(\RR) \simeq \RR^\times$ of $G(\RR)$.
\end{enumerate}

\subsubsection{Unnormalized and normalized Serre duality.} \label{subsubsec:unnorm and norm Serre duality}
We first work with $R = \CC$ and introduce an integral version of Serre duality afterward.

Let $\kappa = (\kappa_0, (\kappa_\sigma)_\sigma)$ be a dominant weight of $T_{H_0}$, as in Section \ref{subsubsec:alg weights}. Define
\[
    a(\kappa) 
        := 
    2\kappa_0 + 
    \sum_\sigma
    \sum_{j=1}^{b_\sigma}
        \kappa_{\sigma, j}
    \ \ ; \ \
    \kappa_0^*
        :=
    -\kappa_0 + a(\kappa)
    \ \ ; \ \
    \kappa_\sigma^*
        :=
    (
        -\kappa_{\sigma, b_\sigma},
        \ldots,
        -\kappa_{\sigma, 1}
    )\,.
\]

The highest weight representation $W_{\kappa^*}$ corresponding to the dominant weight $\kappa^* = (\kappa_0^*, (\kappa_\sigma^*)_{\sigma})$ is a representation of $H_0(\CC)$ such that
\begin{equation} \label{eq:def W kappa *}
    W_{\kappa^*} \cong W_\kappa^\vee \otimes \nu^{a(\kappa)}\,,  
\end{equation}
where we recall that $\nu$ denotes the similitude character of $G$.

As briefly mentioned above, we later work with (Lie algebra) cohomology in degree $d$, hence we are most interested in the $H_0(\CC)$-representation
\[
    \hom_\CC(\wedge^d \p_h^+, W_{\kappa^*})\,,
\]
where the notation is as in Section \ref{subsec:complex unif}.

Since $\hom_\CC(\wedge^d \p_h^+, \CC)$ is the highest weight representation associated to the dominant weight
\[
    \kappa_h^+ := 
    (
        -d, (\kappa_{h, \sigma}^+)_\sigma
    )\,,
\]
where $\kappa_{h, \sigma}^+ = (2a_\sigma, \ldots, 2a_\sigma)$, it follows that
\[
    \hom_\CC(\wedge^d \p_h^+, W_{\kappa^*}) 
        =
    \hom_\CC(\wedge^d \p_h^+, \CC)
        \otimes
     W_{\kappa^*}
\]
is canonically isomorphic to $W_{\kappa^D}$, where
\[
    \kappa^D = \kappa^* + \kappa_h^+\,.
\]

The natural contraction $W_\kappa \otimes W_\kappa^\vee \to \CC$ induces a pairing
\begin{equation} \label{eq:pairing kappa kappa D}
    W_\kappa \otimes W_{\kappa^D} 
        \to
    \hom_\CC(
        \wedge^d \p_h^+, 
        \CC
    )
        \otimes
    \nu^{a(\kappa)}\,,
\end{equation}
by definition of $W_{\kappa^D}$.

Let $L(\kappa)$ denote the automorphic line bundle over $\Sh(V)$ associated to the character $\nu^{a(\kappa)}$, as in Remark \ref{rmk:aut vec bundles from reps}. Namely, $L(\kappa)$ is topologically isomorphic to $\OO_{\Sh(V)}$ but the action of $G(\AA_f)$ on $L(\kappa)$ is given by multiplication via $\nu^{a(\kappa)}$. Then, \eqref{eq:pairing kappa kappa D} induces a map
\[
    \w_\kappa \otimes \w_{\kappa^D}
        \to
    \Omega_{\Sh(V)}^d
        \otimes 
    L(\kappa)\,.
\]

Naturally, one can descend this pairing to $\level{K}{\Sh}$. Furthermore, this pairing extends over any toroidal compactification $\level{K}{\Sh}_{\Omega}$ of $\level{K}{\Sh}$, provided either automorphic vector bundle is replaced by its subcanonical vector bundle. Namely, we have
\[
    \w_\kappa^{\sub} \otimes \w_{\kappa^D}
        \to
    \Omega_{\level{K}{\Sh}_\Omega}^d
        \otimes 
    L(\kappa)
    \;\;\;\text{and}\;\;\;
    \w_\kappa \otimes \w_{\kappa^D}^{\sub}
        \to
    \Omega_{\level{K}{\Sh}_\Omega}^d
        \otimes 
    L(\kappa)\,.
\]

Then, the \emph{unnormalized} Serre duality pairing is the composition of
\[
    H_!^0(\Sh(V), \w_\kappa)
        \otimes
    H_!^d(\Sh(V), \w_{\kappa^D})
        \to
    \varinjlim_{K, \Omega}
        H^d(
            \level{K}{\Sh}_\Omega,
            \Omega_{\level{K}{\Sh}_\Omega}^d
                \otimes 
            L(\kappa)
        )
\]
with the isomorphism
\[
    \varinjlim_{K, \Omega}
        H^d(
            \level{K}{\Sh}_\Omega,
            \Omega_{\level{K}{\Sh}_\Omega}^d
                \otimes 
            L(\kappa)
        )
            \xrightarrow{\sim}
    \varinjlim_{K, \Omega}
        H^d(
            \level{K}{\Sh}_\Omega,
            \Omega_{\level{K}{\Sh}_\Omega}^d
        )
\]
given by multiplication by the global section $g \mapsto ||\nu(g)||^{a(\kappa)}$ of $L(\kappa)^\vee$, and the maps
\[
    \varinjlim_{K, \Omega}
        H^d(
            \level{K}{\Sh}_\Omega,
            \Omega_{\level{K}{\Sh}_\Omega}^d
        )
            \xrightarrow{\sim}
    C(\pi_0(V))
        \to
    \CC\,,
\]
where $C(\pi_0(V))$ is the space of functions on the compact space $\pi_0(V)$ of similitude components of $\Sh(V)$, the isomorphism is the trace map, and the last map is integration over $\pi_0(V)$ with respect to an invariant measure of rational total mass.

The resulting map
\begin{equation} \label{eq:def pairing Ser kappa}
    \brktdotdot_\kappa^{\Ser}
        :
    H_!^0(\Sh(V), \w_\kappa)
        \otimes
    H_!^d(\Sh(V), \w_{\kappa^D})
        \to
    \CC
\end{equation}
is a canonical perfect pairing by \cite[Corollary 2.3]{Har90}. As explained above, we can replace either $H^0_!$ or $H^d_!$ by $H^0$ and $H^d$ respectively, but not both at once.

From Definition \ref{def:M kappa(Kr, [tau],R)} and our discussion in Section \ref{subsec:complex unif}, we see that given a $P$-nebentypus $\tau$, the tautological pairing $(\cdot, \cdot)_{[\tau]} : \MMM_{[\tau]} \otimes \MMM_{[\tau]}^\vee \to \CC$ yields a map
\begin{align*}
    H_!^0(\Sh&(V), \w_{\kappa, [\tau]})
        \otimes_\CC
    H_!^d(\Sh(V), \w_{\kappa^D, [\tau^\vee]})
        \\ &\xrightarrow{(\cdot, \cdot)_{[\tau]}}
    H_!^0(\Sh(V), \w_{\kappa})
        \otimes_\CC
    H_!^d(\Sh(V), \w_{\kappa^D}),.
\end{align*}

Hence, composition of the above with $\brktdotdot_\kappa^\Ser$ induces a duality
\begin{equation} \label{eq:def pairing Serre kappa [tau]}
    \brktdotdot_{\kappa, [\tau]}^\Ser
        :
    H_!^0(\Sh(V), \w_{\kappa, [\tau]})
        \otimes_\CC
    H_!^d(\Sh(V), \w_{\kappa^D, [\tau^\vee]})
        \to
    \CC\,.
\end{equation}

Upon restriction to holomorphic and anti-holomorphic modular forms of $P$-nebentypus $\tau$ and $\tau^\vee$ respectively, we similarly obtain a perfect pairing
\begin{equation} \label{eq:def pairing Serre kappa tau}
    \brktdotdot_{\kappa, \tau}^\Ser
        :
    H_!^0(\Sh(V), \w_{\kappa, \tau})
        \otimes_\CC
    H_!^d(\Sh(V), \w_{\kappa^D, \tau^\vee})
        \to
    \CC\,.
\end{equation}

For our purposes, it is important to compute $\brktdotdot_{\kappa}^{\Ser}$ above in terms of automorphic forms using the isomorphism \eqref{eq:Lie alg coh for H!i Sh V kappa}. Let $\varphi \in H^0(\P_h, K_h; \Ab_0(G) \otimes W_\kappa)$ and $\varphi' \in H^d(\P_h, K_h; \Ab_0(G) \otimes W_{\kappa^D})$. For every $g \in G(\QQ)Z_G(\RR) \backslash G(\AA)$, we have $\varphi(g) \in W_\kappa$ and 
\[
    \varphi'(g) 
        \in 
    \hom_\CC(
        \wedge^d \p_h^-,
        W_{\kappa^D}
    )
        = 
    \hom_\CC(
        \wedge^{2d} \p_h,
        W_\kappa^\vee
            \otimes
        \nu^{a(\kappa)}
    )\,,
\]
where $\p_h = \p_h^- \oplus \p_h^+$. 

Fix a differential form $dx$ on $\p_h$ as in Section \ref{subsubsec:conv on measures}, i.e. a basis of $(\wedge^{2d} \p_h)^\vee$. Using $dx$, we identify this space with $\CC$ and obtain a natural map
\[
    [\cdot, \cdot]_{dx} 
        : 
    W_\kappa
        \otimes
    \hom_\CC(
        \wedge^d \p_h^-,
        W_{\kappa^D}
    )
        \to
    \CC(\nu^{a(\kappa)})
\]

Then, the Serre pairing $\brktdotdot_\kappa^{\Ser}$ can be normalized (i.e. the invariant measure on $\pi_0(V)$ can be normalized) so that
\begin{equation} \label{eq:Serre pairing aut forms}
    \brkt{\varphi}{\varphi'}_{\kappa}^{\Ser}
        =
    \int_{G(\QQ)Z_G(\RR) \backslash G(\AA)}
        [\varphi(g), \varphi'(g)]_{dx}
        ||\nu(g)||^{-a(\kappa)}
    dg\,,
\end{equation}
where $dg$ is as in Section \ref{subsubsec:conv on measures}. 

\begin{remark}
    We later use this formula when $\varphi$ is essentially a value of the Eisenstein measure from Proposition \ref{prop:existence of Eis measure on V3}, as a modular form on $G_3$, and $\varphi'$ is the tensor product of two $P$-anti-ordinary cusp forms, one on $G_1$ and the other on $G_2$. The groups $G_1$, $G_2$ and $G_3$ are defined in Section \ref{sec:comp and comp bw PEL data}.
\end{remark}

To define a normalized version of \eqref{eq:def pairing Ser kappa}, fix a compact open subgroup $K_r = I_rK^p \subset G(\AA_f)$. Denote the volume of $K_r^0 = I_r^0K_p$ with respect to the Tamagawa measure $dg$ above by $\vol(I^0_r)$. 

Then, the \emph{normalized} Serre pairing is the perfect pairing
\begin{equation} \label{eq:def pairing kappa Kr}
    \brktdotdot_{\kappa, K_r}
        :
    H^0_!(\level{K_r}{\Sh(V)}, \w_\kappa)
        \otimes
    H^d_!(\level{K_r}{\Sh(V)}, \w_{\kappa^D})
        \to \CC
\end{equation}
defined via $\brktdotdot_{\kappa, K_r} = \vol(I_r)^{-1} \brktdotdot_{\kappa}^{\Ser}$. Similarly, we set 
\[
    \brktdotdot_{\kappa, K_r, \tau} 
        = 
    \vol(I_r)^{-1} \brktdotdot_{\kappa, \tau}^{\Ser}
        \;\;\;\text{and}\;\;\; 
    \brktdotdot_{\kappa, K_r, [\tau]} 
        = 
    \vol(I_r)^{-1} \brktdotdot_{\kappa, [\tau]}^{\Ser}\,.
\]

The advantage of this normalized pairing is that it commutes with change of level maps. Namely, fix $r' \geq r$, $\varphi \in H_!^0(\level{K_r}{\Sh(V)}, \w_\kappa)$ and $\varphi' \in H_!^d(\level{K_{r'}}{\Sh(V)}, \w_{\kappa^D})$. The trace map from level $K_{r'}$ to level $K_r$ maps $\varphi'$ to
\begin{equation} \label{eq:def of trace map from level r' to r}
    \tr_{K_r/K_{r'}}(\varphi') 
        := 
    \frac{
        \#(I_r^0/I_r)
    }{
        \#(I_{r'}^0/I_{r'})
    } 
    \sum_{\gamma \in K_r/K_{r'}}
        \gamma \cdot \varphi'
        \in
    H_!^d(\level{K_r}{\Sh(V)}, \w_{\kappa^D})
\end{equation}

By definition, we have $I_r^0/I_r \simeq K_r^0/K_r$ and
\[
    \vol(I_r^0)
        =
    \vol(I_{r'}^0)
        \cdot
    \#(I_r^0/I_{r'}^0)
        =
    \vol(I_{r'}^0)
        \cdot
    \frac{
        \#(I_r^0/I_r)
    }{
        \#(I_{r'}^0/I_{r'})
    }
        \cdot
    \#(K_r/K_{r'})\,,
\]
therefore one readily obtains
\begin{equation} \label{eq:norm Serre pairing trace map stable}
    \brkt{
        \varphi
    }{
        \varphi'
    }_{\kappa, K_{r'}}
        =
    \brkt{
        \varphi
    }{
        \tr_{K_r/K_{r'}}(\varphi')
    }_{\kappa, K_r}\,,
\end{equation}
as well as an analogous formula when $\varphi$ has level $K_{r'}$ and $\varphi'$ has level $K_r$.

From \eqref{eq:Lie alg coh for H!i Sh V kappa r tau} and the fact that $L_{H(\ZZ_p)}$ normalizes $K_r$, one readily sees that the formula \eqref{eq:def of trace map from level r' to r} is also well-defined on $H_!^d(\level{K_{r'}}{\Sh(V)}, \w_{\kappa^D, r, \tau^\vee})$ and $H_!^d(\level{K_{r'}}{\Sh(V)}, \w_{\kappa^D, r, [\tau^\vee]})$, for $r' > r$, and yields a trace maps
\begin{equation} \label{eq:norm Serre pairing trace map stable with types}
    \tr_{K_r/K_{r'}}
        :
    H_!^d(\level{K_{r'}}{\Sh(V)}, \w_{\kappa^D, r', \tau^\vee})
        \to
    H_!^d(\level{K_r}{\Sh(V)}, \w_{\kappa^D, r, \tau^\vee})\,.
\end{equation}
and
\begin{equation} \label{eq:norm Serre pairing trace map stable with [types]}
    \tr_{K_r/K_{r'}}
        :
    H_!^d(\level{K_{r'}}{\Sh(V)}, \w_{\kappa^D, r', [\tau^\vee]})
        \to
    H_!^d(\level{K_r}{\Sh(V)}, \w_{\kappa^D, r, [\tau^\vee]})\,.
\end{equation}

It follows from \eqref{eq:norm Serre pairing trace map stable} that the pairings $\brktdotdot_{\kappa, K_r, \tau}$ and $\brktdotdot_{\kappa, K_r, [\tau]}$ are again stable under trace maps.

\subsubsection{Integral structures on (anti-)holomorphic modular forms.} \label{subsubsec:integral struct on aholo mod forms}
Recall that $S_0 = \OO_{\KK', (\p')}$ as in Section \ref{subsubsec:alg weights}. Naturally, we define $S_0$-integral structure of the $\CC$-vector space $H^0_!(\level{K_r}{\Sh(V)}_{/\CC}, \w_\kappa)$ as $S_\kappa(K_r, S_0)$. More generally, for any $S_0$-algebra $R$, we set its $R$-structure to be $S_\kappa(K_r, R)$. This is obviously the structure induced by the $R$-structure of $\level{K_r}{\Sh(V)}$.

On the other hand, we do not define the $R$-integral structure for the space of \emph{anti-holomorphic} forms as the one induced by the $R$-structure of the underlying schemes. This is to avoid the singularities of the special fibers of $\level{K_r}{\Sh(V)}_{S_0}$ as $r$ grows. We instead use duality with respect to $\brktdotdot_{\kappa, K_r}$, following the approach of \cite[Section 6.4.2]{EHLS}.

Firstly, motivated by the identification \eqref{eq:Lie alg coh for H!i Sh V kappa}, we refer to
\[
    \wh{S}_\kappa(K_r; \CC) 
        := 
    H^d_!(\level{K_r}{\Sh(V)}_{/\CC}, \w_{\kappa^D})
\]
as the space of \emph{anti-holomorphic} cusp forms on $G$ of weight $\kappa$ and level $K_r$ over $\CC$ (note the twist by $\kappa^D$). Similarly, given a $P$-nebentypus $\tau$ of level $r$, we set
\[
    \wh{S}_\kappa(K_r, \tau; \CC) 
        := 
    H^d_!(\level{K_r}{\Sh(V)}_{/\CC}, \w_{\kappa^D, r, \tau^\vee})\,.
\]

Then, by definition of $\brktdotdot_{\kappa, K_r}$ and $\brktdotdot_{\kappa, K_r, \tau}$, we have perfect pairings
\[
    S_\kappa(K_r; \CC) 
        \otimes 
    \wh{S}_{\kappa}(K_r; \CC) 
        \to 
    \CC
        \;\;\;\text{and}\;\;\;
    S_\kappa(K_r, \tau; \CC) 
        \otimes 
    \wh{S}_{\kappa}(K_r, \tau; \CC) 
        \to 
    \CC\,,
\]
and 

We define the $S_0$-integral structure $\wh{S}_{\kappa}(K_r; S_0)$ of $H^d_!(\level{K_r}{\Sh}_{/\CC}, \w_{\kappa^D})$ as the $S_0$-dual of $S_\kappa(K_r; S_0)$ via the pairing \eqref{eq:def pairing kappa Kr}. Similarly, we define the $S_0[\tau]$-integral structure $\wh{S}_{\kappa}(K_r, \tau; S_0[\tau])$ of $H^d_!(\level{K_r}{\Sh}_{/\CC}, \w_{\kappa^D, r, \tau^\vee})$ as the $S_0[\tau]$-dual of $S_\kappa(K_r, \tau; S_0[\tau])$ via the pairing $\brktdotdot_{\kappa, K_r, \tau}$.

Given any $S_0$-algebra or $S_0[\tau]$-algebra $R$, let $\wh{S}_{\kappa}(K_r; R) := \wh{S}_{\kappa}(K_r; S_0) \otimes_{S_0} R$ and $\wh{S}_{\kappa}(K_r, \tau; R) := \wh{S}_{\kappa}(K_r, \tau; S_0[\tau]) \otimes_{S_0[\tau]} R$. This yields identifications
\[
    \wh{S}_{\kappa}(K_r; R)
        =
    \hom_{S_0}(
        S_{\kappa}(K_r; S_0),
        R
    )
\]
and
\[
    \wh{S}_{\kappa}(K_r, \tau; R)
        =
    \hom_{S_0[\tau]}(
        S_{\kappa}(K_r, \tau; S_0[\tau]),
        R
    )\,,
\]
via $\brktdotdot_{\kappa, K_r}$ and $\brktdotdot_{\kappa, K_r, \tau}$, respectively. The $R$-integral structure $\wh{S}_\kappa(K_r, [\tau]; R)$ of $\wh{S}_\kappa(K_r, [\tau]; \CC)$ is defined similarly using the pairing $\brktdotdot_{\kappa, K_r, [\tau]}$.

\subsection{$P$-(anti-)ordinary modular forms} \label{subsec:Pord mod forms}
We return to the notation of Section \ref{subsubsec:para subgp of G over Zp}. For instance, let $t_{w, D_w(j)} \in \GL_n(\OO_w)$, for $w \in \Sigma_p$ and $j=1, \ldots, r_w$, be the matrix defined in \eqref{eq:def t w j}, let $t_{w, D_w(j)}^+$ be the corresponding element of $G(\QQ_p)$, and let $t_{w, D_w(j)}^- = (t_{w, D_w(j)}^+)^{-1}$.

Given any $r \geq 1$, we consider the double coset operators
\[
    U_{w, D_w(j)} = [K_r t^+_{w,D_w(j)} K_r]
        \;\;\;\text{and}\;\;\;
    U^-_{w, D_w(j)} = [K_r t^-_{w,D_w(j)} K_r]\,.
\]

One can easily write down a set of right coset representatives for $K_r t^+_{w,D_w(j)} K_r$ (resp. $K_r t^-_{w,D_w(j)} K_r$) that does not depend on $r$, (see Section \ref{subsubsec:explicit coset representatives} for instance). This partly motivates why we omit $r$ from the notation $U_{w, D_w(j)}$ (resp. $U^-_{w, D_w(j)}$).

If $R$ is an $S_0$-algebra in which $p$ is invertible and $\kappa$ be a dominant character of $T_{H_0}$ over $R$, then both $U_{w, D_w(j)}$ and $U^-_{w, D_w(j)}$ define Hecke operators on $M_\kappa(K_r; R)$ and $S_\kappa(K_r; R)$ via \eqref{eq:def Hecke op Tg level Kr} by proceeding as in Section \ref{subsubsec:Hecke ops away from p}. Naturally, they also both define Hecke operators on $M_\kappa(K_r, \tau; R)$ and $S_\kappa(K_r, \tau; R)$ (if $R$ is an $S_0[\tau]$-algebra).

One usually normalize these operators as follows : First, define
\begin{equation} \label{eq:def kappa norm sigma}
    \kappa_{\norm, \sigma} 
        = 
    (
        \kappa_{\sigma, 1} - b_{\sigma}, 
        \ldots, 
        \kappa_{\sigma, b_{\sigma}} - b_{\sigma}
    ) \in \ZZ^{b_\sigma}\,,
\end{equation}
and consider the character $\kappa_{\norm} = (\kappa_0, (\kappa_{\norm, \sigma})_{\sigma \in \Sigma_\KK})$ of $T_{H_0}$. Let $\kappa'$ denote the $p$-adic weight associated to $\kappa_{\norm}$ as in \eqref{eq:tuple definition of kappa p}, i.e. $\kappa' = (\kappa_{\norm})_p$. 

We define the $j$-th normalized Hecke operators at $p$ of weight $\kappa$ as
\begin{equation} \label{definition u w j}
    u_{w,j}^\pm = u_{w,j,\kappa}^\pm
        := 
    \kappa'(t_{w,j}^\pm)
    U_{w,j}^\pm \ ,
\end{equation}
and the Hecke operators at $p$ of weight $\kappa$ with respect to $P$ as
\begin{equation} \label{definition uP}
    u_{P, p}^\pm = u_{P, p, \kappa}^\pm
        :=
    \prod_{w \in \Sigma_p}
    \prod_{j = 1}^{r_w}
        u_{w,D_w(j),\kappa}^\pm\,.
\end{equation}

\begin{remark}
    When working with $t_{w, j}^+$, we often omit the $+$ superscript in the notation above and simply write $u_{w,j}$ and $u_{P,p}$.
\end{remark}

\begin{remark}
    The operators $u_{w,D_w(j),\kappa}$ can be interpreted as correspondences on the Igusa tower associated to $G$, see \cite[Section 2.9.5]{EHLS}, \cite[Section 8.3.1]{Hid04} or \cite{SU02}. We recall this more formally later when introducing $p$-adic modular forms with respect to $P$, see Section \ref{sec:Pord padic mod forms}. This plays a crucial role to $p$-adically interpolate all operators $u_{w,D_w(j),\kappa}$ as $\kappa$ varies.
\end{remark}

In later section, we also consider the action of the center $Z_P$ of $L_P(\ZZ_p)$ on $M_\kappa(K_r; R)$. Namely, any $t \in Z_P$ naturally induces a correspondence on Shimura varieties via the double coset operator $U_p(t) := [K_r t K_r] = [tK_r]$. Clearly, this action factors through the center $Z_{P, r} = Z_P/p^rZ_P$ of $L_P(\ZZ_p/p^r\ZZ_p)$. As above, we normalize these operators by setting $u_{p, \kappa}(t) := \kappa'(t) U_p(t)$.

\subsubsection{$P$-ordinary case} \label{subsubsec:Pord mod forms case}
The $P$-ordinary subspace of a module is later defined as the subspace on which $u_{P,p}$ acts via a generalized eigenvalue which a $p$-adic unit (when this action is well-defined). This feature can be detected using the following $P$-ordinary Hecke projector.

Assume the ring $R$ above is also a $p$-adic ring, i.e. $R = \varprojlim_{i} R/p^iR$. Then, the action of the limit
\begin{equation} \label{definition eP}
    e_P = e_{P, \kappa} 
        := 
    \varinjlim \limits_{n} 
        u_{P, p, \kappa}^{n!}\,.
\end{equation}
on $M_\kappa(K_r; R)$ induced by the action of $u_{P,p,\kappa}$ is well-defined. Since $u_{P, p, \kappa}$ commutes with $G(\AA^p)$ and $L_P(\ZZ_p)$, it stabilizes $S_\kappa(K_r, \tau)$, $M_\kappa(K_r, [\tau]; R)$ and $S_\kappa(K_r, [\tau]; R)$.

\begin{remark}
    In later sections, we also consider the case where $R \subset \CC$ is a localization of a finite $S_0$-algebra at the maximal prime determined by $\incl_p$ or the completion of such a ring. In this situation, the limit operator $e_P$ is again well-defined.
\end{remark}

It is well-known that the eigenvalues of the generalized eigenspaces for each $u_{w, j, \kappa}$ is a $p$-adic integer. Hence, by definition, $e_P$ acts as the identity on the generalized eigenspace of $M_\kappa(K_r; R)$ associated to eigenvalues with $p$-adic valuation 1 and is 0 on all other generalized eigenspaces. 

We write
\[
    M_\kappa^{\Pord}(K_r; R) := e_{P, \kappa} M_\kappa(K_r; R)
\]
and define $S_\kappa^{\Pord}(K_r; R)$, $M_\kappa^{\Pord}(K_r, [\tau]; R)$ and $S_\kappa^{\Pord}(K_r, [\tau]; R)$ similarly.

\subsubsection{$P$-anti-ordinary case.} \label{subsubsec:Paord mod forms case}
An easy computation shows that
\[
    \brkt{U_{w, D_w(j), \kappa} \varphi}{\varphi'}_{\kappa, K_r}
        =
    \brkt{\varphi}{U^-_{w, D_w(j), \kappa^D} \varphi'}_{\kappa, K_r}\,,
\]
for all $\varphi \in S_\kappa(K_r; R)$, $\varphi' \in \wh{S}_\kappa(K_r; R)$, $w \in \Sigma_p$ and $1 \leq j \leq D_w(j)$. 

Therefore, the kernel of the projection
\[
    \wh{S}_{\kappa}(K_r; R) \to e^-_{P, \kappa^D} \wh{S}_{\kappa}(K_r; R)
\]
is exactly the annihilator (in $\wh{S}_\kappa(K_r; R)$) of $S_\kappa^\Pord(K_r; R)$ via $\brktdotdot_{\kappa, K_r}$. In other words, \eqref{eq:def pairing kappa Kr} induces a perfect pairing
\[
    S_\kappa^\Pord(K_r; R) \otimes \wh{S}^{\Paord}_{\kappa}(K_r; R) \to R\,,
\]
where
\[
    \wh{S}^{\Paord}_{\kappa}(K_r; R) 
        := 
    e^-_{P, \kappa^D} \wh{S}_{\kappa}(K_r; R)
\]
is the \emph{$P$-anti-ordinary} subspace of $\wh{S}_{\kappa}(K_r; R)$. Similarly, one can view 
\[
    \wh{S}^{\Paord}_{\kappa}(K_r, \tau; R) 
        := 
    e^-_{P, \kappa^D} \wh{S}_{\kappa}(K_r, \tau; R)
\]
as the $R$-dual of $S^\Pord_{\kappa}(K_r, \tau; R)$ via $\brktdotdot_{\kappa, K_r, \tau}$.

If $R$ is an $S_r[\tau]$-algebra, a similar statement holds for $\wh{S}_\kappa^{\Paord}(K_r, 
[\tau]; R) = e_{P, \kappa^D}^-\wh{S}_\kappa(K_r, [\tau]; R)$ and $S_\kappa^{\Pord}(K_r, [\tau]; R)$ with respect to the pairing $\brktdotdot_{\kappa, K_r, [\tau]}$.

%%%%%%%%%%%%%%%%%%%%%%%%%%%%%%%%%%%%%%%%%%%%%%%%%%%%%%%%%%%%%%%%%%

\section{$P$-(anti-)ordinary (anti-)holomorphic automorphic representations.} \label{sec:P(anti)ord aut representations}

In this section, we frequently use \eqref{eq:Lie alg coh for H!i Sh V kappa} to pass between the language of automorphic forms in $\Ab_0(G)$ and modular forms as global sections on Shimura varieties. We recall the following convenient notions of holomorphic and anti-holomorphic automorphic representations, following \cite[Section 6.5]{EHLS}. We then define $P$-ordinary and $P$-anti-ordinary automorphic representations and their $P$-nebentypus, motivated by Sections \ref{subsec:P nbtp thy of mod forms}--\ref{subsec:Pord mod forms}.

\subsection{(Anti-)holomorphic automorphic representations.} \label{subsec:(A-)holo auto reps}
We continue with the notation of Section \ref{subsec:complex unif}. Recall that we denote the space of cusp forms on $G$ by $\Ab_0(G)$. 

We refer to the irreducible $(\g, K_h) \times G(\AA_f)$-subrepresentations of $\Ab_0(G)$ as \emph{cuspidal automorphic representations} of $G$. In particular, we always assume that a cuspidal automorphic representation $\pi$ is irreducible.

Furthermore, we write $\pi = \pi_\infty \otimes \pi_f$ where $\pi_\infty$ is an irreducible $(\g, K_h)$-module and $\pi_f$ is an irreducible $G(\AA_f)$ admissible representation. 

Let $\kappa$ be a dominant character of $T_{H_0}$, as in Section \ref{subsubsec:alg weights}.

\begin{definition} \label{def:(anti)holo aut reprn}
    We say that $\pi$ is \emph{holomorphic of weight $\kappa$} if
    \[
        H^0(\P_h, K_h; \pi \otimes W_\kappa) \neq 0
    \]
    and say that $\pi$ is \emph{anti-holomorphic of weight $\kappa$} if
    \[
        H^d(\P_h, K_h; \pi \otimes W_{\kappa^D}) \neq 0
    \]
    instead. Clearly, $\pi$ cannot be both holomorphic and anti-holomorphic (except possibly if $d = 0$).    
\end{definition}

Equivalently, $\pi$ is holomorphic of weight $\kappa$ if and only if 
\[
    H^0(\P_h, K_h; \pi_\infty \otimes W_\kappa) \neq 0\,,
\]
in which case the latter is 1-dimensional. Similarly, $\pi$ is anti-holomorphic of weight $\kappa$ if and only if 
\[
    H^d(\P_h, K_h; \pi_\infty \otimes W_{\kappa^D}) \neq 0
\]
is 1-dimensional over $\CC$.

\begin{remark} \label{rmk:E pi rational}
    As mentioned in \cite[Remark 6.5.2]{EHLS}, if $\pi$ is holomorphic or anti-holomorphic, then $\pi_f$ is defined over some number field $E(\pi)$. One can choose $E(\pi)$ to be a CM field. See \cite{BHR94} for more details. Enlarging $E(\pi)$ if necessary, we always assume it contains $\KK'$.
\end{remark}

\subsubsection{Ramified places away from $p$} \label{subsubsec:ramified places away from p}
Let $\pi = \pi_\infty \otimes \pi_f$ be any cuspidal automorphic representation of $G$. Let $K \subset G(\AA_f)$ be any open compact subgroup such that $\pi_f^K \neq 0$. We sometimes say that $\pi_f$ (or $\pi$) has level $K$ in this case.

Let $l \neq p$ be any prime of $\QQ$ and consider the set $\PP_l$ of all primes of $\KK^+$ above $l$. Write $\PP_l = \PP_{l,1} \coprod \PP_{l,2}$, where $\PP_{l,1}$ is the subset of such primes that split in $\KK$ and $\PP_{l,2}$ is the complement. Therefore, one naturally has an identification
\[
    G(\QQ_l) = \prod_{v \in \PP_{l,1}} \GL_n(\KK_v^+) \times G_{l,2} \ ,
\]
where $G_{l,2}$ is the subgroup of elements $((x_w), t) \in \prod_{w \in \PP_{l,2}} \GL_n(\KK_w) \times \QQ_l^\times$ such that each $x_w$ preserve the Hermitian form on $V \times_{\KK} \KK_w$ with the same similitude factor $t$.

Let $S_l = S_l(K_l)$ be the subset of $\PP_l$ consisting of all places at which $K_l$ does not contain a hyperspecial subgroup. Let $S_{l, i} = S_l \cap \PP_{l, i}$ and define
\[
    G(\QQ_l)^{S_l} = 
    \begin{cases}
        \prod_{
            v \in 
                \PP_{l, 1} \backslash S_{l,1}
        } 
            \GL_n(\KK_v^+) \times G_{l, 2} \,, 
        & 
            \text{if } S_{l, 2} = \emptyset \\
        \prod_{
            v \in 
                \PP_{l, 1} \backslash S_{l,1}
        } 
            \GL_n(\KK_v^+) \,, 
        & 
            \text{otherwise.}
    \end{cases}
\]

In particular, we simply have $G(\QQ_l)^{S_l} = G(\QQ_l)$ if $S_l$ is empty. Then, let $S = S(K^p)$ be the set of primes $l \neq p$ such that $S_l$ is nonempty. We define 
\begin{equation} \label{eq:def G(A f S)}
    G(\AA_f^S)  = \prod_{l \notin S} G(\QQ_l) \times \prod_{l \in S} G(\QQ_l)^{S_l}\,,
\end{equation}
and set $S_p = S_p(K^p) := S \cup \{p\}$ (not to be confused with the ring $S_p$ from Section \ref{subsubsec:unitary moduli space} which not consider in what follows).

\subsubsection{Spherical vectors} \label{subsubsec:Spherical vectors}
By definition, for each $l \notin S_p$, $\pi$ contains a $K_l$-spherical vector. We fix such a choice $0 \neq \varphi_{l,0} \in \pi_f^{K_l}$ and consider the corresponding factorization
\begin{equation} \label{eq:facto pi f}
    \pi_f \xrightarrow{\sim} \wh{\bigotimes}_{l}\, \pi_l\,,
\end{equation}
where the restricted tensor product is with respect to our choice of $\varphi_{l,0}$ for each $l \notin S_p$. In particular, $\pi_f^{K}$ is identified with
\begin{equation} \label{eq:facto pi f K fixed}
    \pi_p^{K_p} \otimes \pi_S^{K_S}\,.
\end{equation}

\begin{remark} \label{rmk:E pi rational spherical vectors}
    If $\pi$ is holomorphic or anti-holomorphic, then we may assume that each $\varphi_{l,0}$ is $E(\pi)$-rational, see Remark \ref{rmk:E pi rational}.
\end{remark}

\subsubsection{Contragredient representations and pairings} \label{subsubsec:contragredient rep and pairings}
Let $\pi^\vee$ be the contragredient representation of $\pi$, and write $\pi^\vee = \pi_\infty^\vee \otimes \pi_f^\vee$ for its decomposition as a $(\g, K_h) \times G(\AA_f)$-module.

It is well-known that $\pi^\vee$ is isomorphic to a twist of the complex conjugate $\ol{\pi}$ of $\pi$ (see \eqref{eq:def pi flat} for instance). Therefore, $\pi^\vee$ is again a cuspidal automorphic representation of $G$.

We identify the tautological pairing $\brktdotdot_\pi : \pi \times \pi^\vee$ of contragredient representation with
\[
    \brkt{\varphi}{\varphi^\vee}_\pi
        =
    \int_{Z \cdot G(\QQ) \backslash G(\AA)} \varphi(g) \varphi^\vee(g) dg\,,
        \;\;\;\text{for }
    \varphi \in \pi,\,\varphi^\vee \in \pi^\vee\,,
\]
where $dg$ is the Haar measure on $G(\AA)$ introduced in Section \ref{subsubsec:conv on measures} and $Z$ is the group of real points of the maximal $\QQ$-split subgroup of the center of $G$.

Suppose that $\pi$ has level $K$ and let $S_p = S(K^p) \cup \{p\}$ denote the set of ramified places of $\pi$, as in Section \ref{subsubsec:ramified places away from p}. Then, both $\pi$ and $\pi^\vee$ contain a $K_l$-spherical vector for each $l \notin S_p$. Fix such vectors $\varphi_{l,0} \in \pi$ and $\varphi_{l, 0}^\vee \in \pi^\vee$ for each $l \notin S_p$.

Consider factorization $\pi_f \xrightarrow{\sim} \wh{\otimes}_{l} \pi_l$ and $\pi_f^\vee \xrightarrow{\sim} \wh{\otimes}_{l} \pi_l^\vee$ into restricted tensor products over the finite places of $\QQ$, with respect to the vectors $\varphi_{l,0}$ and $\varphi_{l, 0}^\vee$, as in \eqref{eq:facto pi f}.

For each place $l$ of $\QQ$, we identify $\pi_l^\vee$ as the contragredient of $\pi_l$. Namely, we fix a $G(\QQ_l)$-equivariant perfect pairing $\brktdotdot_{\pi_l} : \pi_l \otimes \pi_l^\vee \to \CC$. We normalize such pairings so that $\brkt{\varphi_{l,0}}{\varphi_{l,0}^\vee} = 1$ for all finite place $l \notin S_p$.

There exists a constant $C$ (depending on all the choices made above) such that for each pure tensor vectors $\varphi = \otimes_l \varphi_l \in \pi$ and $\varphi^\vee = \otimes_l \varphi_l^\vee \in \pi^\vee$, we have
\begin{equation} \label{eq:facto inner product pi}
    \brkt{\varphi}{\varphi^\vee}_\pi
        =
    C 
    \prod_l 
        \brkt{\varphi_l}{\varphi_l^\vee}_{\pi_l}\,,
\end{equation}
where the product is over all places $l$ of $\QQ$.

\subsection{$P$-(anti-)ordinary automorphic representations.} \label{subsec:Paord automorphic representations}

The identifications \eqref{eq:prod G over Zp}, \eqref{eq:GL(Lw pm) basis} and \eqref{eq:GL(Lw) basis} induce an isomorphism
\begin{equation} \label{eq:facto G(Qp)}
    G(\QQ_p) \xrightarrow{\sim} \QQ_p^\times \times \prod_{w \in \Sigma_p} G_w \ ,
\end{equation}
where $G_w = \GL_n(\KK_w)$. Consequently, given any automorphic representation $\pi$, its $p$-factor $\pi_p$ decomposes as
\begin{equation} \label{eq:facto pi p}
    \pi_p 
        \cong 
    \mu_p 
        \otimes 
    \left(
        \bigotimes_{w \in \Sigma_p} 
            \pi_w
    \right)\,,
\end{equation}
where $\mu_p$ is a character of $\QQ_p^\times$ and $\pi_w$ is an irreducible admissible representation of $G_w$.

Consider the groups
\begin{align*}
    P \xrightarrow{\sim} 
        \prod_{w \in \Sigma_p} P_w 
    \ \ \ ; \ \ \
    I^0_r \xrightarrow{\sim} 
        \ZZ_p^\times 
        \times \prod_{w \in \Sigma_p} I_{w, r}^0 
    \ \ \ ; \ \ \
    I_r \xrightarrow{\sim} 
        \ZZ_p^\times \times 
        \prod_{w \in \Sigma_p} I_{w, r}
\end{align*}
constructed in Section \ref{subsubsec:para subgp of G over Zp}.

Fix a compact open subgroup $K_r = I_rK^p \subset G(\AA_f)$ as in Section \ref{subsec:Shi var of P Iwa level at p} such that $\pi_f^{K_r} \neq 0$, i.e. $\pi_f$ has level $K_r$. Then, $\pi_p^{I_r} \neq 0$ and in particular, $\mu_p$ is unramified. 

\subsubsection{$P$-ordinary case}\label{subsubsec:Pord rep cases}
Assume that $\pi$ is holomorphic of weight $\kappa$. Recall that $\kappa$ is a character of $T_{H_0}$ identified with a tuple $(\kappa_0, (\kappa_\sigma)_{\sigma \in \Sigma_\KK})$ such that $\kappa_0 \in \ZZ$ and $\kappa_\sigma \in \ZZ^{b_\sigma}$. Let $\kappa' = (\kappa_\norm)_p$ be the normalized $p$-adic weight related to $\kappa$ as in Section \ref{subsec:Pord mod forms}, see \eqref{eq:def kappa norm sigma}.

To lighten the notation, for every $w \in \Sigma_p$, $1 \leq j \leq r_w$, we set 
\[
    k_w(j) 
        := 
    \kappa'(t^+_{w, D_w(j)}) 
        = 
    \kappa'(t^-_{w, D_w(j)})^{-1}
        =
    \absv{
        \kappa'(t^+_{w, D_w(j)})
    }_p^{-1}\,,
\]
where $t^+_{w, D_w(j)}, t^-_{w, D_w(j)} \in L_P(\QQ_p)$ are introduced at the end of Section \ref{subsubsec:para subgp of G over Zp}, and are both related to the diagonal matrix $t_{w, D_w(j)} \in \GL_n(\OO_w)$ from \eqref{eq:def t w j}.

The normalized Hecke operator $u_{w, D_w(j)}$ defined in \eqref{definition u w j} naturally acts on $\pi_f^{K_r}$ via the action of $k_w(j) [I_r t_{w,D_w(j)} I_r]$ on $\pi_p^{I_r}$. The factorization in \eqref{eq:facto pi p} clearly indicates that this corresponds to the action of $k_w(j) U_{w, D_w(j)}^{\GL}$ on $\pi_w^{I_{w,r}}$, where $U_{w, D_w(j)}^{\GL} = [I_{w,r} t_{w,D_w(j)} I_{w,r}]$.

By abusing notation, we denote all of these normalized double coset operators by $u_{w, D_w(j)}$ (or $u_{w, D_w(j), \kappa}$). It is well-known that the generalized eigenvalues of all the operators $u_{w, D_w(j)}$ on $\pi_p^{I_r}$ are $p$-adically integral. Therefore, the limit in \eqref{definition eP} again induces an operator
\[
    \varinjlim_{n}
    \left(
        \prod_{w \in \Sigma_p}
        \prod_{j=1}^{r_w}
            u_{w, D_w(j)}
    \right)^{n!}
\]
on $\pi_p^{I_r}$, which we still denote $e_P$. It projects $\pi_p^{I_r}$ onto its subspace spanned by generalized eigenspaces associated to generalized eigenvalues that are $p$-adic units.

\begin{definition} \label{def:pi is Pord of level r}
    We say that $\pi$ as above is \emph{$P$-ordinary of level $r$} if 
    \[
        \pi_p^{(\Pord, r)} := e_{P, \kappa} (\pi_p^{I_r}) \neq 0\,.
    \]
\end{definition}

\begin{remark} \label{rmk:Hida thm of unique ordinary vectors}
    When $P = B$, a result of Hida (see \cite[Corollary 8.3]{Hid98} or \cite[Theorem 6.6.9]{EHLS}) implies that the space of $B$-ordinary vectors (or simply \emph{ordinary} vectors) is at most 1-dimensional and does not depend on $r \gg 0$. This is no longer true for general parabolic subgroups $P$. However, Theorem \ref{thm:Pord type tau structure thm} yields an analogous result for $P$-ordinary subspaces.
\end{remark}

By working locally at $w \in \Sigma_p$, we see that equivalently, the limit
\begin{equation} \label{eq:def e P w kappa}
    e_{P, w, \kappa} = e_{P,w}
        := 
    \varinjlim \limits_{n} 
        \left(
            \prod_{j = 1}^{r_w}
                u_{w,D_w(j),\kappa}
        \right)^{n!}\,.
\end{equation}
defines an operator on $\pi_w^{I_{w,r}}$. We refer to $e_{P_w}$ as the \emph{$P_w$-ordinary projection} operator. 

We see that $\pi$ is $P$-ordinary if and only if there exists some $r \gg 0$ such that, for all $w \in \Sigma_p$, there exists $0 \neq \phi_w \in \pi_w^{I_{w,r}}$ satisfying $e_{P, w} \phi_w = \phi_w$. Such a vector $\phi_w$ must then also satisfy $u_{w, D_w(j)} \phi_w = c_{w, D_w(j)} \phi_w$ for some $p$-adic unit $c_{w, D_w(j)}$. 

We say that $\pi_w$ is \emph{$P_w$-ordinary of level $r$} and that $\phi_w$ is a $P_w$-ordinary vector of level $r$. We let 
\[
    \pi_w^{(\Pword, r)} := e_{P, w, \kappa} (\pi_w^{I_{w,r}})
\]
denote the space of all $P_w$-ordinary vectors of level $r$.

By definition, $\pi$ is $P$-ordinary of level $r$ if and only if $\mu_p$ is unramified and $\pi_w$ is $P_w$-ordinary of level $r$ for all $w \in \Sigma_p$.

\subsubsection{$P$-anti-ordinary case} \label{subsubsec:Paord rep cases}

For the dual notion, assume $\pi$ is anti-holomorphic of weight $\kappa$. We retain the assumption that $\pi_f^{K_r} \neq 0$.

For each $w \in \Sigma_p$ and $1 \leq j \leq r_w$, the natural action of the operator $u^-_{w, D_w(j)}$ from Section \ref{subsec:Pord mod forms} on $\pi_f^{K_r}$ factors through the action of $k_w(j)^{-1} [I_r t_{w,D_w(j)}^- I_r]$ on $\pi_p^{I_r}$. Once more, this action is induced by the one of $k_w(j)^{-1} U_{w, D_w(j)}^{\GL, -}$ on $\pi_w^{I_{w,r}}$, where $U_{w, D_w(j)}^{\GL, -} = [I_{w,r} t^-_{w,D_w(j)} I_{w,r}]$. Again, by abuse of notation, we denote all these normalized double coset operators by $u^-_{w, D_w(j)}$.

As in the $P$-ordinary case, the generalized eigenvalues of $u^-_{w, D_w(j)}$ are $p$-adic integers, hence the limits
\begin{equation} \label{eq:def e P -}
    e^-_{P, \kappa} 
        = 
    \varinjlim_{n}
    \left(
        \prod_{w \in \Sigma_p}
            \prod_{j=1}^{r_w} u^-_{w, D_w(j)}
    \right)^{n!}
        \ \ \ \text{and} \ \ \
    e^-_{P,w, \kappa} 
        = 
    \varinjlim_{n}
    \left(
        \prod_{j=1}^{r_w} u^-_{w, D_w(j)}
    \right)^{n!}
\end{equation}
yield well-defined projection operators on $\pi_p^{I_r}$ and $\pi_{w}^{I_{w,r}}$ respectively.

\begin{definition} \label{def:pi is Paord of level r}
    Let $\pi$ as above be an anti-holomorphic cuspidal automorphic representation of weight $\kappa$ such that $\pi_f^{K_r} \neq 0$. We say that $\pi$ as above is \emph{$P$-anti-ordinary of level $r$} if $\pi_p^{(\Paord, r)} := e^-_{P,\kappa} (\pi_p^{I_r}) \neq 0$. Similarly, we say that $\pi_w$ is \emph{$P_w$-anti-ordinary of level $r$} if $\pi_w^{(\Pwaord, r)} := e_{P, w, \kappa}^- (\pi_w^{I_{w,r}}) \neq 0$. 
\end{definition}

\subsection{$P$-(anti-)weight-level-type} \label{subsec:P-(anti-)WLT}
Let $\pi$ be any cuspidal automorphic representation of $G$. Assume that $\pi$ is holomorphic (resp. anti-holomorphic) of weight $\kappa$ and $P$-ordinary (resp. $P$-anti-ordinary) of level $r$. Let $\Pi_r$ denote $\pi_p^{(\Pord, r)}$ (resp. $\pi_p^{(\Paord, r)}$).

Since $I_r$ is normal in the $P$-Iwahori subgroup $I_r^0 = I^0_{P,r} \subset G(\ZZ_p)$ of level $r$ and the matrices $t^\pm_{w, D_w(j)}$, for $w \in \Sigma_p$ and $1 \leq j \leq r_w$, commute with $I_r^0 / I_r = L_P(\ZZ_p/p^r\ZZ_p)$, we see that the action of $I_r^0$ on $\pi_p$ stabilizes $\Pi_r$.

In particular, we can decompose $\Pi_r$ as a direct sum of isotypic components over the (finite-dimensional) irreducible representations of $L_P(\ZZ_p/p^r\ZZ_p)$.

\begin{definition} \label{def:pi has P-(anti-)WLT}
    Let $\pi$ be any cuspidal automorphic representation of $G$ such that $\pi_f^{K_r} \neq 0$, for some $r \gg 0$. Let $\tau$ be some smooth irreducible representation of $L_P(\ZZ_p)$ factoring through $L_P(\ZZ_p/p^r\ZZ_p)$. Let $\kappa$ be some dominant weight of $T_{H_0}$.
    
    We say that $\pi$ has \emph{$P$-weight-level-type} $(\kappa, K_r, \tau)$ if $\pi$ is holomorphic of weight $\kappa$, $\pi_f$ has level $K_r$ and is $P$-ordinary, and the $\tau$-isotypic component $\pi_p^{(\Pord, r)}[\tau]$ of $\pi_p^{(\Pord, r)}$ is nonzero. We often say that $\pi$ has \emph{$P$-WLT} $(\kappa, K_r, \tau)$.

    For the dual notion, we say that $\pi$ has \emph{$P$-anti-weight-level-type} $(\kappa, K_r, \tau)$ if $\pi$ is anti-holomorphic of weight $\kappa$, $\pi_f$ has level $K_r$ and is $P$-anti-ordinary, and the $\tau^\vee$-isotypic component $\pi_p^{(\Paord, r)}[\tau^\vee]$ of $\pi_p^{(\Paord, r)}$ is nonzero. We often say that $\pi$ has \emph{$P$-anti-WLT} $(\kappa, K_r, \tau)$.
\end{definition}

\begin{remark} \label{rmk:I P r vs I SP r}
    In Definition \ref{def:def of PIwahori of level r}, one could replace $I_{P,r}$ with the collection of $g \in G(\ZZ_p)$ such that $g \mod p^r$ is in $(\ZZ_p/p^r\ZZ_p)^\times \times SP(\ZZ_p/p^r\ZZ_p)$. Here, $SP$ is the derived subgroup of $P$ as introduced above Definition \ref{def:def of PIwahori of level r}. Let us write the corresponding group by $I_{SP, r}$ momentarily, so that we have $I_{P, r} \subset I_{SP, r} \subset I_{P,r}^0$.
    
    Then, one can define $P$-ordinary representations of $G$ using $I_{SP, r}$ instead of $I_{P,r}$. By doing so, the space of $P$-ordinary vectors decomposes a direct sum over all $P$-nebentypus of $\tau$ that factor through $\det : L_P(\ZZ_p) \to \ZZ_p^\times$. Doing so is obviously less general but has the advantage of simplifying the theory as only characters of $L_P(\ZZ_p)$ occur as types of $P$-ordinary vectors. On the other hand, systematically developing the more general theory (with $P^u$ instead of $SP$) has the advantage that any holomorphic cuspidal representation $\pi$ of $G$ is trivially $\GL(n)$-ordinary. We discussed our motivation to study this more general notion in the introduction of this paper.
\end{remark}

%%%%%%%%%%%%%%%%%%%%%%%%%%%%%%%%%%%%%%%%%%%%%%%%%%%%%%%%%%%%%%%%%%

\section{Compatibility and comparison between PEL data.}
\label{sec:comp and comp bw PEL data}
\subsection{Unitary groups for the doubling method.} \label{subsec:G1 G2 G3 G4}
In what follows, we introduce the unitary group opposite to $G$ and briefly review the changes necessary in the previous sections for this group. We then introduce comparison results for cohomology of Shimura varieties and automorphic representations by adapting the material of \cite[Section 6.2]{EHLS} to our situation.

\subsubsection{Theory for $G_1$ and $G_2$.} \label{subsubsec:theory for G1 and G2}
Let
$
    \PP_1 
        := 
    \PP
        = 
    (
        \KK, c, \OO, L, 
        2\pi\sqrt{-1}\brkt{\cdot}{\cdot}, h
    )
$ 
be the PEL datum constructed in Section \ref{subsec:unitary pel datum}. We often write $L_1 := L$, $\brktdotdot_1 := 2\pi \sqrt{-1}\brktdotdot$ and $h_1 := h$.

Define
\[
    \PP_2 
        =
    (
        \KK, c, \OO, L_2, \brktdotdot_2, h_2
    )
        :=
    (
        \KK, c, \OO, L, 
        -2\pi\sqrt{-1}\brkt{\cdot}{\cdot}, h(\bar{\cdot})
    )\,,
\]
again a PEL datum of unitary. 

The signature of $\PP_2$ at $w \in \Sigma_p$ is $(b_w, a_w) = (a_{\ol{w}}, b_{\ol{w}})$. We typically write $G_i$ for the similitude unitary group corresponding to $\PP_i$ when we want to distinguish the underlying PEL datum. 

Note that $\PP_1$ and $\PP_2$ are both associated to the same $\KK$-vector space $L \otimes \QQ$ but with opposite Hermitian forms. By abusing notation, we denote the vector space associated to $\PP_1$ by $V$ and the one for $\PP_2$ by $-V$.

Consider the natural decomposition $L \otimes \ZZ_p = L^+ \oplus L^-$ obtained in Section \ref{subsubsec:comp to gen linear groups}. We now write $L_1^\pm := L^\pm$. Considering the signature of $\PP_2$, the analogous decomposition $L_2 \otimes \ZZ_p = L_2^+ \oplus L_2^-$ is given by taking $L_2^\pm = L_1^\mp$.

Obviously, we have $L_2^\pm = \prod_{w \mid p} L_{2,w}^\pm$, where $L_{2,w}^\pm := L_{1,w}^\mp$. The choice of basis for $L_{1,w}^\pm$ therefore naturally determines a choice of basis for each $L_{2,w}^\pm$ and we can proceed as in Section \ref{subsubsec:comp to gen linear groups} for $\PP_2$ to obtain identifications analogous to \eqref{eq:GL(Lw pm) basis} and \eqref{eq:GL(Lw) basis}. However, together with \eqref{eq:prod G over Zp} for $\PP_2$, the corresponding identification of $G_2(\QQ_p)$ with $\QQ_p^\times \times \prod_{w \in \Sigma_p} \GL_n(\OO_w)$ is different than the one for $G_1(\QQ_p)$, although there is a natural identification $G_1(\QQ_p) = G_2(\QQ_p)$.

In general, there is a canonical identification $G_1(\AA) = G_2(\AA)$. Therefore, instead of modifying the identification $\GL_{\OO_w}(L_w)$ with $\GL_n(\OO_w)$ for $\PP_1$ and $\PP_2$, we use the same identification twice (i.e. the one for $\PP_1$).

This is harmless for the theory of Section \ref{subsec:structure of G over Zp}. The only significant change is that the parabolic subgroup $P_2$ of $G_{2\,/\ZZ_p}$ corresponding to the partitions $\d_w$ is equal to $\tp{P_1}$, where $P_1 = P$ is the parabolic subgroup in \eqref{eq:iso P+ with Gm x P}.

In other words, via the identification \eqref{eq:GL(Lw) basis} for $\PP_1$, the parabolic subgroup $P_{2,w}$ of $\GL_n(\OO_w)$ corresponds to $\tp{P}_{1,w}$. Therefore, in what follows we always work with $P_w = P_{1,w}$ when consider $G_1$ and with $\tp{P}_w$ when considering $G_2$.

\begin{remark} \label{rmk:P Iwa vs tp P Iwa}
    This leads to an ambiguity in our notation. For instance, we should refer to $P$-(anti-)ordinary forms on $G_2$ as $\tp{P}$-(anti-)ordinary forms. We avoid this issue and refer to objects on $G_2$ as $P$-(anti-)ordinary.
\end{remark}

\subsubsection{Theory for $G_3$ and $G_4$.} \label{subsubsec:theory for G3 and G4}
For $i=3,4$, define similar PEL datum 
\[
    \PP_i 
        = 
    (\KK, c, \OO, L_i, \la \cdot, \cdot \ra_i, h_i) 
        \;\;\;\text{together with}\;\;\;
    L_i \otimes \ZZ_p = L_i^+ \oplus L_i^-\,,
\]
again of unitary type in the sense of \cite[Section 2.2]{EHLS}, where
\begin{align*}
    \PP_3 &:= (\KK \times \KK, c \times c, \OO \times \OO, L_1 \oplus L_2, \la \cdot, \cdot \ra_1 \oplus \la \cdot, \cdot \ra_2, h_1 \oplus h_2), L_3^\pm := L_1^\pm \oplus L_2^\pm \\
    \PP_4 &:= (\KK, c, \OO, L_3, \la \cdot, \cdot \ra_3, h_3), L_4^\pm := L_3^\pm
\end{align*}

Denote the similitude unitary group in \eqref{eq:def sim uni gp} associated to $\PP_i$ by $G_i$. Similarly, let $\nu_i$ be the similitude character of $G_i$ and $U_i = \ker \nu_i$.

\subsubsection{Compatibility of level structures.} \label{subsubsec:comp of lvl structures}
One readily sees that the reflex field of $\PP_1$, $\PP_2$ and $\PP_3$ are all equal. However, since $\PP_4$ has signature $(n,n)$ at all archimedean places, its reflex field is $\QQ$.

Therefore, all of the theory introduced in the previous sections can be adapted for $G_3$ and $G_4$ over $\OO_F \otimes \ZZ_{(p)}$ and $\ZZ_{(p)}$, respectively.

Note that some of the notation must be adapted for $\PP_3$ since the PEL datum is associated to two copies of $\KK$ instead of a single one. Therefore, all of the associated objects must be adapted to consider two lattices $L_3 = L_1 \oplus L_2$, two vector spaces $V_3 = V_1 \oplus V_2$, and so on, with two idempotent projections $e_1$ and $e_2$ relating the objects with ones on $\PP_1$ and $\PP_2$ respectively. The modifications are mostly obvious, hence we omit precise formulation here. For more details, see \cite[Section 2]{EHLS}.

We have $H_3 = \GL_{(\OO_{\KK} \times \OO_{\KK}) \otimes \ZZ_p}(L_3^+)$ and $H_4 = \GL_{\OO_\KK \otimes \ZZ_p}(L_4^+)$. Furthermore, the obvious inclusions $G_3 \hookrightarrow G_4$ and $G_3 \hookrightarrow G_1 \times G_2$ induce the canonical inclusions $H_3 \hookrightarrow H_4$ and $H_3 \hookrightarrow H_1 \times H_2$. 

The choice of an $\OO_w$-basis of $L_{1,w}^\pm = L_w^\pm$ as in Section \ref{subsubsec:comp to gen linear groups} naturally induces a choice of basis of $L_{i, w}^\pm$ for $i=$ 2, 3 and 4 as well. In fact, we obtain isomorphisms
\begin{equation} \label{eq:identification Gi with GL}
    G_{i\, /\ZZ_p}
        \xrightarrow{\sim}
    \Gm
        \times
    \prod_{w \in \Sigma_p}
    \begin{cases}
        \GL_n(\OO_w)\,, & \text{if } i = 1, 2, \\
        \GL_n(\OO_w) \times \GL_n(\OO_w)\,, & \text{if } i = 3 \\
        \GL_{2n}(\OO_w)\,, & \text{if } i = 4
    \end{cases}
\end{equation}
as well as
\begin{equation} \label{eq:identification Hi with GL}
    H_{i\, /\ZZ_p}
        \xrightarrow{\sim}
    \Gm
        \times
    \prod_{w \mid p}
    \begin{cases}
        \GL_{a_w}(\OO_w)\,, & \text{if } i = 1, \\
        \GL_{b_w}(\OO_w)\,, & \text{if } i = 2, \\
        \GL_{a_w}(\OO_w) \times \GL_{b_w}(\OO_w)\,, & \text{if } i = 3 \\
        \GL_{n}(\OO_w)\,, & \text{if } i = 4
    \end{cases}
\end{equation}

Let $P_1 = P_H \subset H_1 = H$ be the parabolic subgroup introduced in Section \ref{subsubsec:para subgp of G over Zp} associated to partitions $\d_w = (n_{w, 1}, \ldots, n_{w, t_w})$ and $\d_{\ol{w}} = (n_{\ol{w}, 1}, \ldots, n_{\ol{w}, t_{\ol{w}}})$ of the signature $a_w$ and $b_w$ respectively. 

Recall that under the identification $G_1(\AA) = G_2(\AA)$ and our conventions between $\PP_1$ and $\PP_2$, the corresponding parabolic subgroup of $H_2$ is $P_2 = \tp{P_H}$. Adapting the theory of Section \ref{subsubsec:para subgp of G over Zp} for $G_3$ (with the same choice of partitions) amounts to defining the corresponding parabolic $P_3 \subset H_3$ as the preimage of $P_1 \times P_2$ via $H_3 \hookrightarrow H_1 \times H_2$. 

Similarly, consider the two partitions
\begin{equation} \label{eq:partition of n n for G4}
    (n_{w, 1}, \ldots, n_{w, t_w}, n_{\ol{w}, 1}, \ldots, n_{\ol{w}, t_{\ol{w}}})
        \ \ \ \text{and} \ \ \
    (n_{\ol{w}, 1}, \ldots, n_{\ol{w}, t_{\ol{w}}}, n_{w, 1}, \ldots, n_{w, t_w})\,,
\end{equation}
viewed as a partition of $(n,n)$. Let $P_4 \subset H_4$ be the corresponding parabolic, following the approach of Section \ref{subsubsec:para subgp of G over Zp} for $G_4$. Then, the inclusion $G_3 \hookrightarrow G_4$ above induces the canonical inclusion $P_3 \hookrightarrow P_4$.

Let $L_{H, i}$ be the Levi factor of $P_i$. Then, $L_{H, 4} = L_{H, 3} = L_{H, 1} \times L_{H, 2}$. In particular, a $P_4$-nebentypus $\tau$ of level $r$ is also a $P_3$-nebentypus of level $r$. It corresponds to a tensor product $\tau_1 \otimes \tau_2$, where $\tau_i$ is a $P_i$-nebentypus of level $r$ (for $i$ = 1 and 2).

Let $P_i$ is one of those four parabolic subgroup. Let $I_{i, r} := I_{P_i, r}$ and $I_{i, r}^0 := I_{P_i, r}^0$ be the corresponding pro-$p$ $P$-Iwahoric subgroup and $P$-Iwahoric subgroup respectively. By abuse of notation, we still use the terminology ``$P$-Iwahoric'' as opposed to ``$P_i$-Iwahoric''.

Given a compact open subgroup $K_i^p \subset G_i(\AA_f^p)$, let $K_{i, r} = I_{i, r}K_i^p$. We denote the moduli space associated to $\PP_i$ of level $K_i$ by $\MM_{i, K_{i, r}} = \MM_{K_{i, r}}(\PP_i)$ and the corresponding Shimura variety by $\level{K_{i, r}}{\Sh(G_i)}$. If $V_i$ is the vector space associated to the PEL datum $\PP_i$, we sometimes write $\level{K_{i, r}}{\Sh(V_i)}$ instead.

If $K_{3, r} \subset (K_{1,r} \times K_{2,r}) \cap G_3(\AA_f)$ and $K_{3,r} \subset K_{4, r} \cap G_3(\AA_f)$, there are natural maps 
\begin{equation} \label{eq:def embeddings i 3 and i 1 2}
    i_3 : \MM_{3, K_{3, r}} \to \MM_{4, K_{4, r}}
        \;\;\; \text{and} \;\;\;
    i_{1,2} : \MM_{3, K_{3, r}} \to \MM_{1, K_{1, r}} \times \MM_{2, K_{2, r}}
\end{equation}
over $S_p = \OO_{F, (p)}$. For the exact maps at the level of points of moduli problems, see \cite[(37)--(38)]{EHLS}.

All of the above remains compatible if we restrict to Shimura varieties or extend to toroidal compactifications. Furthermore, if $K_{3,r} = (K_{1,r} \times K_{2,r} \cap G_3(\AA_f)$ then the Shimura varieties on both sides, as canonical connected components are identifies, hence we obtain an isomorphism
\begin{equation} \label{eq:iso Sh 3 with Sh 1 x Sh 2}
    i_{1,2} : \level{K_{3,r}}{\Sh(V_3)}\to \level{K_{1,r}}{\Sh(V_1)} \times \level{K_{2,r}}{\Sh(V_2)}
\end{equation}
as well as an analogous isomorphism for toroidal compactifications.

\subsubsection{Compatibility of canonical bundles} \label{subsubsec:comp of canonical bundles}
Recall that in Section \ref{subsubsec:canonical bundle}, we defined a canonical bundle $\EE = \EE_1$ on the toroidal compactification of the moduli space for $\PP_1$. Moreover, for all dominant weight $\kappa$ of the maximal torus $T_{H_0, 1}$ of $H_{0,1}$, we introduced the associated automorphic vector bundle $\w_\kappa$.

Let $V_{1, \CC} = L_1 \otimes \CC$ with its Hodge decomposition $V_{1, \CC} = V_1^{-1, 0} \oplus V_1^{0, -1}$ fixed in Section \ref{subsec:unitary pel datum}. The above is associated to a specific choice of $\OO_F \otimes \ZZ_{(p)}$-module $\Lambda_{0, 1} = \Lambda_0$ in the graded module $W_1 = V_{1, \CC}/V_1^{0,-1}$. Proceeding as in Section \ref{subsec:structure of G over C}, we obtain the groups $G_{0, 1} = G_0$ and $H_{0,1} = H_0$.

For $\PP_2$, consider $V_{2, \CC} = L_2 \otimes \CC$. The corresponding Hodge structure is reversed, i.e. $V_2^{-1, 0} = V_1^{0, -1}$ and $V_2^{0, -1} = V_1^{-1, 0}$. Therefore, one must choose a module in $W_2 = V_{2, \CC}/V_2^{0, -1} = V_{1, \CC}/V_1^{-1,0}$. Using the identifications $\Lambda_1 = \Lambda_{0,1} \oplus \Lambda_{1,0}^\vee$, $V_{1, \CC} = V_{2, \CC} = \Lambda_1 \otimes \CC$ and $\Lambda_{1, 0}^\vee \otimes \CC = V_1^{-1,0}$, we choose $\Lambda_{2, 0}$ to be the image of $\Lambda_{1,0}^\vee$ in $W_2$, and $\Lambda_2 = \Lambda_{2,0} \oplus \Lambda_{2,0}^\vee \cong \Lambda_1$.

Similarly, for $\PP_3$ and $\PP_4$, we pick $\Lambda_{3,0} = \Lambda_{4, 0} = \Lambda_{1,0} \oplus \Lambda_{2,0}$ and $\Lambda_3 = \Lambda_4 = \Lambda_1 \oplus \Lambda_2$. These compatible choices induce obvious inclusions 
\begin{equation} \label{eq:embeddings Hi0}
    H_{3,0} \hookrightarrow H_{4,0}
        \ \ \ \text{and} \ \ \
    H_{3,0} \hookrightarrow H_{1,0} \times H_{2,0}\,.
\end{equation}

Let $\pi_i : \EE_i \to \MM_{K_i}$ be the corresponding canonical bundle for $\PP_i$. Then, the above also induces natural maps
\begin{equation} \label{eq:embeddings EE i}
    i_3 : \EE_3 \to \EE_4
        \ \ \ \text{and} \ \ \
    i_{1,2} : \EE_3 \to \EE_1 \times \EE_2\,,
\end{equation}
compatible with the maps on moduli spaces in \eqref{eq:def embeddings i 3 and i 1 2}. These maps also extend to maps between the $\LL_H(\ZZ_p/p^r\ZZ_p)$-torsors $\EE_{i, r} \to \EE_i$, for all $r \gg 1$.

The choice of basis over $S_0 = \OO_{\KK', (\p')}$ of $\Lambda_{1,0, \sigma}$ and $\Lambda_{1,0, \sigma}^\vee$, for each $\sigma \in \Sigma_\KK$ as in Sections \ref{subsubsec:alg weights}, naturally induces bases for $\Lambda_{i,0}$ and $\Lambda_{i,0}^\vee$ for $i$ = 2, 3 and 4 as well. We obtain identifications
\[
    H_{i, 0\, /S_0}
        \xrightarrow{\sim}
    \Gm
        \times
    \prod_{\sigma \in \Sigma_\KK}
    \begin{cases}
        \GL_{b_\sigma}(S_0)\,, & \text{if } i = 1, \\
        \GL_{a_\sigma}(S_0)\,, & \text{if } i = 2, \\
        \GL_{b_\sigma}(S_0) \times \GL_{a_\sigma}(S_0)\,, & \text{if } i = 3, \\
        \GL_n(S_0)\,, & \text{if } i = 4,
    \end{cases}
\]
as in \eqref{eq:H0 basis}.

Observe that the definition of $H_0$ in Section \ref{subsec:structure of G over C} yields an obvious identification $H_{1,0} = H_{2,0}$ (by switching the roles of $\Lambda_0$ and $\Lambda_0^\vee$). With respect to the identifications above, this corresponds to the automorphism
\begin{equation} \label{eq:iso H1,0 with H2,0}
    (
        h_0, 
        (
            h_\sigma
        )_{\sigma \in \Sigma_\KK}
    )
        \mapsto
    (
        h_0, 
        (
            h_0 \tp{h_{\sigma c}}^{-1}
        )_{\sigma \in \Sigma_\KK}
    )
\end{equation}
of $H_0 = H_{1,0} = H_{2,0}$ over $S_0$.

Furthermore, the embeddings in \eqref{eq:embeddings Hi0} are the obvious ones with respect to the identifications above. Similar identifications can be made about the Borel subgroups $B_{H_0, i}$, maximal torus $T_{H_0, i}$ and parabolic subgroups $P_{H_0, i}$ introduced in Sections \ref{subsubsec:alg weights} and \ref{subsubsec:P parallel weights}.

In particular, $T_{H_0, 4} = T_{H_0, 3} = T_{H_0, 1} \times T_{H_0, 2}$ and a dominant weight $\kappa$ of $T_{H_0, 4}$ corresponds to a dominant weight of $T_{H_0, 3}$. Similarly, a pair $(\kappa_1, \kappa_2)$ consisting of a dominants weight $\kappa_1$ for $T_{H_0, 1}$ and a dominant weight $\kappa_2$ of $T_{H_0, 2}$ also corresponds to a dominant weight of $T_{H_0, 3}$.

\subsubsection{Restriction of algebraic modular forms.} \label{subsubsec:res of alg modular forms}
Let $r \geq 1$ and let $\tau$ be a $P_4$-nebentypus (for $G_4$) of level $r$. Let $R$ be an $S_0[\tau]$-algebra. Let $K_{i,r}$ be an open compact subgroup of $G_i(\AA_f)$ such that $K_{3,r} \subset K_{4, r} \cap G_3(\AA_f)$ and $K_{3, r} \subset (K_{1,r} \times K_{2,r}) \cap G_3(\AA_f)$, as in Section \ref{subsubsec:comp of lvl structures}.

If $\kappa$ is an $R$-valued dominant weight of $T_{H_0, 4}$, then pullback $i_3$ in \eqref{eq:embeddings EE i} induces a restriction map
\begin{equation} \label{eq:restriction classical modular forms G4 to G3}
    \res_3 = (i_3)^* : M_\kappa(K_{4,r}, \tau; R) \to M_\kappa(K_{3,r}, \tau; R)\,.
\end{equation}

Similarly, let $\kappa_1$ (resp. $\kappa_2$) and $\tau_1$ (resp. $\tau_2$) be an $R$-valued dominant weight of $T_{H_0, 1}$ (resp. $T_{H_0, 2}$) and a $P_1$-nebentypus (resp. $P_2$-nebentypus) of level $r$, respectively. Then, set $\kappa = (\kappa_1, \kappa_2)$ and $\tau = \tau_1 \otimes \tau_2$ be the corresponding weight of $T_{H_0, 3}$ and $P_3$-nebentypus. As above, pullback along the second map in \eqref{eq:embeddings EE i} induces a restriction
\begin{equation} \label{eq:restriction classical modular forms G1,2 to G3}
    \res_{1,2} = (i_{1,2})^* 
        : 
    M_\kappa(K_{1,r}, \tau_1; R)
        \otimes
    M_\kappa(K_{2,r}, \tau_2; R)
        \to
    M_\kappa(K_{3,r}, \tau; R)\,.
\end{equation}

Naturally, these maps have analogues when considering modular forms without fixed $P$-nebentype. In situation of \eqref{eq:iso Sh 3 with Sh 1 x Sh 2}, the map $\res_{1,2}$ is an isomorphism.

\subsection{Comparisons between $G_1$ and $G_2$.}  \label{subsec:comp between G1 and G2}
In this section, we discuss various involutions that allow us to compare spaces of holomorphic forms with spaces of anti-holomorphic forms on both $G_1$ or $G_2$. 

Namely, the goal of this section is to explain the functors
\[
\begin{tikzcd}
    \{
        \text{holomorphic AR on $G_1$}
    \}
        \arrow[r, "c_B"]
        \arrow[d, "F^\dagger"]
        \arrow[rd, "F_\infty" above] &
    \{
        \text{anti-holomorphic AR on $G_1$}
    \}
        \arrow[d, "F^\dagger"] \\
    \{
        \text{holomorphic AR on $G_2$}
    \}
        \arrow[r, "c_B"]
        \arrow[ru] &
    \{
        \text{anti-holomorphic AR on $G_2$}
    \}
\end{tikzcd}\,,
\]
where ``AR'' stands for automorphic representations, adapting \cite[Section 6.2]{EHLS} to the $P$-ordinary setting (notably $F^\dagger$), as well as the effect of each arrow on weights, levels and types.

\subsubsection{The involution $c_B$ on $G_1$.}  \label{subsubsec:the involution cB}
Let $\pi$ be a holomorphic cuspidal automorphic representation for $G = G_1$ of weight $\kappa$. All of the following holds for $G_2$ with the obvious modifications. As explained in \cite[Section 6.2.1]{EHLS}, there is a $c$-semilinear, $G(\AA_f)$-equivariant isomorphism
\[
    c_B : 
    H^0(\P_h, K_h; \pi \otimes W_\kappa) 
        \to
    H^d(\P_h, K_h; \ol{\pi} \otimes W_{\kappa^D}) 
\]
induced by the Killing form on $\g$ and the complex conjugation $\pi \to \ol{\pi}$, $\varphi \mapsto \ol{\varphi}$ on $\Ab_0(G)$, where $\ol{\varphi}(g) = \ol{\varphi(g)}$. Equivalently, we have a $c$-semilinear, $G(\AA_f)$-equivariant isomorphism
\begin{equation} \label{eq:iso cB holo to anti holo on G1}
    c_B :
    H_!^0(\Sh(V), \w_\kappa)
        \to
    H_!^d(\Sh(V), \w_{\kappa^D})
\end{equation}
over $\CC$ via \eqref{eq:Lie alg coh for H!i Sh V kappa}.

\begin{remark}
    From Definition \ref{def:(anti)holo aut reprn}, the isomorphism $c_B$ shows that an automorphic representation $\pi$ of $G$ is holomorphic of weight $\kappa$ if and only if $\ol{\pi}$ is anti-holomorphic of weight $\kappa^D$.
\end{remark}

\subsubsection{The involution $F_\infty$.} \label{subsubsec:the involution Foo}
To compare automorphic representations on $G_1$ and $G_2$, we first compare their associated locally symmetric spaces. Let $h$ (resp. $h^c$) be the homomorphism associated to $G_1$ (resp. $G_2$) inducing a Hodge structure on $V \otimes \CC$ and denote the $G(\RR)$-conjugacy class of $h$ (resp. $h^c$) by $X_h$ (resp. $X_{h^c}$). 

Note that the stabilizer $K_h \subset G(\RR)$ of $h$ is also the stabilizer of $h^c$. In this section, we write $U_\infty := K_h$ and identify $X = G(\RR)/U_\infty$ with both $X_h$ and $X_{h^c}$. However, note that the complex structures induced by $h$ and $h^c$ respectively are opposite. Namely, the pullback of these two complex structures to $X$ are conjugate. 

In other words, the natural map $X_h \to X_{h^c}$ given by $ghg^{-1} \mapsto gh^cg^{-1}$ is anti-holomorphic and provides an anti-holomorphic map $\Sh(V)(\CC) \to \Sh(-V)(\CC)$. 

\begin{remark} \label{rmk: holomorphic vs anti holomorphic X h to X h c}
    In Section \ref{subsubsec:the involution dagger}, we study a different map $X_h \to X_{h^c}$ instead. To help distinguish the two, one does not need to apply complex conjugation anywhere in the definition of $F_\infty$. This involution is simply a natural consequence of the relation between the complex structures on $X_h$ and $X_{h^c}$.
\end{remark}

\begin{remark}
    From now on, we use the notation of Section \ref{subsec:G1 G2 G3 G4}. However, we replace the subscripts $i=1$ and $2$ by $V$ and $-V$. In particular, $H_{0, V} := H_{1,0}$ and $H_{0, -V} := H_{2,0}$.
\end{remark}

Let $\kappa = (\kappa_0, (\kappa_\sigma))$ be a dominant character of $T_{H_{0, V}}$. Observe that the notion of a \emph{dominant} character is ``flipped'' as $B_{H_0, -V} = \tp{B_{H_0, V}}$ under the identification \eqref{eq:iso H1,0 with H2,0}. 

Namely, the weight of $T_{H_{0, -V}}$ given by
\[
    \kappa^\flat := (\kappa_0, (\kappa_{\sigma c}))
\]
is dominant. To understand the Lie algebra cohomology of the highest weight representation $W_{\kappa^\flat, -V}$ of $H_{0, -V}$, observe that the Harish-Chandra decomposition of $\g$ induced by $h^c$ is
\[
    \g = \p^-_{h^c} \oplus \k_{h^c} \oplus \p^+_{h^c}\,,
\]
where $\p^\pm_{h^c} = \p^{\mp}_{h^c}$ and $\k_{h^c} = \k_h$.

Hence, given a $(\g, K_h) \times G(\AA_f)$-module $\pi$, we need to consider 
\[
    (
        \pi^{\p_{h^c}^-} 
            \otimes 
        W_{\kappa^\flat, -V}
    )^{K_{h^c}}
        =
    (
        \pi^{\p_h^+} 
            \otimes 
        W_{\kappa^\flat, -V}
    )^{K_h}
\]
and understand $W_{\kappa^\flat, -V}$ as a representation of $H_{0, V}$, via pullback through \eqref{eq:iso H1,0 with H2,0}.

In fact, via \eqref{eq:iso H1,0 with H2,0}, $W_{\kappa^\flat, -V}$ is the irreducible highest weight representation of highest weight $\kappa^*$ for $T_{H_{0, V}}$, where $\kappa^*$ is as in \eqref{eq:def W kappa *}. See \cite[Eq. (121)]{EHLS} for an explicit $H_{0, V}$-isomorphism $W_{\kappa^\flat, -V} \xrightarrow{\sim} W_{\kappa^*, V}$.

Furthermore, by definition of $\kappa^D$, there is a natural map
\[
    i_{\kappa^*}
        :
    W_{\kappa^*, V}
        \hookrightarrow
    \hom(
        \wedge^d \p_h^- 
            \otimes_\CC
        \wedge^d \p_h^+,
        W_{\kappa^*, V}
    )
        =
    \hom(
        \wedge^d \p_h^-,
        W_{\kappa^D, V}
    )
\]
where the first map is induced by the $H_0(\CC)$-equivariant pairing
\[
    \wedge^d \p_h^- 
        \otimes_\CC
    \wedge^d \p_h^+
        \to
    \CC
\]
obtained from the Killing form on $\g$.

Therefore, the above yields
\[
    (
        \pi^{\p_{h^c}^-} 
            \otimes 
        W_{\kappa^\flat, -V}
    )^{K_{h^c}}
        \xrightarrow{\text{id} \otimes i_{\kappa^*}}
    \hom(
        \wedge^d \p_h^- ,
        \pi
            \otimes
        W_{\kappa^D, V}
    )
\]
which induces a $G(\AA_f)$-equivariant isomorphism
\begin{equation} \label{eq:def iso Foo on pi}
    F_\infty
        :
    H^0(
        \P_{h^c}, K_{h^c},
        \pi 
            \otimes 
        W_{\kappa^\flat, -V}
    )
        \xrightarrow{\sim}
    H^d(
        \P_h, K_h,
        \pi 
            \otimes 
        W_{\kappa^D, V}
    )
\end{equation}
over $\CC$.

The case $\pi = \Ab_0(G)$ yields the $G(\AA_f)$-equivariant isomorphism
\begin{equation} \label{eq:iso holo flat on G2 with anti holo on G1}
    F_\infty
        :
    H_!^0(
        \Sh(-V),
        \w_{\kappa^\flat, -V}
    )
        \xrightarrow{\sim}
    H_!^d(
        \Sh(V),
        \w_{\kappa^D, V}
    )
\end{equation}
over $\CC$.

\begin{remark}
    Considering the composition of $c_B$ and $F_\infty$, we see that if $\pi$ is holomorphic (resp. anti-holomorphic) of weight $\kappa$ on $G_1$, then $\ol{\pi}$ is holomorphic (resp. anti-holomorphic) of weight $\kappa^\flat$ on $G_2$
\end{remark}

\subsubsection{The involution $(-)^\flat$} \label{subsubsec:the involution flat}
Let $\pi$ be a cuspidal automorphic representation for $G = G_1$. If $\pi$ is holomorphic of weight $\kappa$, then $\xi_{\pi, \infty}(t) = t^{a(\kappa)}$ for all $t \in Z_G(\RR) \cong \RR^\times$, where $\xi_\pi$ denote the central character of $\pi$. 

It follows that
\[
    \pi \otimes |\xi_\pi \circ \nu|^{-1/2} = \pi \otimes ||\nu||^{-a(\kappa)/2}
\]
is unitary. Hence, considering its conjugate, we see that $\ol{\pi}$ is isomorphic to
\begin{equation} \label{eq:def pi flat}
    \pi^\flat 
        := 
    \pi^\vee \otimes \absv{\xi_\pi \circ \nu}\,.
\end{equation} 

The material of Sections \ref{subsubsec:the involution cB}--\ref{subsubsec:the involution Foo} then implies that $\pi^\flat$ is anti-holomorphic of weight $\kappa^D$ for $G_1$, or equivalently, holomorphic of weight $\kappa^\flat$ for $G_2$.

Since $\pi^\flat$ is a twist of $\pi^\vee$, the pairings $\brktdotdot_\pi$ and $\brktdotdot_{\pi_l}$ from Section \ref{subsubsec:contragredient rep and pairings} induce pairings $\pi \times \pi^\flat \to \CC$ and $\pi_l \otimes \pi_l^\flat \to \CC$, for each place $l$ of $\QQ$, which we again denote $\brktdotdot_\pi$ and $\brktdotdot_{\pi_l}$ respectively.

\begin{remark}
    The necessity of working with $\pi^\flat$ in this paper is due to the doubling method, see Section \ref{subsubsec:zeta integrals}, which requires the integration of an Eisenstein series with the product of a \emph{test} vector $\varphi$ in $\pi$ and a \emph{test} vector $\varphi^\flat$ in a twist of $\pi^\flat \cong \ol{\pi}$ (or equivalently, $\pi^\vee$). It is natural to view the Eisenstein series as a holomorphic modular form on $G_3$ (or rather the restriction to $G_3$ of a modular form on $G_4$) and, dually, $\pi$ and $\pi^\flat$ as anti-holomorphic representations on $G_1$ and $G_2$ respectively.

    Moreover, the advantage of $\pi^\flat$ over $\ol{\pi}$ is that its direct relation with $\pi^\vee$ facilitates the transition between $P$-ordinary properties of $\pi$ and $P$-anti-ordinary properties of $\pi^\flat$ over $G_1$, see Lemma \ref{lma:Pword and Pwaord duality} and Theorem \ref{thm:canonical Pwaord vec of type tau w for G1}.
\end{remark}

\begin{remark}\label{rmk:pi flat instead of pi bar}
    Starting in Section \ref{subsec:lattices of Pord holo forms}, we assume that $\pi$ satisfies the multiplicity one hypothesis (see Hypothesis \ref{hyp:mult one hyp (I)}). In that case, the subspaces of $\Ab_0(G)$ associated to $\pi^\flat$ and $\ol{\pi}$ are in fact equal.
\end{remark}

\subsubsection{The involution $(-)^\dagger$ for level structures.} \label{subsubsec:the involution dagger}
We introduce one last involution ``$\dagger$''. The main feature is to compare level structures between $G_1$ and $G_2$ (and not just weights, as we have done above). Although our approach, arguments and notation follow \cite[Section 6.2.3]{EHLS}, the reader should keep in mind that the results presented here generalize \emph{loc. cit.} by considering $P$-Iwahoric level structures at $p$ for all parabolic subgroups $P$.

As we are assuming Hypothesis \ref{hyp:standard hypothesis}, there exists a $\KK$-basis $\mathcal{B}_h$ of $V$ that diagonalizes the Hermitian pairing $\brktdotdot_V$ associated to $V$. Furthermore, if we write $\mathcal{B}_h^c$ for the image of $\mathcal{B}_h$ under complex conjugation, then $h$ takes values in the space of diagonal matrices under the identification
\[
    \End_{\RR}(V \otimes_{\KK, \sigma} \CC) = \Mat_{2n \times 2n}(\RR)\,.
\]
induced by the $\KK^+$-basis $\mathcal{B}_h \cup \mathcal{B}^c_h$ of $V$, for each $\sigma \in \Sigma$. Clearly, $h^c$ is obtained by conjugating $h$ with the change-of-basis endomorphism interchanging $\mathcal{B}_h$ with $\mathcal{B}_h^c$.

Let $D = \diag(d_1, \ldots, d_n)$, $d_1, \ldots, d_n \in \KK^+$, be the diagonal Hermitian matrix representing $\brktdotdot_V$ with respect to $\mathcal{B}_h$.

Let $L$ be the $\OO$-lattice from Section \ref{subsec:unitary pel datum} associated to $\PP$. As explained in \cite[Section 6.2.3]{EHLS}, using Hypothesis \ref{hyp:ordinary hypothesis}, we can assume that $\mathcal{B}_h$ induces a basis of $L \otimes \ZZ_{(p)}$ such that $D$ is also the diagonalization of the perfect Hermitian pairing on $L \otimes \ZZ_{(p)}$ obtained from $\brktdotdot_V$.

The advantage of $\mathcal{B}_h$ is that it provides a holomorphic map $\Sh(V) \to \Sh(-V)$ (as opposed to the anti-holomorphic one in Section \ref{subsubsec:the involution Foo}). Indeed, first identify $G_{/\QQ}$ as a subgroup of $\res_{\KK/\QQ} \GL_n(\KK)$ using $\mathcal{B}_h$ and consider the automorphism $g \mapsto \ol{g}$ of $G_{\QQ}$ induced by the action of $c$ on $\KK$.

Observe that $\ol{g} = IgI$, where $I : V \to V$ is the $\KK^+$-involution that interchanges $\mathcal{B}_h$ and $\mathcal{B}_h^c$ by sending any vector $v \in \mathcal{B}_h$ to $v^c \in \mathcal{B}_h^c$ and \emph{vice versa}. Note that $I$ stabilizes $L \otimes \ZZ_{(p)}$ and its action on $L \otimes \ZZ_p$ interchanges $L^+$ and $L^-$. Moreover, our explanation above implies that $h^c = \ol{h}$, hence it maps $U_\infty$ to $U_\infty$ and yields an automorphism of $X$.

The composition
\[
    \begin{tikzcd}
        X_h 
            \arrow[r, "\sim"] &
        X 
            \arrow[r, "g \mapsto \ol{g}"] &
        X 
            \arrow[r, "\sim"] &
        X_{h^c} 
            \\
        ghg^{-1} 
            \arrow[rrr] &
            &
            &
        \ol{g}h^c\ol{g}^{-1}
    \end{tikzcd}
\]
is holomorphic and provides a holomorphic map $\Sh(V)(\CC) \to \Sh(-V)(\CC)$ as claimed. Given $K^p \subset G(\AA^p_f)$ and $K_{r, V} = I_{r, V}K^p$ as in Section \ref{subsubsec:lvl subgp KPr}, it provides a holomorphic map between $\level{K_{r,V}}{\Sh(V)}(\CC) \to \level{\ol{K_{r,V}}}{\Sh(-V)}(\CC)$. However, $\ol{K_{r,V}} \neq I_{r, -V} \ol{K^p}$ or equivalently, $\ol{I_{r,V}} \neq I_{r, -V}$.

To resolve this issue, observe that the basis $\mathcal{B}_h$ of $V$ naturally induces a basis $\mathcal{B}_{h, w}$ of $V_w := V \otimes_{\KK} \KK_w$ for any $w \in \Sigma_p$. It would be too restrictive to assume that the basis $\mathcal{B}^w$ of $V_w$ induced by the $\OO_w$-bases of $L_w^\pm$, for $w \in \Sigma_p$, leading to the identifications in \eqref{eq:GL(Lw pm) basis}, is the same as $\mathcal{B}_{h, w}$. In other words, there is no need to assume that all the bases $\mathcal{B}^w$, as $w$ varies in $\Sigma_p$, are all induced by a basis of $V$.

Instead, consider the identification $\GL_{\KK_w}(V_w) = \GL_n(\KK_w)$ induced by $\mathcal{B}_{h, w}$ and let $\beta_w \in \GL_n(\KK_w)$ be the change-of-basis matrix that maps $\mathcal{B}^w$ to $\mathcal{B}_{h,w}$.

\begin{remark}
    This matrix $\beta_w \in \GL_n(\KK_w)$ is the inverse of the analogous change-of-basis matrix, also denoted $\beta_w$, introduced in \cite[Section 6.2.3]{EHLS}.
\end{remark}

Let
\[
    \delta_w = D \cdot \tp{\beta}_w \cdot \beta_w \in \GL_n(\KK_w)\,
\]
and define $\delta_p = (1, (\delta_w)_{w \in \Sigma_p}) \in \QQ_p^\times \times \prod_{w \in \Sigma_p} \GL_n(\KK_w) = G(\QQ_p)$. This identification is with respect to $\mathcal{B}_h$ (and may be different than the identification obtained from \eqref{eq:prod G over Zp} and \eqref{eq:GL(Lw pm) basis}).

Then, one readily checks that
\[
    \ol{\delta}_p = \delta_p^{-1} 
        \ \ ; \ \
    \delta_p^{-1} \ol{G(\ZZ_p)} \delta_p = G(\ZZ_p)
        \ \ ; \ \
    \delta_p^{-1} \ol{I_{r, V}^0} \delta_p = I_{r, -V}^0
        \ \ ; \ \
    \delta_p^{-1} \ol{I_{r, V}} \delta_p = I_{r, -V}\,.
\]

Therefore, by defining an automorphism $(-)^\dagger$ of $G(\AA)$ as 
\[
    g^\dagger := \nu(g)^{-1} \delta_p^{-1} \ol{g} \delta_p\,,
\]
we see that given any $K^p \subset G(\AA_f)$, $K = G(\ZZ_p)K^p$ and $K_{r, V} = I_{r, V}K^p$, we have
\[
    K^\dagger = G(\ZZ_p)\ol{K^p}
        \ \ \ ; \ \ \
    (K_{r, V})^\dagger = (K^\dagger)_{r, -V} = K^\dagger_{r, -V} \,.
\]

In conclusion, the obvious analogue of \cite[Proposition 6.2.4]{EHLS} also holds in our context (namely, one simply changes the meaning of the level $r$ structure at $p$ from ``Iwahoric'' to ``$P$-Iwahoric''). In other words, the holomorphic map $\level{K_{r,V}}{\Sh(V)}(\CC) \to \level{K_{r,-V}^\dagger}{\Sh(-V)}(\CC)$ via $g \mapsto \ol{g} \delta_P$ is well-defined over $S_0 = \OO_{\KK', (\p')}$, i.e. it is induced from an isomorphism 
\begin{equation} \label{eq:iso Sh K r V and K dagger r -V}
    \level{K_{r,V}}{\Sh(V)} 
        \xrightarrow{\sim} 
    \level{K_{r,-V}^\dagger}{\Sh(-V)}
\end{equation}
over $S_0$ comparing $P$-Iwahori level structure on $G_1$ and $G_2$ (see Remark \ref{rmk:P Iwa vs tp P Iwa}).

To understand this involution in terms of modular forms, consider a dominant character $\kappa$ of $T_{H_{0, V}}$. Let $\kappa^\dagger = \kappa^\flat - \ul{a(\kappa)}$, a dominant character of $T_{H_{0, -V}}$, where $\ul{a(\kappa)}$ is the scalar weight corresponding to the character $||\nu(-)||^{a(\kappa)}$.

By definition of $\kappa^\flat$, the natural automorphism $(h_0, (h_\sigma)) \to (h_0, (h_{\sigma c}))$ of $H_0$ induces an isomorphism $W_{\kappa, V} \to W_{\kappa^\flat, -V}$. By twisting by $\nu^{-a(\kappa)}$, we obtain an isomorphism $W_{\kappa, V} \xrightarrow{\sim} W_{\kappa^\dagger, -V}$ given by $\phi \mapsto \phi^\dagger$, where
\[
    \phi^\dagger(
        h_0, (h_\sigma)
    )
        :=
    \phi(
        h_0, (h_{\sigma c})
    )
    h_0^{-a(\kappa)}\,.
\]

One readily sees that this isomorphism is $\dagger$-equivariant for the action of $K_h = K_{h^c}$. Therefore, the isomorphism $g \mapsto g^\dagger$ induces an isomorphism between $\w_{\kappa, V}$ and the pullback to $\w_{\kappa^\dagger, -V}$ from $\Sh(V)$ to $\Sh(-V)$.

In other words, the above induces an isomorphism
\begin{equation} \label{eq:F dagger map level p oo}
    F^\dagger
        : 
    H^i_!(\Sh(V), \w_{\kappa, V})
        \xrightarrow{\sim}
    H^i_!(\Sh(-V), \w_{\kappa^\dagger, -V})\,,
\end{equation}
over $\CC$, that is $\dagger$-equivariant for the action of $G(\AA_f)$. 

\begin{remark} \label{rmk:iso cohomology -V kappa flat and -V kappa dagger}
    Observe that multiplication by the global section $g \mapsto ||\nu(g)||^{a(\kappa)}$ yields an isomorphism
    \[
        H^0_!(\Sh(-V), \w_{\kappa^\flat, -V})
            \xrightarrow{\sim}
        H^0_!(\Sh(-V), \w_{\kappa^\dagger, -V})\,,
    \]
    which can be useful to compare the above with the results of Sections \ref{subsubsec:the involution Foo}--\ref{subsubsec:the involution flat}.
\end{remark}

Using the discussion, \eqref{eq:iso Sh K r V and K dagger r -V} and \eqref{eq:F dagger map level p oo} induce an isomorphism
\begin{equation} \label{eq:F dagger map level K r}
    F^\dagger
        : 
    H^i_!(\level{K_r}{\Sh(V)}, \w_{\kappa, V})
        \xrightarrow{\sim}
    H^i_!(\level{K_r^\dagger}{\Sh(-V)}\,, \w_{\kappa^\dagger, -V})
\end{equation}
over any $S_0$-algebra $R$ and any $r \gg 0$.

Similarly, given a $P$-nebentypus $\tau$ of level $r$, the map $g \mapsto g^\dagger$ induces an isomorphism between $\w_{\kappa, r, \tau, V}$ and the pullback of $\w_{\kappa^\dagger, r, \tau^\vee, -V}$ from $\level{K_r}{\Sh(V)}$ to $\level{K_r^\dagger}{\Sh(-V)}$. Therefore, we also have an isomorphism
\[
    F^\dagger
        : 
    H^i_!(
        \level{K_r}{\Sh(V)}, 
        \w_{\kappa, r, \tau, V}
    )
        \xrightarrow{\sim}
    H^i_!(
        \level{K_r^\dagger}{\Sh(-V)}\,, \w_{\kappa^\dagger, r, \tau^\vee, -V}
    )
\]
over any $S_0[\tau]$-algebra $R$. We now set $\tau^\dagger := \tau^\vee$, motivated by the isomorphism above.

Lastly, to understand this involution in terms of automorphic forms via \eqref{eq:Lie alg coh for H!i Sh V kappa}, let $\pi$ be an arbitrary $(\g, K_h) \times G(\AA_f)$-subrepresentations of $\Ab_0(G)$. The map $g \mapsto g^\dagger$ induces a map $\pi \to \pi^\dagger$, where
\[
    \pi^\dagger
        :=
    \{
        \varphi^\dagger(g) := \varphi(g^\dagger)
            \mid
        \varphi \in \pi
    \}
        \subset
    \Ab_0(G)\,.
\]

As explained in \cite[Section 6.2.3]{EHLS}, the isomorphism
\[
    (\pi \otimes W_{\kappa, V})^{K_h}
        \to
    (\pi^\dagger \otimes W_{\kappa^\dagger, -V})^{K_{h^c}}\,,
\]
given by $\varphi \otimes \phi \mapsto \varphi^\dagger \otimes \phi^\dagger$ is $\dagger$-equivariant for the action of $(\g, K_h) \times G(\AA_f)$. Therefore, one obtains
\[
    F^\dagger
        :
    H^i(
        \P_h, K_h,
        \pi \otimes W_{\kappa, V}
    )
        \xrightarrow{\sim}
    H^i(
        \P_{h^c}, K_{h^c},
        \pi^\dagger \otimes W_{\kappa^\dagger, -V}
    )
\]
over $\CC$. The case $\pi = \Ab_0(G)$ recovers the map \eqref{eq:F dagger map level K r} over $\CC$.

\begin{remark}
    Suppose $\pi$ is (anti-)holomorphic and $P$-(anti-)ordinary such that its $P$-(anti-)WLT $(\kappa, K_r, \tau)$ for $G_1$, then $\pi^\dagger$ is (anti-)holomorphic, $P$-(anti-)ordinary of $P$-(anti-)WLT $(\kappa^\dagger, K^\dagger_r, \tau^\dagger)$ for $G_2$.
\end{remark}

As explained in \cite[Section 6.5.3]{EHLS}, if $\pi$ satisfies the strong multiplicity one hypothesis, then $\pi^\dagger$ and $\pi^\vee$ are equal as subspaces of $\Ab_0(G)$. In that case, one further obtains $\pi^\flat = \pi^\dagger \otimes ||\nu||^{a(\kappa)}$. 

Assume that $\pi$ is (anti-)holomorphic, $P$-(anti-)ordinary of $P$-(anti-)WLT $(\kappa, K_r, \tau)$ for $G_1$. Then, setting 
\[
    K_r^\flat := K_r^{\dagger} \text{ and } \tau^\flat := \tau^\dagger = \tau^\vee\,
\]
we have that $\pi^\flat$ is (anti-)holomorphic, $P$-(anti-)ordinary of $P$-(anti-)WLT $(\kappa^\flat, K_r^\flat, \tau^\flat)$ for $G_2$ (by definition of $\kappa^\flat$).

%%%%%%%%%%%%%%%%%%%%%%%%%%%%%%%%%%%%%%%%%%%%%%%%%%%%%%%%%%%%%%%%%%

\section{$P$-ordinary $p$-adic modular forms.} \label{sec:Pord padic mod forms}
Fix a neat open compact subgroup $K^p \subset G(\AA_f^p)$. In what follows, we use the notation of Sections \ref{subsubsec:canonical bundle} and \ref{subsubsec:lvl subgp KPr} freely. 

In particular, write $K = G(\ZZ_p)K^p$ and $K_r = I_rK^p$ for all $r \geq 0$. Furthermore, recall that $\MM_K^{\tor}$ is the smooth toroidal compactification of $\MM_K$ over $S_p = \OO_{F, p}$ (a tower viewed as a single scheme, see \ref{subsubsec:toroidal compactification}), and $\Ab$ is the universal semiabelian scheme (with extra structure) over $\MM_K^{\tor}$. The dual semiabelian scheme is denoted $\Ab^\vee$ and we write $\w$ for the $\OO_{\MM_K^\tor}$-dual of $\Lie_{\MM_K^\tor} \Ab^\vee$.

In this section, we use the fact that the completion of $\incl_p(S_p)$ is $\ZZ_p$ to view (compactified) moduli spaces and Shimura varieties of level $K$ (and $K_r$) over $\ZZ_p$.

\subsection{Igusa tower.} \label{subsec:Igusa tower}
\subsubsection{Ordinary locus and Igusa cover} \label{subsubsec:Ordinary locus and Igusa cover}
Given $m \geq 1$, let $\SSS_m$ denote the nonvanishing locus of the Hasse invariant on $\level{K}{\Sh^\tor}$ over $\ZZ_p/p^m\ZZ_p$. Let $\SSS_m^0$ be the open subscheme obtained from the intersection of $\SSS_m$ and $\level{K}{\Sh}$. Note that $\SSS_1$ is dense in the special fiber of $\level{K}{\Sh^\tor}$, see \cite[Section 2.8]{EHLS}.

Given $r \geq m$, let $\TTT_{r,m}$ denote the finite étale cover of $\SSS_m$ such that for any $\SSS_m$-scheme $S$
\begin{equation} \label{eq:def level r Igusa cover}
    \TTT_{r,m}(S) 
        = 
    \isom_S(
        L^+ \otimes \mu_{p^r}, 
        \Ab^\vee[p^r]^\circ
    )\,,
\end{equation}
where the superscript $\circ$ denotes the identity component and the isomorphisms are of finite flat group schemes over $S$ endowed with $\OO \otimes \ZZ_p$-actions. One readily sees that $\TTT_{r,m}$ is a closed subscheme of $\level{K_r}{\ol{\Sh}}_{/(\ZZ_p/p^m\ZZ_p)}$.

Furthermore, $\TTT_{r,m} / \SSS_m$ is a Galois cover whose Galois group is canonically isomorphic to $H(\ZZ_p/p^r\ZZ_p)$. We refer to $\TTT_m = \{\TTT_{r,m}\}_r$ as the \emph{Igusa tower over $\SSS_m$}.

Let $\SSS$ denote the non-vanishing locus in $\level{K}{\Sh}^\tor$ of a lift of the Hasse invariant to characteristic 0. This depends on the choice of lift, however its reduction modulo $p$ is isomorphic to $\SSS_1$ and the formal completion $\SSS^\ord$ of $\SSS$ along $\SSS_1$ does not depend on the choice of lift. Then, the \emph{Igusa tower} $\Ig = \varinjlim_m \varprojlim_r \TTT_m$ is a pro-étale cover of $\SSS^\ord$ with Galois group $H(\ZZ_p)$. If we want to emphasize the choice of level away from $p$, we write $\level{K^p}{\Ig}$ instead of $\Ig$.

By taking pullback of $\SSS$ via the $\LL_r$-torsor $\level{K_r}{\ol{\Sh}} \to \level{K}{\Sh}^\tor$, we can similarly define the ordinary locus $\level{K_r}{\SSS}$ of $\level{K_r}{\ol{\Sh}}$. By taking formal completion, we obtain $\level{K_r}{\SSS^\ord}$. Given any dominant weight $\kappa$, restricting a modular form $f \in M_\kappa(K_r, R)$ to this ordinary locus defines an element of $H^0(\level{K_r}{\SSS^\ord}, \w_{\kappa, r})$.

\subsubsection{Embeddings of Igusa towers.} \label{subsubsec:embeddings of Igusa towers}
The above is set with the PEL datum $\PP = \PP_1$. More generally, for $1 \leq i \leq 4$, fix a neat open compact subgroup $K_i^p \subset G_i(\AA_f^p)$ and let $\SSS_{m,i}$ be the analogue of $\SSS_m$ associated to the moduli problem associated to $\PP_i$. Let $\TTT_{r,m, i} \to \SSS_{m,i}$ be the corresponding finite étale cover given by \eqref{eq:def level r Igusa cover}.

If $K_3^p \subset K_4^p \cap G_3(\AA_f^p)$ and $K_3^p \subset (K_1^p \times K_2^p) \cap G_3(\AA_f)$, the maps from \eqref{eq:def embeddings i 3 and i 1 2} extend to embeddings
\begin{equation} \label{eq:embeddings of Igusa towers}
    \TTT_{r, m, 3} \hookrightarrow \TTT_{r, m, 4}
        \ \ \ \text{and} \ \ \
    \TTT_{r, m, 3} \hookrightarrow \TTT_{r, m, 1} \times_{\ZZ_p} \TTT_{r, m, 2}\,,
\end{equation}
see \cite[Equations (42)-(43)]{EHLS}.

However, as explained in \cite[Section 2.1.11]{HLS} and \cite[Remark 3.4.1]{EHLS}, at the level of complex points, the inclusion $\level{K_3^p}{\Ig_3} \hookrightarrow \level{K_4^p}{\Ig_4}$ induced by the first map above is not a restriction of the natural embedding $i_3 : \level{K_{3,r}}{\Sh(V_3)} \hookrightarrow \level{K_{4,r}}{\Sh(V_4)}$

In fact, the first inclusion in \eqref{eq:embeddings of Igusa towers} corresponds to the composition of $i_3$ with the shifted inclusion $G_3(\AA) \hookrightarrow G_4(\AA)$ given by $g \mapsto g \cdot \gamma_p$, where $\gamma_p$ corresponds to the element of $G_4(\AA)$ whose component away from $p$ is trivial and whose component at $p$ is
$
    (1,(\gamma_w)_{w \in \Sigma_p}) 
        \in 
    G_4(\QQ_p) 
        = 
    \Gm 
        \times 
    \prod_{w \in \Sigma_p} \GL_n(\KK_w)
$, where
\[
    \gamma_w 
        = 
    \begin{pmatrix}
        1_{a_w} & 0 & 0 & 0 \\
        0 & 0 & 0 & 1_{b_w} \\
        0 & 0 & 1_{a_w} & 0 \\
        0 & 1_{b_w} & 0 & 0
    \end{pmatrix}\,,
\]
via the identification \eqref{eq:identification Gi with GL} for $i=4$. The reader should keep in mind that this shift by $\gamma_p$ plays an important role in Sections \ref{subsubsec:construction f w s +}--\ref{subsubsec:Local integrals at places above p} for the computation of local zeta integrals at $p$.

Then, we have an inclusion
\begin{equation} \label{eq:incl Ig 3 into Ig 4}
    \gamma_p \circ i_3
        :
    \level{K_3^p}{\Ig_3}
        \hookrightarrow
    \level{K_4^p}{\Ig_4}
\end{equation}
as described above. On the other hand, we obtain an inclusion
\begin{equation} \label{eq:incl Ig 3 into Ig 1, 2}
    i_{1,2}
        :
    \level{K_3^p}{\Ig_3}
        \hookrightarrow
    \level{K_1^p}{\Ig_1}
        \times
    \level{K_2^p}{\Ig_2}
\end{equation}
induced by the second map in \eqref{eq:embeddings of Igusa towers} without any shifts involved.

\subsection{Scalar-valued $p$-adic modular forms with respect to $P$} \label{subsec:scalar valued padic mod forms with respect to P}
In this section, we introduce a slightly unconventional definition of scalar-valued $p$-adic modular forms, generalizing the usual notion (see \cite[Section 2.9]{EHLS}). The key idea is to replace the role of the unipotent radical of some standard Borel subgroup with the unipotent radical of the fixed parabolic $P$, see \eqref{eq:def of scalar valued padic mod forms} below. We recover the usual notion when $P = B$ as in Remark \ref{rmk:trivial partition}. In Section \ref{subsec:padic mod forms valued in loc alg reps}, we introduce another notion that allows us to consider vector-valued $p$-adic modular forms.

In Section \ref{sec:Pord Eisenstein measure}, the goal is to construct a $p$-adic family of such scalar-valued $p$-adic modular forms on $G_4$ from the Eisenstein series constructed in Section \ref{sec:Sgl Eis series for dbl method}.

\subsubsection{Global section over the Igusa tower} \label{subsec: global sections over the Igusa tower}
Fix a $p$-adic ring $R$, i.e. $R = \varprojlim_m R/ p^m R$. Assume that $R$ contains the ring $\OO'$ introduced in Section \ref{subsubsec:padic weights}.

For each $r \geq m \geq 0$, let $D_{r,m}$ be the preimage of $D_m = \SSS_m - \SSS_m^0$ (with its reduced closed subscheme structure) in $\TTT_{r,m}$. Then, define
\[
    \VV_{r,m}(R) 
        = 
    H^0(\TTT_{r,m_{/R}}, \OO_{\TTT_{r,m\,{/R}}})
        \ \ \ \text{and} \ \ \
    \VV^{\cusp}_{r,m}(R) 
        = 
    H^0(\TTT_{r,m_{/R}}, \OO_{\TTT_{r,m\,{/R}}}(-D_{r,m}))\,.
\]

Clearly, there is a natural action of $H(\ZZ_p/p^m\ZZ_p) = \Gal(\TTT_{r,m}/\SSS_m)$ on each of these $R$-modules. 

We define the spaces of scalar-valued $p$-adic modular forms (for $G$) of level $K^p$ over $R$ as
\begin{equation} \label{eq:def of scalar valued padic mod forms}
    \VV(K^p; R)
        =
    \varprojlim_m
    \varinjlim_r
        \VV_{r,m}(R)^{P_H^u(\ZZ_p)}
\end{equation}
and its submodule of $p$-adic cuspidal forms as
\begin{equation} \label{eq:def of scalar valued padic cusp forms}
    \VV^\cusp(K^p; R)
        =
    \varprojlim_m
    \varinjlim_r
        \VV^\cusp_{r,m}(R)^{P_H^u(\ZZ_p)}\,.
\end{equation}

\begin{remark}
    We sometimes write $\VV(G, K^p; R)$ and $\VV^\cusp(G, K^p; R)$ to emphasize the underlying reductive group if there is any risk of confusion.
\end{remark}

\begin{remark}
    When $P = B$ as in Remark \ref{rmk:trivial partition}, these spaces agree with the usual spaces of $p$-adic modular and cuspidal forms (see \cite[Section 2.9]{EHLS}). 
\end{remark}

Naturally, these spaces admit an action by $L_H(\ZZ_p) = P_H(\ZZ_p)/P_H^u(\ZZ_p)$. Consider the maximal torus $T_H(\ZZ_p) \subset L_H(\ZZ_p)$ and let $\kappa_p$ be a $p$-adic weight of $T_H(\ZZ_p)$, as in Section \ref{subsubsec:padic weights}. 

Let $\psi_B$ denote a $\bQQ_p$-valued finite-order character of $T_H(\ZZ_p)$. Let $\OO'[\psi_B] \subset \bQQ_p$ denote the smallest ring extension of $\OO'$ containing the values of $\psi_B$.

For any ring $R$ containing $\OO'[\psi_B]$, we set
\begin{equation} \label{eq:def V kappa p Kp psiB R}
    \VV_{\kappa_p}(K^p, \psi_B; R)
        :=
    \{
        f \in \VV(K^p; R) 
            \mid
        t \cdot f = \psi_B(t) \kappa_p(t) f,\,\forall t \in T_H(\ZZ_p)
    \}\,,
\end{equation}
and we define its submodule $\VV^\cusp_{\kappa_p}(K_r, \psi_B; R)$ similarly.

\subsubsection{Classical to $p$-adic modular forms : scalar case.} \label{subsubsec:classical to padic scalar case}
We now adapt the usual map sending classical forms to $p$-adic forms, see \cite[Section 2.9.4]{EHLS}, to our setup.

For all $n \geq 1$, write $\mu_{p^n} = \spec(\ZZ[x, x^{-1}])/(x^{p^n} - 1)$ and identify $\Lie_{\ZZ_p}(\mu_{p^n})$ with a free $\ZZ$-module of rank 1 generated by $x \frac{d}{dx}$. Given any $m \geq 1$ and any $\ZZ/p^m\ZZ$-scheme $S$, this allows us to view $\Lie_S(\mu_{p^n})$ as the structure sheaf $\OO_S$ of $S$ (compatibly as $n$ varies).

Fix a test object $(\ul{A}, \phi) \in \TTT(S)$ over a $p$-adic $R$-algebra $S$. Here, we write $\phi = (\phi_{n,m})_{n \geq m}$ with $\phi_{n,m} \in \TTT_{n,m}$ and consider the subsequence $(\phi_{m,m})_m$. For any $1 \leq r \leq m$, the map $\phi_{m,m}$ induces the isomorphism
\[
    \phi_{m,m,r} 
        : 
    L^+ \otimes \mu_{p^r} 
        \xrightarrow{\sim}
    \Ab_{/S}^\vee[p^r]^\circ\,.  
\]

Furthermore, using the discussion above, the latter also induces an isomorphism
\[
    \Lie(\phi_{m,m, r}) 
        : 
    L^+ \otimes \OO_S 
        = 
    L^+ \otimes \Lie_S(\mu_{p^r})
        \xrightarrow{\sim}
    \Lie_S(\Ab^\vee_{/S}[p^r]^\circ)
        =
    \Lie_S \Ab^\vee_{/S}\,.
\]

Therefore, using the identification $\Lambda_0 \otimes \ZZ_p = L^+$ from Section \ref{subsubsec:padic weights}, we conclude that the tuple 
\[
    (\ul{A_m}, \phi_{m,m,r}, (\Lie(\phi_{m,m,r})^\vee, \id))
\]
lies in $\EE_r(S)$.

Let $\kappa$ be any dominant weight of $T_{H_0}$. For any $r \geq 1$, this yields a map
\begin{equation} \label{eq:def p adification map of scalar valued forms}
    \Omega_{\kappa, r}
        : 
    M_\kappa(K_r; R) 
        \to 
    \VV(K^p; R)\,.
\end{equation}
which sends a modular form $f \in M_\kappa(K_r; R)$ to
\begin{equation} \label{eq:def classical to padic map}
    \Omega_{\kappa, r}(f)(\ul{A}, \phi) 
        :=
    \varprojlim_m f(
        \ul{A}_m, 
        \phi_{m, m, r}, 
        \left( 
            \Lie(\phi_{m, m, r})^\vee, 
            \id 
        \right)
    )
        \in
    \varprojlim_m S/p^mS = S\,.
\end{equation}

The map $\Omega_{\kappa, r}$ is injective by density of $\level{K_r}{\SSS}_{/(\ZZ/p\ZZ)}$ in $\level{K_r}{\ol{\Sh}}_{/(\ZZ/p\ZZ)}$. This follows immediately from the fact that $\SSS_1$ is dense in $\level{K}{\Sh^{\tor}_{/(\ZZ/p\ZZ)}}$. Considering all dominant weights $\kappa$, we define 
\[
    \Omega_{r} 
        :=
    \bigoplus_{\kappa} \Omega_{\kappa, r}\,.
\]

For $R$ sufficiently large, one readily checks that the restriction of $\Omega_{\kappa, r}$ to $M_\kappa(K_r, \psi_B; R)$ factors through the inclusion
\[
    \VV_{\kappa_p}(K_r, \psi_B; R)
        \hookrightarrow
    \VV(K^p; R)\,,
\]
for any $\kappa_p$ and $\psi_B$ as in \eqref{eq:def V kappa p Kp psiB R}.

Hence, we obtain a map $\Omega_{\kappa, r} : M_\kappa(K_r, \psi_B; R) \to \VV_{\kappa_p}(K_r, \psi_B; R)$, where $\kappa_p$ is the $p$-adic weight associated to $\kappa$ as in \eqref{eq:tuple definition of kappa p}.

In fact, the section $f$ on $\level{K_r}{\Sh}^\tor$ only needs to be defined on the ordinary locus for the formula \eqref{eq:def classical to padic map} to be well-defined. In other words, $\Omega_{\kappa, r}$ naturally extends to a map 
\[
    \Omega_{\kappa, r}
        :
    H^0(\level{K_r}{\SSS^{\ord}}_{/R}, \w_{\kappa, r})
        \to 
    \VV_{\kappa_p}(K_r; R)\,.
\]

\begin{conjecture} \label{conj:density conj scalar padic forms}
    The space
    \[
        \left(
            \varinjlim_r
            \Omega_r\left(
                \bigoplus_{\kappa}
                    H^0(\level{K_r}{\SSS^{ord}}_{/R}, \w_{\kappa, r})
            \right)
        \right)[1/p]
            \cap
        \VV(K^p; R)
    \]
    is $p$-adically dense in $\VV(K^p; R)$.
\end{conjecture}

\begin{remark}
    When $P = B$, this density result is a well-known result. See \cite[Proposition 8.2, Theorem 8.3]{Hid04} or \cite[Theorem 2.6.1]{EFMV18}.
\end{remark}

\subsubsection{$P$-ordinary $p$-adic modular forms : scalar case.} \label{subsubsec:Pord padic mod forms scalar case}
For $w \in \Sigma_p$ and $1 \leq j \leq r_w$, let $t_{w, D_w(j)} = t_{w, D_w(j)}^+ \in G(\QQ_p)$ be the matrix introduced in Section \ref{subsubsec:para subgp of G over Zp}, see \eqref{eq:def t w j}.

It is well-known that the double coset $I_r t_{w, D_w(j)} I_r$ can be written as a disjoint union of right cosets with representatives independent of $r$ (for instance, see the calculations in Section \ref{subsubsec:explicit coset representatives}). Note that $\cap_{r \geq 1} I^{\GL}_r = P_H^u(\ZZ_p)$, hence one can use these same representatives for the double coset $P_H^u(\ZZ_p) t_{w, D_w(j)} P_H^u(\ZZ_p)$.

In \cite[8.3.1]{Hid04}, Hida demonstrates that $u_{w, D_w(j)} = P_H^u(\ZZ_p) t_{w, D_w(j)} P_H^u(\ZZ_p)$ can be interpreted as a correspondence on the Igusa tower. See \cite[Section 2.9.5]{EHLS} as well for further details. This naturally induces an action of $u_{w, D_w(j)}$ on $\VV(K^p; R)$ which stabilizes both $\VV^\cusp(K^p; R)$ and $\VV_{\kappa^p}(K^p, \psi_B; R)$. We set
\[
    u_{P,p} 
        =
    \prod_{w \in \Sigma_p}
    \prod_{j=1}^{r_w}
        u_{w, D_w(j)}
    \;\;\;\text{and}\;\;\;
    e_P = e_{P,p}
        =
    \varinjlim_n
        u_{P,p}^{n!}\,,
\]
and define the space of \emph{$P$-ordinary} $p$-adic modular forms and cuspidal forms as
\[
    \VV^{\Pord}(K^p; R) := e_P\VV^{\Pord}(K^p; R)
        \;\;\;\text{and}\;\;\;
    \VV^{\Pord, \cusp}(K^p; R) := e_P\VV^{\Pord, \cusp}(K^p; R)\,,
\]
respectively. We have similar definitions when fixing a $p$-adic weight $\kappa_p$ and a character $\psi_B$.

In Section \ref{sec:Pord Eisenstein measure}, we use the following conjecture (which is known to hold when $P=B$ as in Remark \ref{rmk:trivial partition}) to compare $p$-adic and classical Eisenstein series.
\begin{conjecture} \label{conj:classicality scalar padic forms}
    Let $\kappa$ be a very regular dominant weight and $\psi_B$ be as in \eqref{eq:def V kappa p Kp psiB R}. The restriction
    \[
        \Omega^{\Pord}_{\kappa, r} 
            : 
        S^{\Pord}_\kappa(K_r, \psi_B; R) 
            \to 
        \VV^{\Pord, \cusp}_{\kappa_p}(K_r, \psi_B; R)
    \]
    of $\Omega_{\kappa, r}$ to $P$-ordinary cusp forms is an isomorphism.    
\end{conjecture}

\begin{remark}
    Although this is not the focus of this paper, the author expects that one can weaken that assumption that $\kappa$ is very regular for a condition that is sensitive to our choice of $P$. In this setting, the assumption of being very regular should be strong enough to hold for all choices of parabolic subgroups $P$.
\end{remark}

\subsubsection{Restrictions of $p$-adic forms.} \label{subsubsec:restrictions of padic forms}~
Let $\PP_i$ be one of the PEL datum introduced in Section \ref{subsubsec:theory for G1 and G2}--\ref{subsubsec:theory for G3 and G4}, and let $(G_i, X_i)$ be the associated Shimura datum. For $1 \leq i \leq 4$, fix a neat open compact subgroup $K_i^p \subset G_i(\AA_f^p)$ and assume that $K_3^p$ is contained in both $(K_1^p \times K_2^p) \cap G_3(\AA_f)$ and $K_4^p \cap G_3(\AA_f)$. Then, we obtain restriction maps
\[
    \res_3
        :
    \VV(G_4, K_4^p; R)
        \to
    \VV(G_3, K_3^p; R)
\]
and
\[
    \res_{1,2}
        :
    \VV(G_1, K_1^p; R)
        \wh{\otimes}
    \VV(G_2, K_2^p; R)
        \to
    \VV(G_3, K_3^p; R)
\]
induced by the embeddings in \eqref{eq:embeddings of Igusa towers}, where $\wh{\otimes}$ is the complete tensor product of the $p$-adic ring $R$. Note that $\res_3$ is induced by restricting forms along $\gamma_p \circ \iota_3$, see \eqref{eq:incl Ig 3 into Ig 4}.

\subsubsection{Evaluation at ordinary points} \label{subsubsec:evaluation at ordinary points}
Now, fix $\PP = \PP_i$ for any $1 \leq i \leq 4$ and set $(G, X) := (G_i, X_i)$. Let $(J_0', h_0) \to (G, X)$ be the embedding of Shimura datum, where $J_0'$ is a torus, from \cite[Section 2.3.2]{EHLS}. In particular, $(J_0', h_0)$ defines a CM Shimura subvariety of $\Sh(J_0', h_0)$ of $\Sh(G, X)$. Given a dominant weight $\kappa$ of $T_{H_0}$ and a level subgroup $K = K_r \subset G(\AA_f)$, we obtain a restriction map
\[
    \res_{J_0', h_0} 
        : 
    M_\kappa(G, K_r; R) 
        \to 
    M_\kappa((J_0', h_0); R)\,,
\]
where the modular forms on $(J_0', h_0)$ are defined with respect to an appropriate level subgroup.

As in \cite[Section 3.2.4]{EHLS}, we say that $(J_0', h_0)$ is \emph{ordinary} if at the level of points of moduli problems, the image of $\Sh(J_0', h_0) \to \Sh(G, X)$ only consists of ordinary abelian varieties (with extra structures). In this case, for all $r \geq m \geq 0$, one can similarly define an Igusa variety $\TTT_{r,m}(J_0', h_0)$ as in \eqref{eq:def level r Igusa cover} for $(J_0', h_0)$. 

The embedding of Shimura datum above similarly induces a map $\TTT_{r,m}(J_0', h_0) \to \TTT_{r,m}(G, X)$ on Igusa varieties, with the obvious notation. We write $\VV_{\kappa_p}((G, X), K^p; R)$ for the space of $p$-adic modular forms of weight $\kappa_p$, level $K^p$ and coefficient $R$ associated to $(G, X)$. For the analogous space on $(J_0', h_0)$, we write $\VV_{\kappa_p}((J_0', h_0); R)$ without specifying any level structure (note that there is no need to specify a parabolic subgroup in \eqref{eq:def of scalar valued padic mod forms} as $J_0'$ is a torus). As above, we obtain a restriction map
\[
    \res_{p, J_0', h_0} 
        : 
    \VV_{\kappa_p}((G, X), K^p; R) 
        \to 
    \VV_{\kappa_p}((J_0', h_0); R)\,.
\]

As in Section \ref{subsubsec:classical to padic scalar case}, we obtain embeddings $\Omega_{\kappa, r}$ for both $(G, X)$ and $(J_0', h_0)$. To distinguish both, we write $\Omega_{\kappa, r, G, X}$ (resp. $\Omega_{\kappa, r, J_0', h_0}$) for the map \eqref{eq:def p adification map of scalar valued forms} with respect to $G$ and $X$ (resp. $J_0'$ and $h_0)$). Therefore, by definition, we obtain the following proposition.

\begin{proposition}
    Using the same notation as above, the following hold :
    \begin{enumerate}
        \item The diagram
        \[
        \begin{tikzcd}
            M_\kappa(G, K_r; R)
                \arrow[rr, "\Omega_{\kappa, r, G, X}"]
                \arrow[d, "\res_{J_0', h_0}"] & &
            \VV_{\kappa_p}((G, X), K^p; R)
                \arrow[d, "\res_{p, J_0', h_0}"] \\
            M_\kappa((J_0', h_0); R)
                \arrow[rr, "\Omega_{\kappa, r, J_0', h_0}"] & &
            \VV_{\kappa_p}((J_0', h_0); R)
        \end{tikzcd}
        \]
        is commutative
        \item Let $f \in \VV^{\Pord}_{\kappa_p}((G,X), K^p; R)$. Suppose that $\res_{p, J_0', h_0}(f) = 0$ for every ordinary CM pair $(J_0', h_0)$ mapping to $(G,X)$. Then, $f = 0$.
    \end{enumerate}
\end{proposition}

\begin{proof}
    The proof is identical to the one of \cite[Proposition 3.2.5]{EHLS}.
\end{proof}

\subsection{$p$-adic modular forms valued in locally algebraic representations} \label{subsec:padic mod forms valued in loc alg reps}

In this section, we introduce a different notion of $p$-adic modular forms by considering non-trivial vector bundles over the Igusa tower. The goal is to develop the necessary material to study Hida theory in the context of $P$-ordinary Hecke algebras acting on $P$-ordinary automorphic representations.

In Section \ref{sec:P-(anti-)ord Hida families}, we use the material discussed here and assume certain results (see Conjectures \ref{conj:indep of weight of padic Hecke alg} and \ref{conj:VCT}), to describe the geometry of $P$-ordinary Hida families of automorphic representations. Furthermore, in Section \ref{subsec:Eis measures and padic L fct}, we again rely on these conjectures to adapt the formalism developed in \cite[Section 7.4]{EHLS} to our situation and construct a $p$-adic $L$-function from the Eisenstein measure of Proposition \ref{prop:existence of Eis measure on V3}.

\subsubsection{Locally algebraic coefficient rings.} \label{subsubsec:locally algebraic coeff rings}
Fix a weight dominant $\kappa$ of $T_{H_0}$ and consider the $P$-parallel lattice $[\kappa]$ passing through $\kappa$ as in \eqref{eq:def [kappa]}. Similarly, fix a $P$-nebentypus $\tau$ of level $r$ and consider the $P$-nebentypus equivalence class $[\tau] = [\tau]_r$ of $\tau$. Fix a $p$-adic ring $R$ as in Section \ref{subsec:scalar valued padic mod forms with respect to P} and further assume that $R$ contains $\incl_p(S_r[\tau])$.

Let $V_\kappa$ and $\MMM_{\tau}$ be the $L_H(\ZZ_p)$-representations over $R$ associated to $\kappa$ and $\tau$ (or equivalently, to $[\kappa]$ and $[\tau]$) respectively, as in Section \ref{subsubsec:padic weights}. We view the $R$-module $\hom_R(\MMM_\tau, V_\kappa)$ as a \emph{locally algebraic} representation of $L_H(\ZZ_p)$.

Let $\ul{[\kappa_p, \tau]}$ denote the trivial vector bundle over $\TTT_{r,m\,/R}$ associated to $\hom_R(\MMM_\tau, V_\kappa)$. We include ``$\kappa_p$'' in the notation here instead of $\kappa$ to emphasize the fact that $V_\kappa$ is viewed as a representation of $L_H(\ZZ_p)$ (and not $L_{H_0}(R)$).

Define
\[
    \VV_{r,m}([\kappa_p, \tau]; R) 
        = 
    H^0 (\TTT_{r,m_{/R}}, \ul{[\kappa_p, \tau]} )
\]
and
\[
    \VV^{\cusp}_{r,m}([\kappa_p, \tau]; R) 
        = 
    H^0 (\TTT_{r,m_{/R}}, \ul{[\kappa_p, \tau]}(-D_{r,m}) )\,.
\]

\begin{definition} \label{def:def padic forms loc alg coeff}
    The space of $p$-adic modular forms (for $G$) of level $K^p$ and coefficient $\hom_R(\MMM_\tau, V_\kappa)$ over $R$ is defined as
    \[
        \VV(K^p, [\kappa_p, \tau]; R)
            =
        \varprojlim_m
        \varinjlim_r
            \VV_{r,m}(
                [\kappa_p, \tau]; R
            )^{P_H^u(\ZZ_p)}
    \]
    and its submodule of $p$-adic cuspidal forms is defined as
    \[
        \VV(K^p, [\kappa_p, \tau]; R)^\cusp
            =
        \varprojlim_m
        \varinjlim_r
            \VV^{\cusp}_{r,m}(
                [\kappa_p, \tau]; R
            )^{P_H^u(\ZZ_p)}\,.
    \]
\end{definition} 

\begin{remark}
    The space $\VV(K^p, [\kappa_p, \tau]; R)$ agrees with $\VV(K^p; R)$ exactly when $\kappa_p$ is a scalar weight and $\tau$ is a character.
\end{remark}

These spaces are again naturally equipped with an action of $L_H(\ZZ_p)$ induced by the action of $P_H(\ZZ_p) \subset H(\ZZ_p)$ on Igusa towers.

We view a $p$-adic modular form $f \in \VV(K^p, [\kappa_p, \tau]; R)$ as vector-valued. More precisely, we view $f$ as a functorial rule such that on each $p$-adic ring $S$ over $R$, a test object $(\ul{A}, \phi)$ is assigned by $f$ to an element of $\hom_S(\MMM_{\tau, S}, V_{\kappa, S})$, where $\ul{A} = (\ul{A_m})_m \in \varprojlim_{m} \SSS_m(S)$ and $\phi = (\phi_{r,m}) \in \varinjlim_m \varprojlim_r \TTT_{r, m}(S)$ with $\phi_{r, m}$ over $\ul{A_m}$ for all $r \gg 0$.

\subsubsection{Classical to $p$-adic modular forms : locally algebraic case.} \label{subsubsec:classical to padic mod forms: vector case}
One can adapt the material of Section \ref{subsubsec:classical to padic scalar case} for vector-valued $p$-adic modular forms as well. Indeed, define
\[
    \VV_{\kappa_p}(K^p, \tau; R) 
        := 
    \{
        f \in V(K^p, [\kappa_p, \tau]; R)
            \mid
        l \cdot f 
            = 
        ((\tau \otimes \rho_{\kappa_p})(l))(f)
    \}\,.
\]

Using the fact that $\Lie(\phi \circ l)^\vee = \tp{l}^{-1} \circ \Lie(\phi)^\vee$, for all $l \in L_H(\ZZ_p)$, as well as the relation \eqref{eq:relation kappa vs kappa p on L H ZZp}, one readily checks that given $f \in M_\kappa(K_r, \tau; R)$, the formula
\begin{equation} \label{eq:def Theta kappa tau}
    \Theta_{\kappa, \tau}(f)(\ul{A}, \phi) 
        :=
    \varprojlim_m f(
        \ul{A}_m, 
        \phi_{m, m, r}, 
        \left( 
            \Lie(\phi_{m, m, r})^\vee, 
            \id 
        \right)
    )
\end{equation}
from \eqref{eq:def classical to padic map} similarly yields an injective map
\[
    \Theta_{\kappa, \tau} 
        : 
    M_\kappa(K_r, \tau; R) 
        \to 
    \VV_{\kappa_p}(K^p, \tau; R)\,.
\]

\subsubsection{$P$-ordinary $p$-adic modular forms : locally algebraic case.} \label{subsubsec:Pord padic mod forms locally algebraic case}
As in Section \ref{subsubsec:Pord padic mod forms scalar case}, for a $p$-adic domain $R$ in which $p$ is nonzero, this action stabilizes the image of $\Omega_{\kappa, \tau}$, and given $f \in M_\kappa(K_r, \tau; R)$, we have
\[
    u_{w, D_w(j)} \Omega_{\kappa, \tau}(f) 
        = 
    \kappa'(t_{w, D_w(j)}) U_{w, D_w(j)} f \,,
\]
where $\kappa' = (\kappa_\norm)_p$ is as in Section \ref{subsec:Pord mod forms}. In other words, these operators agree with the operators denoted $u_{w, D_w(j)}$, for $w \in \Sigma_p$ and $1 \leq j \leq r_w$, from Section \ref{subsec:Pord mod forms}. In particular, the operator $u_{w, D_w(j)}$ on (vector-valued) $p$-adic modular forms only depends on $\kappa$ through the $P$-parallel lattice $[\kappa]$.

Once more, we set
\[
    u_{P,p} 
        =
    \prod_{w \in \Sigma_p}
    \prod_{j=1}^{r_w}
        u_{w, D_w(j)}
    \;\;\;\text{and}\;\;\;
    e_P = e_{P,p}
        =
    \varinjlim_n
        u_{P,p}^{n!}\,,
\]
as operators on $\VV(K^p, [\kappa_p, \tau]; R)$. We define the space of \emph{$P$-ordinary} $p$-adic modular forms with coefficients $\hom_R(\MMM_\tau, V_\kappa)$ over $R$ as
\[
    \VV^{\Pord}(K^p, [\kappa, \tau]; R) := e_P\VV^{\Pord}(K^p, [\kappa, \tau]; R)\,,
\]
and we have similar definitions for $\VV^{\Pord, \cusp}(K^p, [\kappa, \tau]; R)$, $\VV_{\kappa_p}^{\Pord}(K^p, \tau; R)$ and $\VV_{\kappa_p}^{\Pord, \cusp}(K^p, \tau; R)$.

\begin{conjecture} \label{conj:classicality vector padic forms}
    Let $\kappa$ be a very regular dominant weight. Let $\tau$ be a $P$-nebentypus of level $r \geq 1$. The restriction
    \[
        \Omega_{\kappa, \tau}
            : 
        S_\kappa^\Pord(K_r, \tau; R) 
            \to 
        \VV^{\Pord, \cusp}_{\kappa_p}(K^p, \tau; R)
    \]
    of $\Omega_{\kappa, \tau}$ to $P$-ordinary cusp forms is an isomorphism.
\end{conjecture}

\subsection{Hecke operators on $p$-adic modular forms.} \label{subsec:Hecke ops on padic mod forms}
Given $g \in G(\AA_f^p)$, the double coset $T(g) := [KgK]$ natural acts on the space of $p$-adic modular forms $\VV^{\cusp}(K_r; R)$. Namely, one easily adapts \eqref{eq:def Hecke op Tg level Kp} for test objects on the Igusa tower instead of the classical Shimura variety. 

Given $f \in M_\kappa(K_r; R)$, we obviously have $T(g) \Omega_{\kappa, r}(f) = \Omega_{\kappa, r}(T(g)f)$. Moreover, this extends to define an operator $T(g)$ on the spaces $\VV^{\cusp}(K_r; R)$, $\VV(K^p, [\kappa_p, \tau]; R)$ and $\VV^\cusp(K^p, [\kappa_p, \tau]; R)$.

Furthermore, given a matrix $t$ in the center $Z_P$ of $L_H(\ZZ_p)$ and $f \in M_\kappa(K_r, \tau; R)$,
\[
    t \cdot \Omega_{\kappa, \tau}(f)
        =
    \kappa'(t) \w_{\tau}(t) f\,,
\]
where $\w_\tau$ is the central character of $\tau$. More generally, $t \cdot \Theta_{\kappa, \tau}(f) = \kappa'(t) (t \cdot f)$, for all $f \in M_\kappa(K_r, [\tau]; R)$.

Namely, we can again view the operator $u_p(t) = u_{p, \kappa}(t)$ introduced in Section \ref{subsec:Pord mod forms} as an endomorphism of $\VV(K^p, [\kappa_p, \tau]; R)$ via the natural action of $P_H^u(\ZZ_p) t P_H^u(\ZZ_p) = tP_H^u(\ZZ_p)$.

In Section \ref{sec:P-(anti-)ord Hida families}, we study the Hecke algebras generated by the operators above and the endomorphisms $u_{w, D_w(j)}$, for $w \in \Sigma_p$ and $1 \leq j \leq r_w$. We use the compatibility between these endomorphisms on classical forms and on $p$-adic forms on several occasions implicitly.

%%%%%%%%%%%%%%%%%%%%%%%%%%%%%%%%%%%%%%%%%%%%%%%%%%%%%%%%%%%%%%%%%%
%%%%%%%%%%%%%%%%%%%%%%%%%%%%%%%%%%%%%%%%%%%%%%%%%%%%%%%%%%%%%%%%%%

\part{Families of $P$-(anti-)ordinary automorphic representations.} \label{part:families of P(a)ord aut reps}

\section{Structure at $p$ of $P$-(anti-)ordinary automorphic representations.} \label{sec:structure at p of P(a-)ord aut reps}

The main results of this section are Theorems \ref{thm:canonical Pword vec of type tau w for G1} and \ref{thm:canonical Pwaord vec of type tau w for G1}. The idea is to describe the space of $P$-ordinary vectors and $P$-anti-ordinary vectors via types.

We first study the case of $G = G_1$. Then, taking into account the conventions set in Section \ref{subsubsec:comp of lvl structures}, all statements are adapted for $G_2$ in Sections \ref{subsec:Pord theory on G2} and \ref{subsec:Paord theory on G2}.

\subsection{$P$-ordinary theory on $G_1$.} \label{subsec:Pord theory on G1}
In what follows, we use the notation of Section \ref{subsec:Paord automorphic representations} freely. In particular, we work with a cuspidal automorphic representation $\pi$ for $G = G_1$ and write $\pi_p = \mu_p \otimes \left( \otimes_{w \in \Sigma_p} \pi_w \right)$ for its $p$-component. 

Assume that $\pi$ is holomorphic and that its weight $\kappa$ satisfies the inequality :
\begin{equation} \label{eq:ineq kappa sigma}
    \kappa_{\sigma, b_\sigma} + \kappa_{\sigma c, a_\sigma} \geq n, \forall \sigma \in \Sigma_\KK \,.
\end{equation}

\subsubsection{Explicit coset representatives} \label{subsubsec:explicit coset representatives}
To clarify arguments in later proofs, we now describe explicit right coset representatives for $U_{w,D_w(j)}^{\GL} = [I_{w,r} t_{w, D_w(j)} I_{w,r}]$. For simplicity, we only compute the right coset representatives when $j \leq t_w$. The same conclusion applies for $j > t_w$ but writing down the matrices is simply more cumbersome. The reader should keep in mind that $t_w$ only denotes an integer while $t_{w, D_w(j)}$ denotes an element of $\GL_n(\OO_w)$, see Remark \ref{rmk:t w vs t w i}.

Fix $j \leq t_w$ and write $i = D_w(j)$ (making the dependence on $j$ implicit). Fix a uniformizer $\varpi \in \p_w$. Given any matrix $X \in I_{w,r}$, write it as
\[
	X = \begin{pmatrix}
		A & B  \\
		\varpi^r C & D
	\end{pmatrix}
\]
where $A \in \GL_{i}(\OO_w)$, $D \in \GL_{i}(\OO_w)$ and $B \in M_{i \times (n-i)}(\OO_w)$ and $C \in M_{(n-i) \times i}(\OO_w)$.

Fix a set $S_w$ of representatives in $\OO_w$ for $\OO_w/p\OO_w$. Let $B', B'' \in M_{i \times (n-i)}(\OO_w)$ be the unique matrices such that $B'$ has entries in $S_w$ and $BD^{-1} = B' + p B''$. Then, we have
\[
	X = 
	\begin{pmatrix}
		1_j & B' \\
		0 & 1_{n-j}
	\end{pmatrix}
	\begin{pmatrix}
		A - \varpi^r B'C & p B''D \\
		\varpi^r C' & D
	\end{pmatrix} =: X' X''
\] 

In particular, $t_{w,i}^{-1}X''t_{w,i}$ is in $I_{w,r}$. Therefore,
\[
	I_{w,r} t_{w,i} I_{w,r} = \bigsqcup_{x \in M_j} xt_{w,i} I_{w,r}
\]
where $M_j \subset \GL_n(\KK_w)$ is the subset of matrices $\begin{pmatrix} 1_{i} & B' \\ 0 & 1_{n-i}\end{pmatrix}$ such that the entries of $B'$ are in $S_w$.

In particular, this set of representative does not depend on $r$ and one obtains the same result by replacing $I_{w,r}$ with $N_w = \cap_r I_{w,r} = P_w^u(\KK_w) \cap \GL_n(\OO_w)$. As mentioned above, one readily sees that the calculations above still apply for $t_w < j \leq r_w$.

Let $V_w$ be the $\KK_w$-vector space associated to $\pi_w$. By continuity, its $N_w$-invariant subspace $V_w^{N_w}$ is equal to $\cup_r V_w^{I_{w,r}}$. 

\begin{lemma} \label{lma:explicit U_j action on N invariants}
	There is a decomposition $V_w^{N_w} = V^{N_w}_{w, \inv} \oplus V^{N_w}_{w, \nil}$ such that, for $1 \leq j \leq r_w$, $U_{w, D_w(j)}^{\GL}$ is invertible on $V^{N_w}_{w, \inv}$ and nilpotent on $V^{N_w}_{w, \nil}$. Moreover, $U_{w, D_w(j)}^{\GL} = I_{w,r} t_{w, D_w(j)} I_{w,r}$ acts as $\delta_{P_w}(t_{D_w(j)})^{-1} t_{D_w(j)}$ on $V_{w, \inv}^{N_w}$.
\end{lemma}

\begin{proof}
    We keep writing $i = D_w(j)$ in this proof and omit the subscript $w$ in what follows. 
    
    The first part is a consequence of the explanations in \cite[Section 5.2]{Hid98}. Moreover, \cite[Proposition 5.1]{Hid98} shows that the natural projection from $V$ to its $P$-Jacquet module $V_P$ induces an isomorphism $V_{\inv}^N \cong V_P$ that is equivariant for the action of all the $U^{\GL}_i$ operators.
    
    From our explicit computations above, it is clear that $U^{\GL}_i$ acts on $V_P$ via $|M_j|t_{i}$, where $|M_j|$ is the cardinality of $M_j$. To see this, simply note that given any $x \in M_j$, $t_{i}^{-1} x t_{i} \in P_w^u(\KK_w)$ fixes $V_P$. Therefore, the result follows since $M_j$ contains exactly $|p|_w^{-i(n-i)} = \delta_P(t_i)^{-1}$ elements.
\end{proof}

It is clear from Lemma \ref{lma:explicit U_j action on N invariants} that any $P_w$-ordinary vector $\phi \in V_w^{N_w}$ lies in $V^{N_w}_{w,\inv}$ and $\pi_w(t_{w, D_w(j)})$ acts on $\phi$ via multiplication by
\begin{equation} \label{eq:s_j eigenvalue} 
    \kappa'(t^{-1}_{w, D_w(j)})
    \delta_{P_w}(t_{w, D_w(j)})
    c_{w, D_w(j)}
        %=
    %|\kappa'(t_{w, D_w(j)})|_p 
    %\delta_{P_w}(t_{w, D_w(j)})
    %c_{w, D_w(j)}
    \,,
\end{equation}
where $c_{w, D_w(j)}$ is its $u^{\GL}_{w,D_w(j)}$-eigenvalue (a $p$-adic unit), and $\kappa' = (\kappa_{\norm})_p$ is related to $\kappa$ as in Section \ref{subsec:Pord mod forms}. In particular, $\phi$ is a simultaneous eigenvector under the action of $\pi_w$ for all matrices $t_{w,D_w(j)} \in G(\QQ_p)$.

\subsubsection{Bernstein-Zelevinsky geometric lemma for $P_w$-ordinary representations.} \label{subsubsec:BZ geom lemma for Pword rep}
In Section \ref{subsubsec:strucutre thm for Pord rep of G1}, we obtain results about the structure of the $P_w$-ordinary subspace of $\pi_w$ via its relation to its $P_w$-Jacquet module, see the proof of Lemma \ref{lma:explicit U_j action on N invariants}. To understand further the $P_w$-Jacquet module of $\pi_w$, we use a version of the Bernstein-Zelevinsky geometric lemma (see \cite[Section VI.5.1]{Ren10} or \cite[Theorem 6.3.5]{Cas95}) that is adapted to our setting, see Lemma \ref{lma:P-Jacquet filtration}. However, we first need to introduce some notation.

\begin{lemma}\label{lma:Jacquet Lemma}
    Let $\pi_w$ be a $P_w$-ordinary representation of $G_w$. There exists a parabolic subgroup $Q_w \subset P_w$ of $G_w$ and a supercuspidal representation $\sigma_w$ of $Q$ such that $\pi_w \subset \ind{Q_w}{G_w} \sigma_w$.
\end{lemma}

\begin{proof}
    The following is a minor modification of the proof of a theorem of Jacquet, see \cite[Theorem 5.1.2]{Cas95}. We omit the subscript $w$ to lighten the notation.
	
    The fact that $\pi$ is $P$-ordinary implies that $\r{P}{G} \pi \neq 0$. By \cite[Theorem 3.3.1]{Cas95}, the latter is both admissible and finitely generated so it admits an irreducible admissible quotient $\rho$ as a representation of $L$. 
	
    By Frobenius reciprocity \cite[Theorem 2.4.1]{Cas95} and the irreducibility of $\pi$, it follows that $\pi \subset \ind{P}{G} \rho$. Then, it is a theorem of Jacquet \cite[Theorem 5.1.2]{Cas95} that there exists a parabolic $Q_L \subset L$ and a supercuspidal representation $\sigma$ of its Levi factor such that $\rho \subset \ind{Q_L}{L} \sigma$. By transitivity of parabolic induction, the result follows.
\end{proof}

Fix an embedding $\pi_w \hookrightarrow \ind{Q_w}{G_w} \sigma_w$ with the notation as in Lemma \ref{lma:Jacquet Lemma}. Let $M_w$ and $Q_w^u$ denote the Levi factor and unipotent radical of $Q_w$. 

Moreover, let $B_w$ denote the Borel subgroup of $G_w$ corresponding to the trivial partitions, as in Remark \ref{rmk:trivial partition}. Let $T_w$ denote the Levi factor of $B_w$. In particular, $T_w$ is the maximal torus of $G_w$.

Let $W$ be the Weyl group of $G_w$ with respect to $(B_w, T_w)$ and consider
\[
    W(P_w, Q_w) = 
        \{x \in W \mid 
            x^{-1}(L_w \cap B_w)x
                \subset B_w, 
            x(M_w \cap B_w)x^{-1}  
                \subset B_w
        \} \ .
\]

According to \cite[Section V.4.7]{Ren10}, for each $x \in W(P_w, Q_w)$,  $xP_wx^{-1} \cap M_w$ is a parabolic subgroup of $M_w$ with Levi factor equal to $xL_wx^{-1} \cap M_w$. Similarly, the Levi factor of the parabolic subgroup $L_w \cap x^{-1}Q_wx \subset L_w$ is $L_w \cap x^{-1}M_wx$.

Denote the natural \emph{conjugation-by-$x$} functor that sends a representation of $xLx^{-1} \cap M_w$ to a representation of $L_w \cap x^{-1}M_wx$ by $(\cdot)^x$. Moreover, let $W(L_w, M_w)$ be the subset of $x \in W(P_w, Q_w)$ such that $xL_wx^{-1} \cap M_w = M_w$, and so $L_w \cap x^{-1}M_wx = x^{-1}M_wx$. Note that this does not imply that $L_w \cap x^{-1}Q_wx$ is equal to $x^{-1}Q_wx$ but rather that its Levi subgroup is $x^{-1}M_wx$.

The following is a version of \cite[Theorem 6.3.5]{Cas95} that is adapted to our setting and notation.

\begin{lemma} \label{lma:P-Jacquet filtration}
    Let $Q_w \subset P_w$ denote standard parabolic subgroups of $G_w$ as above and let $\sigma_w$ be an irreducible supercuspidal representation of $M_w$.
    
    There exists a filtration, indexed by $W(L_w, M_w)$, of $\r{P_w}{G_w} \ind{Q_w}{G_w} \sigma_w$ as a representation of $L_w$ such that the subquotient corresponding to $x \in W_{L_w}$ is isomorphic to $\ind{L_w \cap x^{-1}Q_wx}{L_w} \sigma_w^x$. One can order the filtration so that subquotient corresponding to $x=1$ is a subrepresentation.
\end{lemma}

\begin{proof}
    As in the previous proofs, we drop the subscript $w$ below.
    
    The Bernstein-Zelevinsky geometric lemma (see \cite[Section VI.5.1]{Ren10}) states that there exits a filtration of $\r{P}{G} \ind{Q}{G} \sigma$ such that the corresponding graded pieces are isomorphic to
    \[
	\ind{L \cap x^{-1}Qx}{L} \left( \r{xPx^{-1} \cap M}{M} \sigma \right)^x
    \]
    as $x$ runs over all elements of $W(P,Q)$. Moreover, one can order the filtration so that the factor corresponding to $\sigma$ (i.e. the graded piece corresponding to $x = 1$) is a subrepresentation of $\r{P}{G} \ind{Q}{G} \sigma$.

    Since $\sigma$ is supercuspidal, the graded piece corresponding to $x \in W(P,Q)$ is nonzero if and only if $xLx^{-1} \cap M = M$, i.e. $x \in W(L, M)$. For such an $x$, the graded piece is clearly isomorphic to $\ind{L \cap x^{-1}Qx}{L} \sigma^x$.
\end{proof}

\subsubsection{Structure theorem for $P$-ordinary representations of $G_1$.} \label{subsubsec:strucutre thm for Pord rep of G1}
For simplicity, we assume that $\pi_p$ satisfies the following hypothesis :
\begin{hypothesis} \label{hyp:Qw equals Pw}
    The parabolic subgroup $Q_w$ for $\pi_w$ from Lemma \ref{lma:Jacquet Lemma} is equal to $P_w$ for all $w \in \Sigma_p$. In particular $\sigma_w$ is a supercuspidal representation of $L_w$.
\end{hypothesis}

\begin{remark} \label{rmk:sc support assumption}
    This hypothesis is certainly restrictive in our context. For instance, if $\pi_p$ is $B$-ordinary, then Lemma \ref{lma:Jacquet Lemma} implies that all local factors $\pi_w$ lie in a principal series. Furthermore, if $\pi_p$ is $B$-ordinary (i.e. \emph{ordinary} in the usual sense) then it follows immediately from our definitions that it is also $P$-ordinary. Therefore, the case $Q_w \neq P_w$ can certainly occurs.
    
    One can argue that this is not a major issue since in the situation above, if $\pi_p$ is $B$-ordinary than there is little interest in considering its structure as a $P$-ordinary representation. One only obtains less information this way. However, if $\pi_p$ is a general $P$-ordinary representation whose local factors $\pi_w$ lie in a principal series, it is not necessarily true that $\pi_p$ is also $B$-ordinary. In general, if $\pi_p$ is $P$-ordinary and the supercuspidal support of all $\pi_w$ is $Q_w$, then $\pi_p$ might not be $Q$-ordinary, where $Q = \prod_w Q_w$. Therefore, the hypothesis above restricts us to study certain $P$-ordinary representations that are not $Q$-ordinary with respect to any smaller parabolic $B \subset Q \subsetneq P$.
    
    In subsequent work, the author plans to adapt the proof Theorem \ref{thm:canonical Pword vec of type tau w for G1} to remove this hypothesis.
\end{remark}

\begin{theorem} \label{thm:canonical Pword vec of type tau w for G1}
    Let $\pi$ be a holomorphic $P$-ordinary representation as above satisfying Hypothesis \ref{hyp:Qw equals Pw} such that its weight $\kappa$ satisfies Inequality \eqref{eq:ineq kappa sigma}. Let $\pi_w \subset \ind{P_w}{G_w} \sigma_w$ be its component at $w \in \Sigma_p$ as above, a $P_w$-ordinary representation.
    \begin{enumerate}
        \item[(i)] For $r \gg 0$, let $\phi, \phi' \in \pi_w^{I_r}$ be $P_w$-ordinary vectors. Let $\varphi$ and $\varphi'$ be their respective image in $\ind{P_w}{G_w} \sigma_w$. If $\phi \neq \phi'$, then $\varphi(1) \neq \varphi'(1)$.
        \item[(ii)] For $r \gg 0$, let $\phi \in \pi_w^{I_r}$ be a simultaneous eigenvector for the $u_{w, D_w(j)}$-operators that is not $P_w$-ordinary. Let $\varphi$ be its image in $\ind{P_w}{G_w} \sigma_w$. Then, $\varphi(1) = 0$.
        \item[(iii)] Let $\tau_w$ be a smooth irreducible representation of $L_w(\OO_w)$. Assume there exists an embedding $\tau_w \hookrightarrow \sigma_w$ over $L_w(\OO_w)$. Let $X_w$ be the vector space associated to $\tau_w$, viewed as a subspace of the one associated to $\sigma_w$. 
        
        Then, given $\alpha \in X_w$, there exists some $r \gg 0$ such that $\tau_w$ factors through $L_w(\OO_w/\p_w^r\OO_w)$ and some (necessarily unique) $P_w$-ordinary $\phi_{r,\alpha} \in \pi_w^{I_r}$ such that $\varphi_{r,\alpha}(1) = \alpha$, where $\varphi_{r,\alpha}$ is the image of $\phi_{r,\alpha}$ in $\ind{P_w}{G_w} \sigma_w$. Furthermore, the support of $\varphi_{r,\alpha}$ contains $P_w I_{w,r}$. The map $\alpha \mapsto \phi_{r,\alpha}$ yields an embedding of $L_w(\OO_w)$-representations 
        \[
            \tau_w \hookrightarrow \pi_{w}^{(\Pword, r)} \ .
        \]
    \end{enumerate}
\end{theorem}

\begin{proof}
    This proof is inspired by the one of \cite[Lemma 8.3.2]{EHLS} which is itself inspired by arguments in \cite[Section 5]{Hid98}. By abuse of notation, we will always write $L$ when we mean $L(\KK_w)$. However, we still write $L(\OO_w)$ when referring to its maximal compact subgroup. From now on, we omit the subscript $w$ in this proof.
	
    From Lemma \ref{lma:explicit U_j action on N invariants} (and its proof), we know the space of $P$-ordinary vector is contained in $V^N_{\inv}$ and $\pr_P : V \to V_P$ induces an isomorphism on $V^N_{\inv} \xrightarrow{\sim} V_P$ which is equivariant for the action of $L(\OO)$ and the $u^{\GL}_{D(j)}$-operators. Let $s_P : V_P \to V^N_{\inv}$ denote its inverse.

    Consider the natural inclusion $V \hookrightarrow \ind{P}{G} \sigma$ and the corresponding embedding $V_P \hookrightarrow (\ind{P}{G} \sigma)_P$ as representations of $L$, using the fact that the $P$-Jacquet module functor is exact. Note here that we are using the unnormalized version of the $P$-Jacquet functor (as opposed to the normalized $\r{P}{G}$).

    Consider the filtration indexed by $W(L, L)$ of $(\ind{P}{G} \sigma)_P$ from Lemma \ref{lma:P-Jacquet filtration}. We use a version with unnormalized $P$-Jacquet functor, hence the graded piece corresponding to $x \in W(L, L)$ is isomorphic to $\sigma^x \delta_P^{1/2}$.
 
    First, we claim that $\pr_P$ maps any simultaneous eigenvector for the $u_{D(j)}$-operators whose eigenvalues are all $p$-adic units inside the subrepresentation $\sigma \delta_P^{1/2}$ corresponding to $x=1$.

    One readily checks that $x \in W(P,P)$ is in $W(L,L)$ if and only if it simply permutes the $\GL_{n_k}(\KK_w)$-blocks of $L$ of the same size. In particular, exactly one such $x \in W(L,L)$ acts trivially on the center $Z(L)$ of $L$, namely $x = 1$, while any other $1 \neq x \in W(L,L)$ stabilizes but acts non-trivially on $Z(L)$.
	
    Using the explicit representatives from Section \ref{subsubsec:explicit coset representatives}, one readily checks that the operator $u^{\GL}_{D(j)}$ acts on $\sigma^x \delta_P^{1/2}$ via multiplication by
    \[
        \beta_x(s_j) 
            := 
        \kappa' (s_j)
        \delta_P^{-1/2}(s_j) 
        \omega_{\sigma}^x(s_j)
            =
        \absv{\kappa' (s_j)}_p^{-1}
        \delta_P^{-1/2}(s_j) 
        \omega_{\sigma}^x(s_j)
    \]
    where $s_j = t_{D(j)}$, $\omega_{\sigma} : Z(L) \to \CC^\times$ is the central character of $\sigma$, and $\w_\sigma^x(-) = \w_\sigma(x (-) x^{-1})$ is the central character of $\sigma^x$.

    These $\beta_x$ define unramified characters of $Z(L)$. The $P$-ordinarity assumption implies that $\beta_1(s_j)$ is a $p$-adic unit for all $1 \leq j \leq t + r$ and therefore $\beta_1(s)$ is a $p$-adic unit for all $s \in Z(L)$. We claim that given any $x \in W(L, L)$, the values of $\beta_x$ on $Z(L)$ are all $p$-adic units if and only if $x = 1$.

    By recalling that $\delta_P$ and $\delta_B$ agree on $Z(L)$ and proceeding exactly as in the proof of \cite[Lemma 8.3.2]{EHLS}, one uses Inequality \eqref{eq:ineq kappa sigma} to show that 
    \[
	\theta = |\kappa'|^{-1} \delta_P^{-1/2}
    \]
    is a regular character of $Z(L)$ and $\beta_x$ satisfies the above property if and only if $\theta^x = \theta$. By regularity, this only occurs when $x = 1$.

    The argument above shows that under the natural map
    \begin{equation} \label{eq:V N inv to rho delta_P}
        V_{\inv}^N \hookrightarrow 
            V \twoheadrightarrow 
                V_P \hookrightarrow 
                    (\ind{P}{G} \sigma)_P \ ,
    \end{equation}
    the subspace of $P$-ordinary vector of $V$ injects into the subrepresentation $\sigma \delta_P^{1/2}$ of $(\ind{P}{G} \sigma)_P$, as desired.

    This map is exactly the composition of $V_{\inv}^N \xrightarrow{\sim} V_P$ with the map $i : V_P \to \sigma \delta_P^{1/2}$ corresponding under the Frobenius reciprocity to the inclusion $v \mapsto f_v$ of $V$ into $\ind{P}{G} \sigma$. In other words, this map is $v \mapsto f_v(1)$. Therefore, a $P$-ordinary vector $v \in V^N$ is uniquely determined by $f_v(1)$. This shows part (i). 

    For part (ii), pick a simultaneous eigenvector $v \in V^N_{\inv}$ for the $u^{\GL}_{D(j)}$-operators that is not $P$-ordinary. Then, as above, the composition $V^N_{\inv} \xrightarrow{\sim} V_P \to \sigma \delta_P^{1/2}$ sends $v$ to $f_v(1)$. By equivariance of the action of the $u^{\GL}_{D(j)}$-operators on both sides, we must have $f_v(1) = 0$.

    To show part (iii), consider $\alpha$ as an element of the vector space associated to $\sigma$, which is also the one associated to $\sigma \delta_P^{1/2} \subset V_P$. Let $\phi = s_P(\alpha) \in V_{\inv}^N$. In particular, $\phi \in \pi^{I_r}$ for some $r \gg 0$. We may assume that $r$ is sufficiently large so that $\tau$ factors through $L(\OO/\p^r\OO)$.

    Finally, since $\pr_P$ is equivariant under the action of the $u^{\GL}_{D(j)}$-operators and these act on $\pr_P(\phi) = \alpha$ via multiplication by the $p$-adic unit $\beta(s_j)$, one concludes that $\phi$ is $P$-ordinary. Proceeding as in the proof of part (i), we obtain $\varphi(1) = \pr_P \phi = \alpha$, where $\varphi \in \ind{P}{G} \sigma$ is the function corresponding to $\phi$. 

    Therefore, $\phi_{r,\alpha} := \phi$ is the desired vector, necessarily unique by part (i). The last statement holds because $s_P$ is $L(\OO_w)$-equivariant.
\end{proof}

\begin{remark} \label{rmk:eigenvalues of Pword representations}
    As a consequence of the proof for part (i) above, we see that $\pi_w$ is $P_w$-ordinary  (of level $r \gg 0$) if and only if
    \begin{equation} \label{eq:central character condition}
        \beta(s) = 
            \absv{\kappa'(s)}_p^{-1}
            \delta_{P_w}^{-1/2}(s)
            \w_\sigma(s)
    \end{equation}
    is a $p$-adic unit for all $s \in Z(L_w(\KK_w))$. In other words, not all supercuspidal representation $\sigma_w$ can occur. Furthermore, when $\pi_w$ is $P_w$-ordinary (of level $r \gg 0$), then all $P_w$-ordinary vectors share the same $u_{w, D_w(j), \kappa}^{\GL}$-eigenvalue, namely $\beta(t_{w, D_w(j)})$.
\end{remark}

\begin{remark} \label{rmk:covers of types to P-Iwahori}
    We now view $\tau_w$ as as a representation of $I_{w, r}^0$ via the identity $I_{w,r}^0 / I_{w,r} = L_w(\OO_w/\p_w^r\OO_w)$. Clearly, the embedding constructed in Theorem \ref{thm:canonical Pword vec of type tau w for G1} (iii) is an embedding of $I_{w,r}^0$-representations. 
    
    This shows $\pi_w$ contains a cover of $\tau_w$ from $L_w$ to $I_{w,r}^0$, in the sense of \cite{BusKut98, BusKut99}, in its subspace $\pi_{w,r}^{\Pword}[\tau_w]$ of $P_w$-ordinary vectors of type $\tau_w$. However, we do not use this point of view explicitly in this paper.
\end{remark}

Recall that in Section \ref{subsubsec:SZ types}, we fixed an ``SZ-types'' $\tau_w$, namely a smooth irreducible representation of $L_w(\OO_w)$ such that $\sigma_w|_{L_w(\OO_w)}$ contains $\tau_w$ with multiplicity one. Such a representation exists but is not necessarily unique.

We sometimes refer to $\tau_w$ as the SZ-type of $\pi_w$. Let $\tau$ be the representation of $L_P(\ZZ_p)$ corresponding to $\otimes_{w \in \Sigma_p} \tau_w$ under the natural identification $L_P = \prod_{w \in \Sigma p} L_w$ induced by \eqref{eq:facto G(Qp)}. We refer to $\tau$ as the \emph{(fixed choice of) SZ-type of $\pi_p$.}

\begin{theorem} \label{thm:Pord type tau structure thm}
    Let $\pi$ be a holomorphic $P$-ordinary representation as above such that its weight $\kappa$ satisfies Inequality \eqref{eq:ineq kappa sigma}. Let $\tau$ be the SZ-type of $\pi_p$. Then,
    \[
        \hom_{L_P(\ZZ_p)}
        (
            \tau, 
            \pi_p^{(\Pord, r)}
        )
    \]
    is 1-dimensional for all $r \gg 0$. In other words, $\pi$ is of $P$-WLT $(\kappa, K_r, \tau)$ for all $r \gg 0$, the space $\pi_{p}^{(\Pord,\tau)} := \pi_{p}^{(\Pord, r)}[\tau]$ of $P$-ordinary vectors of type $\tau$ is independent of $r \gg 0$ and
    \[
        \dim 
        \left(
            \pi_p^{(\Pord, \tau)}
        \right) 
            = 
        \dim \tau\,.
    \]
\end{theorem}

\begin{proof}
    Fix $w \in \Sigma_p$ and consider $\pi_w^{(\Pword,r)} = e_w \pi_w^{I_{w,r}}$ for $r \gg 0$. By Theorem \ref{thm:canonical Pword vec of type tau w for G1} (iii), there is a natural isomorphism
    \[
        \hom_{L_w(\OO_w)}(\uptau_w, \sigma_w) = \hom_{L_w(\OO_w)}(\uptau_w, \pi_{w}^{(\Pword,r)}[\uptau_w]) \ ,
    \]
    where $\uptau_w$ is any smooth irreducible representation of $L_w(\OO_w)$. The result follows by applying the above to $\uptau_w = \tau_w$. 
\end{proof}

\subsection{$P$-anti-ordinary theory on $G_1$.} \label{subsec:Paord theory on G1}
Let $\pi$ be an anti-holomorphic cuspidal representation on $G = G_1$ of weight $\kappa$ such that $\pi_f^{K_r} \neq 0$. Recall that $\pi$ is $P$-anti-ordinary of level $r$ if $\pi_w$ is $P_w$-anti-ordinary of level $r$, for all $w \in \Sigma_p$.

Recall that according to our conventions set in Section \ref{sec:notation and conventions}, given any representation $\rho$, we denote its contragredient representation by $\rho^\vee$. 

\begin{lemma} \label{lma:Pword and Pwaord duality}
    The representation $\pi_w$ is $P_w$-anti-ordinary of level $r \geq 0$ if and only if $\pi_w^\vee$ is $P_w$-ordinary of level $r$. In that case, $\pi_w$ is $P$-anti-ordinary of all level $r \gg 0$.
\end{lemma}

\begin{proof}
    This is a simple generalization of \cite[Lemma 8.3.6 (i)]{EHLS}. The proof goes through verbatim by replacing the pro-$p$ Iwahori subgroup (also denoted $I_{w,r}$) by $I_{P_w, w, r}$ and only considering the Hecke operators $u_{w, D_w(j)}^{\GL, -}$ and $u_{w, D_w(j)}^{\GL}$, for $1 \leq j \leq r_w$. The key part is that all these operators commute with one another.
\end{proof}

\subsubsection{Conventions on contragredient pairings.} \label{subsubsec:conv on contragredient pairings}
Let $\sigma_w$ be an admissible irreducible supercuspidal representation of $L_w(\KK_w)$. Its contragredient $\sigma_w^\vee$ is again an admissible irreducible supercuspidal representation of $L_w(\KK_w)$.

Let $\brkt{\cdot}{\cdot}_{\sigma_w} : \sigma_w \times \sigma_w^\vee \to \CC$ be the tautological pairing on a pair of contragredient representations. Define
\begin{align*}
    \brkt{\cdot}{\cdot}_{w} 
        &: 
    \ind{P_w}{G_w} \sigma_w 
        \times 
    \ind{P_w}{G_w} \sigma_w^\vee
        \to
    \CC 
        \\
    \brkt{\varphi}{\varphi^\vee}_{w} 
        &=
    \int_{G_w(\OO_w)}
        \brkt{\varphi(k)}{\varphi^\vee(k)}_{\sigma_w}
    dk \,,
\end{align*}
a perfect $G_w(\KK_w)$-equivariant pairing. Here $dk$ is the Haar measure on $G_w(\OO_w)$ that such that $\vol(G_w(\OO_w)) = 1$ with respect to $dk$. Then $\brkt{\cdot}{\cdot}_{w}$ naturally identifies $\ind{P_w}{G_w} \sigma_w^\vee$ as the contragredient of $\ind{P_w}{G_w} \sigma_w$.

Let $\pi_w$ be the constituent at $w \in \Sigma_p$ of $\pi_p$ as above. From now on, we assume $\pi_w$ is the unique irreducible quotient $\ind{P_w}{G_w} \sigma_w \twoheadrightarrow \pi_w$. Equivalently, $\pi_w^\vee$ is the unique irreducible subrepresentation $\pi_w^\vee \hookrightarrow \ind{P_w}{G_w} \sigma_w^\vee$, see Remark \ref{rmk:sc support assumption}. If one restricts the second argument of $\brkt{\cdot}{\cdot}_{w}$ to $\pi_w^\vee$, then the first argument factors through $\pi_w$. In other words, $\brkt{\cdot}{\cdot}_{w}$ induces the tautological pairing $\brkt{\cdot}{\cdot}_{\pi_w} : \pi_w \times \pi_w^\vee \to \CC$ and
\[
    \brkt{\phi}{\phi^\vee}_{\pi_w} =
        \int_{G_w(\OO_w)}
            \brkt{\varphi(k)}{\varphi^\vee(k)}_{\sigma_w}
        dk \ , \ \ \ \forall \phi \in \pi_w, \phi^\vee \in \pi_w^\vee \ ,
\]
where $\varphi$ is any lift of $\phi$ and $\varphi^\vee$ is the image of $\phi^\vee$.

Let $(\tau_w, X_w)$ be the SZ-type of $\sigma_w$, a representation of $L_w(\OO_w)$. Then, its contragredient $(\tau_w^\vee, X_w^\vee)$ is the SZ-type of $\sigma_w^\vee$. One can find $L_w(\OO_w)$-embeddings $\tau_w \hookrightarrow \sigma_w$ and $\tau_w^\vee \hookrightarrow \sigma_w^\vee$ (both unique up to scalar) such that for all $\alpha \in X_w$, $\alpha^\vee \in X_w^\vee$,
\[
    \brkt{\alpha}{\alpha^\vee}_{\sigma_w}
        =
    \brkt{\alpha}{\alpha^\vee}_{\tau_w} \ .
\]

More generally, upon restriction of $\sigma_w$ and $\sigma_w^\vee$ to representations of $L_w(\OO_w)$, there are direct sum decompositions
\[
    \sigma_w = \bigoplus_{\uptau_w} \sigma_w[\uptau_w]
    \ \ \ \text{and} \ \ \
    \sigma_w^\vee = \bigoplus_{\uptau_w} \sigma_w^\vee[\uptau_w]
\]
where $\uptau_w$ runs over all smooth irreducible representations of $L_w(\OO_w)$ and the square brackets $[\cdot]$ denote isotypic subspaces. The restriction of $\brkt{\cdot}{\cdot}_{\sigma_w}$ to 
$
    \sigma_w[\uptau_w]
        \times
    \sigma_w^\vee[\uptau'_w]
$
is identically zero if $\uptau'_w \not\cong \uptau_w^\vee$. On the other hand, its restriction to
$
    \sigma_w[\uptau_w]
        \times
    \sigma_w^\vee[\uptau^\vee_w]
$
is a perfect $L_w(\OO_w)$-invariant pairings.

\subsubsection{Structure theorem for $P$-anti-ordinary representations of $G_1$.} \label{subsubsec:strucutre thm for Paord rep of G1}
\begin{theorem} \label{thm:canonical Pwaord vec of type tau w for G1}
    Let $\pi$ be an anti-holomorphic $P$-anti-ordinary representation such that its weight $\kappa$ satisfies Inequality \eqref{eq:ineq kappa sigma}. Let $w \in \Sigma_p$ and $\pi_w$ be a constituent of $\pi$, a $P_w$-anti-ordinary representation of level $r \gg 0$. Assume $\pi_w$ is the unique irreducible quotient $\ind{P_w}{G_w} \sigma_w \twoheadrightarrow \pi_w$ for some supercuspidal $\sigma_w$ and let $(\tau_w, X_w)$ be the SZ-type of $\sigma_w$. 
    
    Given any $\alpha \in X_w$, let $\varphi_{w,r}^{\Pwaord} \in \ind{P_w}{G_w} \sigma_w$ be the unique vector with support $P_wI_{w,r}$ such that $\varphi_{w,r}^{\Pwaord}(1) = \alpha$ and $\varphi_{w,r}^{\Pwaord}$ is fixed by $I_{w,r}$.
    
    Then, the image $\phi_{w,r}^{\Pwaord} \in \pi_w^{I_{w,r}}$ of $\varphi_{w,r}^{\Pwaord}$ is $P_w$-anti-ordinary of level $r \gg 0$. Furthermore, it satisfies :
    \begin{enumerate}[label=(\roman*)]
        \item Let $\phi^\vee \in \pi_{w}^{\vee, I_{w,r}}$ and denote its image in $\ind{P_w}{G_w} \sigma_w$ by $\varphi^\vee$. Then, 
        \[
            \brkt{
                \phi_{w,r}^{\Pwaord}
            }{
                \phi^\vee
            }_{\pi_w} 
            = 
                \vol(I_{w,r}^0)
                \brkt{\alpha}{\varphi^\vee(1)}_{\sigma_w}
        \] 

        In particular, $\brkt{\phi_{w,r}^{\Pwaord}}{\phi^\vee}_{\pi_w} \neq 0$ if and ony if $\phi^\vee$ is $P_w$-ordinary and the component of $\varphi^\vee(1)$ in $\sigma_w^\vee[\tau_w^\vee]$ is non-zero.
        
        \item The vector $\phi_{w,r}^{\Pwaord}$ lies in the $\tau_w$-isotypic space of $\pi_w^{I_{w,r}}$. Moreover, any other $P_w$-anti-ordinary vector of type $\tau_w$ is obtained as above for some other choice $\alpha' \in X_w$.
        
        \item One can pick different choices of $\alpha$ for each $r' \geq r$ so that
        \begin{equation} \label{eq:trace of canonical Pwaord vectors for G1}
            \sum_{
                \gamma \in I_{w,r}/(I_{w,r'}^0 \cap I_{w,r})
            }
                \pi_w(\gamma) \phi_{w,r'}^{\Pwaord} = 
                \phi_{w,r}^{\Pwaord}
        \end{equation}
    \end{enumerate}
\end{theorem}

\begin{proof}
    Write $\phi_{w,r}$ and $\varphi_{w,r}$ instead of $\phi_{w,r}^{\Pwaord}$ and $\varphi_{w,r}^{\Pwaord}$ respectively. We first show that property (i) holds. By Lemma \ref{lma:Pword and Pwaord duality}, $\pi_w^\vee$ is $P_w$-ordinary of level $r$. Write
    \[
        \pi_w^{\vee, I_{w,r}} = \bigoplus_{a=1}^A V_a \ ,
    \]
    where each $V_a$ is a simultaneous generalized eigenspace for the Hecke operators $u_{w,D_w(j)}^{\GL}$. 
    
    From the proof of Theorem \ref{thm:canonical Pword vec of type tau w for G1} and the remark that follows, exactly one $V_a$ has generalized eigenvalues that are all $p$-adic units. We may assume that this holds true for $V_1$. The exact eigenvalue of $u_{w, D_w(j)}^{\GL}$ is given by Equation \eqref{eq:central character condition}, denote it $\beta_{w, D_w(j)}$. For $1 < a \leq A$, at least one generalized eigenvalue for $V_a$ is not a $p$-adic unit.
    
    Given $\phi^\vee \in \pi_w^{\vee, I_{w,r}}$, write it as a sum
    \[
        \phi^\vee = 
            \sum_{a=1}^A \phi_a^\vee \ , 
    \]
    with $\phi_a^\vee \in V_a$. Let $\varphi_a^\vee$ denote the images of $\phi_a^\vee$ in $\ind{P_w}{G_w} \sigma_w^\vee$. Then,
    \[
        \brkt{\phi_{w,r}}{\phi^\vee}_{\pi_w} 
        = 
            \sum_{a=1}^A
                \brkt{
                    \phi_{w,r}
                }{
                    \phi_a^\vee
                }_{\pi_w} 
        = 
            \sum_{a=1}^A
            \int_{G_w(\OO_w)}
                \brkt{
                    \varphi_{w,r}(k)
                }{
                    \varphi_a^\vee(k)
                }_{\sigma_w}
            dk
    \]
    
    Recall that the support of $\varphi_{w,r}$ is $P_wI_{w,r}$. Also, the intersection of $P_wI_{w,r}$ with $G_w(\OO_w)$ is equal to $I_{w,r}^0$ and by Theorem \ref{thm:canonical Pword vec of type tau w for G1} (ii), $\varphi_a^\vee(I_{w,r}^0) = 0$ for all $a \neq 1$. Therefore,
    \[
        \brkt{\phi_{w,r}}{\phi^\vee}_{\pi_w} 
        =
            \int_{I_{w,r}^0}
                \brkt{
                    \varphi_{w,r}(k)
                }{
                    \varphi_1^\vee(k)
                }_{\sigma_w}
            dk
    \]

    Since $I_{w,r}^0 = L_w(\OO_w) I_{w,r}$ and $\varphi_{w,r}^{\Pwaord}$, $\varphi_1^\vee$ are both fixed by $I_{w,r}$, one obtains
    \begin{align*}
        \brkt{\phi_{w,r}}{\phi^\vee}_{\pi_w} 
        =
            \int_{I_{w,r}^0}
                \brkt{
                    \varphi_{w,r}(1)
                }{
                    \varphi_1^\vee(1)
                }_{\sigma_w}
            dk
        =
            \vol(I_{w,r}^0)
                \brkt{
                    \alpha
                }{
                    \varphi_1^\vee(1)
                }_{\sigma_w} \ .
    \end{align*}

    The desired relation holds by noting that $\varphi_1^\vee(1) = \varphi^\vee(1)$. The second part of (i) follows immediately from the discussion about isotypic subspaces at the end of Section \ref{subsubsec:conv on contragredient pairings}.

    As a consequence of property (i), we immediately obtain $\brkt{\phi_{w,r}}{V_a}_{\pi_w} = 0$ for all $a > 1$. Furthermore, for all $\phi^\vee \in V_1$, we have
    \[
        \brkt{
            u_{w, D_w(j)}^{\GL, -} 
            \phi_{w,r}
        }{
            \phi^\vee
        }_{\pi_w}
            =
        \brkt{
            \phi_{w,r}
        }{
            u_{w, D_w(j)}^{\GL} 
            \phi^\vee
        }_{\pi_w}
            =
        \beta_{w, D_w(j)}
        \brkt{
            \phi_{w,r}
        }{
            \phi^\vee
        }_{\pi_w} \ .
    \]

    By combining these two facts, we obtain
    \[
        \brkt{
            u_{w, D_w(j)}^{\GL, -} 
            \phi_{w,r}
        }{
            \phi^\vee
        }_{\pi_w}
            =
        \beta_{w, D_w(j)}
        \brkt{
            \phi_{w,r}
        }{
            \phi^\vee
        }_{\pi_w} \ .
    \]
    for all $\phi^\vee$ in $\pi_w^{\vee, I_{w,r}}$. In other words, $\phi_{w,r}$ is $P_w$-anti-ordinary.

    Furthermore, note that the argument above implies that the subspace of $P_w$-anti-ordinary vectors of type $\tau_w$ in $\pi_w^{I_{w,r}}$ is dual to the subspace of $P_w$-ordinary vectors of type $\tau_w^\vee$. From Theorem \ref{thm:canonical Pword vec of type tau w for G1} (iii), they both have dimension $\dim \tau_w = \dim \tau_w^\vee$. 
    
    In particular, the space generated by the action of $L_w(\OO_w/\p_w^r\OO_w)$ on $\phi_{w,r}$, which is of dimension $\dim \tau_w$, is exactly the subspace of $P_w$-anti-ordinary vectors of type $\tau_w$. Therefore, any other $P_w$-anti-ordinary vector $\phi_{w,r}'$ of type $\tau_w$ is equal to $\pi_w(l) \phi_{w,r}$, for some $l \in L_w(\OO_w/\p_w^r\OO_w)$. One readily sees that it obtained by picking $\alpha' = \tau_w(l)\alpha$ in $X_w$ instead of $\alpha$. This proves the second sentence of part (ii).

    Finally, part (iii) follows immediately from the fact that the analogous properties hold for $\varphi_{w,r}$. 
\end{proof}

Keeping the assumption and notation of Theorem \ref{thm:canonical Pwaord vec of type tau w for G1}, fix a vector $\alpha \in X_w$. From Lemma \ref{lma:Pword and Pwaord duality}, we know $\pi_w^\vee \hookrightarrow \ind{P_w}{G_w} \sigma_w^\vee$ is $P_w$-ordinary. 

Let $(\tau_w^\vee, X^\vee)$ be the SZ-type of $\pi_w^\vee$ and fix any $\alpha^\vee \in X^\vee$ such that $\brkt{\alpha}{\alpha^\vee}_{\tau_w} = 1$. Let $\phi_{w,r}^{\vee, \Pword}$ be the $P_w$-ordinary vector associated to $\alpha^\vee$ obtained from Theorem \ref{thm:canonical Pword vec of type tau w for G1} (iii). 

In fact, as $r$ increases, one may pick compatible choices of $\alpha$ so that \eqref{eq:trace of canonical Pwaord vectors for G1} holds and compatible choices of $\alpha^\vee$ such that $\brkt{\alpha}{\alpha^\vee}_{\tau_w} = 1$ for all $r \gg 0$. Then, as a consequence of Theorem \ref{thm:canonical Pwaord vec of type tau w for G1} (i),
\[
    \frac{
        \brkt
        {
            \phi_{w,r}^{\Pwaord}
        }{
            \phi_{w,r}^{\vee, \Pword}
        }_w 
    }{
        \vol(I_{w,r}^0)
    }
    = 
        \brkt{\alpha}{\alpha^\vee}_{\sigma_w}
    =
        \brkt{\alpha}{\alpha^\vee}_{\tau_w} = 1
\]  
is independent of $r \gg 0$.

Furthermore, one readily obtains a result analogous to Theorem \ref{thm:Pord type tau structure thm} from Theorem \ref{thm:canonical Pwaord vec of type tau w for G1}. Namely, let $\tau = \bigotimes_{w \in \Sigma_p} \tau_w$ be the SZ-type of $\pi_p$, using the identification \eqref{eq:facto G(Qp)}.

\begin{corollary} \label{cor:canonical Paord vec of type tau for G1}
    Let $\pi$ be an anti-holomorphic cuspidal representation of $G$ of weight $\kappa$. Suppose $\kappa$ satisfies Inequality \eqref{eq:ineq kappa sigma}. Then, $\pi$ is $P$-anti-ordinary if and only if $\pi^\vee$ is $P$-ordinary. 
    
    In that case, assume $\pi^\vee$ satisfies Hypothesis \ref{hyp:Qw equals Pw} and let $\tau = \bigotimes_{w \in \Sigma_p} \tau_w$ be the SZ-type of $\pi$. There exists a unique (up to the action of $L_P(\ZZ_p)$) $P$-anti-ordinary vector $\phi_{r}^{\Paord}$ of level $r$ and type $\tau$ in $\pi_{p}^{I_{P, r}}$. Furthermore, for each $w \in \Sigma_p$, there exists $P_w$-ordinary vectors $\phi_{w, r}^{\Pwaord}$ of level $r$ and type $\tau_w$ as in Theorem \ref{thm:canonical Pwaord vec of type tau w for G1} such that, under the identification $\pi_p = \mu_p \otimes \bigotimes_{w \in \Sigma_p} \pi_w$, we have $\phi_{r}^{\Paord} = \bigotimes_{w \in \Sigma_p} \phi_{w, r}^{\Pwaord}$.
\end{corollary}

\subsection{$P$-ordinary theory on $G_2$.} \label{subsec:Pord theory on G2}
In this section, we proceed to compare the theory of $P$-(anti-)ordinary representations on $G_1$ and $G_2$, where $G_i$ is the unitary group associated to the PEL datum $\PP_i$ introduced in Section \ref{subsubsec:theory for G1 and G2}. We add a subscript $V$ (resp. $-V$) in our notation whenever we want to emphasize that we are working with $G_1$ (resp. $G_2$).

\subsubsection{Comparison between representations of $G_1$ and $G_2$.}
Recall that there is a canonical identification $G_1(\AA) = G_2(\AA)$. Furthermore, the identification from isomorphism \eqref{eq:prod G over Zp} remains the same for both $G_1$ and $G_2$. However, the opposite choices of $\OO_\KK \otimes \ZZ_p$-lattices $L_1^\pm = L_2^\mp$ introduce many changes in the notation. 

For instance, under the identification $G_1(\AA) = G_2(\AA)$, the group $H_{0, -V} = H_0$ for $G_1$ corresponds to $H_{0, -V}$ (by switching the roles of $\Lambda_0$ and $\Lambda_0^\vee$.) However, the identification from isomorphism \eqref{eq:prod H0 over S0} interchanges the role of $\sigma \in \Sigma_\KK$ and $\sigma c$ (where $c$ denotes complex conjugation). 

Given a dominant weight $\kappa$ of $T_1 := T_{H_0, V}$, it is identified with a tuple $(\kappa_0, (\kappa_\sigma)_\sigma)$ where $\kappa_0 \in \ZZ$ and $\kappa_{\sigma} \in \ZZ^{b_\sigma}$. The torus $T_2 := T_{H_0, -V}$ is equal to $T_1$ but the corresponding isomorphism \eqref{eq:prod H0 over S0} for $G_2$ identifies $\kappa$ with $(\kappa_0, (\kappa_{\sigma c})_\sigma)$. We denote the latter by $\kappa^\flat$. In particular, $\kappa_{\sigma c} \in \ZZ^{a_\sigma} = \ZZ^{b_{\sigma c}}$ and $\kappa^\flat$ is dominant with respect to $B_{H_0, -V}^\opp$.

As explained in \cite[Sections 6.2.1-6.2.2]{EHLS}, if $\pi$ is a cuspidal (anti-)holomorphic automorphic representation for $G_1$ of weight $\kappa$, then $\pi^\flat = \pi^\vee \otimes ||\nu||^{a(\kappa)}$ (as in Section \ref{subsubsec:the involution flat}) is naturally a cuspidal (anti-)holomorphic automorphic representation for $G_2$ of weight $\kappa^\flat$.

Furthermore, by choosing the same partitions $\d_w$ introduced in Section \ref{subsubsec:para subgp of G over Zp}, the parabolic subgroup $P_w \subset \GL_n(\OO_w)$ for $G_1$ corresponding to $w \in \Sigma_p$ is replaced by the opposite parabolic subgroups, which in our case is simply its transpose $\tp{P}_w \subset \GL_n(\OO_w)$, when working with $G_2$. Similarly, $P$ is replaced by $\tp{P}$ and the (resp. pro-$p$) $P$-Iwahori subgroup of level $r$ is replaced by the (resp. pro-$p$) $\tp{P}$-Iwahori subgroup of level $r$.

In particular, if $\pi_p \cong \mu_p \otimes \bigotimes_{w \in \Sigma_p} \pi_w$ is the identification obtained from \eqref{eq:facto G(Qp)} for $G_1$, the corresponding factorization on $G_2$ induces
\[
    \pi_p^\flat \cong \mu_p^\flat \otimes \bigotimes_{w \in \Sigma_p} \pi_w^\flat \ ,
\]
where $\pi_w^\flat = \pi_w^\vee$ and $\mu_p^\flat = \mu_p^{-1}|\nu|_p^{a(\kappa)}$, by definition of $\pi^\flat$.

\subsubsection{Structure theorem for $P$-ordinary representations of $G_2$.}  \label{subsubsec:strucutre thm for Pord rep of G2}
The discussion above shows that $\pi_w$ is $P_w$-ordinary of level $r \gg 0$ (for $G_1$) if and only if $\pi_w^\flat$ is $P_w$-ordinary of level $r \gg 0$ (for $G_2$). As explained in Section \ref{subsubsec:comp of lvl structures}, adapting the definitions for $P$-ordinary theory from $G_1$ to $G_2$ requires to change $P_w$ for $\tp{P}_w$ and the double coset operators $U_{w,j}^{\GL}$ for $U_{w,j}^{\flat, \GL} = \tp{I}_{w,r} t_{w,j}^{-1} \tp{I}_{w,r}$.The analogue of Theorem \ref{thm:canonical Pword vec of type tau w for G1} is the following.

\begin{lemma} \label{lma:canonical Pword vec of type tau w for G2}
    Let $\pi$ be an holomorphic cuspidal representation of $G_1$. Suppose its weight $\kappa$ satisfies Inequality \eqref{eq:ineq kappa sigma}. Assume that $\pi_w$ is $P_w$-ordinary of level $r \gg 0$ (for $G_1$) and that it is the unique irreducible subrepresentation of $\ind{P_w}{G_w} \sigma_w$ for some irreducible supercuspidal $\sigma_w$. Let $(\tau_w, X_w)$ be the SW-type of $\pi_w$.

    \begin{enumerate}
        \item[(i)] The unique irreducible quotient of $\ind{P_w}{G_w} \sigma_w^\vee$ is isomorphic to $\pi_w^\flat$.
        
        \item[(ii)] Let $(\tau_w^\vee, X_w^\vee)$ be the contragredient of $(\tau_w, X_w)$, the BK-type of $\sigma_w^\vee$. Consider $X_w^\vee$ as a subspace of the vector space associated to $\sigma_w^\vee$, via a fix embedding (unique up to scalar) $\tau_w^\vee \hookrightarrow \sigma_w^\vee$. 
        
        For any $\alpha^\vee \in X_w^\vee$, let $\varphi^\flat_w \in \ind{P_w}{G_w} \sigma_w^\vee$ be the unique function with support $P_w\tp{I}_{w,r}$ (for all $r \gg 0$) such that $\varphi^\flat_{w}(1) = \alpha^\vee$ and $\varphi^\flat_{w}$ is fixed by $\tp{I}_{w,r}$ (for all $r \gg 0$). Let $\phi^\flat_w$ denote its image in $\pi_w^\flat$.

        Then, $\phi^\flat_w$ is $P_w$-ordinary of type $\tau_w^\vee = \tau_w^\flat$ of level $r \gg 0$. In particular, $\pi_w^\flat$ is $P_w$-ordinary of level $r \gg 0$ (for $G_2$).
        
        This induces a natural isomorphism between $\tau_w^\flat$ and the subspace of $\pi_w^\flat$ of $P_w$-ordinary vectors of type $\tau_w^\flat$ of level $r \gg 0$. The latter is independent of $r \gg 0$ and has dimension $\dim \tau_w^\flat = \dim \tau_w$.
    \end{enumerate}
\end{lemma}

\begin{proof}
    Consider the composition of $\pi_w \hookrightarrow \ind{P_w}{G_w} \sigma_w$ with the map (of vector spaces)
    \begin{align*}
        \ind{P_w}{G_w} \sigma_w 
            &\to 
                \ind{\tp{P}_w}{G_w} \sigma^\vee_w \\
        \phi 
            &\mapsto
                \phi^\vee(g) := \phi(\tp{g}^{-1})
    \end{align*}

    Its image is $\pi_w^\flat = \pi_w^\vee$ and the above realizes $\pi_w^\flat$ as the unique irreducible subrepresentation of $\ind{\tp{P}_w}{G_w} \sigma^\vee_w$. In particular, all the consequences of Theorem \ref{thm:canonical Pword vec of type tau w for G1} hold for $\pi_w^\flat$ by replacing $P_w$ by $\tp{P}_w$ and $\sigma_w$ by $\sigma_w^\vee$. 
    
    Given $\alpha^\vee \in X_w^\vee$ as above, let $\phi_w^\vee \in \pi_w^{\vee, I_{w,r}}$ and $\varphi_w^\vee \in \ind{\tp{P}_w}{G_w} \sigma^\vee_w$ be the vectors obtained from Theorem \ref{thm:canonical Pword vec of type tau w for G1} (iii) associated to $\alpha^\vee$. In particular, $\phi_w^\vee$ is a $P_w$-ordinary vector of type $\tau_w^\vee$ and the subspace generated by the action of $L_w(\OO_w)$ on $\phi_w^\vee$ is exactly the space of all $P_w$-ordinary vectors of type $\tau_w^\vee$. In particular, the latter is independent of $r \gg 0$ and isomorphic to $\tau_w^\vee$ as a representation of $L_w(\OO_w)$.

    Lastly, consider the standard intertwining operator $\ind{P_w}{G_w} \sigma^\vee_w \to \ind{\tp{P_w}}{G_w} \sigma^\vee_w$. Its image is $\pi_w^\flat$, hence identifies $\pi_w^\flat$ as the unique irreducible quotient of $\ind{P_w}{G_w} \sigma^\vee_w$.
    
    One readily checks that under this intertwining operator, the vector $\varphi_w^\flat \in \ind{P_w}{G_w} \sigma^\vee_w$ described in the statement of the lemma maps to $\varphi_w^\vee$. Therefore, the desired properties for $\phi_w^\flat$ can be verified through $\varphi_w^\vee$.
\end{proof}

\subsection{$P$-anti-ordinary theory on $G_2$.} \label{subsec:Paord theory on G2}
Going back to the discussion of Section \ref{subsec:Pord theory on G2}, we know that $\pi_w$ is $P_w$-anti-ordinary of level $r \gg 0$ (for $G_1$) if and only if $\pi_w^\flat$ is $P_w$-anti-ordinary of level $r \gg 0$ (for $G_2$). As explained in Section \ref{subsubsec:comp of lvl structures}, adapting the definitions for $P$-anti-ordinary theory from $G_1$ to $G_2$ requires to change $P_w$ for $\tp{P}_w$ and the double coset operators $U_{w,j}^{\GL, -}$ for $U_{w,j}^{\flat, \GL, -} = \tp{I}_{w,r} t_{w,j} \tp{I}_{w,r}$. The analogue of Theorem \ref{thm:canonical Pwaord vec of type tau w for G1} is the following.

\begin{lemma} \label{lma:canonical Pwaord vec of type tau w for G2}
    Let $\pi$ be an anti-holomorphic cuspidal representation of $G_1$. Suppose its weight $\kappa$ satisfies Inequality \eqref{eq:ineq kappa sigma}. Assume that $\pi_w$ is $P_w$-anti-ordinary of level $r \gg 0$ (for $G_1$) and that it is the unique irreducible quotient of $\ind{P_w}{G_w} \sigma_w$ for some irreducible supercuspidal $\sigma_w$. Let $(\tau_w, X_w)$ be the SZ-type of $\pi_w$.

    \begin{enumerate}
        \item[(i)] The unique irreducible subrepresentation of $\ind{P_w}{G_w} \sigma_w^\vee$ is isomorphic to $\pi_w^\flat$.
        
        \item[(ii)] Let $(\tau_w^\vee, X_w^\vee)$ be the contragredient of $(\tau_w, X_w)$, the SZ-type of $\sigma_w^\vee$. Consider $X_w^\vee$ as a subspace of the vector space associated to $\sigma_w^\vee$, via a fix embedding (unique up to scalar) $\tau_w^\vee \hookrightarrow \sigma_w^\vee$. 
        
        For each $r \gg 0$ and $\alpha \in X_w^\vee$, there exists some unique $P_w$-anti-ordinary $\phi^\flat_{w,r} \in \pi_w^{\flat, \tp{I}_r}$ of type $\tau_w^\vee$ and level $r$ such that $\varphi^\flat_{w,r}(1) = \alpha$, where $\varphi^\flat_{w,r}$ is the image of $\phi^\flat_{w,r}$ in $\ind{P_w}{G_w} \sigma^\vee_w$, and the support of $\varphi^\flat_{w,r}$ contains $P_w \tp{I}_{w,r}$. In particular, $\pi_w^\flat$ is $P_w$-anti-ordinary of level $r \gg 0$ (for $G_2$).

        \item[(iii)] For $r' > r \gg 0$, one can choose $\alpha$, $\alpha' \in X_w^\vee$ such that the vectors $\phi^\flat_{w,r}$ and $\phi^\flat_{w,r'}$ corresponding to $\alpha$ and $\alpha'$ respectively satisfy
        \[
            \sum_{
                \gamma \in \tp{I}_{w,r}/(\tp{I}_{w,r'}^0 \cap \tp{I}_{w,r})
            }
                \pi^\flat_w(\gamma) \phi_{w,r'}^\flat = 
                \phi_{w,r}^\flat
        \]
    \end{enumerate}
\end{lemma}

\begin{proof}
    As in the proof of Lemma \ref{lma:canonical Pword vec of type tau w for G2}, the map
    \begin{align*}
        \ind{\tp{P}_w}{G_w} \sigma^\vee_w 
            &\to 
                \ind{P_w}{G_w} \sigma_w \\
        \phi 
            &\mapsto
                \phi^\vee(g) := \phi(\tp{g}^{-1}) \ ,
    \end{align*}
    realizes $\pi_w^\flat = \pi_w^\vee$ as the unique irreducible quotient of $\ind{\tp{P}_w}{G_w} \sigma^\vee_w$. 
    
    In particular, all the consequences of Theorem \ref{thm:canonical Pwaord vec of type tau w for G1} hold for $\pi_w^\flat$ by replacing $P_w$ by $\tp{P}_w$ and $\sigma_w$ by $\sigma_w^\vee$. Given $\alpha \in X_w^\vee$ as above, let $\varphi_{w,r}' \in \ind{\tp{P}_w}{G_w} \sigma^\vee_w$ be the vectors obtained from Theorem \ref{thm:canonical Pwaord vec of type tau w for G1} associated to $\alpha$.
    
    Furthermore, consider the standard intertwining operator 
    $
        \ind{\tp{P_w}}{G_w} \sigma^\vee_w
            \xrightarrow{\sim} 
        \ind{P_w}{G_w} \sigma^\vee_w 
    $. Its image is both the unique irreducible quotient of $\ind{\tp{P_w}}{G_w} \sigma^\vee_w$, namely $\pi_w^\flat$, and the unique irreducible subrepresentation of $\ind{P_w}{G_w} \sigma^\vee_w$. This proves part (i).
    
    To conclude, let $\phi_{w,r}^\flat$ (resp. $\varphi_{w,r}^\flat$) be the image of $\varphi_{w,r}'$ in $\pi_w^\vee$ (resp. $\ind{P_w}{G_w} \sigma^\vee_w$) via this intertwining operator. The fact that $\phi_{w,r}^\flat$ is $\tp{P}_w$-anti-ordinary of type $\tau_w^\vee$ and level $r$ follows from Theorem \ref{thm:canonical Pwaord vec of type tau w for G1} (ii). Similarly, part (iii) follows from Theorem \ref{thm:canonical Pwaord vec of type tau w for G1} (iii) (upon making the appropriate adjustments between $G_1$ and $G_2$). The properties of $\varphi_{w,r}'$ are obtained from an easy computation using the definition of $\varphi_{w,r}'$ and the exact formula for the intertwining operator above.
\end{proof}

%%%%%%%%%%%%%%%%%%%%%%%%%%%%%%%%%%%%%%%%%%%%%%%%%%%%%%%%%%%%%%%%%%

\section{Explicit choice of $P$-(anti-)ordinary vectors.} \label{sec:explicit choice of P-(anti-)ord vectors}
In what follows, we freely use the notation from Sections \ref{subsubsec:ramified places away from p} and \ref{subsubsec:contragredient rep and pairings}. In particular, let $\pi = \pi_\infty \otimes \pi_f$ be a cuspidal automorphic representation for $G_1$ of level $K \subset G_1(\AA_f)$ and unramified away from $S = S(K^p)$ and $p$, and let $\pi^\vee$ denote its contragredient.

The goal of this section is to single out a set of \emph{test vectors} in a $P$-anti-ordinary anti-holomorphic cuspidal automorphic representation $\pi$ on $G_1$. Our strategy is to construct local test vectors $\varphi_l \in \pi_l$ for all places $l$ of $\QQ$ and consider $\varphi = \otimes_l \varphi_l \in \pi$ via \eqref{eq:facto pi f}.

Then, we use the involutions in Section \ref{subsec:comp between G1 and G2} to obtain a compatible space of test vectors for $\pi^\flat$ on $G_2$. Recall that $\pi^\flat$ is defined as a twist of $\pi^\vee$, hence it suffices specify a space of test vectors in $\pi^\vee$. 

Throughout this section, we assume that $\pi$ is anti-holomorphic of a certain weight $\kappa$, hence $\pi_f$ (and $\pi_f^\flat$) is defined over some number field $E(\pi)$, see Remark \ref{rmk:E pi rational}. Recall that we always assume that $E(\pi)$ contains $\KK'$. We further assume that $\pi$ (resp. $\pi^\flat$) is $P$-anti-ordinary of level $r \gg 0$ for $G_1$ (resp. for $G_2$).

We work with the $G(\QQ_l)$-equivariant perfect pairing $\brktdotdot_{\pi_l} : \pi_l \times \pi_l^\vee \to \CC$, for each place $l \leq \infty$ of $\QQ$, as in Section \ref{subsubsec:contragredient rep and pairings}.

\begin{remark}
    Using the involution $F_\infty$ from Section \ref{subsubsec:the involution Foo}, this leads to test vectors for holomorphic, $P$-ordinary cuspidal automorphic representations.
\end{remark}

\subsection{Local test vectors at places away from $p$ and $\infty$.} \label{subsec:Local test vectors at places away from p and oo}

\subsubsection{Local test vectors at unramified places.} \label{subsubsec:Local test vectors at finite unramd places}
For each finite prime $l \notin S \cup \{p\}$, we fix $E(\pi)$-rational $K_l$-spherical vectors $\varphi_{l, 0} \in \pi_l$ and $\varphi_{l, 0}^\vee \in \pi_l^\vee$ such that $\brkt{\varphi_{l,0}}{\varphi_{l,0}^\vee}_{\pi_l} = 1$, as in Section \ref{subsubsec:contragredient rep and pairings}.

\subsubsection{Local test vectors at ramified places.} \label{subsubsec:Local test vectors at finite ramd places}
On the other hand, the choice of local test vectors at $l \in S$ is non-canonical. We adapt the same conventions as in \cite[Section 4.2.2]{EHLS}.

Given $l \in S$, fix an arbitrary irreducible $U_1(\QQ_l)$-subrepresentation $\ul{\pi}_l$ of $\pi_l$. The dual $\ul{\pi}_l^\vee$ of $\ul{\pi}_l$ occurs as an irreducible $U_1(\QQ_l)$-subrepresentation of $\pi_l^\vee$. Furthermore, the bilinear $\brktdotdot_{\pi_l} : \pi_l \times \pi_l^\vee \to \CC$ induces a perfect $U_1(\QQ_l)$-equivariant pairing between $\ul{\pi}_l$ and $\ul{\pi}_l^\vee$, again denoted $\brktdotdot_{\pi_l}$.

Since $U_1$ is the restriction of scalar of a reductive group from $\KK^+$ to $\QQ$, we have $U_1(\QQ_l) = \prod_{v \mid l} U_{1, v}$, where the product is over the places $v$ of $\KK^+$ above $l$ and $U_{1, v}$ is the set of $\KK_v^+$-points of a unitary group over $\KK^+$. Similarly, we obtain
\begin{equation} \label{eq:facto pi l ramified}
    \ul{\pi}_l \cong \bigotimes_{v \mid l} \ul{\pi}_v
        \;\;\;\text{and}\;\;\;
    \ul{\pi}_l^\vee \cong \bigotimes_{v \mid l} \ul{\pi}_v^\vee
\end{equation}
for irreducible admissible representations $\ul{\pi}_v$ and $\ul{\pi}_v^\vee$ of $U_{1,v}$. 

Naturally, there are $U_{1,v}$-equivariant perfect pairings $\brktdotdot_{\pi_v} : \ul{\pi}_v \times \ul{\pi}_v^\vee \to \CC$, identifying $\ul{\pi}_v^\vee$ as the contragredient of $\ul{\pi}_v$, such that $\brktdotdot_{\pi_l} = \prod_{v \mid l} \brktdotdot_{\pi_v}$. 

Fix \emph{any} nonzero vectors $\varphi_v \in \ul{\pi}_v$ and $\varphi_v^\vee \in \ul{\pi}_v^\vee$ such that $\brkt{\varphi_v}{\varphi_v^\vee}_{\pi_v} = 1$. Our choice of local test vectors $\varphi_l \in \pi_l$ and $\varphi_l^\vee \in \pi_l^\vee$ are
\[
    \varphi_l 
        := 
    \otimes_{v \mid l} \varphi_v
        \;\;\;\text{and}\;\;\;
    \varphi_l^\vee 
        := 
    \otimes_{v \mid l} \varphi_v^\vee\,,
\]
via \eqref{eq:facto pi l ramified}. In Section \ref{subsubsec:I pi and test vectors}, we restrict our attention slightly and choose an integral structure for such local test vectors.

\begin{remark} \label{rmk:test vectors, SW section and volume}
    In Section \ref{subsec:Local SW section away from p and oo}, we use this naive choice of test vectors at $l \in S$ suffices to obtain non-zero constant local zeta integrals, essentially volume factors, that are insensitive to the variation of $\pi$ in a $p$-adic family. This approach is standard in the literature, see \cite[Section 4.2.2]{EHLS}.
\end{remark}

\subsection{Local test vectors at $p$} \label{subsec:Local test vectors at p}
In this section, we choose test vectors at $p$ following the strategy developed in \cite[Section 4.3.4]{EHLS}. However, to generalize their results, we need to work out various extra details due to the fact that spaces of $P$-ordinary vectors are not 1-dimensional in general, see Theorem \ref{thm:canonical Pwaord vec of type tau w for G1} and Lemma \ref{lma:canonical Pwaord vec of type tau w for G2}. The theory of types of $P$-(anti-)ordinary vectors here is used as a substitute for the lack of ``ordinary nebentypus'', see \cite[Section 6.6.6]{EHLS}, in the general $P$-(anti-)ordinary setting.

\subsubsection{Local representations over CM type at $p$.} \label{subsubsec:Local rep over CM type at p}
Let $w \in \Sigma_p$ and set $G_w := \GL_n(\KK_w)$. As in Section \ref{subsec:Paord automorphic representations}, the isomorphisms \eqref{eq:prod G over Zp} and \eqref{eq:GL(Lw) basis} induce an identification $G(\QQ_p) = \QQ_p^\times \times \prod_{w \in \Sigma_p} G_w$ as well as an isomorphism
\begin{equation} \label{eq:pi p iso mu p otimes mu w}
    \pi_p \cong \mu_p \otimes \left( \bigotimes_{w \in \Sigma_p} \pi_w \right) \,,
\end{equation}
where $\mu_p$ is some character of $\QQ_p$ and $\pi_w$ 
is an irreducible admissible representations of $G_w$. Since $\pi$ is $P$-anti-ordinary of level $r \gg 0$, we know $\mu_p$ is unramified and
\[
    \pi_p^{I_{P, r}} \cong \bigotimes_{w \in \Sigma_p} \pi_w^{I_{w, r}} \neq 0 \ .
\]

Similarly, for the contragredient $\pi^\vee$ of $\pi$, we have
\begin{equation}  \label{eq:pi p flat iso mu p flat otimes mu w flat}
    \pi^\vee_p \cong \mu_p^{-1} \otimes \left( \bigotimes_{w \in \Sigma_p} \pi^\vee_w \right)
\end{equation}
and $\left( \pi^\vee_w \right)^{\tp{I_{w, r}}} \neq 0$ for each $w \in \Sigma_p$. Note that $\pi_w^\vee = \pi_w^\flat$ for each $w \in \Sigma_p$.

\subsubsection{Compatibility of parabolic subgroups.} \label{subsubsec:Comp of parabolic subgroups}
For each $w \in \Sigma_p$ and integer $d \geq 1$, let $G_w(d)$ denote the algebraic group $\GL(d)$ over $\OO_w = \OO_{\KK_w}$. However, when $d = n$, we still write $G_w$ instead of $G_w(n)$. Let $(a_w, b_w)$ be the signature at $w \in \Sigma_p$ associated to the PEL datum $\PP = \PP_1$, as in Section \ref{subsec:unitary pel datum}. 

Proceeding as in Section \ref{subsubsec:para subgp of G over Zp}, let $P_{a_w} \subset G_w(a_w)$, $P_{b_w} \subset G_w(b_w)$ and $P_{a_w, b_w} \subset G_w$ be the standard upper triangular parabolic subgroups associated to partitions
\[
    \d_{a_w} = 
        (n_{w, 1}, \ldots, n_{w, t_w}) 
    \ \ \ ; \ \ \ 
    \d_{b_w} = 
        (n_{w, t_w+1}, \ldots, n_{w, r_w}) 
    \ \ \ ; \ \ \
    \d_{w} = 
        (a_w, b_w) 
\]
of $a_w$, $b_w$ and $n$, respectively. We also work with the parabolic subgroup $P_w \subset G_w$ constructed in \eqref{eq:def of local parabolic Pw}. Note that $P_w \subset P_{a_w, b_w} \subset G_w$.

For any one of these parabolic subgroup $P_\bullet$, let $L_\bullet$ denote its standard Levi subgroup consisting of block-diagonal matrices (corresponding to the decomposition defining $P_\bullet$). Similarly, consider the pro-$p$ Iwahori subgroup $I_{\bullet, r}$ of level $r$ associated to $P_\bullet$ consisting of invertible matrices $g$ (of the appropriate size) over $\OO_w$ such that $g$ mod $\p_w^r$ is in $P_\bullet^u(\OO_w/\p_w^r\OO_w)$, see Definition \ref{def:def of PIwahori of level r} and \eqref{eq:facto PIwahori over Sigma p}.

Let $K_\bullet = L_\bullet(\OO_w)$ and $I_{\bullet, r}^0 = K_\bullet I_{\bullet, r}$. Setting $K_{w,j} = \GL_{n_{w,j}}(\OO_w)$, we have
\[
    K_{a_w} = 
        \prod_{j=1}^{t_w}
            K_{w,j} \ \ \ ; \ \ \
    K_{b_w} = 
        \prod_{j=t_w+1}^{r_w}
            K_{w,j} \ \ \ ; \ \ \
    K_{w} = 
        K_{a_w} \times K_{b_w} \ ,
\]
where the products take place in $G_w(a_w)$, $G_w(b_w)$ and $G_w$, respectively.

\subsubsection{Compatibility of local representations} \label{subsubsec:comp of local representations}
Since $\pi$ is $P$-anti-ordinary, we may assume (see Lemma \ref{lma:Jacquet Lemma}, Lemma \ref{lma:Pword and Pwaord duality} and Section \ref{subsubsec:conv on contragredient pairings}) that there exists an admissible irreducible representation $\sigma_w$ of $L_w$ such that $\pi_w$ is the unique irreducible quotient of $\ind{P_w}{G_w} \sigma_w$. Equivalently, $\pi^\vee_w$ is the unique irreducible subrepresentation of $\ind{P_w}{G_w} \sigma^\vee_w$.

\begin{remark}
    We do not assume that $\sigma_w$ is supercuspidal.
\end{remark}

Write $\sigma_w = \boxtimes_{j=1}^{r_w} \sigma_{w,j}$ and consider the representations  
\[
    \sigma_{a_w} = 
        \boxtimes_{j=1}^{t_w}
            \sigma_{w,j} \ \ \ ; \ \ \
    \sigma_{b_w} = 
        \boxtimes_{j=t_w+1}^{r_w}
            \sigma_{w,j}
\]
of $L_{a_w}$ and $L_{b_w}$. Let $\pi_{a_w}$ and $\pi_{b_w}$ be the unique irreducible quotients
\begin{equation} \label{def pi aw and pi bw}
    \ind{P_{a_w}}{G_w(a_w)} \sigma_{a_w} 
        \twoheadrightarrow \pi_{a_w}
    \ \ \ \text{and} \ \ \
    \ind{P^{\opp}_{b_w}}{G_w(b_w)} \sigma_{b_w} 
        \twoheadrightarrow \pi_{b_w} \ ,
\end{equation}
and set $\pi_{a_w, b_w} := \pi_{a_w} \boxtimes \pi_{b_w}$. Under the canonical isomorphism
\begin{equation}\label{iso induced rep}
    \ind{P_w}{G_w} \sigma_w 
        \xrightarrow{\sim}
    \ind{P_{a_w, b_w}}{G_w} 
    \left(
        \ind{
            P_{a_w} \times P_{b_w}^{\opp}
        }{
            G_w(a_w) \times G_w(b_w)
        }
        \sigma_{a_w} \boxtimes \sigma_{b_w}
    \right) \ ,
\end{equation}
given by $\phi \mapsto (g \mapsto (h \mapsto \phi(hg))$, $\pi_w$ is the unique irreducible quotient
\begin{equation} \label{pi aw x pi bw to pi w}
    \ind{P_{a_w,b_w}}{G_w}
    \left(
        \pi_{a_w, b_w} 
    \right)
        \twoheadrightarrow 
    \pi_w \ .
\end{equation}

\subsubsection{Conventions on local pairings above $p$} \label{subsubsec:Conventions on local pairings above p}
In this section, we refine the conventions on pairings set in Section \ref{subsubsec:conv on contragredient pairings} to local places above $p$. This follows the approach of \cite[Section 4.3.3]{EHLS}.

Let $\la \cdot, \cdot \ra_{\sigma_{w,j}}$ be the tautological pairing between $\sigma_{w,j}$ and its contragredient $\sigma^\vee_{w,j}$. Then, define $(\cdot, \cdot)_{a_w} = \otimes_{i=1}^{t_w} \la \cdot, \cdot \ra_{\sigma_{w,j}}$ so that
\begin{align*}
    \la \cdot, \cdot \ra_{a_w} &: 
        \left( 
            \ind{P_{a_w}}{G_w(a_w)} 
                \sigma_{a_w}
        \right)
            \times 
        \left( 
            \ind{P_{a_w}}{G_w(a_w)} 
                \sigma^\vee_{a_w}
        \right) 
    \to \CC \\
    \la 
        \varphi, 
        \varphi^\vee 
    \ra_{a_w} 
    &= 
    \int_{K_{a_w}} 
        ( 
            \varphi(k), 
            \varphi^\vee(k) 
        )_{a_w} 
    dk
\end{align*}
is the perfect $G_w(a_w)$-invariant pairing that identify the above pair as contragredient representations. A similar logic applies for $(\cdot, \cdot)_{b_w} = \otimes_{i=t_w+1}^{r_w} \la \cdot, \cdot \ra_{\sigma_{w,j}}$ and
\begin{align*}
    \la \cdot, \cdot \ra_{b_w} &: 
        \left( 
            \ind{P^{\opp}_{b_w}}{G_w(b_w)} 
                \sigma_{b_w}
        \right)
            \times 
        \left( 
            \ind{P^{\opp}_{b_w}}{G_w(b_w)} 
                \sigma^\vee_{b_w}
        \right) 
    \to \CC \\
    \la 
        \varphi, 
        \varphi^\vee 
    \ra_{b_w} 
    &= 
    \int_{K_{b_w}} 
        ( 
            \varphi(k), 
            \varphi^\vee(k) 
        )_{b_w} 
    dk \ .
\end{align*}

Taking the dual of the surjections in Equation \eqref{def pi aw and pi bw} yields injections
\begin{equation} \label{def pi aw and pi bw dual}
    \pi^\vee_{a_w} \hookrightarrow 
        \ind{P_{a_w}}{G(a_w)} 
            \sigma^\vee_{a_w} 
    \ \ \ \text{ and } \ \ \ 
    \pi^\vee_{b_w} \hookrightarrow
        \ind{P_{b_w}^{\opp}}{G(b_w)} 
            \sigma^\vee_{b_w} 
\end{equation}
and restricting the second argument of $\brkt{\cdot}{\cdot}_{a_w}$ to $\pi^\vee_{a_w}$ makes the first argument of the pairing factor through $\pi_{a_w}$. It is identified with the tautological pairing $\brkt{\cdot}{\cdot}_{\pi_{a_w}} : \pi_{a_w} \times \pi^\vee_{a_w} \to \CC$. Again, a similar logic applies for $\brkt{\cdot}{\cdot}_{\pi_{b_w}} : \pi_{b_w} \times \pi^\vee_{b_w} \to \CC$.

Let $(\cdot, \cdot)_w = \la \cdot, \cdot \ra_{\pi_{a_w}} \otimes \la \cdot, \cdot \ra_{\pi_{b_w}}$. As above, it determines a pairing
\[
    \la \cdot, \cdot \ra_w : 
        \ind{P_{a_w, b_w}}{G_w}  
            \left(
                \pi_{a_w, b_w}
            \right)
            \times 
        \ind{P_{a_w, b_w}}{G_w} 
            \left( 
                \pi^\vee_{a_w, b_w}
            \right)
    \to \CC
\]
as well as a pairing $\brkt{\cdot}{\cdot}_w : \pi_w \times \pi^\vee_w \to \CC$, using the dual 
$
    \pi^\vee_w \hookrightarrow 
    \ind{P_{a_w, b_w}}{G_w} 
        \left( 
            \pi^\vee_{a_w, b_w}
        \right)
$
induced from Equation \eqref{pi aw x pi bw to pi w}.

\begin{remark}
    One may normalize these pairings so that $\brktdotdot_{\pi_p} = \prod_{w \in \Sigma_p} \brktdotdot_{\pi_w}$.
\end{remark}

For any $\phi \in \pi_w$, $\phi^\vee \in \pi^\vee_w$, if $\varphi$ is a lift of $\phi$ and $\varphi^\vee$ is the image of $\phi^\vee$, then
\begin{equation} \label{inner product pi w integral}
    \brkt{\phi}{\phi^\vee}_{\pi_w} =
    \int_{\GL_n(\OO_w)}
        \left(
            \varphi(k),
            \varphi^\vee(k)
        \right)_w
    dk
\end{equation}

\subsubsection{Compatibility of test vectors.} \label{subsubsec:comp of test vectors}
For each $1 \leq j \leq r_w$, let $\tau_{w,j}$ be a smooth (finite-dimensional) irreducible representation of $K_{w,j}$. We assume that $r$ is large enough so that $\tau_{w,j}$ factors through $\GL_{n_{w,j}}(\OO_w/\p_w^r\OO_w)$. Assume there exists an embedding $\alpha_{w,j}$ of $\tau_{w,j}$ in the restriction of $\sigma_{w,j}$ as a representation of $K_{w,j}$. 

Let $\alpha_{a_w} : \tau_{a_w} \to \sigma_{a_w}$ and $\alpha_{b_w} : \tau_{b_w} \to \sigma_{b_w}$ be the corresponding embeddings over $K_{a_w}$ and $K_{b_w}$ respectively, where
\[
    \tau_{a_w} = 
        \boxtimes_{j=1}^{t_w}
            \tau_{w,j} \ \ \ ; \ \ \
    \tau_{b_w} = 
        \boxtimes_{j=t_w+1}^{r_w}
            \tau_{w,j} \ .
\]

\begin{remark} \label{tau BK type sigma}
    Implicitly, we think of $\tau_{a_w}$ as the SZ-type of $\sigma_{a_w}$, in the sense Section \ref{subsubsec:SZ types}. In that case, there exists a unique such embedding $\alpha_{w,j}$ (up to scalar) and Theorem \ref{thm:canonical Pwaord vec of type tau w for G1} is concerned about constructing a canonical lift of $\alpha_{a_w}$ to an embedding of $\tau_{a_w}$ into $\pi_{a_w}^{(\Pwaord, r)}$. For now, Theorem \ref{thm:canonical Pwaord vec of type tau w for G1} only deals with $\sigma_{a_w}$ supercuspidal. However, in the following we proceed as if this theorem held for arbitrary admissible $\sigma_{a_w}$. In other words, we conjecture that we can omit the supercuspidal aHypothesis \ref{hyp:Qw equals Pw} and proceed without comments. Note that similar statements can be made about $\tau_{b_w}$ and $\tau_w := \tau_{a_w} \boxtimes \tau_{b_w}$.
\end{remark}

For each $j = 1, \ldots, r_w$, fix a vector $\phi_{w,j}$ in the image of $\alpha_{w,j}$ and consider
\begin{equation} \label{eq:def phi a w 0 and phi b w 0}
    \phi_{a_w}^0 := 
        \bigotimes \limits_{j=1}^{t_w}
            \phi_{w,j}
            \ \ \ ; \ \ \
    \phi_{b_w}^0 := 
        \bigotimes \limits_{j=t_w+1}^{r_w}
            \phi_{w,j}
\end{equation}
as vectors in the image of $\alpha_{a_w}$ and $\alpha_{b_w}$ respectively.

\begin{remark} \label{rmk:twists of local vectors in local types}
    In Section \ref{sec:Pord Eisenstein measure}, given local representations $\tau_{w,j}$ and $\sigma_{w,j}$ as above, we work with such local vectors with respect $\tau_{w,j} \otimes \psi_{w,j}$ and $\sigma_{w,j} \otimes \psi_{w,j}$, where $\psi_{w,j}$ is a finite-order character of $K_{w,j}$ (viewing $\tau_{w,j}$ as fixed and $\psi_{w,j}$ as varying). We always assume that the corresponding test vectors in the image of $\alpha_{w,j} \otimes \mathrm{id}$ are $\phi_{w,j} \otimes 1$, i.e. essentially the ``same'' local vectors. See Remark \ref{rmk:SZ types of twists and contragredient}.
\end{remark}

Let $\varphi_{a_w} \in \ind{P_{a_w}}{G_w(a_w)} \sigma_{a_w}$ be the unique function fixed by $I_{a_w, r}$ that has support $P_{a_w} I_{a_w, r}$ and
\begin{equation} \label{def varphi aw}
    \varphi_{a_w}(\gamma) = \sigma_{a_w}(\gamma) \phi_{a_w}^0 = \tau_{a_w}(\gamma) \phi_{a_w}^0 \ ,
\end{equation}
for all $\gamma \in I_{a_w, r}^0$. Denote its image in $\pi_{a_w}$ by $\phi_{a_w}$. 

\begin{remark} \label{covers of types to P-Iwahori}
    Here, we implicitly identify $\tau_{a_w}$ with its image in $\sigma_{a_w}$ and as a representation of $I_{a_w, r}^0$ that factors through $I_{a_w, r}^0/I_{a_w, r} \cong L_{a_w}(\OO_w/\p_w^r\OO_w)$. In what follows, we similarly identify $\tau_{b_w}$ (resp. $\tau^\vee_{a_w}$, $\tau^\vee_{b_w}$) with its cover as a representation of $\tp{I}_{b_w, r}^0$ (resp. $\tp{I}_{a_w, r}^0$, ${I}_{b_w, r}^0$) contained in $\sigma_{b_w}$ (resp. $\sigma^\vee_{a_w}$, $\sigma^\vee_{b_w}$).
\end{remark}

Let $\varphi_{b_w} \in \ind{P^{\opp}_{b_w}}{G_w(b_w)} \sigma_{b_w}$ be the unique function whose support is $P_{b_w}^{\opp} \tp{I}_{b_w, r}$ such that
\begin{equation} \label{def varphi bw}
    \varphi_{b_w}(\gamma) = \tau_{b_w}(\gamma) \phi_{b_w}^0 \ ,
\end{equation}
for all $\gamma \in \tp{I}_{b_w, r}^0$. Let $\phi_{b_w}$ denote its image in $\pi_{b_w}$. 

Lastly, consider the unique function $\varphi_w \in \ind{P_w}{G_w} \sigma_w$ fixed by $I_{w,r}$ whose support is $P_wI_{w,r}$ and
\begin{equation}\label{def varphi w}
    \varphi_w(\gamma) = \tau_w(\gamma) (\phi_{a_w}^0 \otimes \phi_{b_w}^0) \ ,
\end{equation} 
for all $\gamma \in I_{w,r}^0$, where $\tau_w = \tau_{a_w} \boxtimes \tau_{b_w}$. Here, we view $\tau_w$ as a $I_{w,r}^0$-subrepresentation of $\sigma_w$, see Remark \ref{covers of types to P-Iwahori}.

For our purposes, it is more convenient to work with the vector corresponding to $\varphi_w$ via the map 
$
    \ind{P_w}{G_w} \sigma_w
        \to 
    \ind{P_{a_w,b_w}}{G_w} \pi_{a_w, b_w}
$ induced by the maps in \eqref{def pi aw and pi bw} and \eqref{iso induced rep}. We denote this image by $\varphi_w$ again, which should not cause any confusion since we will only ever work with $\varphi_w$ in $\ind{P_{a_w,b_w}}{G_w} \pi_{a_w, b_w}$ from now on. 

One easily checks that the support of $\varphi_w$ is $P_{a_w,b_w}I_{w,r}$ and
\[
    \varphi_w(\gamma) = \tau_w(\gamma) (\phi_{a_w} \otimes \phi_{b_w}) \ ,
\]
for all $\gamma \in I_{w,r}^0$. Let $\phi_w$ be the image of $\varphi_w$ in $\pi_w$.

\begin{remark} \label{If pi Paord}
    If $\sigma_w$ is supercuspidal, for each $w \in \Sigma_p$, then $\phi_{a_w}$ (resp. $\phi_{b_w}$, $\phi_w$) is a $P_{a_w}$-anti-ordinary (resp. $\tp{P}_{b_w}$-anti-ordinary, $P_w$-anti-ordinary) vector of level $r$ and type $\tau_{a_w}$ (resp. $\tau_{b_w}$, $\tau_w$) as in \ref{thm:canonical Pwaord vec of type tau w for G1}.
\end{remark}

We now proceed similarly by constructing explicit vectors related to the contragredient representations. Since $\sigma_{w,j}$ is admissible, for $j = 1, \ldots, r_w$, we also have an embedding $\alpha^\vee_{w,j} : \tau^\vee_{w,j} \to \sigma^\vee_{w,j}$ of $K_{w,j}$-representations. We identify the natural contragredient pairing on $\tau_{w,j} \times \tau^\vee_{w,j}$ with the restriction of $\brkt{\cdot}{\cdot}_{\sigma_{w,j}}$ via their fixed embedding in $\sigma_{w,j} \times \sigma^\vee_{w,j}$. 

\begin{remark} \label{cg tau BK type cg sigma}
    If $\tau_{w,j}$ is the SZ-type of $\sigma_{w,j}$ as in Remark \ref{tau BK type sigma}, then $\tau^\vee_{w,j}$ is also the SZ-type of $\sigma^\vee_{w,j}$. In that case, such maps $\alpha^\vee_{w,j}$ again exist and are unique up to scalar.
\end{remark}

Fix a vector $\phi^\vee_{w,j} \in \sigma^\vee_{w,j}$ in the image of $\alpha^\vee_{w,j}$ such that $\brkt{\phi_{w,j}}{\phi^\vee_{w,j}}_{\sigma_{w,j}} = 1$ and define
\begin{equation} \label{eq:def phi a w vee 0 and phi b w vee 0}
    \phi^{\vee,0}_{a_w} := 
        \bigotimes \limits_{j=1}^{t_w}
            \phi^\vee_{w,j}
            \ \ \ ; \ \ \
    \phi^{\vee,0}_{b_w} := 
        \bigotimes \limits_{j=t_w+1}^{r_w}
            \phi^\vee_{w,j}
\end{equation}
as vectors in $\sigma^\vee_{a_w}$ and $\sigma^\vee_{b_w}$ respectively.

\begin{remark} \label{rmk:twists of local vectors in local types - contragredients}
    As in Remark \ref{rmk:twists of local vectors in local types}, if we replace $\tau_{w,j}$ by $\tau_{w,j} \otimes \psi$, for some finite-order character $\psi_{w,j}$ of $K_{w,j}$, then we always assume that the corresponding choice of local vector in the image of $\alpha_{w,j}^\vee \otimes \mathrm{id}$ is $\phi_{w,j}^\vee \otimes 1$. Once again, see Remark \ref{rmk:SZ types of twists and contragredient} for further details.
\end{remark}

Assume there exists a vector $\phi^\vee_{a_w}$ in $\pi^\vee_{a_w}$ fixed by $\tp{I}_{a_w,r}$ such that the support of its image $\varphi^\vee_{a_w}$ in $\ind{P_{a_w}}{G_w(a_w)} \sigma^\vee_{a_w}$ contains $P_{a_w} \tp{I}_{a_w, r}$ and that
\begin{equation} \label{def cg varphi aw}
    \varphi^\vee_{a_w}(\gamma) 
    = 
        \tau^\vee_{a_w}(\gamma) \phi^{\vee,0}_{a_w} \ , \ \forall \gamma \in \tp{I}^0_{a_w,r} \ .
\end{equation}

Similarly, assume there exists a vector $\phi^\vee_{b_w}$ in $\pi^\vee_{b_w}$ fixed by $I_{b_w,r}$ such that the support of its image $\varphi^\vee_{b_w}$ in $\ind{P_{b_w}}{G_w(b_w)} \sigma^\vee_{b_w}$ contains $P_{b_w} I_{b_w, r}$ and that
\begin{equation} \label{def cg varphi bw}
    \varphi^\vee_{b_w}(\gamma) 
        = 
    \tau^\vee_{b_w}(\gamma) 
    \phi^{\vee,0}_{b_w} \,,\,
    \forall \gamma \in I^0_{b_w,r} \ .
\end{equation}

Lastly, assume there exists a vector $\phi^\vee_w$ in $\pi^\vee_w$ fixed by $\tp{I}_{w,r}$ such that the support of its image $\varphi^\vee_w$ in $\ind{P_{a_w,b_w}}{G_w} \pi^\vee_{a_w, b_w}$ contains $P_w \tp{I}_{w, r}$ and that
\begin{equation} \label{def cg varphi w}
    \varphi^\vee_w(\gamma) 
    = 
        \tau^\vee_w(\gamma) (\phi^\vee_{a_w} \otimes \phi^\vee_{b_w}) \ , \ \forall \gamma \in \tp{I}^0_{w,r} \ .
\end{equation}

\begin{remark} \label{If cg pi Paord}
    As in Remark \ref{If pi Paord}, assume that $\sigma_w$ is supercuspidal for each $w \in \Sigma_p$. In that case, Lemma \ref{lma:canonical Pwaord vec of type tau w for G2} proves the existence of the vectors $\phi^\vee_{a_w}$, $\phi^\vee_{b_w}$ and $\phi^\vee_w$. In the last case, we are implicitly using the isomorphism \eqref{iso induced rep} to compare \emph{loc. cit.} with our notation here.

    In particular, in that case $\phi^\vee_{a_w}$ (resp. $\phi^\vee_{b_w}$, $\phi^\vee_w$) is $\tp{P}_{a_w}$-anti-ordinary (resp. $P_{b_w}$-anti-ordinary, $\tp{P}_w$-anti-ordinary) of type $\tau^\vee_{a_w}$ (resp. $\tau^\vee_{b_w}$, $\tau^\vee_w$) in the sense of Section \ref{subsec:Paord theory on G2}.
\end{remark}

\subsubsection{Choice of $P$-anti-ordinary test vectors (and twists).} \label{subsubseq:choice of Paord test vectors}
Our choice of test vectors at $p$ is
\begin{equation} \label{eq:def test vectors varphi p and varphi vee p}
    \varphi_p 
        = 
    1 
        \otimes 
    \left(    
        \bigotimes_{w \in \Sigma_p}
            \phi_w
    \right)
        \in
    \pi_p
        \;\;\;\text{and}\;\;\;
    \varphi^\vee_p 
        = 
    1 
        \otimes 
    \left(    
        \bigotimes_{w \in \Sigma_p}
            \phi^\vee_w
    \right)
        \in
    \pi^\vee_p\,,
\end{equation}
via \eqref{eq:pi p iso mu p otimes mu w} and \eqref{eq:pi p flat iso mu p flat otimes mu w flat}. By definition of $\pi^\flat$, $\varphi_p^\vee$ naturally corresponds to some $\varphi_p^\flat \in \pi_p^\vee$.

Observe that for each $w \in \Sigma_p$, the construction of $\phi_w$ not only depends on the choice of SZ-type $\tau_{w,j}$ of $\sigma_{w,j}$ but also on the choice of nonzero vectors $v_{w,j} \in \tau_{w,v}$, see \eqref{eq:def phi a w 0 and phi b w 0}.

Let $\pi'$ be some other anti-holomorphic $P$-anti-ordinary automorphic representation of $G_1$, with SZ-type $\tau'$ at $p$. Let $\sigma'_{w,j}$ and $\tau'_{w,j}$ be the analogues for $\pi'$ of and $\sigma_{w,j}$ and $\tau_{w,j}$ for $\pi$, as in Section \ref{subsubsec:comp of local representations}.

If $\tau' = \tau \otimes \psi$ for some character $\psi$ of $L(\ZZ_p)$, for instance if $\sigma_{w,j} = \sigma'_{w,j} \otimes \psi_{w,j}$ for some unramified character $\psi_{w,j}$ of $\GL_{n_{w,j}}(K_w)$ (see conventions set in Section \ref{subsubsec:SZ types}), then the vectors spaces for $\tau_{w,j}$ and $\tau'_{w,j}$ are canonically identified. 

We always assume that the vector $\varphi_p'$ for $\pi_p'$ is obtained from the same choices of vectors $v_{w,j}$ in this situation. We impose a similar convention for the dual vectors $\varphi_p^\vee \in \pi_p^\vee$ and $\varphi_p'^{,\vee} \in \pi_p'^{,\vee}$.

\subsubsection{Inner products between test vectors.} \label{subsubsec:inner product between test vectors}
Observe that the intersection of the support of $\varphi_w$ with $\GL_n(\OO_w)$ is $P_{a_w,b_w}I_{w,r} \cap \GL_n(\OO_w) = I^0_{a_w,b_w, r}$. Therefore,
\begin{align*}
    \la \phi_w, \phi^\vee_w \ra_{\pi_w} 
    = 
    \int_{I^0_{a_w, b_w, r}} 
        (
            \varphi_w(k),
            \varphi^\vee_w(k)
        )_{a_w,b_w} 
    d^\times k \ ,
\end{align*}

Write any $k \in I^0_{a_w, b_w, r}$ as
\[
    k = 
    \begin{pmatrix}
	1 & B \\
	0 & 1
    \end{pmatrix}
    \begin{pmatrix}
	A & 0 \\
	0 & D
    \end{pmatrix}
    \begin{pmatrix}
	1 & 0 \\
	C & 1
    \end{pmatrix} 
\]
where $A \in \GL_{a_w}(\OO_w)$, $D \in \GL_{b_w}(\OO_w)$, $B \in M_{a_w \times b_w}(\OO_w)$ and $C \in \p_w^rM_{b_w \times a_w}(\OO_w)$. Since
$
    \begin{pmatrix}
        1 & B \\ 0 & 1
    \end{pmatrix}
$ 
is in $P_{a_w, b_w}$ and 
$
    \begin{pmatrix}
        1 & 0 \\ C & 1
    \end{pmatrix}
$
is in both $I_{w,r}$ and $\tp{I}_{w,r}$, we see that
\[
    \varphi_w(k) = 
        \varphi_w
        \left( 
		\begin{pmatrix}
			A & 0 \\
			0 & D
		\end{pmatrix} 
	\right)
	= 
        \pi_{a_w}(A) \phi_{a_w} 
        \otimes \pi_{b_w}(D) \phi_{b_w}
\]
and
\[
    \varphi^\vee_w(k) = 
	\varphi^\vee_w
        \left(
		\begin{pmatrix}
			A & 0 \\
			0 & D
		\end{pmatrix}
	\right)
	= 
        \pi^\vee_{a_w}(A) \phi^\vee_{a_w} 
        \otimes \pi^\vee_{b_w}(D) \phi^\vee_{b_w} \ ,
\]
so we obtain
\begin{equation} \label{eq:size phi w}
    \la 
        \phi_w, 
        \phi^\vee_w 
    \ra_{\pi_w} 
        = 
    \vol(I^0_{a_w,b_w,r}) 
    \la 
            \phi_{a_w}, 
            \phi^\vee_{a_w}
        \ra_{\pi_{a_w}} 
    \la 
            \phi_{b_w}, 
            \phi^\vee_{b_w}
        \ra_{\pi_{b_w}} \ .
\end{equation}

Similar arguments yield
\begin{equation} \label{eq:size phi a w}
    \la 
        \phi_{a_w}, 
        \phi^\vee_{a_w} 
    \ra_{\pi_{a_w}} = 
	\vol(I^0_{P_{a_w}, r})
	( 
            \phi^0_{a_w}, 
            \phi^{\vee,0}_{a_w}
        )_{{a_w}}= 
	\vol(I^0_{P_{a_w}, r})
\end{equation}
and
\begin{equation} \label{eq:size phi b w}
    \la 
        \phi_{b_w}, 
        \phi^\vee_{b_w} 
    \ra_{\pi_{b_w}} = 
	\vol(I^0_{P_{b_w}, r})
	( 
            \phi^0_{b_w},
            \phi^{\vee,0}_{b_w} 
        )_{{b_w}} = 
	\vol(I^0_{P_{b_w}, r}) \ ,
\end{equation}
using the fact that $(\phi^0_{w, j}, \phi^{\vee,0}_{w, j}) = 1$ for each $1 \leq j \leq r_w$. Ultimately, we obtain
\begin{equation} \label{eq:inner prod Pwaord vectors at w}
    \la 
        \phi_w, 
        \phi^\vee_w 
    \ra_{\pi_w} 
    = 
        \vol(I^0_{a_w,b_w,r}) 
        \vol(I^0_{P_{a_w}, r})
        \vol(I^0_{P_{b_w}, r})
    =
        \vol(I^0_{w,r}) \ ,
\end{equation}
which in particular is nonzero.

\subsection{Local test vectors at $\infty$}
\label{subsec:Local test vectors at oo}
In this section, we choose local test vectors for $\pi_\infty$ and $\pi_\infty^\vee$. This material is well-establish in the literature. The author redirects the reader to \cite[Section 4.4]{EHLS} for ample details. 

\subsubsection{Anti-holomorphic modules for $G_1$.}
First consider $G^* = R_{\KK/\QQ} \GU^+(V, \brktdotdot)$, where $\GU^+(V, \brktdotdot)$ is the full unitary group associated to $\PP = \PP_1$. We have 
\[
    G^*(\RR) 
        = 
    \prod_{\sigma \in \Sigma}
        G_\sigma\,,
\]
where $G_\sigma = \GU^+(V)_{\KK_\sigma} \simeq \GU^+(a_\sigma, b_\sigma)$. Here, we implicitly use the identification between $\Sigma_{\KK^+}$ and $\Sigma$. Note that $G(\RR)$ consists of the subgroup of elements $(g_\sigma)_{\sigma \in \Sigma}$ for which the similitude factors $v(g_\sigma)$ are independent of $\sigma \in \Sigma$.

We view the map $h$ introduced in Section \ref{subsec:unitary pel datum} associated to $\PP$ as a homomorphism
\[
    h 
        = 
    \prod_{\sigma \in \Sigma} h_\sigma
        : 
    R_{\CC/\RR}(\Gm_{/\CC}) 
        \to 
    G^*_{\RR}
\]
whose image is contained in $G_{\RR}$. Note that $\g = \Lie(G(\RR))_\CC = \Lie(G^*(\RR))_\CC = \oplus_\sigma \g_\sigma$, where $\g_\sigma = \Lie(G_\sigma)$.

Let $U_\infty = C(\RR) \subset G(\RR)$ be the stabilizer of $h$ via conjugation, as in Section \ref{subsubsec:complex structure}. Then, $\pi_\infty$ and $\pi_\infty^\vee$ are both irreducible $(\g, U_\infty)$-modules.

For each $\sigma \in \Sigma_{\KK^+}$, let $U_\sigma = U_\infty \cap G_\sigma$ and let $K_\sigma^\circ \subset U_\sigma$ be its maximal compact subgroup. One readily checks that $K_\sigma$ is isomorphic to $U(a_\sigma) \times U(b_\sigma)$. 

As in Section \ref{subsubsec:Lie algebra cohomology}, the Harish-Chandra decomposition for $\g_\sigma$ is
\[
    \g_\sigma = \p_\sigma^- \oplus \k_\sigma \oplus \p_\sigma^+\,,
\]
where $\p_{\sigma}^\pm$ is the $(\pm 1)$-eigenspace of $\ad h_\sigma(\sqrt{-1})$ on $\g_\sigma$ and $\k_\sigma$ is the $0$-eigenspace. In particular, $\k_\sigma = \Lie(U_\sigma) = \mathfrak{z}_\sigma \oplus \Lie(K_\sigma^\circ)$, where $\mathfrak{z}_\sigma$ is the $\RR$-split center of $\g_\sigma$.

Recall that we assume that $h$ is \emph{standard}, see Hypothesis \ref{hyp:standard hypothesis}. This implies the above decomposition is rational over $\sigma(\KK) \subset \CC$. 

Furthermore, the fact that $h$ is standard is equivalent to the existence of a specific maximal rational torus $T$ of $G$ such that $h$ factors through $T \hookrightarrow G$. See \cite[Section 2.3.2]{EHLS} for further details and the exact construction of $T$ (denoted $J_0^{(n)}$). In particular, $(T, h)$ is a Shimura datum.

We decompose $\pi_\infty$ and $\pi_\infty^\vee$ as
\[
    \pi_\infty 
        = 
    \bigotimes_{\sigma \in \Sigma}
        \pi_\sigma
        \;\;\;\text{and}\;\;\;
    \pi_\infty^\vee
        =
    \bigotimes_{\sigma \in \Sigma}
        \pi_\sigma^\vee\,,
\]
for contragredient pairs of irreducible $(\g_\sigma, U_\sigma)$-modules $\pi_\sigma$ and $\pi_\sigma^\vee$.

The fact that $\pi_\infty$ is anti-holomorphic (for $G_1$) of weight $\kappa = (\kappa_0, (\kappa_\sigma))$ implies that for each $\sigma \in \Sigma$,
\begin{equation}\label{eq:pi sigma iso anti holomorphic module}
    \pi_\sigma 
        \cong
    U(\g_\sigma) 
        \bigotimes_{
            U(\k_\sigma \oplus \p_\sigma^+)
        }
    W_{\kappa_\sigma} 
        =: 
    \DD_c(\kappa_\sigma)\,,
\end{equation}
where $U(-)$ is the universal enveloping algebra functor, and $W_{\kappa_\sigma}$ is the irreducible representation of $U_\sigma$ of highest weight $\kappa_\sigma$. 

\subsubsection{Anti-holomorphic modules for $G_2$.}
If we consider $\pi^\flat$ as a representation of $G_2$ instead of $\pi$ as a representation of $G_1$, all of the theory above remains the same. However the roles of $\p_\sigma^+$ and $\p_\sigma^-$ are reversed.

One therefore obtains the isomorphism
\begin{equation}\label{eq:pi sigma flat iso holomorphic module}
    \pi_\sigma^\flat
        \cong
    U(\g_\sigma) 
        \bigotimes_{
            U(\k_\sigma \oplus \p_\sigma^-)
        }
    W_{\kappa_\sigma^\flat}
        =
    \DD_c(\kappa_\sigma^\flat)\,,
\end{equation}
and it follows from our discussion in Section \ref{subsubsec:the involution Foo} that $\DD(\kappa_\sigma) := \DD_c(\kappa_\sigma)^\vee \cong \DD_c(\kappa_\sigma^\flat)$ as representations of $U(a_\sigma, b_\sigma)$.

Furthermore, $\DD_c(\kappa_\sigma)$ is isomorphic to the complex conjugate of $\DD_c(\kappa_\sigma^\flat)$ with respect to the $\RR$-structure on $\g_\sigma$.

\begin{remark} \label{rmk:discrete series for G sigma}
    It is well-known that when the weight $\kappa$ satisfies Inequality \eqref{eq:ineq kappa sigma}, i.e. $\kappa$ is \emph{strongly positive}, then the modules $\DD_c(\kappa_\sigma)$ and $\DD(\kappa_\sigma)$ are the \emph{anti-holomorphic} and \emph{holomorphic discrete series} representations of $G_\sigma$ respectively. See \cite[Section 4.4.1]{EHLS}.
\end{remark}

\subsubsection{Choice of anti-holomorphic test vectors.} \label{subsubsec:choice of a-holo test vectors}
One refers to the subspaces $1 \otimes W_{\kappa_\sigma}$ and $1 \otimes W_{\kappa_\sigma^\flat}$ as the \emph{minimal $U_\sigma$-type} of $\DD_c(\kappa_\sigma)$ and $\DD_c(\kappa_\sigma^\flat)$ respectively. 

For each $\sigma$, let $\varphi_{\kappa_\sigma, -} \in \pi_\sigma$ be a lowest-weight vector in the minimal $U_\sigma$-type of $\DD_c(\kappa_\sigma)$ and $\varphi^\flat_{\kappa_\sigma^\flat, -} \in \pi^\flat_\sigma$ be a lowest-weight vector in the minimal $U_\sigma$-type of $\DD_c(\kappa_\sigma^\flat)$, both unique up to scalar. We normalize them so that $\brkt{\varphi_{\kappa_\sigma, -}}{\varphi^\flat_{\kappa^\flat_\sigma, -}}_{\sigma} = 1$.

Recall that we assume that the homorphism $h$ associated to the PEL data $\PP_1$ is \emph{standard}. In particular, $\varphi_{\kappa_\sigma, -}$ (resp. $\varphi^\flat_{\kappa_\sigma^\flat, -}$) is an eigenvector for $T_\sigma$ of weight $-\kappa$ (resp. $-\kappa_\sigma^\flat$). Here, $T_\sigma \subset G_\sigma$ is the $\sigma$-component of $T(\RR)$. 

Our choice of local test vectors $\varphi_\infty$ and $\varphi_{\infty}^\flat$ are
\begin{equation} \label{eq:def test vectors varphi oo and varphi vee oo}
    \varphi_\infty = \otimes_{\sigma} \varphi_{\kappa_\sigma, -}
        \;\;\;\text{and}\;\;\;
    \varphi^\flat_\infty = \otimes_{\sigma} \varphi^\flat_{\kappa_\sigma^\flat, -}\,.
\end{equation}

%%%%%%%%%%%%%%%%%%%%%%%%%%%%%%%%%%%%%%%%%%%%%%%%%%%%%%%%%%%%%%%%%%

\section{$P$-(anti-)ordinary Hida families.} \label{sec:P-(anti-)ord Hida families}
\subsection{Hecke algebras for modular forms with respect to $P$.} \label{subsec:Hecke alg for mod forms wrt P}

\subsubsection{(Anti-)holomorphic Hecke algebras.} \label{subsubsec:(Anti-)holomorphic Hecke alg}
We now construct the Hecke algebra of level $K_r = I_rK^p$ generated by Hecke operators at unramified places and at $p$. Let $S = S(K^p)$ as in Section \ref{subsubsec:ramified places away from p}.

Let $R$ be a $p$-adic algebra over $S_0 = \OO_{\KK', (\p')}$, as in Section \ref{subsec:scalar valued padic mod forms with respect to P}. Let $\bTT_{K_r, \kappa, R}$ denote the $R$-subalgebra of $\End_\CC(S_\kappa(K_r; \CC))$ generated by the operators 
\begin{enumerate}
    \item $T(g) = T_r(g)$, for all $g \in G(\AA_f^{S,p})$, 
    \item $u_{w, D_w(j)} = u_{w, D_w(j), \kappa}$, for all $w \in \Sigma_p$, $1 \leq j \leq r_w$, and
    \item $u_p(t) = u_{p, \kappa}(t)$, for all $t \in Z_P$.
\end{enumerate}

In particular, $\bTT_{K_r, \kappa, R}$ is an algebra over $R[Z_P]$, where $Z_P$ is the center of $L_P$. In fact, setting $Z_{P,r} = Z_P/(1+p^rZ_P)$, then it is equivalently an algebra over $R[Z_{P, r}]$.

If $R$ is also an $S_0[\tau]$-algebra, we define $\bTT_{K_r, \kappa, \tau, R}$ as the quotient algebra obtained by restricting each operator to an endomorphism of $S_\kappa(K_r, \tau; \CC)$. Finally, if $R$ is also an $S_r[\tau]$-algebra, we define $\bTT_{K_r, \kappa, [\tau], R}$ as the quotient algebra obtained upon restriction to $S_\kappa(K_r, [\tau]; \CC)$, where we recall that $[\tau] = [\tau]_r$ denotes the equivalence class of $\tau$ as a $P$-nebentypus of level $r$.

If $R = S_0$, $S_0[\tau]$ or $S_r[\tau]$, we omit $R$ from the notation. Moreover, if $r$ is clear from the context or does not affect the argument, we omit $K_r$ from the notation and simply write $\bTT_{\kappa, \tau}$ or $\bTT_{\kappa, [\tau]}$.

Similarly, we define the Hecke algebra $\bTT^d_{K_r, \kappa, R}$ as we constructed $\bTT_{K_r, \kappa, R}$ above but we replace $S_\kappa(K_r; \CC)$ by $\wh{S}_\kappa(K_r; \CC)$, and each $u_{w, D_w(j)}$ by $u^-_{w, D_w(j)}$, for all $w \in \Sigma_p$ and $1 \leq j \leq r_w$. Lastly, we define $\bTT^d_{K_r, \kappa, \tau, R}$ (resp. $\bTT^d_{K_r, \kappa, [\tau], R}$) analogously as a subalgebra of $\End_\CC(\wh{S}_\kappa(K_r, \tau; \CC))$ (resp. $\End_\CC(\wh{S}_\kappa(K_r, [\tau]; \CC))$).

For each of these Hecke algebras $\bTT^?_\bullet$, we write $\bTT_\bullet^{?, p}$ for the subalgebra generated by the operators $T_r(g)$, $g \in G(\AA_f^S)$, i.e. by omitting the Hecke operators at $p$.

\begin{remark}
    Implicitly, all of the above is stated for $G = G_1$. The definitions for $G = G_2$ are identical, considering the conventions set in Section \ref{subsec:comp between G1 and G2}. If we want to distinguish the two situations, we write $T_{V, \bullet}^?$ for $G_1$ and $T_{-V, \bullet}^?$ for $G_2$.
\end{remark}

\subsubsection{Hecke equivariance} \label{subsubsec:Hecke equivariance}
Let $\varphi \in H^0_!(\level{K_r}{\Sh(V)}, \w_\kappa)$ and $\varphi' \in H^d_!(\level{K_r}{\Sh(V)}, \w_{\kappa^D})$. By definition of \eqref{eq:def pairing kappa Kr}, one readily checks that
\[
    \brkt{
        T(g) \varphi
    }{
         \varphi'
    }_{\kappa, K_r}
        =
    \brkt{
        \varphi
    }{
         T(g)^d \varphi'
    }_{\kappa, K_r}
\]
and
\[
    \brkt{
        u_{w, D_w(j), \kappa} \varphi
    }{
        \varphi'
    }_{\kappa, K_r}
        =
    \brkt{
        \varphi
    }{
        u_{w, D_w(j), \kappa^D}^- \varphi'
    }_{\kappa, K_r}\,,
\]
for all $g \in G(\AA_f^{S,p})$ and $w \in \Sigma_p$, $1 \leq j \leq r_w$, where $T(g)^d := ||\nu(g)||^{a(\kappa)}T(g^{-1})$. 

Similarly, using notation from Section \ref{subsubsec:the involution dagger} and the isomorphism in Remark \ref{rmk:iso cohomology -V kappa flat and -V kappa dagger}, if $\varphi^\flat = F^\dagger(\varphi) \in S_{\kappa^\flat}(-V, K^\flat_r; R)$, we have
\[
    T(g)^\flat \varphi^\flat 
        = 
    F^\dagger(T(g) \varphi)
        \; ; \;
    u_{w, D_w(j), \kappa}^\flat \varphi^\flat 
        = 
    F^\dagger(u_{w, D_w(j), \kappa} \varphi)\,,
\]
where $T(g)^\flat := T(g^\dagger) = T(\ol{g})$ and $u_{w, D_w(j), \kappa}^\flat := u_{w, n, \kappa^\flat}^{-1} u_{w, n - D_w(j), \kappa^\flat}$.

We obtain the next result as a consequence.

\begin{lemma} \label{lma:iso holo aholo flat Hecke alg}
    Let $R \subset \CC$ be any subring. 

    \begin{enumerate}
        \item The map $\bTT_{K_r, \kappa, R} \to \bTT^d_{K_r, \kappa^D, R}$ induced by
        \[
            T(g) 
                \mapsto 
            T(g)^d 
                \ \ \ \text{and} \ \ \ 
            u_{w, D_w(j), \kappa} 
                \mapsto 
            u_{w, D_w(j), \kappa^D}^-\,,
        \]
        is an isomorphism.
        \item The map in (i) induces an isomorphism $\bTT_{K_r,\kappa, \tau, R} \xrightarrow{\sim} \bTT^d_{K_r, \kappa^D, \tau^\vee, R}$.
        \item The map $\bTT_{V, K_r, \kappa, R} \to \bTT_{-V, K_r^\flat, \kappa^\flat, R}$ induced by
        \[
            T(g) 
                \mapsto 
            T(g)^\flat 
                \ \ \ \text{and} \ \ \ 
            u_{w, D_w(j), \kappa} 
                \mapsto 
            u_{w, D_w(j), \kappa}^\flat
        \]
        is an isomorphism.
        \item The map in (iii) induces an isomorphism $\bTT_{V, K_r,\kappa, \tau, R} \xrightarrow{\sim} \bTT^d_{-V, K_r^\flat, \kappa^\flat, \tau^\flat, R}$.
    \end{enumerate}
\end{lemma}

We use the isomorphisms of Lemma \ref{lma:iso holo aholo flat Hecke alg} to view $\wh{S}_\kappa(K_r; R)$ and $S_{\kappa^\flat}(-V, K_r^\flat; R)$ (resp. $\wh{S}_\kappa(K_r, \tau; R)$ and $S_{\kappa^\flat}(-V, K_r^\flat, \tau^\flat; R)$) as modules over $\bTT_{K_r, \kappa, R}$ (resp. $\bTT_{K_r, \kappa, \tau, R}$).

\subsubsection{$P$-(anti-)ordinary Hecke algebras.} \label{subsubsec:P-(anti-)ordinary Hecke alg}

Let $\bTT^{\Pord}_{K_r, \kappa, R} := e_\kappa \bTT_{K_r, \kappa, R}$ and $\bTT^{\Paord}_{K_r, \kappa, R} := e^-_\kappa \bTT_{K_r, \kappa, R}$. We define $\bTT^{\Pord}_{K_r, \kappa, \tau, R}$, $\bTT^{\Pord}_{K_r, \kappa, [\tau], R}$, $\bTT^{\Paord}_{K_r, \kappa, \tau, R}$ and $\bTT^{\Paord}_{K_r, \kappa, [\tau], R}$ similarly.

\begin{lemma} \label{lma:comp Pord and Paord hecke alg}
    The isomorphism of Lemma \ref{lma:iso holo aholo flat Hecke alg} (i) induces an isomorphism
    \[
        \bTT^{\Pord}_{K_r, \kappa, R} 
            \xrightarrow{\sim}
        \bTT^{d, \Paord}_{K_r, \kappa^D, R}\,.
    \]

    Similarly, The isomorphism of Lemma \ref{lma:iso holo aholo flat Hecke alg} (iii) induces an isomorphism
    \[
        \bTT^{\Pord}_{V, K_r, \kappa, R} 
            \xrightarrow{\sim}
        \bTT^{\Pord}_{-V, K_r^\flat, \kappa^\flat, R}\,.
    \]
\end{lemma}

\begin{remark}
    Similar isomorphisms exists for Hecke algebras associated to a (class of) type but we omit the explicit statement.
\end{remark}

Consequently, $\wh{S}_\kappa^{\Paord}(K_r; R)$ (resp. $\wh{S}_\kappa^{\Paord}(K_r, \tau; R)$, $\wh{S}_\kappa^{\Paord}(K_r, [\tau]; R)$) has a natural structure as a module over $\bTT^{\Pord}_{K_r, \kappa, R}$ (resp. $\bTT^{\Pord}_{K_r, \kappa, \tau, R}$, $\bTT^{\Pord}_{K_r, \kappa, [\tau], R}$).

\subsection{Lattices of holomorphic $P$-ordinary forms.} \label{subsec:lattices of Pord holo forms}
Let $\pi$ be a holomorphic cuspidal automorphic representation on $G$ of weight $\kappa$ and level $K = K_r = K_{P,r} = I_{P,r} K^p$, for some $r \geq 0$. In what follows, we use the notation of Section \ref{subsec:(A-)holo auto reps} without comments.

In particular, we identify $(\pi^{p,S})^{K^{p, S}}$ as a 1-dimensional $\CC$-vector space with a natural $E(\pi)$-rational structure, see Remark \ref{rmk:E pi rational spherical vectors}.

Furthermore, we fix a choice of a highest weight vector $\varphi_\infty$ in $\pi_\infty$. By definition of the weight of $\pi$, this is equivalent to the choice a nonzero vector in the 1-dimensional $\CC$-vector space
\[
    H^0(\P_h, K_h; \pi_\infty \otimes W_\kappa)\,.
\]

Using the above and \eqref{eq:Lie alg coh for S kappa Kr CC}, we obtain an embedding
\[
    \pi_p^{I_{P, r}} \otimes \pi_S^{K_S} \hookrightarrow S_\kappa(K_r; \CC)\,,
\]
over $\CC$, which is equivariant for the action of $\bTT^p_{K_r, \kappa}$. 

Let $\lambda_\pi^p$ be the character via which $\bTT^p_{K_r, \kappa}$ acts on $\pi^{K_r}$, namely its action on $(\pi^{p,S})^{K^{p, S}}$. Then the embedding above factors through
\[
    j_\pi : \pi_p^{I_{P,r}} \otimes \pi_S^{K_s} \hookrightarrow S_\kappa(K_r; \CC)(\pi)\,,
\]
where $S_\kappa(K_r; \CC)(\pi)$ denotes the $\lambda^p_\pi$-isotypic component of $S_\kappa(K_r; \CC)$.

For the remainder of this article, we assume the following :
\begin{hypothesis}[Multiplicity one for $\pi$] \label{hyp:mult one hyp (I)}
    For any holomorphic cuspidal automorphic representation $\pi'$ of weight $\kappa$ such that $(\pi_f')^{K_r} \neq 0$, if $\pi' \neq \pi$, then $\lambda^p_{\pi'} \neq \lambda^p_{\pi}$.
\end{hypothesis}

\begin{remark} \label{rmk:mult one hyp and limitations}
    This is the same multiplicity one hypothesis as \cite[Hypothesis 6.6.4]{EHLS}. See the comments below \emph{loc. cit} to see the limitations of this hypothesis and the cases where it is known to hold.
\end{remark}

\begin{lemma} \label{lma:jpi iso (I)}
    Let $\pi$, $\kappa$ and $K_r = K_{P,r}$ be as above. Assume that $\pi$ satisfies Hypothesis \ref{hyp:mult one hyp (I)}. Then, the embedding $j_\pi$ is an isomorphism.
\end{lemma}

To study this isomorphism further, assume that $\pi$ is $P$-ordinary. Let $(\tau, \MMM_\tau)$ is the SZ-type of $\pi_p$, as in Section \ref{subsubsec:strucutre thm for Pord rep of G1}. We assume $r$ is large enough so that $\tau$ is a $P$-nebentypus of level $r$.

It follows from Remark \ref{rmk:eigenvalues of Pword representations} that the Hecke operator $u_{w, D_w(j)} = u_{w, D_w(j), \kappa}$, for $w \in \Sigma_p$ and $1 \leq j \leq r_w$, acts as a scalar on $\pi_p^{(\Pord, r)}[\tau]$, independent of $r \gg 0$ and our choice of SZ-type from Section \ref{subsubsec:SZ types}. Hence, the character $\lambda_\pi^p$ extends uniquely to a character $\lambda_\pi$ of $\bTT_{K_r, \kappa}$ corresponding to its action on $\pi_p^{(\Pord, r)} \otimes \pi_S^{K_S}$, and $\lambda_\pi$ factors through $\bTT^{\Pord}_{K_r, \kappa, \tau, R}$.

Let $E(\lambda_\pi)$ denote the smallest extension of $E(\pi)$ also containing the values of $\lambda_\pi$. Let $R(\lambda_\pi)$ denote the localization of the ring of integers of $E(\lambda_\pi)$ at the maximal ideal determined by $\incl_p$. One readily sees that $\lambda_\pi$ is $R(\lambda_\pi)$-valued.

We always denote the residue field of $R(\lambda_\pi)$ by $k(\pi)$, the reduction of $\lambda_{\pi}$ in $k(\pi)$ by $\ol{\lambda}_\pi$, and the $p$-adic completion of $R(\lambda_\pi)$ by $\OO_\pi$. In particular, we view $\ol{\lambda}_\pi$ as being valued in a fixed algebraic closure of $\mathbb{F}_p = \ZZ/p\ZZ$.

Let
\[
    \varphi^\circ 
        = 
    \left( 
        \bigotimes_{l \notin S} \varphi_{l,0} 
    \right) 
        \otimes
    \varphi^\circ_S 
        \otimes
    \varphi_\infty 
        \otimes 
    \varphi_p 
        \in 
    H^0(
        \P_h, K_h;
        \pi^{K_r}
            \otimes 
        W_\kappa
    )\,,
\]
where each local factor is a test vector for $\pi$ chosen as in Section \ref{sec:explicit choice of P-(anti-)ord vectors}. In particular, $\varphi_p = \varphi_{p, \iota, v} := \iota(v)$ depends on the choice of an $\LL_r$-embedding $\iota : \tau \hookrightarrow \pi_p^{(\Pord, r)}$ and a nonzero vector $v \in \MMM_\tau$.

From \eqref{eq:Lie alg coh for H!i Sh V kappa}, one readily sees that the (canonical) choice of test vectors away from $S \cup \{p\}$ induces a map
\begin{equation} \label{eq: map pi p Ir x pi S K S to S kappa K r}
    \pi_p^{I_r} \otimes \pi_S^{K_S}
        \to
    S_\kappa(K_r; \CC)\,,
\end{equation}
that is equivariant under the action of $\bTT^p_{K_r, \kappa, \tau}$

The above can be improved via \eqref{eq:Lie alg coh for H!i Sh V kappa r tau} to incorporate our choice of $\varphi_p$ as follows. Following our discussion from Section \ref{subsubsec:complex modular forms}, we can tensor this map by $\MMM_\tau^\vee$ (and apply $P$-ordinary projections), to obtain a map
\begin{equation} \label{eq: map pi p tau Ir x pi S K S to S kappa K r  tau}
    \hom_{\LL_r}(\tau, \pi_p^{(\Pord, r)}) \otimes \pi_S^{K_S}
        \hookrightarrow
    S_\kappa^{\Pord}(K_r, \tau; \CC)
\end{equation}
that is equivariant under the action of $\bTT_{K_r, \kappa, \tau}$. Let $f^\circ$ and $F^\circ$ denote the image of $\varphi_p \otimes \varphi_S^\circ$ and $\iota \otimes \varphi_S^\circ$ via \eqref{eq: map pi p Ir x pi S K S to S kappa K r} and \eqref{eq: map pi p tau Ir x pi S K S to S kappa K r  tau} respectively. Then, one readily sees that $F^\circ(v) = f^\circ$. 

By definition of $\lambda_\pi$, we in fact have a $\bTT_{K_r, \kappa, \tau}$-equivariant embedding
\[
    j_\pi
        :
    \hom_{\LL_r}(\tau, \pi_p^{(\Pord, r)}) 
        \otimes 
    \pi_S^{K_S}
        \hookrightarrow
    S_\kappa^{\Pord}(
        K_r, \tau; E(\lambda_\pi)
    )[\lambda_\pi] 
        \otimes_{E(\lambda_\pi)}
    \CC\,,
\]
where $[\lambda_\pi]$ indicates the $\lambda_\pi$-isotypic component.

By Theorem \ref{thm:Pord type tau structure thm}, the space $\hom_{\LL_r}(\tau, \pi_p^{(\Pord, r)})$ is 1-dimensional and $\iota$ corresponds to a basis element. Therefore, the following is an immediate consequence of the above together with Lemma \ref{lma:jpi iso (I)}.

\begin{proposition} \label{prop:jpi iso (II)}
    Let $\pi$, $\kappa$ and $K_r$ be as in Lemma \ref{lma:jpi iso (I)}. Let $\tau$ and $\iota$ be as above. Suppose that $\pi$ satisfies Hypothesis \ref{hyp:mult one hyp (I)}.
    
    Let $R \subset \CC$ be the localization of a finite extension of $R(\lambda_\pi)$ at the prime determined by $\incl_p$ or the $p$-adic completion of such a ring. Let $E = R[1/p]$. Then, $j_\pi$ induces an isomorphism between
    \[
        \hom_{\LL_r}(\tau, \pi_p^{(\Pord, r)}) 
            \otimes 
        \pi_S^{K_S} 
            = 
        \pi_S^{K_S}\,.
    \]
    and $S_\kappa^{\Pord}(K_r, \tau; E)[\lambda_\pi] \otimes_E \CC$.

    Furthermore, let $\m_\pi \subset \bTT_{K_r, \kappa, \tau, R}$ be the kernel of the reduction of $\lambda_\pi$ modulo the maximal ideal of $R$. Let $S_\kappa^{\Pord}(K_r, \tau; R)_\pi$ denote the localization of $S_\kappa^{\Pord}(K_r, \tau; R)$ at the maximal ideal $\m_\pi$, and set
    \begin{align*}
        S_\kappa^{\Pord}(K_r, \tau; R)[\pi] 
            :=\
        &S_\kappa^{\Pord}(K_r, [\tau]; R)_\pi 
            \cap 
        S_\kappa^{\Pord}(K_r, \tau; E)[\lambda_\pi]
            \\ =\
        &S_\kappa^{\Pord}(K_r, \tau; R)_\pi 
            \cap 
        S_\kappa^{\Pord}(K_r, \tau; E)[\lambda_\pi]\,.
    \end{align*}

    Then, $j_\pi$ identifies $S_\kappa^{\Pord}(K_r, \tau; R)[\pi]$ with an $R$-lattice in $\pi_S^{K^S}$.
\end{proposition}

To finish this section, we also identify $S_\kappa^{\Pord}(K_r, [\tau]; R)_\pi$ as a lattice in a space of automorphic forms. To do so, we consider congruence between automorphic forms modulo $p$.

Namely, define the set $\SSS(\pi, \kappa, K_r, [\tau])$ as the collection of $P$-ordinary holomorphic cuspidal automorphic representation $\pi'$ of $P$-WLT $(\kappa, K_r, \tau')$ such that $[\tau]_r = [\tau']_r$ and $\ol{\lambda}_\pi = \ol{\lambda}_{\pi'}$. Here, $\tau'$ is again chosen to be the SZ-type of $\pi'$.

In particular, for any $\pi' \in \SSS(\pi, \kappa, K_r, \tau)$, both $\lambda_\pi$ and $\lambda_\pi'$ both factor through characters of $\bTT_{K_r, \kappa, [\tau], R}$ (for some sufficiently large ring $R$), and $\m_\pi = \m_{\pi'}$. 

\begin{proposition}\label{prop:jpi iso (III)}
    With notation as in Proposition \ref{prop:jpi iso (II)}, $S_\kappa^{\Pord}(K_r, [\tau]; R)_\pi$ is identified with an $R$-lattice in
    \[
        \bigoplus_{
            \pi^\prime 
                \in 
            \SSS(\pi, \kappa, K_r, [\tau])
        } 
            \hom_{\LL_r}(\tau', \pi_p^{(\Pord, r)})
                \oplus
            (\pi^\prime_S)^{K_S}
            =
        \bigoplus_{\pi^\prime \in \SSS(\pi, \kappa, K_r, [\tau])} (\pi^\prime_S)^{K_S}
    \]
    via the map $\oplus_{\pi^\prime} j_{\pi^\prime}$.
\end{proposition}

\subsection{Lattices of anti-holomorphic $P$-anti-ordinary forms.}
\label{subsec:lattices of aholo Paord forms}
We now adjust the theory above in the anti-holomorphic case for $\pi^\flat$ on $G$, where $\pi$ is as in the previous section. We keep our assumption that $\pi$ satisfies Hypothesis \ref{hyp:mult one hyp (I)}, hence $\pi^\flat = \ol{\pi}$ as subspaces of $\Ab_0(G)$, see Remark \ref{rmk:pi flat instead of pi bar}.

\subsubsection{Lattices in $\pi^\flat$.} \label{subsubsec:lattices in pi flat}

Let
\[
    \varphi^{\flat,\circ} 
        = 
    \left( 
        \bigotimes_{l \notin S} \varphi^\flat_{l,0} 
    \right) 
        \otimes
    \varphi^{\flat, \circ}_S 
        \otimes
    \varphi^\flat_\infty 
        \otimes 
    \varphi^\flat_p 
        \in 
    H^d(
        \P_h, K_h;
        \pi^{\flat, K_r}
            \otimes 
        W_{\kappa^D}
    )\,,
\]
where each local factor is a test vector for $\pi^\flat$ chosen as in Section \ref{sec:explicit choice of P-(anti-)ord vectors}. Again, $\varphi^\flat_p = \iota^\vee(v^\vee)$ depends on the choice of an $\LL_r$-embedding $\iota : \tau^\vee \hookrightarrow (\pi_p^\flat)^{(\Paord, r)}$ and a nonzero vector $v^\vee \in \MMM_\tau^\vee$.

Similar to the holomorphic case, after fixing a basis of the 1-dimensional complex vector space $H^d(\P_h, K_h; \pi^\flat \otimes W_{\kappa^D})$ and unramified local vectors, we obtain an embedding
\[
    \pi_p^{\flat, I_r} 
        \otimes 
    \pi_S^{\flat, K_S} 
        \hookrightarrow
    H^d_{\kappa^D}(K_r, \CC)
        =
    H^d_!(\level{K_r}{\Sh(V)}, \w_{\kappa^D})\,,
\]
via the identification \eqref{eq:Lie alg coh for H!i Sh V kappa}.

Assume $\pi$ is $P$-ordinary with SZ-type $\tau$, or equivalently, that $\pi^\flat$ is $P$-anti-ordinary with SZ-type $\tau^\flat$. Then, $\bTT^d_{K_r, \kappa^D, \tau^\flat}$ acts on $(\pi_p^\flat)^{(\Paord, r)}[\tau^\flat] \otimes \pi_S^{\flat, K_S}$ via some character $\lambda_\pi^\flat$. We use Lemma \ref{lma:iso holo aholo flat Hecke alg} to view $\lambda_\pi^\flat$ as a character of $\bTT_{K_r, \kappa, \tau}$. 

\begin{remark} \label{rmk:lambda pi flat and m pi flat}
    It follows from the definition of the isomorphism in Lemma \ref{lma:iso holo aholo flat Hecke alg} (ii) that $\lambda_{\pi}^\flat = \lambda_\pi$ as characters of $\bTT_{K_r, \kappa, \tau}$. In particular, the ring $R(\lambda_\pi)$ and its $p$-adic completion $\OO_\pi$ defined in the previous section are the same when working with $\pi$ or $\pi^\flat$. Furthermore, the kernel of $\lambda_\pi^\flat$ is again the maximal $\m_\pi$ of $\bTT_{K_r, \kappa, \tau}$.
\end{remark}

Again, the map above further induces an embedding
\[
    j_{\pi^\flat}
        :
    \hom_{\LL_r}(\tau^\flat, \pi_p^{\flat, (\Paord, r)})
        \otimes
    \pi_S^{\flat, K_S}
        \hookrightarrow
    \wh{S}_\kappa
    (
        K_r, \tau; E(\lambda_\pi)
    )[\lambda_\pi]
        \otimes_{E(\lambda_\pi)}
    \CC\,,
\]
using the identification \eqref{eq:Lie alg coh for H!i Sh V kappa r tau}, that is $\bTT_{K_r, \kappa, \tau}$-equivariant. From Corollary \ref{cor:canonical Paord vec of type tau for G1}, we know $\hom_{\LL_r}(\tau^\flat, \pi_p^{\flat, (\Paord, r)})$ is 1-dimensional and $\iota^\vee$ corresponds to a basis element.

Let $R \subset \CC$ be any ring as in Prop \ref{prop:jpi iso (II)}, and let $E = R[1/p]$. Given any $\bTT_{K_r, \kappa, \tau, R}$-module $M$, we again denote its $\lambda_\pi$-isotypic (or equivalently, $\lambda_\pi^\flat$-isotypic) component by $M[\lambda_\pi]$ and its localization at $\m_\pi$ by $M_\pi$. Moreover, we define
\begin{align*}
    \wh{S}_{\kappa}^{\Paord}(K_r, \tau; R)[\pi] 
        :=\
    &\wh{S}_{\kappa}^{\Paord}(K_r, [\tau]; R)_{\pi}
        \cap
    \wh{S}_{\kappa}^{\Paord}(K_r, \tau; E)[\lambda_\pi]
        \\ =\
    &\wh{S}_{\kappa}^{\Paord}(K_r, \tau; R)_{\pi}
        \cap
    \wh{S}_{\kappa}^{\Pord}(K_r, \tau; E)[\lambda_\pi]
\end{align*}

\begin{lemma} \label{lma:jpi flat iso}
    Let $\pi$ be as above, $R$, and $E$ be as above. Assume $\pi$ satisfies Hypothesis \ref{hyp:mult one hyp (I)}. Then,
    \begin{enumerate}
        \item The embedding $j_\pi^\flat$ induces an isomorphism
        \[
            \pi_S^{\flat, K_S}
                \xrightarrow{\sim}
            \wh{S}_\kappa
            (
                K_r, \tau; E
            )[\lambda_\pi]
                \otimes_{E}
            \CC\,.
        \]
        
        \item The isomorphism from part (i) identifies $\wh{S}_{\kappa}^{\Paord}(K_r, \tau; R)[\pi]$ with an $R$-lattice in $\pi_S^{\flat, K_S}$. Similarly, $\wh{S}_{\kappa}^{\Paord}(K_r, [\tau]; R)_{\pi}$ with an $R$-lattices in
        \[
            \bigoplus_{
                \pi^\prime \in \SSS(\pi, \kappa, K_r, [\tau])
            }
                (\pi_S^{\prime, \flat})^{ K_S}
        \]
        via the map $\oplus_{\pi^\prime} j_{\pi^\prime}^\flat$.

        \item The pairing $\brktdotdot_{\kappa, K_r, \tau}$ induces perfect $\bTT_{K_r, \kappa, \tau, R}^{\Pord}$-equivariant pairings
        \[
            S_\kappa^{\Pord}(K_r, \tau; R)[\pi]
                \otimes
            \wh{S}_{\kappa}^{\Paord}(K_r, \tau; R)[\pi] \to R
        \]
        and the pairing $\brktdotdot_{\kappa, K_r, [\tau]}$ induces perfect $\bTT_{K_r, \kappa, [\tau], R}^{\Pord}$-equivariant pairings
        \[
            S_\kappa^{\Pord}(K_r, [\tau]; R)_\pi
                \otimes
            \wh{S}_{\kappa}^{\Paord}(K_r, [\tau]; R)_{\pi} \to R
        \]
    \end{enumerate}
\end{lemma}

\subsection{Big Hecke algebra and $P$-anti-ordinary Hida families.} \label{subsec:big Hecke alg and Paord Hida families}

\subsubsection{Independence of weights.} \label{subsubsec:indep of weights}
Let $R$, $\kappa$, $\tau$ and $K_r$ be as above. Consider the algebra
\[
    \varprojlim_{r}
        \bTT^{\Pord}_{K_r, \kappa, [\tau], R}
\]
over $\Lambda_R := R[[Z_P]]$. It follows from the discussion at the end of Section \ref{subsec:Hecke ops on padic mod forms} that this algebra can be viewed as a subquotient of $\End_R(\VV^\Pord(K^p, [\kappa_p, \tau], R)$.

\begin{conjecture} \label{conj:indep of weight of padic Hecke alg}
    Let $\kappa_1$ and $\kappa_2$ be two dominant characters such that $[\kappa_1] = [\kappa_2]$. There is a canonical isomorphism
    \[
        \varprojlim_{r}
            \bTT^{\Pord}_{K_r, \kappa_1, [\tau], R}
                \xrightarrow{\sim}
        \varprojlim_{r} 
            \bTT^{\Pord}_{K_r, \kappa_2, [\tau], R}\,.
    \]
\end{conjecture}

From now on, we assume that this conjecture holds without comments. Furthermore, we write $\bTT^{\Pord}_{K^p, [\kappa, \tau], R}$ instead of 
$
    \varprojlim_{r}
        \bTT^{\Pord}_{K_r, \kappa, [\tau], R}
$
to emphasize the fact that this algebra (conjecturally) only depends on the set $[\kappa]$ of dominant weights obtained as $P$-parallel shifts of $\kappa$.

\begin{remark} \label{rmk:Hida independence of weight of Hecke alg}
    When $P = B$ as in Remark \ref{rmk:trivial partition}, this result holds and is due to Hida, see \cite[Theorem 7.1.1]{EHLS}.
\end{remark}

Recall that the normalized Serre pairing is stable under the trace map, see \eqref{eq:norm Serre pairing trace map stable} and \eqref{eq:norm Serre pairing trace map stable with types}. In particular, for $\tau$ of level $r' > r \gg 0$, we have a map
\[
    \tr_{K_r/K_{r'}}
        :
    \wh{S}_\kappa(K_{r'}, [\tau]; R)
        \to
    \wh{S}_\kappa(K_r, [\tau]; R)\,.
\]

Therefore, the above induces natural maps
\[
    \bTT^{d, \Paord}_{K_{r'}, \kappa^D, [\tau^\vee], R}
        \to
    \bTT^{d, \Paord}_{K_r, \kappa^D, [\tau^\vee], R}
\]
that are compatible with the isomorphisms of Lemma \ref{lma:comp Pord and Paord hecke alg} and the maps 
\[
    \bTT^{\Pord}_{K_{r'}, \kappa, [\tau], R} \to \bTT^{\Pord}_{K_r, \kappa, [\tau], R}\,.
\]

In other words, the algebra
\[
    \bTT^{d, \Paord}_{K^p, [\kappa^D, \tau^\vee], R} 
        :=
    \varprojlim_{r} \bTT^{d, \Paord}_{K_r, \kappa^D, [\tau^\vee], R}\,.
\]
is well-defined and isomorphic to $\bTT^{\Pord}_{K^p, [\kappa, \tau], R}$ via Lemma \ref{lma:comp Pord and Paord hecke alg}. In particular, Conjecture \ref{conj:indep of weight of padic Hecke alg} implies a similar independence of weight for $\bTT^{d, \Paord}_{K^p, [\kappa^D, \tau^\vee], R}$.

\subsubsection{Classical points of $P$-anti-ordinary families.} \label{subsubsec:classical points of P-anti-ordinary families}
Let $Z_P^\circ$ denote the maximal pro-$p$-subgroup of $Z_P$. There exists a finite group $\Delta_P \subset Z_P$ of order prime-to-$p$ such that $Z_P = \Delta_P \times Z_P^\circ$.

For a $p$-adic ring $R$ as in the previous section, let $\Lambda^\circ_R \subset \Lambda_R$ be the complete group algebra associated to $Z_P^\circ$ over $R$. We refer to $\WWW = \spec \Lambda_R^\circ$ as the \emph{weight space} over $R$ (associated to the parabolic $P$). The \emph{weight map} is the structure homomorphism $\Omega : \Lambda_R^0 \to \bTT_{K^p, [\kappa, \tau], R}^\Pord$ sending $t \mapsto u_p(t)$ for all $t \in Z_P^\circ$.

Let $\kappa$, $K_r$ and $\tau$ be as in the previous sections. Let $\kappa_p$ be the $p$-adic weight corresponding to $\kappa$, viewed as an algebraic character of $T_H(\ZZ_p)$. Denote its restriction to a character of $Z_P$ by $\kappa_p$ again. Furthermore, let $\w_\tau$ denote the central character of $\tau$, a finite order character of $Z_P$.

\begin{definition} \label{def:arith homomorphism of Lambda}
    We say that a homomorphism $\Lambda_R^\circ \to R$ is \emph{arithmetic} if it is induced by an $R$-valued character of $Z_P$ of the form $\kappa_p \cdot \w_\tau$ for some $\kappa$ and $\tau$ as above. We sometimes say that $\kappa_p \cdot \w_\tau$ is an \emph{arithmetic character} of $Z_P$.
\end{definition}

\begin{definition} \label{def:arith homomorphism of bTT}
    Let $\lambda : \bTT_{K^p, [\kappa_p, \tau], R}^{\Pord} \to R^\times$ be a continuous character. We say that $\lambda$ is \emph{arithmetic} if its composition $\lambda \circ \Omega : \Lambda_R^\circ \to R$ with the weight map is arithmetic.
\end{definition}

If we fix ``base points'' $\kappa$ and $\tau$ of $[\kappa]$ and $[\tau]$ respectively, note that any arithmetic character of $\Lambda_R^\circ$ corresponds to a product of an algebraic character $(\kappa + \theta)_p$ and a finite-order character $\w_{\tau \otimes \psi}$, for some $P$-parallel weight $\theta$ and some finite-order character $\psi$ of $L_H(\ZZ_p)$. Recall that we use additive notation for the binary operation on the set of algebraic weights.

Furthermore, one readily sees that for all $P$-anti-ordinary automorphic representation $\pi$, the associated character $\lambda_\pi$ constructed in Section \ref{subsubsec:lattices in pi flat} is arithmetic, valued in $\OO_\pi$ and factors through $(\bTT^{\Pord}_{K^p, [\kappa, \tau], \OO_\pi})_{\m_\pi}$.

\begin{definition} \label{def:classical arithmetic caracters}
    We say that an arithmetic character $\lambda$ is \emph{classical} if it arises as $\lambda = \lambda_\pi$ for some $\pi$ as above. If $\pi$ is of $P$-anti-WLT $(\kappa, K_r, \tau)$, we say $\lambda$ has weight $\kappa$, level $r \gg 0$ and $P$-nebentypus $\tau$.
\end{definition}

For any tame character $\epsilon$ of $Z_P$, we write $\Lambda_{R, \epsilon}$ (resp. $\Lambda^\circ_{R, \epsilon}$) for the localization of $\Lambda_R$ (resp. $\Lambda^\circ_R$) at the maximal ideal of $\Lambda_R$ (resp. $\Lambda^\circ_R$) defined by $\epsilon$. Note that the quotient map $\Lambda_\pi \to \Lambda_\pi/(\m_\pi \cap \Lambda_\pi)$ is the homomorphism induced by some tame character of $Z_P$. 

\begin{conjecture} \label{conj:VCT}
    Let $R$, $\kappa$, $\tau$ and $K_r$ be as above, and assume Conjecture \ref{conj:indep of weight of padic Hecke alg}.
    \begin{enumerate}
        \item For each tame character $\epsilon$, the localization $\bTT^{\Pord}_{K^p, [\kappa, \tau], R, \epsilon}$ of the Hecke algebra $\bTT^{\Pord}_{K^p, [\kappa, \tau], R}$ at the maximal ideal defined by $\epsilon$ is finite free over $\Lambda_{R, \epsilon}^\circ$.
        \item Let $\kappa$ be a $P$-very regular weight and let $\kappa_p$ be the corresponding $p$-adic weight of $T_H(\ZZ_p)$, as in \eqref{eq:relation kappa kappa p}. Let $I_\kappa$ be the kernel of the homomorphism $\Lambda_P^\circ \to R \subset \CC_p$ induced by the restriction of $\kappa_p$ to $Z_P^\circ$. Then, the natural homomorphism
        \[
            \bTT_{K_p, [\kappa, \tau], R}^\Pord
                \otimes
            \Lambda_R^\circ / I_\kappa
                \to
            \bTT_{K_r, \kappa, [\tau], R}^\Pord
        \]
        is an isomorphism.
    \end{enumerate}
\end{conjecture}

\begin{remark} \label{rmk:Hida VCT}
    Again, when $P = B$ as in Remark \ref{rmk:trivial partition}, this result holds and is due to Hida, see \cite[Theorem 7.2.1]{EHLS}.
\end{remark}

Now, let $\pi$ be cuspidal automorphic representation of $G = G_1$. Assume that $\pi$ is anti-holomorphic and $P$-anti-ordinary of anti-$P$-WLT $(\kappa, K_r, \tau)$. In particular, $\pi^\flat$ is holomorphic $P$-ordinary on $G_1$, or equivalently, anti-holomorphic $P$-anti-ordinary on $G_2$. In what follows, we work with $R = \OO_\pi$ and set $\Lambda_\pi := \Lambda_{\OO_\pi} = \OO_\pi[[Z_P]]$, $\Lambda^\circ_\pi := \Lambda^\circ_{\OO_\pi} = \OO_\pi[[Z_P^\circ]]$.

Let $\lambda_\pi$ be the classical character, see Definition \ref{def:classical arithmetic caracters}, of the $\Lambda_\pi$-algebra $\bTT_{K^p, [\kappa, \tau], \OO_\pi}$ associated to $\pi$ as in Section \ref{subsec:lattices of aholo Paord forms}.

Denote the localization of $\bTT^{\Pord}_{K^p, [\kappa, \tau], \OO_\pi}$ at $\m_\pi$ by $\TT = \TT_\pi$. Similarly, denote the localization of $\bTT^{\Pord}_{K_r, \kappa, [\tau], \OO_\pi}$ at $\m_\pi$ by $\TT_{K_r, \kappa, [\tau], \OO_\pi}$. We do not include the superscript ``$\Pord$'' in the notation of the localized Hecke algebras $\TT_?$ as we do not ever consider such localization of ``non-$P$-ordinary'' Hecke algebras in what follows.

\begin{proposition} \label{prop:VCT for big Hecke alg}
    Assume Conjectures \ref{conj:indep of weight of padic Hecke alg} and \ref{conj:VCT}. Then,
    \begin{enumerate}
        \item The Hecke algebra $\TT$ is finite, free over $\Lambda_\pi^\circ$.
        \item Let $\kappa$ be a very regular weight and let $\kappa_p$ be the corresponding $p$-adic weight of $T_H(\ZZ_p)$, as in \eqref{eq:relation kappa kappa p}. Let $I_\kappa$ be the kernel of the homomorphism $\Lambda_P^\circ \to \OO_\pi \subset \CC_p$ induced by the restriction of $\kappa_p$ to $Z_P^\circ$. Then, the natural homomorphism
        \[
            \TT \otimes \Lambda_R^\circ/I_\kappa
                \to
            \TT_{K_r, \kappa, [\tau], \OO_\pi}
        \]
        is an isomorphism.
    \end{enumerate}
\end{proposition}

\begin{definition} \label{def:P-anti-ord Hida family of pi}
    The representation $\pi$, or more precisely the homomorphism $\lambda_\pi$, corresponds to an $\OO_\pi$-valued point $\spec \TT_\pi$. We refer to $\TT_\pi$ as a $P$-anti-ordinary Hida family associated to $\pi$.
\end{definition}

\begin{remark}
    Note that this Hida family is not an irreducible component of $\TT_\pi$. The $p$-adic $L$-function constructed in Section \ref{sec:pairing} is well-defined on all of $\spec \TT_\pi$. This (connected) space is implicitly a branch corresponding to our choice of SZ-type $\tau$ associated to $\pi$. However, the choice of $\tau$ does not affect the $p$-adic interpolation formula of the $p$-adic $L$-function constructed in this paper. Namely, the reader should note that expression at the end of Theorem \ref{thm:main thm} does not depend on $\tau$. 
\end{remark}

Let $\pi'$ be an anti-holomorphic, $P$-anti-ordinary cuspidal automorphic representation of $G = G_1$ of anti-$P$-WLT $(\kappa', K_{r'}, \tau')$. Assume that $[\kappa'] = [\kappa]$ and $[\tau'] = [\tau]$. 

The canonical isomorphism provided by Conjecture \ref{conj:indep of weight of padic Hecke alg} identifies the maximal ideal $\m_{\pi'}$ associated to $\pi'$ as a maximal ideal of $\TT_{K^p, [\kappa, \tau], R}$, for some $\OO_\pi$-algebra $R$. This allows us to generalize the set $\SSS(K_r, \kappa, [\tau], \pi)$ defined at the end of Section \ref{subsec:lattices of Pord holo forms} for other levels and weights, i.e. let 
\begin{equation} \label{eq:def SSS K r' kappa' tau pi}
    \SSS(K_{r'}, \kappa', [\tau], \pi)
        :=
    \{
        \pi' \text{ as above such that } \m_{\pi'} = \m_{\pi}
    \}\,.
\end{equation}

Similarly, let
\begin{equation} \label{eq:def SSS Kp pi}
    \SSS(K^p, \pi) 
        = 
    \SSS(K^p, [\kappa], [\tau], \pi)
        :=
    \bigcup_{r \geq 1} \bigcup_{\kappa' \in [\kappa]}
        \SSS(K_{r'}, \kappa', [\tau], \pi)\,,
\end{equation}
hence a classical character corresponds to a point $\lambda = \lambda_{\pi'}$ of $\spec \TT_\pi$, for some representation $\pi' \in \SSS(K^p, \pi)$.

The image of $\pi' \in \SSS(K^p, \pi)$ in $\WWW$ is $\w_{\tau \otimes \psi} \cdot (\kappa_p + \theta_p)$. Implicitly, in what follows, we view $\pi$ as a choice of ``base point'' and $\pi'$ as a ``shift'' from $\pi$ by $\psi \cdot \theta_p$. 

\begin{remark} \label{rmk:base point pi and shift by psi theta}
    Naturally, our constructions in the following sections do not depend on the choice of a base point. However, this perspective of ``shifting'' (or ``twisting'') $\pi$ by $\psi \cdot \theta_p$ is useful to understand the construction of the $P$-ordinary Eisenstein measure, see Proposition \ref{prop:existence of Eis measure on V3}.
\end{remark}

\subsubsection{$P$-anti-ordinary vectors and minimal ramification.} \label{subsubsec:Paord vectors and minimal ramification}
In what follows, we view 
\[
    \wh{S}_\kappa^{\Paord}(K_r, [\tau]; \OO_\pi)_\pi 
        := 
    \hom_{\OO_\pi}(S^\Paord_\kappa(K_r, [\tau]; \OO_\pi), \OO_\pi)_{\m_\pi}
\]
as a module over $\TT_{K_r, \kappa, [\tau], \OO_\pi}$. Similarly, we view
\[
    \wh{S}_\kappa^{\Paord}(K^p, [\tau]; \OO_\pi)_\pi
        := 
    \varprojlim_r 
        \wh{S}_\kappa^{\Pord}(K_r, [\tau]; \OO_\pi)_\pi
\]
as a $\TT$-module. 

\begin{hypothesis}[Gorenstein Hypothesis] \label{hyp:Gorenstein hyp}
    Let $\wh{\TT}$ denote the $\Lambda_\pi^\circ$-dual of $\TT$.
    \begin{enumerate}
        \item The $\TT$-module $\wh{\TT}$ is free of rank one. Fix an isomorphism $G_\pi : \TT \xrightarrow{\sim} \wh{\TT}$ of $\TT$-modules.
        \item The $\TT$-module $\wh{S}_\kappa^{\Paord}(K^p, [\tau]; \OO_\pi)_\pi$ is finite, free.
    \end{enumerate}
\end{hypothesis}

From now on, we always assume that the Gorenstein hypothesis above holds. We fix any $\TT$-basis of $\wh{S}_\kappa^{\Paord}(K^p, [\tau]; \OO_\pi)_\pi$ and let $\wh{I}_\pi$ denote the $\OO_\pi$-lattice spanned by this basis. In particular, we have an isomorphism
\[
    \TT \otimes_{\OO_\pi} \wh{I}_\pi
        \xrightarrow{\sim}
    \wh{S}_\kappa^{\Paord}(K^p, [\tau]; \OO_\pi)_\pi\,.
\]

Assume that the weight $\kappa$ of $\pi$ is very regular. Then, the vertical control theorem Proposition \ref{prop:VCT for big Hecke alg} (ii) implies that taking tensor with $\Lambda_\pi^\circ / I_{\kappa}$ on both sides yields an isomorphism
\begin{equation} \label{eq:basis wh I pi iso}
    \TT_{K_r, \kappa, [\tau], \OO_\pi}
        \otimes
    \wh{I}_\pi
        \xrightarrow{\sim}
    \wh{S}_\kappa^{\Paord}(K_r, [\tau]; \OO_\pi)_\pi\,.
\end{equation}

Similarly, the $\lambda_\pi$-isotypic component $\TT_{K_r, \kappa, [\tau], \OO_\pi}[\lambda_\pi] = \bTT^\Pord_{K_r, \kappa, \tau, \OO_\pi}[\lambda_\pi]$ is free of rank 1 over $\OO_\pi$, by the multiplicity one hypothesis \ref{hyp:mult one hyp (I)}. Hence, the identification \eqref{eq:basis wh I pi iso} also induces an isomorphism
\[
    \wh{I}_\pi 
        \xrightarrow{\sim}
    (
        \TT_{K_r, \kappa, [\tau], \OO_\pi}
            \otimes
        \wh{I}_\pi
    )[\lambda_\pi]
        \xrightarrow{\sim}
    \wh{S}_\kappa^{\Paord}(K_r, [\tau]; \OO_\pi)[\lambda_\pi] \,,
\]
and note that the last term is equal to $\wh{S}_\kappa^{\Paord}(K_r, \tau; \OO_\pi)[\lambda_\pi]$.

Therefore, the isomorphism $j_\pi^\flat$ from Lemma \ref{lma:jpi flat iso} (i) induces an embedding
\begin{equation} \label{eq:lattice wh I pi identification}
    \wh{I}_\pi
        \xrightarrow{\sim}
    \wh{S}_\kappa^{\Paord}(K_r, \tau; \OO_\pi)[\lambda_\pi]
        \xhookrightarrow{(j_\pi^\flat)^{-1}}
    \hom_{\LL_r}(\tau^\flat, \pi_p^{\flat, (\Paord, r)})
        \otimes 
    \pi_S^{\flat, K_S}
        =
    \pi_S^{\flat, K_S}\,.
\end{equation}

For the dual picture, we map both sides of \eqref{eq:basis wh I pi iso} to their quotients modulo $\ker(\lambda_\pi)$ and obtain
\[
    \wh{I}_\pi
        \xrightarrow{\sim}
    (
        \TT_{K_r, \kappa, [\tau], \OO_\pi}/\ker(\lambda_\pi)
    )
        \otimes
    \wh{I}_\pi
        \xrightarrow{\sim}
    \hom_{\OO_\pi}(
        S_\kappa^{\Pord}(
            K_r, \tau; \OO_\pi
        )[\lambda_\pi],
        \OO_\pi
    )
    \,.
\]

Define $I_\pi$ as the $\OO_\pi$-dual of $\wh{I}_\pi$. Then the above, together with the isomorphism $j_\pi$ from Lemma \ref{lma:jpi iso (I)} (i), induces an embedding
\begin{equation} \label{eq:lattice I pi identification}
    I_\pi
        \xrightarrow{\sim}
    S_\kappa^{\Pord}(
        K_r, \tau; \OO_\pi
    )[\lambda_\pi]
        \xhookrightarrow{j_\pi^{-1}}
    \hom_{\LL_r}(\tau, \pi_p^{(\Pord, r)})
        \otimes 
    \pi_S^{K_S}
        =
    \pi_S^{K_S} \,.
\end{equation}

\begin{remark} \label{rmk:I pi and wh I pi as pi varies}
    Note that $I_\pi$ and $\wh{I}_\pi$ only depend on $\m_\pi$. Therefore, the embedding \eqref{eq:lattice wh I pi identification} (resp. \eqref{eq:lattice I pi identification}) identifies a lattice of $P$-anti-ordinary (resp. $P$-ordinary) anti-holomorphic (resp. holomorphic) automorphic forms shared by all $(\pi')^\flat$ (resp. $\pi'$) such that $\m_\pi = \m_{\pi'}$ as a maximal ideal of $\bTT^\Pord_{K^p, [\kappa, \tau], \OO_{\pi}}$. 
    
    In other words, the embedding \eqref{eq:lattice wh I pi identification} (resp. \eqref{eq:lattice I pi identification}) obtained by assuming the Gorenstein Hypothesis \ref{hyp:Gorenstein hyp} implies that the $\CC$-dimension of local representations at ramified places of all $(\pi')^\flat$ (resp. $\pi'$) as above is constant. This can therefore be viewed as a certain \emph{minimality hypothesis} on the behavior over the ramified places of the $P$-ordinary Hida family associated to $\pi$.
\end{remark}

The discussion above, together with Remark \ref{rmk:I pi and wh I pi as pi varies}, proves the following proposition (the analogue of \cite[Proposition 7.3.5]{EHLS} in the context of $P$-ordinary representations).

\begin{proposition} \label{prop:minimality hypothesis}
    Let $\pi$, $r$, $\kappa$ and $\tau$ be as above. Let $\pi' \in \SSS(\pi, \kappa', K_{r'}, [\tau])$, for some $r' \geq 1$ and some very regular weight $\kappa'$ such that $[\kappa] = [\kappa']$.

    There is an isomorphism of $\TT_{K_{r'}, \kappa', [\tau], \OO_{\pi}}$-module
    \begin{equation} \label{basis wh I pi' iso}
        \TT_{K_{r'}, \kappa', [\tau], \OO_{\pi}} \otimes \wh{I}_\pi
            \xrightarrow{\sim}
        \wh{S}_{\kappa'}^{\Paord}(
            K_{r'}, [\tau]; \OO_\pi
        )_\pi
    \end{equation}
    such that for $r'' \geq r'$, the ``change-of-level'' diagram
    \[
    \begin{tikzcd}
        \TT_{K_{r''}, \kappa', [\tau], \OO_{\pi}} \otimes \wh{I}_\pi
            \arrow[r, "\sim"]
            \arrow[d] &
        \wh{S}_{\kappa'}^{\Paord}(
            K_{r''}, [\tau]; \OO_\pi
        )_\pi
            \arrow[d] \\
        \TT_{K_{r'}, \kappa', [\tau], \OO_{\pi}} \otimes \wh{I}_\pi
            \arrow[r, "\sim"] &
        \wh{S}_{\kappa'}^{\Paord}(
            K_{r'}, [\tau]; \OO_\pi
        )_\pi
    \end{tikzcd}
    \]
    commutes. 
    
    Furthermore, tensoring \eqref{basis wh I pi' iso} with $\TT_{K_{r'}, \kappa', [\tau], \OO_\pi}/\ker(\lambda_{\pi'})$ over $\TT_{K_{r'}, \kappa', [\tau], \OO_\pi}$, i.e. specializing this isomorphism at the $\OO_\pi$-valued point $\lambda_{\pi'}$ of $\TT_{K_{r'}, \kappa', [\tau], \OO_\pi}$ corresponding to $\pi'$, yields the commutative diagram
    \[
    \begin{tikzcd}
        (
            \TT_{K_{r'}, \kappa', [\tau], \OO_\pi}/\ker(\lambda_{\pi'})
        )
            \otimes
        \wh{I}_\pi
            \arrow[r, "\sim"]
            \arrow[d, "="] 
                &
        \hom_{\OO_\pi}(
            S_{\kappa'}^{\Pord}(
                K_{r'}, \tau; \OO_\pi
            )[\lambda_{\pi'}],
            \OO_\pi
        ) 
             \arrow[d] 
                \\
        \OO_\pi \otimes_{\OO_\pi} \wh{I}_\pi
            \arrow[r, "\cong"] 
                &
        \hom_{\OO_\pi}(I_\pi, \OO_\pi)
    \end{tikzcd}\,,
    \]
    where the bottom map is the tautological identification of $\wh{I}_\pi$ as the $\OO_\pi$-dual of $I_\pi$.
\end{proposition}

Observe that all of this section can be rephrased for $G_2$. Namely, we can rewrite all of the above for $\pi^\flat$ as an anti-holomorphic $P$-anti-ordinary automorphic representation on $G_2$. Then, considering the analogue of Hypothesis \ref{hyp:Gorenstein hyp}, we similarly obtain an isomorphism
\begin{equation}  \label{eq:basis wh I pi flat iso}
    \TT_{\pi^\flat} \otimes_{\OO_\pi} \wh{I}_{\pi^\flat}
        \xrightarrow{\sim}
    \wh{S}_{\kappa^\flat, -V}^{\Paord}(K^{\flat, p}, [\tau^\flat]; \OO_\pi)_{\pi^\flat}\,,
\end{equation}
of finite free $\TT_{\pi^\flat}$-modules. Considering \eqref{eq:iso holo flat on G2 with anti holo on G1} and the isomorphism $\TT_{\pi^\flat} \cong \TT_{\pi}$ induced from the second part of Lemma \eqref{lma:comp Pord and Paord hecke alg}, we naturally identify $\wh{I}_{\pi^\flat}$ with $I_\pi$.

\subsubsection{$I_\pi$ and test vectors} \label{subsubsec:I pi and test vectors}
Fix any $\varphi_S \in \pi_S^{K_S}$ and $\varphi_S^\flat \in \pi_S^{\flat, K_S}$ such that $j_\pi(\varphi_S) \in I_\pi$ and $j_\pi^\flat(\varphi_S^\flat) \in \wh{I}_\pi$. Furthermore, let $\varphi_{l, 0} \in \pi_l$ and $\varphi^\flat_{l, 0} \in \pi^\flat_l$ be local test vectors at $l$ for all finite places $l \notin S \cup \{p\}$ of $\QQ$ as well as $\varphi_\infty \in \pi_\infty$ and $\varphi_\infty^\flat$ be local test vector at $\infty$, as in Section \ref{sec:explicit choice of P-(anti-)ord vectors}.

Fix a basis $\iota$ of $\hom_{L_P}(\tau, \pi_p^{(\Paord, r)})$ and a basis $\iota^\flat$ of $\hom_{L_P}(\tau^\flat, \pi_p^{\flat, (\Paord, r)})$. For any $v \in \tau$ and $v^\flat \in \tau^\flat$, let $\varphi_{p, v} = \iota(v)$ and $\varphi_{p, v^\flat}^\flat = \iota^\flat(v^\flat)$. By definition,
\begin{equation} \label{eq:test vector for pi determined by I pi and iota}
    \varphi 
        = 
    \left( 
        \bigotimes_{l \notin S \cup \{p\}} \varphi_{l, 0} 
    \right)
        \otimes
    \varphi_{p,v}
        \otimes
    \varphi_\infty
        \otimes
    \varphi_S
\end{equation}
and
\begin{equation} \label{eq:test vector for pi flat determined by I pi flat and iota flat}
    \varphi^\flat
        = 
    \left( 
        \bigotimes_{l \notin S \cup \{p\}} \varphi_{l, 0}^\flat 
    \right)
        \otimes
    \varphi_{p,v}^\flat
        \otimes
    \varphi_\infty^\flat
        \otimes
    \varphi_S^\flat
\end{equation}
are test vectors of $\pi$ and $\pi^\flat$ respectively. By construction and \eqref{eq:inner prod Pwaord vectors at w}, the inner product between $\varphi$ and $\varphi^\flat$ only depend on the choice of $\varphi_S$ and $\varphi_S^\flat$, i.e.
\begin{equation} \label{eq:inner product global test vectors}
    \brkt{\varphi}{\varphi^\flat}
        =
    C \cdot \vol(I_{P,r}^0) \cdot \brkt{\varphi_S}{\varphi^\flat_S}_S\,,
\end{equation}
where $C$ is the constant from \eqref{eq:facto inner product pi} and $\brktdotdot_S = \bigotimes_{l \in S} \brktdotdot_{\pi_l}$.

By abuse of terminology, we still refer to $j_\pi(\varphi_S)$ and $j_{\pi}^\flat(\varphi_S^\flat)$ as ``test vectors'', leaving the choice of basis of $\hom_\LL(\tau, \pi_p^{(\Paord,r)}$ and $\hom_\LL(\tau^\flat, \pi_p^{\flat, (\Paord,r)}$ implicit.

Let $\pi' \in \SSS(K_r, \kappa, [\tau], \pi)$ be a $P$-anti-ordinary automorphic representation of $P$-anti-WLT $(\kappa, K_r, \tau')$. Using Remarks \ref{rmk:SZ types of twists and contragredient} and \ref{rmk:I pi and wh I pi as pi varies}, one readily sees that $\varphi_v$ and $\varphi_v^\flat$ similarly determine test vectors of $\pi'$ and $\pi'^{,\flat}$, which we again denote $\varphi_v$ and $\varphi_v^\flat$. 

This yields embeddings
\[
    I_{\pi^\flat} = \wh{I}_\pi \hookrightarrow (\pi'_S)^{K_S}
        \;\;\;\text{and}\;\;\;
    \wh{I}_{\pi^\flat} = I_\pi \hookrightarrow (\pi'_S)^{\flat, K_S}
\]
into the subspaces of test vectors. Therefore, using Proposition \ref{prop:minimality hypothesis}, we identify 
\[
    \wh{I}_\pi \otimes I_\pi
        =
    \ehom_{\OO_\pi}(\wh{I}_\pi)
        =
    \ehom_{\OO_\pi}(I_{\pi^\flat})
\]
as the space of test vectors in $\pi_S' \otimes \pi_S'^{,\flat}$, for all $\pi' \in \SSS(K^p, \pi)$.

%%%%%%%%%%%%%%%%%%%%%%%%%%%%%%%%%%%%%%%%%%%%%%%%%%%%%%%%%%%%%%%%%%

\part{$P$-ordinary family of Siegel Eisenstein series}\label{part:Pord family of Eis series}

\section{Siegel Eisenstein series for the doubling method.} \label{sec:Sgl Eis series for dbl method}
Given any number field $F/\QQ$, we write $\absv{\cdot}_F$ for the standard absolute value on $\AA^\times_F$ (instead of $\absv{\cdot}_{\AA_F}$). For $F = \QQ$, we keep writing $\AA$ for $\AA_{\QQ}$.

\subsection{Siegel Eisenstein series.} \label{subsec:Sgl Eis series}
\subsubsection{Siegel parabolic.} \label{subsubsec:def sgl parabolic}
Let $W = V \oplus V$, equipped with $\brkt{\cdot}{\cdot}_W := \brkt{\cdot}{\cdot}_V \oplus (- \brkt{\cdot}{\cdot}_V)$, be the Hermitian vector space associated to $G_4$. We work with $G_4$ for most of what follows, hence we set $G := G_4$ in all of Section \ref{sec:Sgl Eis series for dbl method}.

Consider the subspaces $V^d = \{(x,x) \in W : x \in V\}$ and $V_d = \{(x,-x) \in W : x \in V\}$. We identify both with $V$ via projection on their first factor. The direct sum $W = V_d \oplus V^d$ is a polarization of $\brkt{\cdot}{\cdot}_W$.

Let $P_\sgl \subset G$ denote the stabilizer of $V^d$ under the right-action of $G$, a maximal $\QQ$-parabolic subgroup. Let $M \subset P_\sgl$ denote the Levi subgroup that also stabilizes $V_d$. The unipotent radical of $P_\sgl$ is the subgroup $N$ that fixes both $V^d$ and $W/V^d$ and clearly, $P_{\sgl}/N \cong M$. Furthermore, there is a canonical identification $M \xrightarrow{\sim} \GL_{\KK}(V) \times \Gm$ via $m \mapsto (\Delta(m), \nu(m))$, where $\Delta$ is the projection
\[
     P_\sgl \to \GL_\KK(V^d) = \GL_\KK(V) \ ,
\]
whose inverse is given by 
$
    (A, \lambda) 
        \mapsto 
    \diag(\lambda (A^*)^{-1}, A)
$, where $A^* = \tp{A}^c$. 

\subsubsection{Induced Representations.} \label{subsubsec:Induced representations}
Let $\chi : \KK^\times \backslash \AA_{\KK}^\times \to \CC^\times$ be a unitary Hecke character. It factors as $\chi = \bigotimes \limits_{w} \chi_w$, where $w$ runs over all places of $\KK$. In later section, we assume that $\chi$ is of type $A_0$, i.e. we impose certain conditions on $\chi_\infty = \bigotimes \limits_{w \mid \infty} \chi_w$.

For convenience, define the character $\nabla$ of $P_\sgl(\AA)$ as
\[
    \nabla(-) = 
    \absv{ 
        \nm_{\KK/\KK^+} \circ 
        \det \circ 
        \Delta(-)
    }_{\KK^+} \cdot
    \absv{
        \nu(-)
    }_{\KK^+}^{-n} = 
    \absv{ 
        \det \circ 
        \Delta(-)
    }_\KK \cdot
    \absv{
        \nu(-)
    }_\KK^{-n/2} \ ,
\]
where $\nm_{\KK/E}$ is the usual norm homomorphism $\AA_{\KK} \to \AA_E$. One readily checks that $G_1(\AA)$, via its natural diagonal inclusion in $G_4(\AA)$, is in the kernel of $\nabla$. Moreover, the modulus character $\delta_{\sgl}$ of $P_\sgl(\AA_\QQ)$ equals $\nabla^{n}$.

Let $s \in \CC$, and define the smooth and normalized induction
\begin{equation} \label{eq:definition I chi s normalized}
    I(\chi, s) 
    = 
    \ind{P_\sgl(\AA)}{G(\AA)}
        \left(
            \chi
            \left(
                \det \circ \Delta(-)
            \right)
                \cdot
            \nabla(-)^{-s}
        \right)\,.
\end{equation}

This degenerate principal series is identical to the one in \cite[Section 4.1.2]{EHLS}. It is also equal to the smooth, unnormalized parabolic induction
\begin{equation} \label{eq:definition I chi s unnormalized}
    I(\chi, s) =
    \Ind_{P_\sgl(\AA)}^{G(\AA)}
        \left(
            \chi
            \left(
                \det \circ \Delta(-)
            \right) 
                \cdot
            \nabla(-)^{-s - \frac{n}{2}}
        \right)\,,
\end{equation}
and factors as a restricted tensor product of local induced representations
\[
    I(\chi, s) = \bigotimes_{v} I_v(\chi_v, s)\,,
\]
where $v$ runs over all places of $\QQ$ and $\chi_v = \bigotimes \limits_{w | v} \chi_w$. The definition of $I_v(\chi_v, s)$ is given by the obvious local analogue of \eqref{eq:definition I chi s unnormalized} at $v$.

\begin{remark} \label{rmk:comp s sE and sEL}
    To compare with results in \cite{Eis15} and \cite{EisLiu20}, let us write $s_E$ and $s_{EL}$ for the variable $s$ appearing in the unnormalized parabolic induction functor for these articles respectively. Then, the relations with our variable $s$ are $s_E = s + \frac{n}{2}$ and $s_{EL} = -s$.
\end{remark}

\subsubsection{Siegel-Weil sections and Eisenstein series.} \label{subsubsec:SW sections and Eis series}
Given a Siegel-Weil section $f = f_{\chi, s}$ of $I(\chi, s)$, one constructs the \emph{standard (nonnormalized)} Eisenstein series
\begin{equation} \label{def Eis series}
    E_f(g) = 
        \sum_{
            \gamma \in P(\QQ) \backslash G(\QQ)
        } 
        f(\gamma g)
\end{equation}
as a function on $G(\AA)$. It converges on the half-plane Re$(s) > n/2$ and if $f$ is right-$K$-finite, for some maximal compact open subgroup $K \subset G$, it admits a meromorphic continuation on $\CC$.

\begin{remark} \label{rmk:local SW sections and comparison to EHLS}
    In the following sections, we choose explicit $f_v \in I_v(\chi_v, s)$ for each place $v$ of $\QQ$. Our choices are parallel to the ones in \cite[Section 4]{EHLS} and are standard in the literature, especially at finite unramified places and at archimedean places.
    
    However, our choice of section at $p$ requires several adjustments to construct a Siegel Eisenstein series that interpolates properly $p$-adically along a $P$-ordinary family. 

    The main difference is that the locally constant function $\mu$ in \cite[Section 4.3.1]{EHLS}, which is essentially the nebentypus of an ordinary cuspidal automorphic representation, is replaced by a (matrix coefficient of a) type of a $P$-ordinary cuspidal automorphic representations.
\end{remark}

\subsubsection{Zeta integrals.} \label{subsubsec:zeta integrals}
Let $f = f_{\chi, s} \in I(\chi, s)$. Let $\pi$ be \emph{any} cuspidal automorphic representation for $G_1$ and let $\varphi \in \pi$ and $\varphi^\vee \in \pi^\vee$ be any vectors. The doubling method consists of relating the Rankin-Selberg integral
\[
    I(\varphi, \varphi^\vee, f; \chi, s) := 
    \int_{
        Z_3(\AA)G_3(\QQ) \backslash G_3(\AA)
    } 
        E_f(g_1, g_2)
        \varphi(g_1)
        \varphi^\vee(g_2)
        \chi^{-1}(\det g_2)
    d(g_1,g_2)
\]
and relate it to the $L$-function associated to $\pi$ and $\chi$. In Section \ref{sec:pairing}, we reinterpret this integral algebraically as a pairing between a holomorphic modular form on $G_3$ and a anti-holomorphic cusp form on $G_3$. To do so, as explained in \cite{GPSR87}, we use that for Re$(s)$ large enough,
\[
    I(\varphi, \cg{\varphi}, f; \chi, s) 
        = 
    \int_{U_1(\AA)} 
        f_{\chi, s}(u, 1)
        \brkt{
            \pi(u) \varphi
        }{
            \varphi^\vee
        }_\pi
    du\,.
\]

In the following sections, we choose some $f$ for which this can be done. More precisely, we construct $f$ as a pure tensor $f = \bigotimes_l f_l$ over all places $l$ of $\QQ$. Assuming that $\pi$ is $P$-anti-ordinary of $P$-anti-WLT $(\kappa, K_r, \tau)$, we construct these local Siegel-Weil sections so that $f_p$ depends on $\chi_p$ and $\tau$, $f_\infty$ depends on $\chi_\infty$ and $\kappa$, and for all finite prime $l$ away from $p$, $f_l$ depends on $K_r^p$.

Assume $\varphi$ and $\varphi^\vee$ are ``pure tensors'', i.e. $\varphi = \bigotimes_l \varphi_l$ and $\varphi^\vee = \bigotimes_l \varphi^\vee_l$ according to the factorization \eqref{eq:facto pi f}, e.g. $\varphi$ and $\varphi^\vee$ are test vectors as in Section \ref{sec:explicit choice of P-(anti-)ord vectors}. 

Then
\[  
    I(\varphi, \varphi^\vee, f; \chi, s) 
        = 
    \prod_l
        I_l(\varphi_l, \varphi^\vee_l, f_l; \chi_l, s) \cdot  
        \brkt{\varphi}{\varphi^\vee}\,,
\]
where
\begin{equation} \label{eq:def local zeta integrals}
    I_l(\varphi_l, \varphi^\vee_l, f_l; \chi_l, s) 
        =
    \frac{
        \int_{U_{1,l}} 
            f_{\chi, s, l}(u, 1)
            \brkt{
                \pi_l(u) \varphi_l
            }{
                \varphi^\vee_l
            }_{\pi_l}
        du
    }{
        \brkt{\varphi_l}{\varphi^\vee_l}_{\pi_l}
    }\,,
\end{equation}
for any place $l$ of $\QQ$. Let $Z_l$ denote the numerator of the fraction on the right-hand side of \eqref{eq:def local zeta integrals}. We compute each zeta integral $Z_l$ individually (by factoring it over places of $\KK^+$ above $l$), in Section \ref{sec:doubling method integral}.

\subsection{Local Siegel-Weil section at $p$.} \label{subsec:Local SW section at p}
For each places $w \in \Sigma_p$ of $\KK$, fix an isomorphism $\KK_w = \KK_{\overline{w}}$. Then, the identification \eqref{eq:prod G over Zp} for $G_4$ induces an identification of $P_\sgl(\QQ_p)$ with $\QQ_p^\times \times \prod_{w \in \Sigma_p} P_n(\KK_w)$, where $P_n \subset \GL_{\KK}(W)$ is the parabolic subgroup stabilizing $V^d$.

Let $\chi_p = \otimes_{w \mid p} \chi_w$ and, given $s \in \CC$, view $\chi_p \cdot \absv{-}_p^{-s}$ as a character of $P_\sgl(\QQ_p)$. One readily checks that its restriction to $\prod_{w \in \Sigma_p} P_n(\KK_w)$ corresponds to the product over $w \in \Sigma_p$ of the characters $\psi_{w,s} : P_n(\KK_w) \to \CC^\times$ defined as
\[
    \psi_{w,s}
    \left( 
        \begin{pmatrix}
            A & B \\
            0 & D
        \end{pmatrix}
    \right)
        =
    \chi_w(\det D) \chi_{\ol{w}}(\det A^{-1}) \cdot \absv{\det A^{-1}D }_w^{-s} \ ,
\]
by writing element of $P_n$ according to the direct sum decomposition $W = V_d \oplus V^d$.

Let $W_w = W \otimes_{\KK} \KK_w$ and choose any $f_{w,s} \in \ind{P_n(\KK_w)}{\GL_{\KK_w}(W_w)} \psi_{w,s}$, for each $w \in \Sigma$. Then, it is clear that the section
\begin{equation} \label{eq:def SW f p}
    f_p(g) = f_{p, \chi, s}(g) := 
        \absv{\nu}_p^{(s + \frac{n}{2})\frac{n}{2}} 
        \prod_{w \in \Sigma_p} f_{w,s}(g_w) \,, \;\;\ g = (\nu, (g_w)_w) \in G(\QQ_p)\,,
\end{equation}
is in $I_p(\chi_p, s)$.

\begin{remark} \label{rmk:f w s from Schwartz functions}
    The strategy below is to construct such $f_{w,s} = f_{w,s}^{\Phi_w}$, and hence $f_p = f_p$, from a specific Schwartz function $\Phi_w = \Phi_w^{\tau_w}$ (that depends on the type $\tau_w$), see \eqref{eq:def f phi w}. This approach is already used in \cite[Section 2.2.8]{Eis15} and \cite[Section 4.3.1]{EHLS}. In fact, our argument owes a great deal to their work and the details they carefully provide. 
    
    The novelty here is that we associate Schwartz functions to finite dimensional representations (namely the SZ types from Section \ref{subsubsec:SZ types}), instead of characters.
\end{remark}

\subsubsection{Locally Constant Matrix Coefficients.} \label{subsubsec:Locally constant matrix coeff}
In what follows, we use the notation of Section \ref{subsec:Local test vectors at p} freely. Let $\chi_{w,1} := \chi_w$ and $\chi_{w,2} := \chi_{\ol{w}}^{-1}$. Increasing the level $r$ of the SZ-types $\tau$ and $\tau^\vee$ at $p$ if necessary, we assume that the following inequality holds :
\begin{equation} \label{eq:large enough r}
	r \geq 
	\max
	(
		1, 
		\ord_w(\cond(\chi_{w,1})) , 
		\ord_w(\cond(\chi_{w,2}))
	)\,,
\end{equation}
for each $w \in \Sigma_p$. In what follows, we consider $\chi_{w,1}$ and $\chi_{w,2}$ as characters of general linear groups of any rank via composition with the determinant without comment.

Let $\mu'_{w,j} : K_{w,j} \to \CC$ be the matrix coefficient defined as
\[
    \mu'_{w,j}(X) = 
    \begin{cases}
        \langle 
            \phi_{w,j},
            \tau^\vee_{w,j}(X)
            \phi^\vee_{w,j} 
        \rangle_{\sigma_{w,j}}, & 
        \text{if $j=1, \ldots, t_w$}, \\
        \langle 
            \tau_{w,j}(X)
            \phi_{w,j},
            \phi^\vee_{w,j} 
        \rangle_{\sigma_{w,j}}, & 
        \text{if $j=t_w+1, \ldots, r_w$}.
    \end{cases}
\]

\begin{remark} \label{rmk:dependence of matrix coefficients on local vectors}
    We do not make the choice of $\phi_{w,j}$ and $\phi_{w,j}^\vee$ (see \eqref{eq:def phi a w 0 and phi b w 0} and \eqref{eq:def phi a w vee 0 and phi b w vee 0}) explicit in our notation for $\mu_{w,j}'$. See Remark \ref{rmk:psi w locally constant} below for further details.
\end{remark}

The products $\mu'_{a_w} = \bigotimes_{j=1}^{t_w} \mu'_{w,j}$ and $\mu'_{b_w} = \bigotimes_{j=t_w+1}^{r_w} \mu'_{w,j}$ on $K_{a_w}$ and $K_{b_w}$, respectively, are the matrix coefficients
\[
    \mu'_{a_w}(X) = 
    (
        \phi_{a_w}^0, 
        \tau^\vee_{a_w}(X) 
        \phi_{a_w}^{\vee, 0} 
    )_{a_w}
        \ \ \ ; \ \ \
    \mu'_{b_w}(X) = 
    ( 
        \tau_{b_w}(X)
        \phi_{b_w}^0, 
        \phi_{b_w}^{\vee, 0}
    )_{b_w}
\]
of $\cg{\tau}_{a_w}$ and $\tau_{b_w}$ respectively. 

We now consider $\mu'_{a_w}$ as a locally constant function on $M_{a_w}(\KK_w)$ supported on $\X_w^{(1)} := \tp{I}_{a_w, r}^0 I_{a_w, r}^0$. More precisely, one readily verifies that given $X \in \X_w^{(1)}$ and any $\tp{\gamma}_1, \gamma_2 \in I_{a_w, r}^0$ such that $X = \gamma_1\gamma_2$, then
\begin{equation} \label{eq:def extn mu aw}
    \mu'_{a_w}(X) := 
    ( 
        \tau_{a_w}(\gamma_1^{-1})
        \phi_{a_w}^0,
        \tau^\vee_{a_w}(\gamma_2)
        \phi_{a_w}^{\vee, 0} 
    )_{a_w}
\end{equation}
is well-defined. Indeed, given $\tp{\gamma}_1', \gamma_2' \in I^0_{a_w, r}$ such that $X = \gamma_1\gamma_2 = \gamma_1' \gamma_2'$, we have
\begin{align*}
    ( 
        \tau_{a_w}(\gamma_1^{-1})
        \phi_{a_w}^0,
        \tau^\vee_{a_w}(\gamma_2)
        \phi_{a_w}^{\vee, 0}
    )_{a_w}
        &=
    ( 
        \tau_{a_w}(\gamma_2 (\gamma_2')^{-1})
        \tau_{a_w}((\gamma_1')^{-1})
        \phi_{a_w}^0,
        \tau^\vee_{a_w}(\gamma_2)
        \phi_{a_w}^{\vee, 0}
    )_{a_w} 
        \\ &=
    (
        \tau_{a_w}((\gamma_1')^{-1})
        \phi_{a_w}^0,
        \tau^\vee_{a_w}(\gamma_2' \gamma_2^{-1})
        \tau^\vee_{a_w}(\gamma_2)
        \phi_{a_w}^{\vee, 0}
    )_{a_w} 
        \\ &=
    (
        \tau_{a_w}((\gamma_1')^{-1})
        \phi_{a_w}^0,
        \tau^\vee_{a_w}(\gamma_2')
        \phi_{a_w}^{\vee, 0} 
    )_{a_w} 
\end{align*}
where the first and second equality holds since $\gamma_1^{-1} \gamma_1' = \gamma_2 (\gamma_2')^{-1}$ modulo $\p_w^r$ lies in $L_{a_w}(\OO_w/\p_w^r\OO_w)$.

Similarly, we extend $\mu'_{b_w}$ to a locally constant function on $M_{b_w}(\KK_w)$ supported on $\X_w^{(4)} := \tp{I}_{b_w, r}^0 I_{b_w, r}^0$ via
\begin{equation} \label{eq:def extn mu bw}
    \mu'_{b_w}(X) := 
    ( 
        \tau_{b_w}(\gamma_2)
        \phi_{b_w}^0,
        \tau^\vee_{b_w}(\gamma_1^{-1})
        \phi_{b_w}^{\vee, 0} 
    )_{b_w},
\end{equation}
where $X \in \X_w^{(4)}$ and $\tp{\gamma}_1, \gamma_2 \in I_{b_w, r}^0$ are any elements such that $X = \gamma_1\gamma_2$.

Let $\mu_{a_w}(A) := \chi_{2,w}^{-1}\mu'_{a_w}$ and $\mu_{b_w} := \chi_{1,w}\mu'_{b_w}$. Let $\X_w \subset M_n(\OO_w)$ be the set of matrices 
$
    \begin{pmatrix}
        A & B \\
        C & D
    \end{pmatrix}
$
such that $A \in \X_w^{(1)}$, $B \in M_{a_w \times b_w}(\OO_w)$, $C \in M_{b_w \times a_w}(\OO_w)$ and $D \in \X_w^{(4)}$. 

We define a locally constant function $\mu_w$ on $M_{n}(\KK_w)$ supported on $\X$ via
\begin{equation} \label{eq:def ext mu w}
    \mu_w\left(
        \begin{pmatrix}
            A & B \\
            C & D
        \end{pmatrix}
    \right)
    = \mu_{a_w}(A)\mu_{b_w}(D) \ ,
\end{equation}
for all 
$
    \begin{pmatrix}
        A & B \\
        C & D
    \end{pmatrix}
    \in \X_w
$.

Observe that the set $\X_w$ contains the subgroup $\GG_w = \GG_w(r) \subset \GL_n(\OO_w)$ consisting of matrices whose terms below the $(n_{w,j} \times n_{w,j})$-blocks along the diagonal are in $\p_w^r$ and such that the upper right $(a_w \times b_w)$-block is also in $\p_w^r$. Similarly, let $\GG_{l,w} = \GG_{l,w}(r)$ (resp. $\GG_{u,w} = \GG_{u,w}(r)$) be the largest subgroup  of $\GL_n(\OO_w)$ such that $\GG_{l,w} \cap P_{\d_w}^u = 1$ (resp. $\GG_{l,w} \cap P_{\d_w}^{u, \opp} = 1$).

In particular, we have the natural decomposition $\GG_w = \GG_l (I^0_{a_w,r} \times I^0_{b_w,r}) \GG_u$. By abuse of notation, given $B \in M_{a_w \times b_w}(\KK_w)$ or $C \in M_{b_w \times a_w}(\KK_w)$, we sometimes write $B \in \mathfrak{G}_{u,w}$ or $C \in \mathfrak{G}_{l,w}$ when we mean
\[
	\begin{pmatrix}
		1 & B \\
		0 & 1
	\end{pmatrix} \in \mathfrak{G}_{u,w} \ \ \ \ \ \text{or} \ \ \ \ \ 
	\begin{pmatrix}
		1 & 0 \\
		C & 1
	\end{pmatrix} \in \mathfrak{G}_{l,w} \ .
\]

\subsubsection{Choice of Schwartz functions.} \label{subsubsec:Choice of Schwartz functions}
Let $\Phi_{1,w}: M_n(\KK_w) \to \CC$ be the locally constant function supported on $\GG_w$ such that 
\begin{equation} \label{eq:def Phi 1 w}
    \Phi_{1,w}(X) = 
    \mu_w(X)
\end{equation}
for all $X \in \GG_w$. Furthermore, define the locally constant functions
\begin{equation} \label{eq:def nu w and phi nu w}
    \nu_{\bullet}(z) = 
        \chi_{w,1}^{-1} \chi_{w,2} 
        \mu_{\bullet}(z)
    \ \ \ ; \ \ \
    \phi_{\nu_{\bullet}}(z) = 
        \nu_{\bullet}(-z) \ , 
\end{equation}
where $\bullet$ denotes $a_w$, $b_w$ or $w$, and $z$ is in the appropriate domain.

Let $\Phi_{2,w} : M_n(\QQ_p) \to \CC$ be 
\begin{equation} \label{eq:def Phi 2 w}
	\Phi_{2,w}(x) = (\nu_w)^\wedge(x) = \int_{M_n(\KK_w)} \phi_{\nu_w}(y) e_w(\tr(yx)) dy
\end{equation}

\begin{remark} \label{rmk:telescoping product}
    The definition of $\mu_w$ and its twist $\nu_w$ on $\X_w$ allows us to generalize the function denoted $\phi_{\nu_v}$ in \cite[Section 4.3.1]{EHLS}. In \emph{loc. cit.}, the SZ-types are all characters, in which case $\tau_\bullet$ is equal to $\mu'_\bullet$ (and there is no need to talk about types). The ``telescoping product'' in the definition of $\phi_{\nu_v}$ is simply a formula that expresses the extension of these characters to $I_{\bullet}^0$ and $\tp{I}_{\bullet}^0$ simultaneously. Our alternative is to use extensions of (matrix coefficients of) SZ-types such as in Equations \eqref{eq:def extn mu aw},  \eqref{eq:def extn mu bw} and \eqref{eq:def ext mu w}.
\end{remark}

\begin{remark} \label{rmk:Fourier transform conv for Phi 2 w}
	This Fourier transform in the definition of $\Phi_{2,w}$ is slightly different than the one in \cite[Section 2.2.8]{Eis15} and \cite[Section 4.3.1]{EHLS}. It is the same as the one involved in the Godement-Jacquet functional equation \cite{Jac79}.
\end{remark}

\begin{lemma} \label{lma:Phi 2 w support}
    Given 
    $X = 
        \begin{pmatrix} 
            A & B \\ 
            C & D 
        \end{pmatrix}$ 
    with $A \in M_{a_w \times a_w}(\KK_w)$, $B, \tp{C} \in M_{a_w \times b_w}(\KK_w)$ and $D \in M_{b_w \times b_w}(\KK_w)$, one can write
    \[
	\Phi_{2,w}(X) = 
            \Phi_w^{(1)}(A)
            \Phi_w^{(2)}(B)
            \Phi_w^{(3)}(C)
            \Phi_w^{(4)}(D)
    \]
    with
    \begin{align*}
	\Phi_w^{(2)} = 
            \Char_{M_{a_w \times b_w}(\OO_w)},
            &\hspace*{1cm} 
        \Phi_w^{(3)} = 
            \Char_{M_{b_w \times a_w}(\OO_w)}, \\
        \supp(\Phi_w^{(1)}) \subset 
            \p_w^{-r}M_{a_w \times a_w}(\OO_w), 
            &\hspace*{1cm} 
        \supp(\Phi_w^{(4)}) \subset 
            \p_w^{-r}M_{b_w \times b_w}(\OO_w),
    \end{align*}
    where $r$ is as in Inequality \eqref{eq:large enough r}.
\end{lemma}

\begin{proof}
    The definitions of $\phi_{\nu_{a_w}}$, $\phi_{\nu_{b_w}}$ and $\phi_{\nu_w}$ immediately imply
    \begin{align*}
        \Phi_{2,w}(X) 
        &= 
        \int_{\X_w} 
            \phi_{\nu_w}
            \left(
                \begin{pmatrix}
			     \alpha & \beta \\
			     \gamma & \delta
            \end{pmatrix}
            \right)
            e_w(
                \tr(
                    \alpha A + \beta B + \gamma C + \delta D
                )
            )
        d\alpha d\beta d\gamma d\delta \\
        &= 
        \int_{\X_w^{(1)}}
		\phi_{\nu_{a_w}}
            \left(
                \alpha
            \right)
            e_w(\tr(\alpha A))
        d\alpha 
        \int_{\X_w^{(4)}}
            \phi_{\nu_{b_w}}
            (
                \delta
            )
            e_w(\tr(\delta D))
        d\delta \\
        &\times 
        \Char_{M_{a_w \times b_w}(\OO_w)}(B) \Char_{M_{b_w \times a_w}(\OO_w)}(C) \\ 
    \end{align*}
	
    Then, we may conclude as in the proof of \cite[Lemma 4.3.2 (ii)]{EHLS} by observing
    \begin{align*}
	&\Phi_w^{(1)}(A) 
        := 
        \int_{\X_w^{(1)}}
		\phi_{\nu_{a_w}}
            \left(
                \alpha
            \right)
		e_w(\tr(\alpha A))
	d\alpha \\
	&= 
        \vol(\p_w^rM_{a_w}(\OO_w))
	\sum_{
            \alpha \in \X_w^{(1)} \text{ mod } \p_w^r
        }
    	\phi_{\nu_{a_w}}
                \left(
                    \alpha
                \right)
    	e_w(\tr{\alpha A}) 
    	\Char_{\p_w^{-r}M_{a_w}(\OO_w)}(A) \ ,
    \end{align*}
    and
    \begin{align*}
	&\Phi_w^{(4)}(D) 
        := 
        \int_{\X_w^{(4)}}
		\phi_{\nu_{b_w}}
            \left(
                \delta
            \right)
		e_w(\tr(\delta D))
	d\delta \\
	&= 
        \vol(\p_w^r M_{b_w}(\OO_w))
	\sum_{
            \delta \in \X_w^{(4)} \text{ mod } p^r
        }
    	\phi_{\nu_{b_w}}
            \left(
                \delta
            \right)
    	e_w(\tr{\delta D}) 
	   \Char_{\p_w^{-r}M_{b_w}(\OO_w)}(D) \ .
    \end{align*}
\end{proof}

Define the Schwartz function $\Phi_w : M_{n \times 2n}(\QQ_p) \to \CC$ as
\begin{equation} \label{eq:def Schwartz Phi w}
	\Phi_w(X) = \Phi_w(X_1, X_2) = \frac{\dim \tau_w}{\vol(\GG_w)} \Phi_{1,w}(-X_1) \Phi_{2,w}(X_2) \ .
\end{equation}

\begin{remark} \label{rmk:mu tau nu tau and Phi tau}
    In this section, the local type $\tau_w$ is fixed, hence we do not include it in our notation. However, in Section \ref{sec:Pord Eisenstein measure}, the type varies along a $P$-ordinary Hida family. Therefore, we write $\mu_w^{\tau_w}$, $\nu_w^{\tau_w}$, $\Phi_w^{\tau_w}$ and so on to emphasize the role of $\tau_w$.
\end{remark}

\subsubsection{Construction of $f_{w,s}$ for $w \in \Sigma_p$} \label{subsubsec:construction of f w s for w in Sigma p}
For each $w \in \Sigma_p$, write $V_w = V \otimes_{\KK} \KK_w$ and use similar notation for $V_{d, w}$ and $V^d_w$. Consider the decomposition
\[
    \hom_{\KK_w}(V_w, W_w) = \hom_{\KK_w}(V_w, V_{w,d}) \oplus \hom_{\KK_w}(V_w, V_w^d), \ \ \ X = (X_1, X_2) 
\]
and its subspace
\[
    \mathbf{X} 
        := 
    \{
        X \in \hom_{\KK_w}(V_w, W_w) \mid X(V_w) = V_w^d
    \}
        =
    \{
        (0,X) \mid X : V_w \xrightarrow{\sim} V_w^d
    \}.
\]

In fact, any $X \in \mathbf{X}$ can be viewed as an automorphism of $V_w$ (by composing with the identification of $V^d$ with $V$) and hence, we identify $\mathbf{X}$ with $\GL_{\KK_w}(V_w)$. Let $d^\times X$ be the Haar measure on the latter.

Furthermore, recall that we fixed an $\OO_w$-basis of $L_{1,w}$ in Section \ref{subsubsec:comp to gen linear groups}. This provides a $\KK_w$-basis of $V_w$ and, via their identification to $V$, a $\KK_w$-basis of $V_{d, w}$ and of $V^d_w$. Hence, it also induces a $\KK_w$-basis of $W_w = V_{d, w} \oplus V^d_w$.

It identifies $\isom(V_w^d, V_w)$ with $\isom(V_{w,d}, V_w)$, $\GL_{\KK_w}(V_w)$ with $\GL_n(\KK_w)$, $\GL_{\KK_w}(W_w)$ with $\GL_{2n}(\KK_w)$, $P_n(\KK_w)$ with the subgroup of $\GL_{2n}(\KK_w)$ consisting of upper-triangular $n \times n$-block matrices, and $\hom_{\KK_w}(V_w, W_w)$ with $M_{n \times 2n}(\QQ_p)$

Therefore, we now view the Schwartz function $\Phi_w : M_{n \times 2n}(\QQ_p) \to \CC$ constructed above as a function on $\hom_{\KK_w}(V_w, W_w)$. We define $f_{w,s} = f_{w,s}^{\Phi_w} = f^{\Phi_w}$, an element of $\ind{P_n(\KK_w)}{\GL_{2n}(\KK_w)} \psi_{w,s}$, as
\begin{equation} \label{eq:def f phi w}
    f^{\Phi_w}(g) = 
        \chi_{2,w}(g)
        \absv{\det g}_w^{\frac{n}{2} + s}
        \int_{\mathbf{X}}
            \Phi_w(Xg)
            \chi_{w,1}^{-1}
            \chi_{w,2}(X)
            |\det X|_w^{n+2s}
        d^\times X
    \ ,    
\end{equation}
as in \cite[Equation (55)]{EHLS}. To emphasize the role of $\chi_p$ and $\tau$ in the construction of the local Siegel-Weil section $f_p = f_{p, \chi, s}$ via \eqref{eq:def SW f p} obtained from $f^{\Phi_w}$ in \eqref{eq:def f phi w}, for each $w \in \Sigma_p$, we sometimes denote it by
\begin{equation} \label{eq:def SW f p tau chi p}
    f_p(\bullet) = f_p(\bullet; \tau, \chi_p, s)\,.
\end{equation}

\begin{remark} \label{rmk:psi w locally constant}
    In Section \ref{sec:Pord Eisenstein measure}, we consider the section $f_p(\bullet; \tau \otimes \psi, \chi_p, s)$ as $\tau$ remains fixed and $\psi = \otimes \psi_w$ varies over finite-order characters of $L_P(\ZZ_p)$. Note that this section depends on the choice of local vectors $\phi_{w,j}$ and $\phi_{w,j}^\vee$, for $w \in \Sigma_p$ and $1 \leq j \leq r_w$, with respect to $\tau \otimes \psi$ instead of $\tau$, see \eqref{eq:def phi a w 0 and phi b w 0} and \eqref{eq:def phi a w vee 0 and phi b w vee 0}.
    
    However, we do not make this dependence explicit in our notation. This is because the choice of such local vectors $\phi_{w,j}^?$ are fixed given any $\tau$, and our conventions ensure that the corresponding local vector for $\tau \otimes \psi$ are the local vectors $\phi_{w,j}^? \otimes 1$, i.e. essentially the ``same'' local vectors, see Remarks \ref{rmk:SZ types of twists and contragredient}, \ref{rmk:twists of local vectors in local types}, \ref{rmk:twists of local vectors in local types - contragredients}, and \ref{rmk:I pi and wh I pi as pi varies}. 
\end{remark}

\subsection{Local Siegel-Weil section at $\infty$.} \label{subsec:Local SW section at oo}

In what follows, we continue with the notation of Section \ref{subsec:Local test vectors at oo} (for $G = G_4$ instead of $G_1$). In particular, we have $G = G_4 \subset G^*$ and $G^*(\RR) = \prod_{\sigma} G_\sigma$, where $G_\sigma \cong \GU^+(n,n)$, identifying $\sigma \in \Sigma$ with its restriction in $\Sigma_{\KK^+} = \hom(\KK^+, \RR)$. Similarly, the homomorphism $h = h_4$ from the PEL datum $\PP_4$, valued in $G^*_{/\RR}$, factors as $h = \prod_\sigma h_\sigma$.

By fixing a basis of of $L_1$, we naturally obtain bases for $V_d$ and $V^d$ (via their identification with $V = L_1 \otimes \RR$) and $W = V_d \oplus V^d$. We use these to view $G_\sigma$ as a subgroup of $\GL_{2n}(\CC)$ and write 
$ 
    g_\sigma 
        = 
    \begin{pmatrix}
        a_\sigma & b_\sigma \\
        c_\sigma & d_\sigma
    \end{pmatrix}
$
where $a_\sigma$, $b_\sigma$, $c_\sigma$ and $d_\sigma$ are all $n \times n$-matrices. We always use this convention of symbols without comments. Furthermore, this choice of basis induces an identification between $\p_\sigma^+ = \p_{4, \sigma}^+$ and $M_n(\CC)$.

One readily checks that the $G^*(\RR)$-conjugacy class of $h$ is again equal to the $G(\RR)$-conjugacy class $X$ of $h$. Therefore, $X = \prod_{\sigma \in \Sigma} X_\sigma$, where $X_\sigma$ is the $G_\sigma$-conjugacy class of $h_\sigma$. Let $X_\sigma^+ \subset X_\sigma$ be the connected component containing $h_\sigma$. 

It is well-known that the space $X^+_\sigma$ is holomorphically isomorphic to a tube domain in $\p_\sigma^+ \simeq M_n(\CC)$, see \cite[(5.3.2)]{Har86}, \cite[Section 2.1]{Eis15} or \cite[Section 4.4.2]{EHLS}. Namely, let $\gimel_\sigma \in M_n(\CC)$ be the fixed point of $U_\sigma = U_\infty \cap G_\sigma$. Without loss of generality, we may assume that $\gimel_\sigma$ is a diagonal matrix whose entries are trace-zero elements of $\sigma(\KK)$. Then, $X_\sigma^+$ is naturally identified with
\[
    X_{n,n} 
        := 
    \{
        z \in M_n(\CC)
            \mid
        \gimel_\sigma(\tp{\ol{z}} - z) \text{ is positive-definite}
    \}\,.
\]

The action of $g_\sigma \in G_\sigma$ on $z \in X_{n,n}$ is given by
\[
    g_\sigma(z) = (a_\sigma z + b_\sigma) \cdot (c_\sigma z + d_\sigma)^{-1}\,.
\]

\subsubsection{Unitary Hecke characters of type $A_0$.} \label{subsubsec:Unitary Hecke char of type A0}
Let $\chi = \otimes_w \chi_w$ be the unitary Hecke character introduced in Section \ref{subsubsec:Induced representations}. Let $\chi_\infty = \otimes_{\sigma \in \Sigma} \chi_\sigma$. We assume that for each $\sigma \in \Sigma$, there exists integer $k_\sigma \in \ZZ_{\geq 0}$ and $\nu_\sigma \in \ZZ$ such that
\begin{equation} \label{eq:unitary Hecke char of type A0 k nu}
    \chi_\infty(z) 
        =
    \prod_{\sigma \in \Sigma}
        z_\sigma^{-(k_\sigma + 2\nu_\sigma)}
        (z_\sigma \ol{z_\sigma})^{\frac{k_\sigma}{2} + \nu_\sigma}
        =
    \prod_{\sigma \in \Sigma}
        \left(
            \frac{
                \absv{z_\sigma}_\sigma
            }{
                z_\sigma
            }
        \right)^{k_\sigma + 2\nu_\sigma}
    \,,
\end{equation}
for all $z = (z_\sigma)_\sigma \in \AA_{\KK, \infty}^\times = \prod_{\sigma \in \Sigma} \CC$.

\begin{remark} \label{rmk:comp m a b k nu with EHLS}
    To compare with the notation with \cite[Section 4.4.2]{EHLS}, consider the Hecke character $\chi_\infty(\bullet)|\bullet|_\infty^{-s-\frac{n}{2}}$. Assume there exists some integer $k$ such that $k = k_\sigma$ for all $\sigma \in \Sigma$. In that case, let $s = \frac{k-n}{2}$ so that
    \begin{equation} \label{eq:relation m k nu}
        |z|_\infty^{-\frac{k}{2}}\chi_\infty(z)
            =
        \prod_{\sigma \in \Sigma}
            z_\sigma^{-(k+\nu_\sigma)}
            \ol{z_\sigma}^{\nu_\sigma}\,,
    \end{equation}
    for all $z = (z_\sigma)_\sigma \in \AA_{\KK, \infty}^\times$. 
    
    For the expression \eqref{eq:relation m k nu} to be in the same form as the character denoted $|\bullet|^m\chi_0$ from \cite[p.72]{EHLS}, there is a lot of freedom on the integers $m$, $a(\chi_\sigma)$ and $b(\chi_\sigma)$ introduced in \emph{loc.cit}. For instance, we can pick $m$ arbitrarily and let $a(\chi_\sigma) = m + k + \nu_\sigma$ and $b(\chi_\sigma) = m - \nu_\sigma$.    
\end{remark}

\begin{remark} \label{rmk k nu vs m a b}
    The relations of the previous remark are the ones unstated in the explanations of \cite[Section 5.3]{EHLS}. We prefer to work with the notation of $k$, $\nu_\sigma$ and a unitary character $\chi$ as it is easier to compare with the computations of \cite{Eis15} and \cite{EisLiu20} and the work of Shimura more generally. However, for applications towards motivic conjectures, the notation with $m$, $a(\chi_\sigma)$ and $b(\chi_\sigma)$ is often more appropriate. %\DM{(to compare with \cite{EisLiu20}, we also need $r = -k - 2\nu_\sigma$.}
\end{remark}

\subsubsection{Canonical automorphy factors for $\GU(n,n)$.} \label{subsubsec:canonical automorphy factors for GU n n}
For $\sigma \in \Sigma$, $z \in X_{n,n}$ and $g_\sigma \in G_\sigma$, let
\[
    J_\sigma(g_\sigma, z) 
        = 
    c_\sigma z + d_\sigma
        \ \ \ \text{and} \ \ \
    j_\sigma(g_\sigma, z) 
        =
    \det(J_\sigma(g_\sigma, z))\,.
\]

Similarly, for $z = (z_\sigma) \in X = \prod_\sigma X_\sigma$ and $g = (g_\sigma) \in G^*(\RR) = \prod_\sigma G_\sigma$, let
\[
    J(g, z) 
        = 
    \prod_\sigma 
        J_\sigma(g_\sigma, z_\sigma)
        \ \ \ \text{and} \ \ \
    J'(g, z)
        = 
    \prod_\sigma 
        J'_\sigma(g_\sigma, z_\sigma)\,.
\]

The functions
\[
    J_\sigma(g_\sigma) = J_\sigma(g_\sigma, \gimel_\sigma) 
        \ \ \ \text{and} \ \ \
    J'_\sigma(g_\sigma) = J'_\sigma(g_\sigma, \gimel_\sigma) 
\]
are $C^\infty$-functions on $G_\sigma$ valued in $\GL_n(\CC)$, and so the functions
\[
    j_\sigma(g_\sigma) = \det(J_\sigma(g_\sigma)) 
        \ \ \ \text{and} \ \ \
    j'_\sigma(g_\sigma) = \det(J'_\sigma(g_\sigma)) 
\]
are $C^\infty$-functions on $G_\sigma$ valued in $\CC^\times$.

Using the integers $k_\sigma$ and $\nu_\sigma$ from above, let
\[
    j_{\chi_\sigma}(g_\sigma, z)
        :=
    \det(g_\sigma, z)^{-\nu_\sigma}
    j_\sigma(g_\sigma, z)^{-k_\sigma}
\]
and given $s \in \CC$, define $f_\sigma(g_\sigma; \gimel_\sigma, \chi_\sigma, s)$ as
\[
    j_{\chi_\sigma}(g_\sigma, \gimel_\sigma)
        \cdot
    \absv{
        j_\sigma(g_\sigma, \gimel_\sigma)
    }_{\sigma}^{s - \frac{k_\sigma - n}{2}}
        \cdot
    \absv{
        \nu(g_\sigma)
    }_\sigma^{\frac{n}{2}\left(s + \frac{n}{2} \right)}
    \,,
\]
a function on $G_\sigma$. 

From this point on, assume $\chi$ satisfies the following hypothesis :

\begin{hypothesis} \label{hyp:hyp integer k}
    There exists some integer $k \geq 0$ such that $k_\sigma = k$ for all $\sigma \in \Sigma$.
\end{hypothesis}

\begin{remark} \label{rmk:prod j chi sigma holomorphic}
    This hypothesis is exactly the necessary condition to ensure that the function 
    $
        \prod_{\sigma \in \Sigma}
            j_{\chi_\sigma}(g_\sigma, z)
                \cdot
            \absv{
                j_\sigma(g_\sigma, z)
            }_{\sigma}^{s - \frac{k_\sigma - n}{2}}
    $
    is holomorphic as a function of $z$ at $s = \frac{k-n}{2}$.
\end{remark}

Let $U(\g_\sigma)$ denote the universal enveloping algebra of $\g_\sigma$. Consider the $U(\g_\sigma)$-submodule $C_{\chi_\sigma}(G_\sigma)$ generated by $f_\sigma(g_\sigma; \gimel_\sigma, \chi_\sigma, \frac{n-k}{2})$ of $C^\infty(G_\sigma)$. It naturally carries the structure of a $(U(\g_\sigma), U_\sigma)$-module and as explained in \cite[Section 4.4.2]{EHLS}, it is isomorphic to the holomorphic $(U(\g_\sigma), U_\sigma)$-module $\DD^2(\chi_\sigma)$ with highest $U_\sigma$-type
\[
    \Lambda(\chi_\sigma)
        :=
    (
        \bullet;
        \nu_\sigma, 
        \ldots, 
        \nu_\sigma;
        k + \nu_\sigma,
        \ldots,
        k+\nu_\sigma
    )\,,
\]
where $\bullet$ denotes some character of the $\RR$-split center of $U_\sigma$ (whose exact description is irrelevant for our purpose), the next $n$ entries are identical, and the last $n$-entries are also identical.

One readily checks that
\begin{equation} \label{eq:def SW f oo gimel}
    f_\infty(g)
        =
    f_\infty(g; \gimel, \chi_\infty, s)
        :=
    \prod_{\sigma \in \Sigma}
        f_\sigma(g_\sigma; \gimel_\sigma, \chi_\sigma, s)
        \in
    I_\infty(\chi_\infty, s)
        \,,
\end{equation}
where $\gimel = (\gimel_\sigma)_\sigma \in X$ and $g = (g_\sigma)_\sigma \in G_4(\RR) \subset G^*(\RR) = \prod_\sigma G_\sigma$.

More generally, replacing $\gimel$ by any $z = (z_\sigma)_\sigma \in X$, we define $f_\sigma(g_\sigma; z_\sigma, \chi_\sigma, s)$ and $f_\infty(g; z, \chi_\infty, s)$ similarly.

\subsubsection{$C^\infty$-differential operators.} \label{subsubsec:Coo diff operators}

\begin{remark}
    The Eisenstein measure involved in the construction of our $p$-adic $L$-functions uses the Siegel section $f_\infty$ constructed above at $s = \frac{k-n}{2}$. The idea is to view the corresponding Siegel Eisenstein series as a $p$-adic modular form, using the theory of Section \ref{sec:Pord padic mod forms}, and apply $p$-adic differential operators.
    
    However, to relate this measure to standard $L$-functions, we compare this $p$-adic Eisenstein series to a smooth (non-holomorphic) Eisenstein series obtained by replacing the $p$-adic differential operators with more familiar $C^\infty$ differential operators. We obtain special values of $L$-functions by applying the doubling method to the $C^\infty$ Eisenstein series.
    
    As opposed to our choice of Siegel-Weil sections at $p$ from Section \ref{subsec:Local SW section at p} (and the corresponding computation of Zeta integrals in Section \ref{subsec:Local zeta integral at p}), the objects and calculus needed here are already explained thoroughly in the literature, see \cite{Har97, Har08}, \cite[Sections 4.4-4.5]{EHLS} or \cite{EisLiu20}. Therefore, we simply recall the material and results upon which we rely.
\end{remark}

Let $\kappa$ be a dominant character of $T_{H_0}$ as in Section \ref{subsubsec:alg weights}. We modify our notation slightly in this section by identifying $\kappa$ as a tuple in $\ZZ \times \prod_{\sigma \in \Sigma} \ZZ^{a_\sigma} \times \ZZ^{b_\sigma}$. Therefore, we momentarily write
\[
    \kappa_\sigma
        =
    (
        \kappa_{\sigma, 1},
        \ldots,
        \kappa_{\sigma, a_{\sigma}};
        \kappa^c_{\sigma, 1},
        \ldots,
        \kappa^c_{\sigma, b_{\sigma}}
    )
        \in
    \ZZ^{a_\sigma} \times \ZZ^{b_\sigma}
\]
and $\kappa = (\kappa_0, (\kappa_\sigma)_{\sigma \in \Sigma})$.

\begin{definition} \label{def:critical pairs}
    We say a pair $(\kappa, \chi)$ is \emph{critical} if $\kappa_\sigma \in C_3(\chi_\sigma)$ for all $\sigma \in \Sigma$, where $C_3(\chi_\sigma) \subset \ZZ^{a_\sigma} \times \ZZ^{b_\sigma}$ is defined as
    \begin{center}
    $\biggl\{$
        \stackanchor{
            $(
                - \nu_\sigma - r_{a_\sigma},
                \ldots,
                - \nu_\sigma - r_1;$
        }{
            $- k - \nu_\sigma + s_1,
                \ldots,
                - k - \nu_\sigma + s_{b_\sigma}
            )$
        }
                $\biggm|$
        \stackanchor{
            $r_1 \geq \ldots \geq r_{a_\sigma} \geq 0;$
        }{
            $s_1 \geq \ldots \geq s_{b_\sigma} \geq 0$
        }
    $\biggr\}$
    \end{center}
\end{definition}

Given a critical pair $(\kappa, \chi)$, we set
\begin{equation} \label{eq:def rho sigma and rho sigma v shifts}
    \rho_\sigma 
        := 
    (
        -r_{a_\sigma}, 
        \ldots,
        -r_1;
        s_1,
        \ldots,
        s_{b_\sigma}
    )
        \ \ \ \text{and} \ \ \
    \rho^v_\sigma 
        := 
    (
        r_1, 
        \ldots,
        r_{a_\sigma};
        s_1,
        \ldots,
        s_{b_\sigma}
    )
\end{equation}
and write $\rho = (\rho_\sigma)_{\sigma \in \Sigma}$, $\rho^v = (\rho^v_\sigma)_{\sigma \in \Sigma}$. Obviously, given a fixed unitary Hecke character $\chi$ of type $A_0$, the tuples $\kappa_\sigma$ and $\rho_\sigma$ determine one another, however we do not make this relation explicit in our notation. Informally, we think of $\rho_\sigma$ as the ``shift'' from $\Lambda(\chi_\sigma)$ to $\kappa_\sigma$ (but note the change of signs). 

\begin{remark}
    See \cite[Remark 4.4.6]{EHLS} and the discussion that precedes it for more information about the relevance of the set $C_3(\chi_\sigma)$.
\end{remark}

Fix any $\kappa_\sigma \in C_3(\chi_\sigma)$. Following the discussion in Section \ref{subsec:Local test vectors at oo}, note that $\DD(\kappa_\sigma) \otimes \DD(\kappa_\sigma^\flat \otimes \chi_\sigma)$ is a holomorphic discrete series for $U_{3, \sigma} := U(a_\sigma, b_\sigma) \times U(b_\sigma, a_\sigma)$, see Remark \ref{rmk:discrete series for G sigma}, for all $\kappa_\sigma \in C_3(\chi_\sigma)$. Furthermore, for $k$ sufficiently large, the restriction of $\DD^2(\chi_\sigma)$ to $U_{3, \sigma}$ is isomorphic to
\[
    \bigoplus_{\kappa_\sigma \in C_3(\chi_\sigma)} 
        \DD(\kappa_\sigma) \otimes \DD(\kappa_\sigma^\flat \otimes \chi_\sigma)\,.
\]

Observe that $\DD(\kappa_\sigma) \otimes \DD(\kappa_\sigma^\flat \otimes \chi_\sigma)$ is defined over the field of definition $E(\kappa_\sigma, \chi_\sigma)$ of $\kappa_\sigma \boxtimes (\kappa_\sigma^\flat \otimes \chi_\sigma)$. Let $v_{\kappa_\sigma} \otimes v_{\kappa_\sigma \otimes \chi_\sigma}$ be a highest weight vector in the minimal $U_{3,\sigma}$-type, rational over $E(\kappa_\sigma, \chi_\sigma)$. Note that $v_{\kappa_\sigma}$ and $v_{\kappa_\sigma \otimes \chi_\sigma}$ are dual to the choice of anti-holomorphic test vectors from Section \ref{subsubsec:choice of a-holo test vectors}.

Similarly, let $v_{\chi_\sigma}$ be the tautological generator of the $\Lambda(\chi_\sigma)$-isotypic subspace of $\DD^2(\chi_\sigma)$. We fix a map $\iota_{\chi_\infty} : \DD^2(\chi_\sigma) \to C^\infty(G_\sigma)$ mapping $v_\sigma$ to $f_\sigma(\bullet; \gimel_\sigma, \chi_\sigma, \frac{k-n}{2})$.

For each $\kappa_\sigma$, there is a natural projection
\[
    \pr_{\kappa_\sigma}
        :
    \DD^2(\chi_\sigma)
        \to
    \DD(\kappa_\sigma) 
        \otimes 
    \DD(\kappa_\sigma^\flat \otimes \chi_\sigma)
\]
and its composition with the orthogonal projection onto the highest weight component also yields
\begin{equation} \label{eq:def pr hol kappa sigma}
    \pr^{\hol}_{\kappa_\sigma}
        :
    \DD^2(\chi_\sigma)
        \to
    \Span(v_{\kappa_\sigma} \otimes v_{\kappa_\sigma^\flat \otimes \chi_\sigma})\,.
\end{equation}

Then, as explained in \cite[Section 4.4.7]{EHLS}, there exists a differential operator $D(\rho_\sigma^v) \in U(\p_{4, \sigma}^+)$ such that :
\[
    \pr_{\kappa_\sigma}^{\hol}( D(\rho_\sigma^v) v_{\chi_\sigma} )
        =
    P_{\kappa_\sigma, \chi_\sigma}
        \cdot
    v_{\kappa_\sigma} \otimes v_{\kappa_\sigma^\flat \otimes \chi_\sigma}
    \,,
\]
for some $P_{\kappa_\sigma, \chi_\sigma} \in E(\kappa_\sigma, \chi_\sigma)^\times$. Furthermore, let
\[
    D(\rho^v) 
        = 
    \prod_{\sigma} D(\rho^v_\sigma)
        \;\;\;;\;\;\;
    D^\hol(\rho^v_\sigma) 
        = 
    \pr_{\kappa_\sigma}^\hol D(\rho^v_\sigma)
        \;\;\;;\;\;\;
    D^\hol(\rho^v) 
        = 
    \prod_{\sigma} D^\hol(\rho^v_\sigma)\,.
\]

Then, for any other dominant weight $\kappa^\dagger \leq \kappa$ of $T_{H_0}$ (in particular, $(\kappa^\dagger, \chi)$ is again critical), there exists a differential operator $\delta(\kappa, \kappa^\dagger) \in U(\p_3^+)$ such that
\begin{equation} \label{eq:decomposition D kappa chi with D hol kappa dagger chi}
    D(\kappa, \chi)
        =
    \sum_{\kappa^\dagger \leq \kappa}
        \delta(\kappa, \kappa^\dagger)
            \circ
        D^\hol(\kappa^\dagger, \chi)\,,
\end{equation}
and $\delta(\kappa, \kappa) = \prod_\sigma P_{\kappa_\sigma, \chi_\sigma}$. See \cite[Corollary 4.4.9]{EHLS} for further details.

\begin{remark}
    We sometimes denote $D(\rho_\sigma^v)$ and $D(\rho)$ by $D(\chi_\sigma, \kappa_\sigma)$ and $D(\chi, \kappa)$ respectively. We have similar conventions for the holomorphic differential operators.
\end{remark}

Lastly, if $\kappa_\sigma \in C_3(\chi_\sigma)$, let
\[  
    f_{\sigma, \kappa_\sigma}(
        g_\sigma; 
        \gimel_\sigma, 
        \chi_\sigma, 
        s
    )
        := 
    D(\chi_\sigma, \kappa_\sigma)
    f_\sigma(
        g_\sigma; 
        \gimel_\sigma, 
        \chi_\sigma, 
        s
    )
\]
where $s = \frac{k-n}{2}$ as previously set. For each critical pair $(\kappa, \chi)$, our choice of Siegel-Weil section at $\infty$ is
\begin{equation} \label{eq:def SW f oo kappa gimel}
    f_{\infty, \kappa}(g)
        =
    f_{\infty, \kappa}(
        g; 
        \gimel, 
        \chi_\infty, 
        s
    )
        :=
    \prod_{\sigma \in \Sigma}
        f_{\sigma, \kappa_\sigma}(
            g; 
            \gimel_\sigma,
            \chi_\sigma, 
            s
        )
    \,, \ \ \ g = (g_\sigma)_\sigma\,,
\end{equation}
which lies in $I_\infty(\chi_\infty, \frac{k-n}{2})$ by \cite[Lemma 4.5.2]{EHLS}.

\subsection{Local Siegel-Weil section away from $p$ and $\infty$.} \label{subsec:Local SW section away from p and oo}

Our choice of Siegel-Weil section away from $p$ and $\infty$ is
\[
    f^{p, \infty}(\bullet) 
        := 
    \bigotimes_{l \neq p, \infty} f_l
        \in
    \bigotimes_{l \neq p, \infty}
        I_l(\chi_l, s)\,,
\]
where $f_l \in I_l(\chi_l, s)$ is defined as in the following two section, for each finite prime $l \neq p$ of $\QQ$, .

\subsubsection{Local Siegel-Weil section at finite unramified places.} \label{subsubsec:Local SW section at finite unramd places}

Fix a prime $l \notin S \cup \{p\}$. Then, $G_4(\ZZ_l)$ is a hyperspecial maximal compact subgroup of $G_4(\QQ_l)$. Furthermore, $G_4(\QQ_l)$ factors as $\prod_{v \mid l} G_{4,v}$ and so does $I_l(\chi_l, s) = \bigotimes_{v \mid l} I_v(\chi_v, s)$. Our choice of local Siegel-Weil section at $l$ is the unique $G_4(\ZZ_l)$-invariant function 
\begin{equation} \label{eq:def SW f l at unramd places}
    f_l = \bigotimes_{v \mid l} f_v
\end{equation}
such that $f_l(G_4(\ZZ_l)) = 1$.

\subsubsection{Local Siegel-Weil section at finite ramified places.} \label{subsubsec:Local SW section at finite ramd places}
For a prime $l \in S$, we choose the same section $f_l$ as in \cite[Section 4.2.2]{EHLS} and \cite[Section 2.2.9]{Eis15} (with some minor adjustments). 

Consider $P_{U, \sgl} = P_\sgl \cap U$, where $U = U_4$ is the unitary subgroup of $G = G_4$. Then, $P_\sgl = \Gm \ltimes P_{U, \sgl}$, where $\Gm$ is the similitude factor.

We have $U_4(\QQ_l) = \prod_{v \mid l} U_{4,v}$, where the products run over all primes $v$ of $\KK^+$ above $l$. There are similar decompositions $U_i(\QQ_l) = \prod_{v \mid l} U_{i, v}$ for $i$ = 1, 2 and 3, with the obvious inclusions (see Section \ref{sec:comp and comp bw PEL data} for more details). Similarly, we have $P_{U,\sgl}(\QQ_l) = \prod_{v \mid l} P_{\sgl, v}$. Fix any place $v \mid l$ of $\KK^+$ and write $P_v := P_{\sgl, v}$. Let $N_v$ be the unipotent radical of $P_v$. 

As explained in \cite[Section 4.2.2]{EHLS}, we know $A = P_v \cdot (U_{1, v} \times 1_n) \cap P_v w P_v$ is open in $U_{4,v}$. Furthermore,
\[
    P_v w = P_v \cdot (-1_n, 1_n) \subset P_v \cdot U_{3,v}
        \;\;\;\text{and}\;\;\;
    P_v \cap (U_{1,v} \times 1_n) = (1_n, 1_n) \in U_{3,v}\,,
\]
hence $A$ is an open neighborhood of $w$ in $P_v w P_v$ and can be written as $A = P_v w \mathfrak{U}$ for some open subgroup $\mathfrak{U}$ of $N_v$.

Let $\varphi_v \in \ul{\pi}_v$ be any nonzero vector (i.e. a local test vector at $v$ as in Section \ref{subsec:Local test vectors at places away from p and oo}). Let $K_v \subset U_{1,v}$ be an open compact subgroup that fixes $\varphi_v$ and fix any lattice $L_v$ sufficiently small so that
\[
    N(L_v)
        :=
    \left\{
        \begin{pmatrix}
            1_n & L_v \\
            0 & 1_n
        \end{pmatrix}
    \right\}
        \subset
    \mathfrak{U}\,,
\]
and
\[
    P_v w N(L_v)
        \subset
    P_v \cdot ( (-1 \cdot K_v) \times 1_n ) 
        \subset 
    P_v \cdot U_{3,v}\,.
\]

For such $L_v$, we have $P_v w N(L_v) = P_v \cdot (\mathcal{U}_v \times 1_n) \subset -1 \cdot K_v$, for some open neighborhood $\mathcal{U}_v$ of $-1_n$. Since $-1 \cdot K_n$ is open in $U_{1,v}$, so is $\mathcal{U}_v$.

Let $I_{U_4, v}(\chi_v, s)$ be the principal series defined as in \eqref{eq:definition I chi s normalized} (and its local version) for $P_v$ and $U_{4,v}$. Then, there exists some $f_{L_v} \in I_v(\chi_v, s)$ supported on $P_v w P_v$ such that
\[
    f_{L_v}(w x) = \delta_{L_v}(x)\,,
\]
for all $x \in N_v$, where $\delta_{L_v}$ is the characteristic function of $N(L_v)$, see \cite[p. 449-450]{HLS}.

Let 
\begin{equation} \label{eq:def SW f l at ramd places}
    f_v = f_{L_v}^-\,,
\end{equation}
where $f_{L_v}^-(g) = f_{L_v}(g \cdot (-1_n, 1_n))$ for all $g \in U_{4,v}$, and define $f_U = \otimes f_v$ on $U(\QQ_l) = \prod_{v \mid l} U_v$.

Our choice of local Siegel-Weil section at $l$ is \emph{any} $f_l \in I_l(\chi_l, s)$ whose restriction from $G_4(\QQ_l)$ to $U_4(\QQ_l)$ equals $f_U$. This section depends on many choices which we do not make explicit in our notation.

\begin{remark}
    One easily checks that any element of $\bigotimes_{v \mid l} I_v(\chi_v, s)$ can be extended to a function in $I_l(\chi_l, s)$. Furthermore, the choice of extension $f_l$ of $f_U$ is irrelevant for our purpose as all of our later constructions (Fourier coefficients and local zeta integrals) only depends on the restriction of $f_l$ to $U_4(\QQ_l)$. In fact, the Siegel-Weil section at $l$ in \cite[Section 4.2.2.]{EHLS} is only described on $U_4(\QQ_l)$ as its full description on $G_4(\QQ_l)$ is unnecessary.
\end{remark}

\subsubsection{Comparison to other choices in the literature.} \label{subsubsec:comp to other choices in the literature}

For each finite place $v$ of $\KK^+$ away from $p$, the local section $f_v$ at $v$ is constructed as in \cite[Section 4.2]{EHLS}. However, they differ slightly from \cite[Section 18]{Shi97}, \cite[Section (3.3.1)-(3.3.2)]{HLS}, \cite[Section 2.2.9]{Eis15}, see \cite[Section 4.2.2]{EHLS}.

Namely, given an ideal $\bb$ of $\OO_{\KK^+}$, let $f_v^\bb = f_v^\bb(\bullet; \chi_v, s)$ be the local Siegel-Weil section in \cite[Section 2.2.9]{EHLS}. Then, one can choose $\bb$ prime to $p$ (depending on the lattice $L_v$ for each ramified $v$) such that
\[
    f_v^\bb
        =
    \begin{cases}
        f_v, & v \text{ is unramified,} \\
        f_{L_v}, & v \text{ is ramified.}
    \end{cases}
\]

Since $f_v^\bb$ is at most a translation of $f_v$, the Fourier coefficients associated to each of them are equal, see Section \ref{subsubsec:Local coeff at places away from p and oo}. For more details, see \cite[Remark 12]{Eis15}.

\subsection{Siegel Eisenstein series as $C^\infty$-modular forms.} \label{subsec:Sgl Eis series as Coo mod forms}
Let $\pi$ be a $P$-anti-ordinary automorphic representation of $P$-anti-WLT $(\kappa, K_r, \tau)$. Furthermore, fix a unitary Hecke character $\chi$ as in Section \ref{subsubsec:Unitary Hecke char of type A0} that satisfies Hypothesis \ref{hyp:hyp integer k} for some integer $k \geq 0$.

To $K_r$, $\tau$ and $\chi$, we associate the Siegel-Weil section
\begin{equation} \label{eq:def f tau chi}
    f^\tau_\chi 
        = 
    f^\tau_{\chi}(\bullet; s)
        := 
    f_p(\bullet; \tau; \chi_p, s)
        \otimes
    f_\infty(\bullet; i 1_n, \chi_\infty, s)
        \otimes
    f^{p, \infty}(\bullet)\,,
\end{equation}
using notation from \eqref{eq:def SW f p tau chi p}, \eqref{eq:def SW f oo gimel}, \eqref{eq:def SW f l at unramd places} and \eqref{eq:def SW f l at ramd places}.

Similarly, if $(\kappa, \chi)$ is critical, we also define
\begin{equation} \label{eq:def f tau kappa chi}
    f^{\tau, \kappa}_\chi 
        = 
    f^\tau_\chi(\bullet; s)
        := 
    f^{p, \infty}(\bullet) 
        \otimes 
    f_p(\bullet; \tau, \chi_p, s)
        \otimes
    f_{\infty, \kappa}(\bullet; i 1_n, \chi_\infty, s)\,,
\end{equation}
using \eqref{eq:def SW f oo kappa gimel}.

Let $\psi$ be any finite order character of $L_P(\ZZ_p)$. When considering $\tau \otimes \psi$ instead of $\tau$ (thinking of $\tau$ as fixed and $\psi$ as varying), we set
\begin{equation} \label{eq:def f tau kappa chi psi}
    f^{\tau}_{\chi, \psi} := f^{\tau \otimes \psi}_\chi 
        \;\;\;\text{and}\;\;\;
    f^{\tau, \kappa}_{\chi, \psi} := f^{\tau \otimes \psi, \kappa}_\chi\,.
\end{equation}

Set $E_{f_{\chi, \psi}^\tau} = E_{f_{\chi, \psi}^\tau}(\bullet; \frac{k-n}{2})$ for the Eisenstein series associated to $f_{\chi, \psi}^\tau(\bullet; s)$ at $s = \frac{k-n}{2}$, for $k$ as in Hypothesis \ref{hyp:hyp integer k}. 

Let $L(\chi)$ be the 1-dimensional vector space on which $U_\infty$ acts via $\Lambda(\chi)$, and $\LL(\chi)$ denote the automorphic line bundle on $\Sh(G_4)$ determined by the dual of $\Lambda(\chi)$. Its fiber at the fixed point $h_4$ of $U_\infty$ is isomorphic to $L(\chi)$. Let $\LL(\chi)^\can$ denote its canonical extension to the toroidal compactification $\Sh(G_4)^\tor$. 

\begin{remark}
    Although the notation is different, the automorphic line bundle $\LL(\chi)$ is an automorphic bundle associated to a highest weight representation as in Section \ref{subsubsec:canonical bundle} (replacing $G_1$ with $G_4$).
\end{remark}

Then, $E_{f_{\chi, \psi}^\tau}$ corresponds to an Eisenstein modular form
\begin{equation} \label{eq:def Eis mod form E tau chi psi}
    E_{\chi, \psi}^\tau \in H^0(\Sh(G_4)^\tor, \LL(\chi)^\can)
\end{equation}
via complex uniformization and pullback to functions on $G_4(\AA)$. It descends to a modular form of level $K_4 = K_{4,r} = I_{4,r} K_4^p \subset G_4(\AA_f)$, for $r$ as in \eqref{eq:large enough r} and where $K_4^p$ contains the maximal subgroup $G_4(\ZZ_l)$ for each $l \notin S \cup \{p\}$ and depends on $K_v = K_{1,v}$ for each $v \mid l \in S$ (see Section \ref{subsubsec:Local SW section at finite unramd places}). 

Most importantly, if $K_r = K_{1,r}$ is the level of the anti-holomorphic $P$-anti-ordinary representation $\pi$ on $G_1$ fixed in Section \ref{subsubsec:zeta integrals}, and $K_{2,r} = K_{1,r}^\flat$, then $K_{4,r} \cap G_3(\AA_f) \supset (K_{1,r} \times K_{2,r}) \cap G_3(\AA_f)$.

The differential operators $D(\kappa, \chi)$ used to define $f_{\infty, \kappa}$ can be interpreted at the level of modular forms as well. Set
\[
    d 
        = 
    \sum_{\sigma \in \Sigma}
        r_{1, \sigma} + s_{1, \sigma}
        = 
    \sum_{\sigma \in \Sigma}
        k - \kappa_{\sigma, a_\sigma} + \kappa_{1, \sigma}^c\,,
\]
where $r_1 = r_{1, \sigma}$ and $s_1 = s_{1, \sigma}$ are the integers appearing in the definition of $\rho^v = (\rho^v_\sigma)_\sigma$ in \eqref{eq:def rho sigma and rho sigma v shifts}. Consider the $C^\infty$-differential operator
\[
    \delta^d(\kappa, \chi)
        =
    \delta_\chi^d(\rho^v)
        :
    H^0(\Sh(V_4)^\tor, \LL(\chi)^\can)
        \to
    \Ab(G, \LL(\chi)_h)
\]
defined \cite[Equation (105)]{EHLS}.

Furthermore, let $K_1 = K_{1,r}$, $K_2 = K_{2,r}$ and $K_4 = K_{4,r}$ be the  level subgroups above, and set $K_3 = K_{3,r} := (K_{1,r} \times K_{2,r}) \cap G_3(\AA_f)$. If we compose the above with restriction of functions from $G_4(\AA)$ to $G_3(\AA)$, we obtain
\[
    \res_3 \circ \delta^d(\kappa, \chi)
        :
    H^0(\level{K_4}{\Sh(V_4)^\tor}, \LL(\chi)^\can)
        \to
    H^0(\level{K_3}{\Sh(3_4)^\tor}, i_3^*\LL(\chi)^\can)\,,
\]
where $i_3$ is as in \eqref{eq:def embeddings i 3 and i 1 2}. Then, \cite[Proposition 4.4.11]{EHLS} show that the above is the pullback to functions on $G_4(\AA)$ of a differential operator
\[
    D(\kappa, \chi)
        :
    H^0(\level{K_4}{\Sh(V_4)^\tor}, \LL(\chi)^\can)
        \to
    M_{\kappa, V}^\infty
    (
        K_1,  \CC
    )
        \otimes
    M_{\kappa^\flat, -V}^\infty
    (
        K_2,  \CC
    )
        \otimes
    (\chi \circ \det)
\]
where the superscript $\infty$ stands for smooth modular forms (as opposed to holomorphic ones). The above respect cuspidal forms, namely we also have
\[
    D(\kappa, \chi)
        :
    H^0_!(\level{K_4}{\Sh(V_4)^\tor}, \LL(\chi)^\can)
        \to
    S_{\kappa, V}^\infty
    (
        K_1,  \CC
    )
        \otimes
    S_{\kappa^\flat, -V}^\infty
    (
        K_2,  \CC
    )
        \otimes
    (\chi \circ \det)\,.
\]

Lastly, one can compose with the holomorphic projections from \eqref{eq:def pr hol kappa sigma} to obtain
\[
    D^\hol(\kappa, \chi)
        :
    H^0(\level{K_4}{\Sh(V_4)^\tor}, \LL(\chi)^\can)
        \to
    M_{\kappa, V}
    (
        K_1,  \CC
    )
        \otimes
    M_{\kappa^\flat, -V}
    (
        K_2,  \CC
    )
        \otimes
    (\chi \circ \det)
\]
and
\[
    D^\hol(\kappa, \chi)
        :
    H^0_!(\level{K_4}{\Sh(V_4)^\tor}, \LL(\chi)^\can)
        \to
    S_{\kappa, V}
    (
        K_1,  \CC
    )
        \otimes
    S_{\kappa^\flat, -V}
    (
        K_2,  \CC
    )
        \otimes
    (\chi \circ \det)
\]

Therefore,
\[
    E_{\chi, \psi}^{\tau, \kappa} 
        := 
    D(\kappa, \chi)
    E_{\chi, \psi}^{\tau}
\]
is (the restriction to $G_3$) of the Eisenstein $C^\infty$-modular form associated to $f^{\tau, \kappa}_{\chi, \psi}$.

\subsubsection{Family of Siegel Eisenstein series} \label{subsubsec:family of Siegel Eisenstein series}
Fix two dominant weights $\kappa$ and $\kappa'$ such that $[\kappa] = [\kappa']$, say $\kappa' = \kappa + \theta$ for some $P$-parallel weight $\theta$. Assume that both $(\kappa, \chi)$ and $(\kappa', \chi)$ are critical. In later section, it is convenient to think of $\kappa$ as fixed and vary $\theta$, see Remark \ref{rmk:base point pi and shift by psi theta}, hence we sometimes write $D(\kappa, \theta, \chi)$ (resp. $D^\hol(\kappa, \theta, \chi)$) instead of $D(\kappa', \chi)$ (resp. $D^\hol(\kappa', \chi)$). Similarly, we set
\[
    E_{\chi, \psi, \theta}^{\tau, \kappa}
        :=
    E_{\chi, \psi}^{\tau, \kappa'} 
        = 
    D(\kappa', \chi) E_{\chi, \psi}^\tau 
        = 
    D(\kappa, \theta, \chi) E_{\chi, \psi}^\tau\,,
\]
and
\[
    E_{\chi, \psi, \theta}^{\tau, \kappa, \hol}
        =
    E_{\chi, \psi}^{\tau, \kappa', \hol} 
        := 
    D^\hol(\kappa', \chi) E_{\chi, \psi}^\tau 
        = 
    D^\hol(\kappa, \theta, \chi) E_{\chi, \psi}^\tau\,.
\]

%%%%%%%%%%%%%%%%%%%%%%%%%%%%%%%%%%%%%%%%%%%%%%%%%%%%%%%%%%%%%%%%%%

\section{Zeta integrals and the doubling method.} \label{sec:doubling method integral}

In this section, we compute the zeta integral $Z_l$ introduced in Section \ref{subsubsec:zeta integrals}, for each place $l$ of $\QQ$. The most technical case is objectively for $l = p$. Our method adapts the calculations of \cite[Section 4.3.6]{EHLS}, where we resolve various issues arising when working with a $P$-nebentypus (i.e. finite-dimensional representations) instead of an ordinary nebentypus (i.e. a character). The calculations of the other integrals are more common in the literature, and we recall the necessary results for our purposes.

\subsection{Local zeta integrals at $p$.} \label{subsec:Local zeta integral at p}

\subsubsection{Construction of $f_{w,s}^+$.} \label{subsubsec:construction f w s +}
Let $f_p$ be the corresponding local Siegel-Weil section at $p$, as in Equation \eqref{eq:def SW f p}. Ahead of our computations in the next section, we write down an explicit expression for $f_p(u,1)$ for any $u \in U_1(\QQ_p)$. 

Firstly, the isomorphism \eqref{eq:prod G over Zp}, restricted to $U_1$, yields an identification $U_1(\QQ_p) = \prod_{w \in \Sigma_p} U_{1,w}$, where $U_{1,w} = \GL_{\KK_w}(V_w) = \GL_n(\KK_w)$. We write $u = (u_w)_{w \in \Sigma_p}$ accordingly and wish to evaluate \eqref{eq:def f phi w} at $g = (u, 1)$, where this notation is with respect to the embedding $G_1 \times G_2 \hookrightarrow G_3$. To simplify the expression, it is therefore more convenient to replace the decomposition $W_w = V_{w, d} \oplus V_w^d$ with $W_w = V_w \oplus V_w$. In that case, an element $X \in \GL_n(\KK_w) = \GL_{\KK_w}(V_w)$ corresponds to an element $(X,X)$ in $\mathbf{X}$ instead of $(0,X)$.

Secondly, using the decomposition $W_w = V_w \oplus V_w$ again and the corresponding identification $W_w = V_w \oplus V_w$, consider the element
\[
    S_w = 
    \begin{pmatrix}
        1_{a_w} & 0 & 0 & 0 \\
        0 & 0 & 0 & 1_{b_w} \\
        0 & 0 & 1_{a_w} & 0 \\
        0 & 1_{b_w} & 0 & 0
    \end{pmatrix} .
\]

\begin{remark} \label{rmk:inclusion of Igusa towers and matrix S w}
    As explained in \cite[Section 2.1.11]{HLS} and \cite[Remark 3.1.4]{EHLS}, the natural inclusion of Shimura varieties associated to $G_3$ and $G_4$ does not induce the natural inclusion on Igusa tower. In fact, one needs to twist the former by the matrix $S_w$ to obtain the latter, see Section \ref{subsubsec:embeddings of Igusa towers} (where we wrote $\gamma_w$ for $S_w$).
\end{remark}

Lastly, replace each $f_{w,s}$ by its translation $f_{w,s}^+$ via $g \mapsto gS_w$, and let $f_p^+$ be the corresponding local Siegel-Weil section at $p$ defined by Equation \eqref{eq:def SW f p}. In that case, for $g=(u,1)$, we obtain that $f^+_p(u,1)$ is equal to a product over $w \in \Sigma_p$ of
\[
    \chi_{2,w}(\det u_w)
    \absv{\det u_w}_w^{\frac{n}{2} + s}
    \int_{\GL_n(\KK_w)}
        \Phi_w((Xu_w,X) S_w)
        \chi_{w,1}^{-1}
        \chi_{w,2}(\det X)
        |\det X|_w^{n+2s}  
    d^\times X
\]
and we denote the above expression by $f_{w,s}^+(u_w,1) = f_w^+(u_w,1)$, as a function of $u_w \in \GL_n(\KK_w)$.

In the following two subsections, we prove the following formula :
\begin{equation}
    I_p(\varphi_p, \varphi^\vee_p, f_p^+, \chi_p, s) 
        =
    E_p\left(
        s+\frac{1}{2}, \Pord, \pi_p, \chi_p
    \right)
        \cdot
    (\dim \tau_p)
        \cdot
    \frac{\vol(I^0_{P, r})\vol(\tp{I}^0_{P, r})}{\vol(I^0_{P, r} \cap \tp{I}_{P,r}^0)}
\end{equation}

\subsubsection{Local integrals at places above $p$.} \label{subsubsec:Local integrals at places above p}
We proceed with the same notation as in Sections \ref{subsec:Local test vectors at p} and \ref{subsec:Local SW section at p}. In particular, for each $w \in \Sigma_p$, consider the local test vectors $\phi_w \in \pi_w$ and $\phi^\vee_w \in \cg{\pi}_w$ at $w$ defined in Section \ref{subsubsec:comp of test vectors}. We now compute the $p$-adic local zeta integral $Z_p$ defined in \eqref{eq:def local zeta integrals} for the local section $f_p^+$ constructed above and the test vectors
\[
    \varphi_p = 1 \otimes \left( \bigotimes_{w \in \Sigma_p} \phi_w \right)
        \ \ \ ; \ \ \ 
    \varphi^\vee_p = 1 \otimes \left( \bigotimes_{w \in \Sigma_p} \phi^\vee_w \right)
\]
as in \eqref{eq:def test vectors varphi p and varphi vee p}.

Then by definition, for
\begin{align*}
    Z_w :=
        &\int_{\GL_n(\KK_w)} 
            f_{w,s}^+(g,1) \la \pi_w(g) \phi_w, \cg{\phi}_w \ra_{\pi_w}
        d^\times g \\
    =
        &\int_{\GL_n(\KK_w)} 
            \chi_{2,w}(\det g)
            \absv{\det g}_w^{\frac{n}{2} + s}
            \int_{\GL_n(\KK_w)}
                \Phi_w((Xg, X) S_w) \\
        &\times
                \chi_{w,1}^{-1}
                \chi_{w,2}(\det X)
                |\det X|_w^{n+2s}  
            \la \pi_w(g) \phi_w, \cg{\phi}_w \ra_{\pi_w}
            d^\times X 
        d^\times g\,,
\end{align*}
we have $Z_p = \prod_{w \in \Sigma_p} Z_w$.

According to the decomposition $M_{n \times n}(\KK_w) = M_{n \times a_w}(\KK_w) \times M_{n \times b_w}$, write $Z_1 := Xg = [Z_1', Z_1'']$ and $Z_2 := X = [Z_2', Z_2'']$, where $Z_1'$ and $Z_2'$ (resp. $Z_1''$ and $Z_2''$) are $n \times a$-matrices (resp. $n \times b$-matrices). Then,
\[
    (Xg, X) S_w = ([Z_1', Z_2''], [Z_2', Z_1''])
\]
and 
\[
    \la \pi_w(g) \phi_w, \cg{\phi}_w \ra_{\pi_w} 
        = 
    \la \pi_w(Xg) \phi_w, \cg{\pi}_w(X) \cg{\phi}_w \ra_{\pi_w}
        = 
    \la \pi_w(Z_1) \phi_w, \cg{\pi}_w(Z_2) \cg{\phi}_w \ra_{\pi_w}.
\]

Therefore, using \eqref{eq:def Schwartz Phi w}, we obtain
\begin{align*}
    Z_w &= 
    \frac{\dim \tau_w}{\vol(\GG_w)}
    \int_{\GL_n(\KK_w)} 
        \chi_{w,2}(Z_1) \chi_{w,1}(Z_2)^{-1} 
        \absv{ \det Z_1Z_2 }_w^{s+\frac{n}{2}} 
        \\ &\times 
        \Phi_{1,w}(Z_1', Z_2'') 
        \Phi_{2,w}(Z_2', Z_1'') 
        \la 
            \pi_w(Z_1) \phi_w, 
            \cg{\pi}_w(Z_2) \cg{\phi}_w
        \ra_{\pi_w} 
    d^\times Z_1 d^\times Z_2 \ .
\end{align*}

We take the integrals over the following open subsets of full measure. We take the integral in $Z_1$ over
\begin{align*}
    \left\{ 
	\begin{pmatrix}
		1 & 0 \\ 
            C_1 & 1
	\end{pmatrix}
	\begin{pmatrix}
		A_1 & 0 \\ 
            0 & D_1
	\end{pmatrix}
	\begin{pmatrix}
		1 & B_1 \\ 
            0 & 1
	\end{pmatrix} \mid
            B_1, \tp{C}_1 \in 
                M_{a_w \times b_w}(\KK_w), 
            A_1 \in 
                \GL_{a_w}(\KK_w), 
            D_1 \in 
                \GL_{b_w}(\KK_w) 
    \right\},
\end{align*}
with the measure
\[
    \absv{ 
        \det A_1^{b_w} 
        \det D_1^{-a_w}
    }_w 
    dC_1 
    d^\times A_1 
    d^\times D_1 
    dB_1 \ .
\]

Similarly, we take the integral in $Z_2$ over
\begin{align*}
    \left\{ 
	\begin{pmatrix}
		1 & B_2 \\ 0 & 1
	\end{pmatrix}
	\begin{pmatrix}
		A_2 & 0 \\ 0 & D_2
	\end{pmatrix}
	\begin{pmatrix}
		1 & 0 \\ C_2 & 1
	\end{pmatrix} \mid 
            B_2, \tp{C}_2 \in 
                M_{a_w \times b_w}(\KK_w), 
            A_2 \in 
                \GL_{a_w}(\KK_w), 
            D_2 \in 
                \GL_{b_w}(\KK_w) 
    \right\},
\end{align*}
with the measure
\[
    \absv{ 
        \det A_2^{b_w} 
        \det D_2^{-a_w} 
    }_w 
    dC_2 
    d^\times A_2 
    d^\times D_2 
    dB_2 \ .
\]

Therefore, one has
\begin{align*}
	\Phi_{1,w}(Z_1', Z_2'') &= 
	\Phi_{1,w}\left(
		\begin{pmatrix}
			A_1 & B_2D_2 \\
			C_1A_1 & D_2
		\end{pmatrix}
	\right)	
	\\
	\Phi_{2,w}(Z_2', Z_1'') &=
	\Phi_{2,w}\left(
		\begin{pmatrix}
			A_2 + B_2D_2C_2 & A_1B_1 \\
			D_2C_2 & C_1A_1B_1 + D_1
		\end{pmatrix}
	\right)	
\end{align*}
and both can be simplified by considering their support.
\begin{lemma} \label{lma:Z_i conditions for Phi 1 Phi 2 nonzero}
    The product 
    \[
        \Phi_{1,w}
        \left(
		\begin{pmatrix}
			A_1 & B_2D_2 \\
			C_1A_1 & D_2
		\end{pmatrix}
	\right)	
	\Phi_{2,w}
        \left(
		\begin{pmatrix}
			A_2 + B_2D_2C_2 & A_1B_1 \\
			D_2C_2 & C_1A_1B_1 + D_1
		\end{pmatrix}
	\right)	
    \]
    is zero unless all of the following conditions are met:
    \begin{align*}
        A_1 \in 
            I_{a_w, r}^0 \ \ \ ; \ \ \
        D_2 \in 
            I_{b_w, r}^0 \ \ \ &; \ \ \
        C_1 \in 
            \GG_l\ \ \ ; \ \ \
        B_2 \in 
            \GG_u \\
        B_1 \in 
            M_{a_w \times b_w}(\OO_w) \ \ \ &; \ \ \
        C_2 \in 
            M_{b_w \times a_w}(\OO_w) \\
        A_2 \in 
            \p_w^{-r} M_{a_w \times a_w}(\OO_w) \ \ \ &; \ \ \
        D_1 \in 
            \p_w^{-r} M_{b_w \times b_w}(\OO_w)
    \end{align*}
	
	Moreover, in this case, the product is equal to 
    $
        \mu_{a_w}(A_1) 
        \mu_{b_w}(D_2) 
        \Phi_{w}^{(1)}(A_2)
        \Phi_{w}^{(4)}(D_1)
    $.
\end{lemma}

\begin{proof}
	Using Lemma \ref{lma:Phi 2 w support} and the definition of $\Phi_{w,1}$, it is clear that the product above is nonzero if and only if the conditions above are satisfied. Moreover, if they are satisfied, one has
    \[
        \Phi_{1,w}
        \left(
	    \begin{pmatrix}
                A_1 & B_2D_2 \\
                C_1A_1 & D_2
            \end{pmatrix}
        \right)	
	   = 
        \mu_{a_w}(A_1) 
        \mu_{b_w}(D_2) 
\]
    by definition of $\mu_w$. One also obtains
    \[	
	\Phi_{2,w}
        \left(
            \begin{pmatrix}
                A_2 + B_2D_2C_2 & A_1B_1 \\
                D_2C_2 & C_1A_1B_1 + D_1
            \end{pmatrix}
        \right)
        =
        \Phi_w^{(1)}(A_2)
        \Phi_w^{(4)}(D_1)
    \]
	as in the proof of Lemma \ref{lma:Phi 2 w support}.
\end{proof}

\begin{lemma}
Under the conditions of Lemma \ref{lma:Z_i conditions for Phi 1 Phi 2 nonzero}, one has
\[
    \la 
        \pi_w(Z_1) \phi_w, 
        \pi^\vee_w(Z_2) \phi^\vee_w 
    \ra_{\pi_w}
    =
    \la 
        \pi_w
        \left(
            \begin{pmatrix}
                A_1 & 0 \\ 0 & 1
            \end{pmatrix}
        \right) 
        \phi_w,
        \pi^\vee_w
        \left(
            \begin{pmatrix}
                A_2 & 0 \\ 0 & D_1^{-1}D_2
            \end{pmatrix}
        \right)
        \phi^\vee_w
    \ra_{\pi_w}
\]
\end{lemma}

\begin{proof}
    We write
    \[
        Z_1 = \begin{pmatrix}
            1 & 0 \\ 
            C_1 & 1
        \end{pmatrix}
        \begin{pmatrix}
            A_1 & 0 \\ 
            0 & D_1
        \end{pmatrix}
        \begin{pmatrix}
            1 & B_1 \\ 
            0 & 1
        \end{pmatrix} 
        \ \ \ \text{and} \ \ \
        Z_2 = \begin{pmatrix}
            1 & B_2 \\ 
            0 & 1
        \end{pmatrix}
        \begin{pmatrix}
            A_2 & 0 \\ 
            0 & D_2
        \end{pmatrix}
        \begin{pmatrix}
            1 & 0 \\ 
            C_2 & 1
        \end{pmatrix}
    \]
    under the conditions of Lemma \ref{lma:Z_i conditions for Phi 1 Phi 2 nonzero}. As 
    $
        \begin{pmatrix}
            1 & B_1 \\
            0 & 1
        \end{pmatrix} 
        \in I_{w,r}
    $ 
    and
    $
        \begin{pmatrix}
            1 & 0 \\ 
            C_2 & 1
        \end{pmatrix} 
        \in \tp{I}_{w,r}
    $
    fix $\phi_w$ and $\phi^\vee_w$ respectively, the pairing
    \[
        \la 
            \pi_w(Z_1)
            \phi_w, 
            \pi^\vee_w(Z_2) 
            \phi^\vee_w 
        \ra_{\pi_w}
    \]
    is equal to
    \begin{align*}
        \la
            \pi_w
            \left(
                \begin{pmatrix}
                    1 & -B_2 \\
                    0 & 1
                \end{pmatrix}
                \begin{pmatrix}
                    1 & 0 \\
                    C_1 & D_1
                \end{pmatrix}
            \right)
            \pi_w
            \left(
                \begin{pmatrix}
                    A_1	& 0 \\
                    0 & 1
                \end{pmatrix}
            \right)
            \phi_w,
            \pi^\vee_w
           \left(
                \begin{pmatrix}
                    A_2 & 0 \\
                    0 & D_2
                \end{pmatrix}
            \right)
            \phi^\vee_w
        \ra_{\pi_w}
    \end{align*}
	
    Furthermore, write
    \[
        \begin{pmatrix}
            1 & -B_2 \\
            0 & 1
        \end{pmatrix}
        \begin{pmatrix}
            1 & 0 \\
            C_1 & D_1
        \end{pmatrix} 
        =
        \begin{pmatrix}
            1 & 0 \\
            C & 1
        \end{pmatrix}
        \begin{pmatrix}
            A & 0 \\
            0 & D
        \end{pmatrix}
        \begin{pmatrix}
            1 & B \\
            0 & 1
        \end{pmatrix}
    \]
    where 
    \begin{align*}
        A &= 
            1 - B_2C_1 \in 
                1 + \p_w^{2r}M_{a_w}(\OO_w), \\
        CA &= 
            C_1 \in 
                \p_w^rM_{b_w \times a_w}(\OO_w), \\
        AB &= 
            -B_2D_1 \in 
                M_{a_w \times b_w}(\OO_w) \\
        D_1 &= 
            D + CAB \in 
                \p_w^{-r}M_{b_w}(\OO_w) \ .
    \end{align*}
    
    Note that $1 = A^{-1} + B_2C_1A^{-1}$, so
    \begin{align*}
        &A^{-1} = 
            1 - B_2C \in 
                1 + \p_w^{2r}M_{a_w}(\OO_w), \\
        &C \in 
            \p_w^rM_{b_w \times a_w}(\OO_w), \ \ \ \ \ 
        B \in 
            M_{a_w \times b_w}(\OO_w), \\
        &D = 
            (1+CB_2)D_1 \in 
                (1+\p_w^{2r})M_{b_w}(\OO_w)D_1 \ .
    \end{align*}
    
    Therefore,
    \begin{align*}
        \begin{pmatrix}
            1 & -B_2 \\
            0 & 1
        \end{pmatrix}
        \begin{pmatrix}
            1 & 0 \\
            C_1 & D_1
        \end{pmatrix} =
        \begin{pmatrix}
            1 & 0 \\
            C & 1 + CB_2
        \end{pmatrix}
        \begin{pmatrix}
            1 & 0 \\
            0 & D_1
        \end{pmatrix}
        \begin{pmatrix}
            A & AB \\
            0 & 1
        \end{pmatrix}
    \end{align*}
	
    Setting
    \begin{align*}	
        \gamma_0 = 
        \begin{pmatrix}
            A & AB \\
            0 & 1
        \end{pmatrix} 
        \ \ \ \text{ and } \ \ \
        \cg{\gamma}_0 =
        \begin{pmatrix}
            1 & 0 \\
            C & 1 + CB_2
        \end{pmatrix} \ ,
    \end{align*}
    one obtains
    \begin{align*}
        \la 
            &\pi_w(Z_1)\phi_w, 
            \pi^\vee_w(Z_2) \phi^\vee_w 
        \ra_{\pi_w} \\
            &= 
        \la 
            \pi_w
            \left(
                \gamma_0
                \begin{pmatrix}
                    A_1 & 0 \\ 
                    0 & 1
                \end{pmatrix}
            \right) 
            \phi_w,
            \pi^\vee_w
            \left(
                \begin{pmatrix}
                    1 & 0 \\ 
                    0 & D_1^{-1}
                \end{pmatrix}
                \cg{\gamma}_0
                \begin{pmatrix}
                    A_2 & 0 \\ 
                    0 & D_2
                \end{pmatrix}
            \right) 
            \phi^\vee_w
        \ra_{\pi_w}
    \end{align*}
	
    The desired result follows since $\gamma_0, \tp{\cg{\gamma}_0} \in I_{w,r}$.
\end{proof}

\begin{proposition} \label{prop:pairing pi phi}
    Under the conditions of Lemma \ref{lma:Z_i conditions for Phi 1 Phi 2 nonzero}, we have 
    \[
        \Phi_{w,1}(Z_1', Z_2'') 
        \Phi_{w,2}(Z_2', Z_1'') 
        \la 
            \pi_w(Z_1) 
            \phi_w, 
            \pi^\vee_w(Z_2) 
            \phi^\vee_w
        \ra_{\pi_w} =
        \vol(I^0_{a_w,b_w, r}) 
        \cdot J_{a_w} 
        \cdot J_{b_w}
    \]
    where
    \begin{align*}
        J_{a_w} 
        &= 
            \mu_{a_w}(A_1) 
            \Phi_w^{(1)}(A_2)
            \absv{
                \det A_2
            }_w^{b_w/2}
            \la 
                \pi_{a_w}(A_1) 
                \phi_{a_w}, 
                \pi^\vee_{a_w}(A_2) 
                \phi^\vee_{a_w}
            \ra_{\pi_{a_w}}, 
        \\
        J_{b_w} 
        &= 
            \mu_{b_w}(D_2) 
            \Phi_w^{(4)}(D_1)
            \absv{
                \det D_1
            }_w^{a_w/2}
            \la 
                \pi_{b_w}(D_1) 
                \phi_{b_w},
                \pi^\vee_{b_w}(D_2)
                \phi^\vee_{b_w}
            \ra_{\pi_{b_w}}
    \end{align*}
\end{proposition}

\begin{proof}
    Using the conditions on $Z_1$ and $Z_2$, we have
    \begin{align*}	
        &\la 
            \pi_w(Z_1)
            \phi_w, 
            \pi^\vee_w(Z_2) 
            \phi^\vee_w 
        \ra_{\pi_w} \\
        &= 
        \la 
            \pi_w
            \left(
                \begin{pmatrix}
                    A_1 & 0 \\ 
                    0 & 1
                \end{pmatrix}
            \right) 
            \phi_w,
            \pi^\vee_w
            \left(
                \begin{pmatrix}
                    A_2 & 0 \\ 
                    0 & D_1^{-1}D_2
                \end{pmatrix}
            \right) 
            \phi^\vee_w
        \ra_{\pi_w} \\
        &= 
        \int_{\GL_n(\OO_w)} 
            (
                \varphi_w
                \left(
                    k
                    \begin{pmatrix}
                        A_1 & 0 \\ 
                        0 & 1
                    \end{pmatrix}
                \right),
                \varphi^\vee_w
                \left(
                    k
                    \begin{pmatrix}
                        A_2 & 0 \\ 
                        0 & D_1^{-1}D_2
                   \end{pmatrix}
                \right)
            )_{w} 
        d^\times k \ ,
    \end{align*}
    using Equation \eqref{inner product pi w integral}.
	
    As the support of $\varphi_w$ is $P_{a_w,b_w} I_{w,r}$ and its intersection with $\GL_n(\OO_w)$ is equal to $I^0_{a_w,b_w,r}$, the integrand above is nonzero if and only if $k \in I^0_{a_w,b_w,r}$. Write such a $k \in I^0_{a_w,b_w,r}$ as
    \[
        k = 
        \begin{pmatrix}
            1 & B \\ 0 & 1
        \end{pmatrix}
        \begin{pmatrix}
            A & 0 \\ 0 & D
        \end{pmatrix}
        \begin{pmatrix}
            1 & 0 \\ C & 1
        \end{pmatrix}
    \]
    with 
    $
        A \in \GL_{a_w}(\OO_w)
    $, 
    $
        D \in \GL_{b_w}(\OO_w)
    $, 
    $
        B \in M_{a_w \times b_w}(\OO_w)
    $ and 
    $
        C \in \p_w^r M_{b_w \times a_w}(\OO_w)
    $.

    An short computation shows that
    \begin{align*}
        \varphi_w
        \left(
            k
            \begin{pmatrix}
                A_1 & 0 \\ 
                0 & 1
            \end{pmatrix}
        \right)
        &= 
        \varphi_w \left( 
        \begin{pmatrix}
            AA_1 & 0 \\ 0 & D
        \end{pmatrix}
        \right),\text{ and} \\
        \varphi^\vee_w
        \left(
            k
            \begin{pmatrix}
                A_2 & 0 \\ 
                0 & D_1^{-1}D_2
            \end{pmatrix}
        \right) 
        &=
        \varphi^\vee_w
        \left(
            \begin{pmatrix}
                AA_2 & 0 \\ 
                0 & DD_1^{-1}D_2
            \end{pmatrix}
        \right).
    \end{align*}
    
    Observe that the determinant of the matrices $A$, $D$, $A_1$ and $D_2$ are all integral $p$-adic units. Therefore, using the definition of $\varphi_w$ (resp. $\varphi^\vee_w$) and its relation to $\phi_{a_w} \otimes \phi_{b_w}$ (resp. $\cg{\phi}_{a_w} \otimes \cg{\phi}_{b_w})$, the integrand above is equal to
    \begin{align*}
        &\absv{
            \det A_2
        }_w^{b_w/2} 
        \absv{
            \det D_1^{-1}
        }_w^{-a_w/2} 
            \\ &\times
        \la 
            \pi_{a_w}(AA_1) \phi_{a_w}
            \otimes 
            \pi_{b_w}(D) \phi_{b_w},
            \pi^\vee_{a_w}(AA_2) \phi^\vee_{a_w} 
            \otimes 
            \pi^\vee_{b_w}(DD_1^{-1}D_2) \phi^\vee_{b_w}
        \ra_{a_w,b_w} 
            \\ =\,
        &\absv{
            \det A_2
        }_w^{b_w/2} 
        \absv{
            \det D_1
        }_w^{a_w/2}
            \\ &\times
        \la 
            \pi_{a_w}(A_1) \phi_{a_w}, 
            \pi^\vee_{a_w}(A_2) \phi^\vee_{a_w}
        \ra_{\pi_{a_w}}
        \la 
            \pi_{b_w}(D_1) \phi_{b_w},
            \pi^\vee_{b_w}(D_2) \phi^\vee_{b_w}
        \ra_{\pi_{b_w}} \ ,
    \end{align*}
    which does not depend on $k \in I_{a_w, b_w, r}^0$. The result follows by using the second part of Lemma \ref{lma:Z_i conditions for Phi 1 Phi 2 nonzero}.
\end{proof}

\begin{corollary} \label{cor:Corollary Zw Ia Ib}
    The zeta integral $Z_w$ is equal to
    \[
        \dim \tau_w
            \cdot
        \frac{
             \vol(I^0_{a_w, b_w, r})
        }{
            \vol(I_{a_w, r}^0)\vol(I_{b_w, r}^0)
        }
            \cdot 
        \mathcal{I}_{a_w}
            \cdot 
        \mathcal{I}_{b_w}
    \]
    where
    \begin{align*}
        \mathcal{I}_{a_w} = 
        \int_{I_{a_w, r}^0} 
            &\int_{\GL_{a_w}(\KK_w)} 
                \mu_{a_w}(A_1) 
                \chi_{w,2}(A_1) 
                \chi_{w,1}^{-1}(A_2) 
                \\
                &\times
                \Phi_w^{(1)}(A_2) 
                \absv{
                    \det A_2
                }_w^{s + \frac{a_w}{2}} 
                \la 
                    \pi_{a_w}(A_1) 
                    \phi_{a_w}, 
                    \pi^\vee_{a_w}(A_2)
                    \phi^\vee_{a_w} 
                \ra_{\pi_{a_w}} 
            d^\times A_2 
        d^\times A_1 \\
        \mathcal{I}_{b_w} = 
        \int_{I_{b_w, r}^0} 
            &\int_{\GL_{b_w}(\KK_w)} 
                \mu_{b_w}(D_2) 
                \chi_{w,2}(D_1) 
                \chi_{w,1}^{-1}(D_2)
                \\
                &\times
                \Phi_w^{(4)}(D_1) 
                \absv{
                    \det D_1
                }_w^{s + \frac{b_w}{2}} 
                \la 
                    \pi_{b_w}(D_1) 
                    \phi_{b_w}, 
                    \pi^\vee_{b_w}(D_2) 
                    \phi^\vee_{b_w} 
                \ra_{\pi_{b_w}} 
            d^\times D_1 
        d^\times D_2
    \end{align*}
\end{corollary}

\begin{proof}
    Using Lemma \ref{lma:Z_i conditions for Phi 1 Phi 2 nonzero} and Proposition \ref{prop:pairing pi phi},
    \begin{align*}
        Z_w 
        &= 
            \frac{\dim \tau_w}{\vol(\GG_w)} \\
            &\times 
            \int_{
                A_1, B_1, C_1, A_2, B_2, D_1, C_2, D_2
            } 
                \chi_{w,2}(A_1)
                \chi_{w,2}(D_1) 
                \chi_{w,1}^{-1}(A_2)
                \chi_{w,1}^{-1}(D_2) \\
                &\times 
                \absv{
                    \det A_1 
                    \det D_1 
                    \det A_2 
                    \det D_2
                }_w^{s+\frac{n}{2}} \\
                &\times 
                \vol(
                    I^0_{a_w, b_w, r}
                )
                J_{a_w} J_{b_w} \\
            &\times 
            \absv{
                \det A_1^{b_w} 
                \det D_1^{-a_w}
            } 
            d^\times A_1 
            dB_1 
            dC_1 
            d^\times D_1 \\
            &\times 
            \absv{
                \det A_2^{-b_w}
                \det D_2^{a_w}
            } 
            d^\times A_2 
            dB_2 
            dC_2 
            d^\times D_2
    \end{align*}
    where the domain of integration for the matrices $A_i$, $B_i$, $C_i$ and $D_i$ ($i=1,2$) is given by the conditions of Lemma \ref{lma:Z_i conditions for Phi 1 Phi 2 nonzero}.
	
    Note that the integrand is independent of $B_1 \in M_{a_w \times b_w}(\OO_w)$, $B_2 \in \GG_{u,w}$, $C_1 \in \GG_{l,w}$ and $C_2 \in M_{b_w \times a_w}(\OO_w)$. Moreover, the determinants of the matrices $A_1$ and $D_2$ are both $p$-adic units. Therefore, the above simplifies to
    \begin{align*}	
        Z_w 
        &= 
            \frac{\dim \tau_w}{\vol(\GG_w)} 
            \vol(I^0_{a_w, b_w, r})
            \vol(M_{a_w \times b_w}(\OO_w))^2 \vol(\GG_{l,w}) 
            \vol(\GG_{u,w}) \\
            &\times 
            \int_{I_{a_w, r}^0} 
            \int_{I_{b_w, r}^0} 
            \int_{\GL_{a_w}(\KK_w)} 
            \int_{\GL_{b_w}(\KK_w)}
                \chi_{w,2}(A_1)
                \chi_{w,2}(D_1) 
                \chi_{w,1}^{-1}(A_2)
                \chi_{w,1}^{-1}(D_2) \\
                &\times 
                \mu_{a_w}(A_1) 
                \Phi_w^{(1)}(A_2)
                \absv{
                    \det A_2
                }_w^{s+\frac{a_w}{2}}
                \la 
                    \pi_{a_w}(A_1) 
                    \phi_{a_w}, 
                    \pi^\vee_{a_w}(A_2) 
                    \phi^\vee_{a_w}
                \ra_{\pi_{a_w}} \\
                &\times
                \mu_{b_w}(D_2) 
                \Phi_w^{(4)}(D_1)
                \absv{
                    \det D_1
                }_w^{s+\frac{b_w}{2}}
                \la 
                    \pi_{b_w}(D_1)
                    \phi_{b_w},
                    \pi^\vee_{b_w}(D_2) 
                    \phi^\vee_{b_w}
                \ra_{\pi_{b_w}} \\
            &\times 
            d^\times D_1 
            d^\times A_2 
            d^\times D_2 
            d^\times A_1 \ ,
    \end{align*}	
    and using the decomposition $\GG_w = \GG_{l,w} (I_{a_w, r}^0 \times I_{b_w, r}^0) \GG_{u,w}$, the result follows immediately.
\end{proof}

\begin{theorem} \label{thm:main thm - comp of Zw above p}
    The integrals $\mathcal{I}_{a_w}$ and $\mathcal{I}_{b_w}$ are equal to
    \begin{align*}
        \mathcal{I}_{a_w} &= 
        \displaystyle{
            \frac{
                \epsilon(
                    -s + \frac{1}{2}, 
                    \pi_{a_w} 
                    \otimes 
                    \chi_{w,1}
                ) 
                L(
                    s + \frac{1}{2}, 
                    \pi^\vee_{a_w} 
                    \otimes 
                    \chi_{w,1}^{-1}
                )
            }{
                L(
                    -s+\frac{1}{2},
                    \pi_{a_w} 
                    \otimes 
                    \chi_{w,1}
                )
            } \cdot 
            \frac{
                \vol(\X^{(1)}) \vol(I_{a_w, r}^0)
            }
            {
                (\dim \tau_{a_w})^2
            } 
            \la 
                \phi_{a_w}, 
                \phi^\vee_{a_w} 
            \ra_{\pi_{a_w}}
        } \\
        \mathcal{I}_{b_w} &= 
        \displaystyle{
            \frac{
                L(
                    s+\frac{1}{2},
                    \pi_{b_w} 
                        \otimes 
                    \chi_{w,2})
            }{
                \epsilon(
                    s + \frac{1}{2},
                    \pi_{b_w} 
                        \otimes 
                    \chi_{w,2}) 
                L(
                    -s + \frac{1}{2},
                    \pi^\vee_{b_w} 
                        \otimes 
                    \chi_{w,2}^{-1})
            } \cdot 			
            \frac{
                \vol(\X^{(4)}) \vol(I_{b_w, r}^0)
            }
            {
                (\dim \tau_{b_w})^2
            } \la 
                \phi_{b_w}, 
                \phi^\vee_{b_w} 
            \ra_{\pi_{b_w}}
        }
    \end{align*}
    Therefore, by setting 
    \begin{align*}
        &L\left(
            s + \frac{1}{2},
            \ord,
            \pi_w, 
            \chi_w
        \right) 
            \\&\;\;\;\;:=\,
        \displaystyle{
            \frac{
                \epsilon(
                    -s + \frac{1}{2},
                    \pi_{a_w}
                        \otimes 
                    \chi_{w,1}) 
                L(
                    s + \frac{1}{2},
                    \pi^\vee_{a_w}
                        \otimes 
                    \chi_{w,1}^{-1})
                L(
                    s+\frac{1}{2}, 
                    \pi_{b_w} 
                        \otimes 
                    \chi_{w,2})
            }{
                L(
                    -s+\frac{1}{2}, 
                    \pi_{a_w} 
                    \otimes 
                    \chi_{w,1})
                \epsilon(
                    s + \frac{1}{2},
                    \pi_{b_w} 
                    \otimes 
                    \chi_{w,2}
                ) 
                L(
                    -s + \frac{1}{2},
                    \cg{\pi}_{b_w} 
                    \otimes 
                    \chi_{w,2}^{-1}
                )
            }
        }
    \end{align*}
    one has
    \[
        Z_w = 
        \frac{1}{\dim \tau_w}
        L\left(
            s + \frac{1}{2},
            \ord, 
            \pi_w, 
            \chi_w
        \right) 
            \cdot 
        \frac{
            \vol(I_{w, r}^0)
            \vol(\tp{I}_{w, r}^0)
        }{
            \vol(I_{w, r}^0 \cap 
            \tp{I}_{w, r}^0)
        }
            \cdot 
        \la 
            \varphi_w, 
            \varphi^\vee_w
        \ra_{\pi_w}
    \]
\end{theorem}

\begin{proof}
    This proof is inspired by the argument of \cite[Theorem 4.3.10]{EHLS}. First, write
    \begin{align*}
        \mathcal{I}_{a_w} = 
        \int_{I_{a_w, r}^0}
            \mu_{a_w}(A_1) 
            \chi_{w,2}(A_1)
             \mathcal{I}_{a_w, 2}(A_1)
        d^\times A_1 \ ,
    \end{align*}
    where $\mathcal{I}_{a_w, 2} = \mathcal{I}_{a_w, 2}(A_1)$ is defined as
    \begin{align*}
        \int_{\GL_{a_w}(\KK_w)} 
                \Phi_w^{(1)}(A_2) 
                \absv{
                    \det A_2
                }_w^{s + \frac{a_w}{2}}
                \la
                    \pi_{a_w}(A_1)
                    \phi_{a_w}, 
                    (
                        \chi_{w,1}^{-1} 
                            \otimes 
                        \pi^\vee_{a_w}
                    )(A_2)
                    \phi^\vee_{a_w}
                \ra_{\pi_{a_w}} 
            d^\times A_2 \ .
    \end{align*}
    
    The above is a ``Godement-Jacquet'' integral, as defined in \cite[Equation (1.1.3)]{Jac79}. Therefore, we use its functional equation to obtain
    \begin{align*}
        \mathcal{I}_{a_w,2} 
            &= 
        \frac{
            \epsilon(
                -s + \frac{1}{2},
                \pi_{a_w} 
                \otimes 
                \chi_{w,1}
            ) 
            L(
                s + \frac{1}{2},
                \pi^\vee_{a_w} 
                \otimes 
                \chi_{w,1}^{-1}
            )
        }{
            L(
                -s+\frac{1}{2},
                \pi_{a_w} 
                \otimes 
                \chi_{w,1}
            )
        } \\
        &\times
        \int_{\GL_{a_w}(\KK_w)}
            \left( 
                \Phi_w^{(1)} 
            \right)^{\wedge}(A_2)
            \absv{
                \det A_2
            }_w^{-s + \frac{a_w}{2}}
        \\&\times\,
            \chi_{w,1}(A_2) 
            \la 
                \pi_{a_w}(A_1) 
                \phi_{a_w},
                \pi^\vee_{a_w}(A_2^{-1})
                \phi^\vee_{a_w}
            \ra_{\pi_{a_w}}
        d^\times A_2
    \end{align*}
    
    Let $L_{a_w, \ord}$ denote the quotient of $L$-factors and $\epsilon$-factors leading the expression above. Recall that 
    $
        \left( 
            \Phi_w^{(1)} 
        \right)^{\wedge}(A_2)
    $ is supported on $\X^{(1)}$. Furthermore, for $A_2 \in \X^{(1)}$, we have 
    $
        \left( 
            \Phi_w^{(1)} 
        \right)^{\wedge}(A_2)
            =
        \nu_{a_w}(A_2)
    $
    and 
    $
        \absv{
            \det A_2
        }_w = 1
    $. Then, $\mathcal{I}_{a_w,2}$ is equal to
    \begin{align*}
        L_{a_w, \ord} \times	
        \int_{\X^{(1)}}
            \chi_{w,1}(A_2)
            \nu_{a_w}(A_2)
            \la 
                \pi_{a_w}(A_1) 
                \phi_{a_w},
                \pi^\vee_{a_w}(A_2^{-1}) 
                \phi^\vee_{a_w}
            \ra_{\pi_{a_w}}
        d^\times A_2
    \end{align*}

    By definition of $\X^{(1)}$, we can write $A_2 = \gamma_1k_2\gamma_2$ uniquely for some $k_2 \in K_{a_w}$, $\gamma_1 \in \X^{(1)}_l := \tp{I}^0_{a_w, r} \cap \tp{P}_{a_w}^u$, and $\gamma_2 \in \X^{(1)}_u := I_{a_w, r}^0 \cap P_{a_w}^u$. It follows that
    \begin{align*}
        \la 
            \pi_{a_w}(A_1) 
            \phi_{a_w},
            \pi^\vee_{a_w}(A_2^{-1}) 
            \phi^\vee_{a_w}
        \ra_{\pi_{a_w}} 
            &=
        \la 
            \pi_{a_w}(k_2 \gamma_2 A_1) 
            \phi_{a_w},
            \pi^\vee_{a_w}(\gamma_1^{-1}) 
            \phi^\vee_{a_w}
        \ra_{\pi_{a_w}}
            \\ &=
        \la 
            \pi_{a_w}(k_2 A_1) 
            \phi_{a_w},
            \phi^\vee_{a_w}
        \ra_{\pi_{a_w}}
            \\ &=
        \int_{\GL_{a_w}(\OO_w)}
            (
                \varphi_{a_w} (k k_2 A_1),
                \varphi^\vee_{a_w}(k)
            )_{a_w}
        d^\times k
    \end{align*}
    
    The support of $\varphi_{a_w}$ is $P_{a_w} I_{a_w, r} = P_{a_w} I^0_{a_w, r}$. Since $k_2 A_1 \in I^0_{a_w, r}$, the integrand vanishes unless $k \in P_{a_w} I_{a_w, r} \cap \GL_{a_w}(\OO_w) = I^0_{a_w, r}$. Using the fact that such $k$ is in $P_{a_w}$ as well as Equation \eqref{def varphi aw}, we obtain
    \[
        (
            \varphi_{a_w} (k k_2 A_1),
            \varphi^\vee_a(k)
        )_{a_w}
        =
        (
            \varphi_{a_w} (k_2 A_1),
            \varphi^\vee_{a_w}(1)
        )_{a_w}
        =
        (
            \tau_{a_w}(k_2 A_1) 
            \phi^0_{a_w},
            \phi^{\vee,0}_{a_w}
        )_{a_w} .
    \]
    
    Then, using the above, Equation \eqref{eq:def extn mu aw}, the definition of $\nu_{a_w}$, and orthogonality relations of matrix coefficients, we obtain
    \begin{align*}
        \mathcal{I}_{a_w, 2}
            &=
        L_{a_w, \ord} \times
        \int_{\X^{(1)}}
            \mu'_{a_w}(A_2)
            \la 
                \pi_{a_w}(A_1) 
                \phi_{a_w},
                \pi^\vee_{a_w}(A_2^{-1}) 
                \phi^\vee_{a_w}
            \ra_{\pi_{a_w}}
        d^\times A_2 \\
            &=
        L_{a_w, \ord}
        \vol(I_{a_w, r}^0)
        \vol(\X^{(1)}_l) 
        \vol(\X^{(1)}_u) \\
            &\times
        \int_{K_{a_w}}
            ( 
                \phi^0_{a_w},
                \tau^\vee_{a_w}(k_2)
                \phi^{\vee,0}_{a_w}
            )_{a_w}
            (
                \tau_{a_w}(k_2)
                \tau_{a_w}(A_1) 
                \phi^0_{a_w},
                \phi^{\vee,0}_{a_w}
            )_{a_w}
        d^\times k_2 \\
            &=
        L_{a_w, \ord}
        \vol(I_{a_w, r}^0)
        \frac{
            \vol(\X^{(1)})
        }{
            \dim \tau_{a_w}
        }
        ( 
            \phi^0_{a_w},
            \phi^{\vee,0}_{a_w}
        )_{a_w}
        (
            \tau_{a_w}(A_1) 
            \phi^0_{a_w},
            \phi^{\vee,0}_{a_w}
        )_{a_w}
    \end{align*}
    
    Using Equation \eqref{eq:size phi a w}, orthogonality relations of matrix coefficients once more, and the normalization $(\phi^0_{a_w}, \phi^{\vee,0}_{a_w})_{a_w} = 1$, we ultimately obtain that $\mathcal{I}_{a_w}$ is equal to
    \begin{align*}
        &L_{a_w, \ord}
        \la
            \phi_{a_w}, 
            \phi^\vee_{a_w} 
        \ra_{\pi_{a_w}}
        \frac{
            \vol(\X^{(1)})
        }
        {
            \dim \tau_{a_w}
        } 
        \int_{I_{a_w, r}^0} 
            \mu'_{a_w}(A_1)
            (
                \tau_{a_w}(A_1) 
                \phi^0_{a_w},
                \phi^{\vee,0}_{a_w}
            )_{a_w}
        d^\times A_1 \\
        = 
        &L_{a, \ord}
        \frac{
            \vol(\X^{(1)}) 
            \vol(I_{a_w, r}^0)
        }
        {
            (\dim \tau_{a_w})^2
        } 
        \la 
            \phi_{a_w}, 
            \phi^\vee_{a_w} 
        \ra_{\pi_{a_w}}
    \end{align*}
    
    A similar argument yields
    \begin{align*}
        \mathcal{I}_{b_w}
        &=
        L_{b_w, \ord}
        \frac{
            \vol(\X^{(4)})
            \vol(I_{b_w, r}^0)
        }
        {
            (\dim \tau_{b_w})^2
        }
        \la 
            \phi_{b_w}, 
            \phi^\vee_{b_w} 
        \ra_{\pi_{b_w}} \ ,
    \end{align*}
    where
    \begin{align*}
        L_{b_w, \ord} =
        \frac{
            L(
                s+\frac{1}{2},
                \pi_{b_w} 
                \otimes 
                \chi_{w,2}
            )
        }{
            \epsilon(
                s + \frac{1}{2},
                \pi_{b_w} 
                \otimes 
                \chi_{w,2}
            ) 
            L(
                -s + \frac{1}{2},
                \pi^\vee_{b_w} 
                \otimes 
                \chi_{w,2}^{-1}
            ) \ .
        }
    \end{align*}
    
    Therefore, the result follows using Corollary \ref{cor:Corollary Zw Ia Ib}, Equation \eqref{eq:inner prod Pwaord vectors at w}, and the identity
    \[
        \vol(\X^{(1)})\vol(\X^{(4)}) 
        = 
            \frac{
                \vol(I_{a_w, r}^0)
                \vol(\tp{I}_{a_w, r}^0)
                \vol(I_{b_w, r}^0)
                \vol(\tp{I}_{b_w, r}^0)
            }{
                \vol(I_{a_w, r}^0 \cap 
                \tp{I}_{a_w, r}^0)
                \vol(I_{b_w, r}^0 \cap 
                \tp{I}_{b_w, r}^0)
            } 
        =
            \frac{
                \vol(I_{w, r}^0)
                \vol(\tp{I}_{w, r}^0)
            }{
                \vol(I_{w, r}^0 \cap 
                \tp{I}_{w, r}^0)
            }    
    \]
\end{proof}

\subsubsection{Main Local Theorem.} \label{subsubsec:Main local thm}
Keeping with the notation of Theorem \ref{thm:main thm - comp of Zw above p}, define
\[
    I_p\left( 
        s+\frac{1}{2}, \Pord, \pi, \chi
    \right)
        := 
    \prod_{w \in \Sigma_p} 
        L\left( 
            s+\frac{1}{2}, \Pord, \pi_w, \chi_w
        \right).
\]

Then, from Theorem \ref{thm:main thm - comp of Zw above p} and \eqref{eq:def local zeta integrals}, we immediately obtain our main result.

\begin{theorem} \label{thm:main local thm}
    Let $\chi$ be a unitary Hecke character of $\KK$, $\chi_p = \otimes_{w \mid p} \chi_w$, and let $s \in \CC$. Let $f_p = f_p(\bullet; \tau, \chi_p, s) \in I_p(\chi_p, s)$ be the local Siegel-Weil section at $p$ in \eqref{eq:def SW f p tau chi p}. Let $\varphi_p \in \pi_p$ and $\varphi^\vee_p \in \pi^\vee_p$ be the test vectors defined in \eqref{eq:def test vectors varphi p and varphi vee p}.
    
    Then, the $p$-adic local zeta integral $I_p(\varphi_p, \varphi^\vee_p, f_p; \chi_p, s)$ from \eqref{eq:def local zeta integrals} is equal to
    \begin{equation} \label{eq:def I p}
        \frac{1}{\dim \tau}
        I_p\left( 
            s+\frac{1}{2}, \Pord, \pi, \chi
        \right)
            \cdot
        \frac{
            \vol(I_{P, r}^0)
            \vol(\tp{I}_{P, r}^0)
        }{
            \vol(I_{P, r}^0 \cap 
            \tp{I}_{P, r}^0)
        }
    \end{equation}
\end{theorem}

\begin{remark}
    Using the same minor manipulation explained in \cite[Remark 4.3.11]{EHLS}, we see that the $p$-Euler factor $I_p(s+\frac{1}{2}, \Pord, \pi_p, \chi_p)$ takes the form of a modified Euler factor at $p$ as predicted in \cite[Section 2, Equation (18b)]{Coa89} for the conjectures of Coates and Perrin-Riou on $p$-adic $L$-functions.
\end{remark}

\subsection{Local zeta integrals at $\infty$.} \label{subsec:Local zeta integrals at oo}
Assume the unitary character $\chi$ satisfies Hypothesis \ref{hyp:hyp integer k} for some $k \geq 0$. Let $f_{\infty, \kappa} = f_{\infty, \kappa}(\bullet; \gimel, \chi_\infty, s)$ be the Siegel-Weil section in \eqref{eq:def SW f oo kappa gimel}. Let $\varphi_\infty \in \pi_\infty$ and  $\varphi_\infty^\vee \in \pi_\infty^\vee$ be test vectors as in \eqref{eq:def test vectors varphi oo and varphi vee oo}.

If $k \geq n$ and $(\kappa, \chi)$ is critical, then \cite[Theorem 1.3.1]{EisLiu20} yields
\begin{align*}
    &Z_\infty(
        \varphi_\infty,
        \varphi_\infty^\flat,
        f_{\infty, \kappa};
        \chi_\infty,
        s
    )|_{s = \frac{k-n}{2}}
       \\ &=
    \frac{
        A(\pi_\infty, \chi_\infty)
    }{
        \left(
            2^{(n-1)n}
            (-2\pi i)^{-n k} 
            \pi^{\frac{n(n-1)}{2}}
            \prod \limits_{j=0}^{n-1}
                \Gamma(k - j)
        \right)^{[\KK^+:\QQ]}
    }
    E_\infty\left(
        \frac{k-n+1}{2},
        \pi,
        \chi
    \right)\,
\end{align*}
where $A(\pi_\infty, \chi_\infty)$ is some algebraic number depending on $\pi_\infty$ and $\chi_\infty$, and $E_\infty$ is the modified archimedean Euler factor (both introduced in \cite[Section 1.3]{EisLiu20}). Let $D_\infty(\pi_\infty, \chi_\infty)$ denote the fraction on the right-hand side.

\begin{remark}
    The denominator of the leading term on the right-hand side of the equation above also appears in the archimedean Fourier coefficients of the Siegel Eisenstein series associated to $f_\infty = f_\infty(\bullet; \gimel, \chi_\infty, \frac{n-k}{2})$ (the holomorphic Siegel-Weil section introduced in Section \ref{subsubsec:canonical automorphy factors for GU n n}), see \eqref{eq:archimedean fourier coeff expression}. 
    
    We later normalize this Siegel-Weil section so that this terms does not appear in either the local archimedean zeta integral nor Fourier coefficient.
\end{remark}

We set
\begin{equation} \label{eq:def I infty}
    I_\infty\left(
        \frac{k-n+1}{2},
        \pi,
        \chi
    \right)
        :=
    A(\pi_\infty, \chi_\infty)
    E_\infty\left(
        \frac{k-n+1}{2},
        \pi,
        \chi
    \right)\,.
\end{equation}

\subsection{Local zeta integrals away from $p$ and $\infty$.} \label{subsec:Local zeta integral away from p and oo}

\subsubsection{Local zeta integrals at finite unramified places.} \label{subsubsec:Local zeta integrals at finite unramd places}

For each $l \notin S \cup \{p\}$, let $\varphi_l$ and $\varphi_l^\vee$ be local test vectors at $l$ as in Section \ref{subsubsec:Local test vectors at finite unramd places}. Similarly, let $f_l \in I_l(\chi_l, s)$ be as in Section \ref{subsubsec:Local SW section at finite unramd places}.

It follows from \cite{Jac79}, \cite[Section 6]{GPSR87} and \cite[Section 3]{Li92} that
\[
    d^{S,p}(s + \frac{1}{2}, \chi) 
    \prod_{l \notin S \cup \{p\}}
        I_l(\varphi_l, \varphi_l^\vee, f_l, s)
        =
    L^{S, p}(s+\frac{1}{2}, \pi, \chi)\,,
\]
where $d^{S,p}(s, \chi) = \prod_{l \notin S \cup \{p\}} \prod_{v \mid l} d_v(s, \chi)$ and
\begin{equation} \label{eq:def d v s chi}
    d_v(s, \chi)
        =
    \prod_{r=1}^{n}
        L_v(
            2s + n - r,
            \chi^+
                \cdot
            \eta^r
        )\,,
\end{equation}
where $\chi^+$ is the restriction of $\chi$ to $\AA_{\KK^+}$, and $\eta = \eta_{\KK/\KK^+}$ is the quadratic character of $\AA_{\KK^+}$ associated to the extension $\KK/\KK^+$. For more details, see \cite[Section 4.2.1]{EHLS}.

\subsubsection{Local zeta integrals at finite ramified places.} \label{subsubsec:Local zeta integrals at finite ramd places}

For each $l \in S$, let $\varphi_l = \bigotimes_{v \mid l} \varphi_v$ and $\varphi_l^\vee = \bigotimes_{v \mid l} \varphi_v^\vee$ be local test vectors at $l$ as in Section \ref{subsubsec:Local test vectors at finite ramd places}. Similarly, let $f_l = \bigotimes_{v \mid l} f_v \in I_l(\chi_l, s)$ be as in Section \ref{subsubsec:Local SW section at finite ramd places}.

For each place $v \mid l$ of $\KK^+$, let $\mathcal{U}_v$ be the open neighborhood of $-1$ in $K_v$ as in Section \ref{subsubsec:Local SW section at finite ramd places}. It follows from \cite[Lemma 4.2.3]{EHLS} that
\[
    I_l(\varphi_l, \varphi_l^\vee, f_l, \chi)
        =
    \prod_{v \mid l}
        \vol(\mathcal{U}_v)\,,
\]
where the volume is respect to the local Haar measure discussed in Section \ref{subsubsec:conv on measures}. We normalize the product over all primes in $S$ as
\begin{equation} \label{eq:def I S}
    I_S 
        = 
    \prod_{l \in S}
    \prod_{v \mid l}
        d_v(s+\frac{1}{2}, \chi)
        \vol(\mathcal{U}_v)\,,
\end{equation}
where $d_v(s, \chi)$ is defined as in \eqref{eq:def d v s chi}. Note that $I_S$ is independent of $\pi$, see Remark \ref{rmk:test vectors, SW section and volume}.

\subsection{Global formula.} \label{subsec:Global formula}

Let $D^{p, \infty}(\chi) = \prod_{v \nmid p \infty} d_v(s+\frac{1}{2}, \chi)$, where the product runs over all finite places $v$ of $\KK^+$ away from $p$, and $d_v(s, \chi)$ is again as in \eqref{eq:def d v s chi}.

\begin{theorem} \label{thm:main formula global zeta integral}
    Let $\pi$ be a $P$-anti-ordinary, anti-holomorphic cuspidal automorphic representation for $G_1$ of $P$-anti-WLT $(\kappa, K_r, \tau)$. Let $S = S(K^p)$ as in Section \ref{subsubsec:ramified places away from p} and let $\chi$ be a unitary Hecke character of type $A_0$ satisfying Hypothesis \ref{hyp:hyp integer k} for some integer $k \geq n$. Assume that $(\kappa, \chi)$ is critical.
    
    Let $\varphi \in \pi$ and $\varphi^\vee \in \pi^\vee$ be test vectors as in Section \ref{sec:explicit choice of P-(anti-)ord vectors}.  Let $f = f_{\chi}^{\tau, \kappa} \in I(\chi, s)$ be as in \eqref{eq:def f tau kappa chi} for $s = \frac{k-n}{2}$. Then,
    \begin{align*}
        &D^{p, \infty}(\chi) 
        I\left(
            \varphi, \varphi^\vee, f; \chi, s
        \right)
            \\ =\,
        &\frac{
            \brkt{\varphi}{\varphi^\vee}
        }
        {
            \dim \tau
        }
            \cdot
        \frac{
            \vol(I_{P,r}^0) \vol(\tp{I}^0_{P,r})
        }
        {
            \vol(I_{P,r}^0 \cap \tp{I}^0_{P, r})
        }
            \cdot
        I_p\left(
            s + \frac{1}{2},
            \Pord,
            \pi, \chi
        \right)
            \\ \times\,
        &D_\infty(\pi_\infty, \chi_\infty) 
        E_\infty\left(
            s + \frac{1}{2};
            \pi, \chi
        \right)
        I_S
        L^S\left(
            s + \frac{1}{2};\pi, \chi_u
        \right)
    \end{align*}
    at $s = \frac{k-n}{2}$.
\end{theorem}

%%%%%%%%%%%%%%%%%%%%%%%%%%%%%%%%%%%%%%%%%%%%%%%%%%%%%%%%%%%%%%%%%%

\section{$P$-ordinary Eisenstein measure.} \label{sec:Pord Eisenstein measure}

We now construct an Eisenstein measure, in the sense of \cite[Section 5]{EHLS}, by $p$-adically interpolating the (holomorphic) Eisenstein series associated to the Siegel-Weil sections chosen in the previous section. To do so, we follow the approach of \cite{Eis15} and in particular use several results of \cite{Shi97}. Therefore, it is convenient to work with a specific choice of basis for the Hermitian vector space associated to $G_4$.

Namely, let $(V, \brkt{\cdot}{\cdot}_V)$ and $(W, \brkt{\cdot}{\cdot}_W)$ be the Hermitian vector spaces associated to $G_{1_{/\RR}}$ and $G_{4_{/\RR}}$ respectively. Let $\mathcal{B}_1 := \{e_1, \ldots, e_n\}$ be any orthogonal basis of $V$ and let $\phi$ be the corresponding diagonal matrix for $\brkt{\cdot}{\cdot}_V$. Let $\mathcal{B}_4 := \{(e_1, 0), \ldots, (e_n, 0), (0, e_1), \ldots, (0, e_n)\}$ be the corresponding basis of $W$. Then, we momentarily identify $G_1(\RR)$ with the group of matrices (written with respect to $\mathcal{B}_1$) preserving some form $\phi$ up to scalar, and $G_4(\RR)$ with the group of matrices preserving
\[
    \begin{pmatrix}
        \phi & 0 \\
        0 & -\phi
    \end{pmatrix} .
\]

Let $\alpha \in K$ be any totally imaginary element, define
\[
    S = \begin{pmatrix}
        1_n & -\frac{\alpha}{2}\phi \\
       -1_n & -\frac{\alpha}{2}\phi
    \end{pmatrix}
\]
and consider the basis $\mathcal{B}_4' := S\mathcal{B}_4$ of $W$. In this section, and only in this section, we identify $G_{4_{/\RR}}$ with the group of matrices preserving the matrix
\[
    \eta := \begin{pmatrix}
        0 & -1_n \\
        1_n & 0
    \end{pmatrix}
\]

This way, the unitary group $U_4(\RR)$ is equal to the group denoted $G(\eta)$ in \cite{Eis15}. The results of Shimura to compute Fourier coefficients of Eisenstein series are stated with respect to this $G(\eta)$, motivating our change in notation. Our choice of local Siegel-Weil sections in Section \ref{subsec:Local SW section at p}, \ref{subsec:Local SW section at oo}, and \ref{subsec:Local SW section away from p and oo} make no mention of an explicit global basis for $V$ or $W$, hence this does not introduce any unintentional technicalities.

Observe that $U_4$ is the restriction of scalar to $\QQ$ of an algebraic group $U := U_{\KK^+}$ on $\KK^+$. In what follows, it is more convenient to work with $U_4(\AA_{\QQ}) = U(\AA_{\KK^+})$. 

Fix a unitary Hecke character $\chi$ that satisfies Hypothesis \ref{hyp:hyp integer k} for some integer $k \geq 0$, a $P$-nebentypus $\tau = \bigotimes_{w \in \Sigma_p} \tau_w$ of level $r \gg 0$ and a finite-order character $\psi = \bigotimes_{w \in \Sigma_p} \psi_w$ of $L_P(\ZZ_p)$. Let $f = f_{\chi, \psi}^\tau \in I(\chi, s)$ be the associated Siegel-Weil section in \eqref{eq:def f tau kappa chi psi}. By construction, its restriction $f_U$ to $U$ factors as
\[
    f_U = \bigotimes_v f_v
\]
where the tensor product runs over all the places $v$ of $\KK^+$.

Let $P_{U, \sgl} \subset U_4$ be the maximal $\QQ$-parabolic subgroup that stabilizes $V^d$, i.e. $P_{U, \sgl} = P_\sgl \cap U$. Its Levi subgroup $M_U$ is identified with $\GL_\KK(V)$ via $\Delta$ or equivalently, with $\GL_n(\KK)$ using the basis $\mathcal{B}_1$. Its unipotent radical $N_U$ is identified with the group of matrices $
    \begin{pmatrix}
        1 & m \\
        0 & 1
    \end{pmatrix}
$, where $m \in \Her_n(\KK)$.

With this notation, we can adapt the content of Section \ref{subsec:Sgl Eis series} by replacing $G_4$ with $U_4$. Let $E_{f, U}$ be the Eisenstein series associated to $f_U$ on $U_4$ or equivalently, the restriction of $E_f$ from $G_4(\AA)$ to $U_4(\AA) = U(\AA_{\KK^+})$.

\subsection{Fourier coefficients of Eisenstein series.} \label{subsec:Fourier coefficients of Eisenstein series}
The Siegel-Weil section $f_U$ satisfies Conditions 4 and 5 of \cite[Section 2.2.3]{Eis15}. Therefore by \cite[Proposition 18.3]{Shi97}, $E_{f, U}$ admits a Fourier expansion: For all $m \in \Her_n(\AA_\KK)$, $h \in \GL_n(\AA_\KK)$, we have
\begin{equation}
    E_{f, U}\left(
        \begin{pmatrix}
            1 & m \\
            0 & 1
        \end{pmatrix}
        \begin{pmatrix}
            \tp{\ol{h}}^{-1} & 0 \\
            0 & h
        \end{pmatrix}
    \right)
        =
    \sum_{\beta \in \Her_n(\KK)}
        c(\beta, h; f_U)
        e_{\AA_{\KK^+}}(\tr \beta m)\,,
\end{equation}
where $c(\beta, h; f_U)$ is a complex number that depends on $f_U$, the Hermitian matrix $\beta$, and $h$.

Furthermore, by \cite[Sections 18.9, 18.10]{Shi97}, for each non-degenerate matrix $\beta$, the Fourier coefficient $c(\beta, h; f)$ factors over the places $v$ of $\KK^+$. More precisely, write $\beta = (\beta_v)_v$ and $h = (h_v)_v$ as $v$ runs overs the places $v$ of $\KK^+$, and define $c(\beta_v, h_v; f_v)$ as
\[
    \int_{\Her(\KK \otimes_{\KK^+} \KK^+_v)}
        f_v
        \left(
            \begin{pmatrix}
               0 & -1 \\
               1 & 0
            \end{pmatrix} 
            \begin{pmatrix}
               1 & N_v \\
               0 & 1
            \end{pmatrix} 
            \begin{pmatrix}
                \tp{\ol{h}}_v^{-1} & 0 \\
                0 & h_v
            \end{pmatrix} 
        \right)
        e_v(-\tr \beta_v N_v)
    d N_v\,.
\]

Then, we have
\[
    c(\beta, h; f) = C(n, \KK) \prod_v c(\beta_v, h_v; f_v) ,
\]
where 
\begin{equation} \label{eq:def C n K}
    C(n, \KK)
        =
    2^{n(n-1)[\KK^+:\QQ]/2}
    \absv{D_{\KK^+}}^{-n/2}
    \absv{D_{\KK}}^{-n(n-1)/4}\,,
\end{equation}
and $dN_v$ denotes the Haar measure on $\Her(\KK \otimes_{\KK^+} \KK^+_v)$ such that
\[
    \int_{\Her_n(\OO_{\KK} \otimes_{\OO_{\KK^+}} \OO_{\KK^+_v})}
        dN_v 
    = 1\,, \text{ for each finite place $v$}\,,
\]
and
\[
    dN_v 
        :=
    \absv{
        \bigwedge_{j=1}^n
            dN_{jj}
        \bigwedge_{j < k}
            2^{-1}
            dN_{jk}
                \wedge
            d\ol{N}_{kj}
    }\,, \text{ for each archimedean place $v$}\,,
\]
where $N_{jk}$ is the $(j,k)$-th entry of the matrix $N_v$.

In the following sections, we generalize the approach of \cite{Eis15} to compute the local Fourier coefficients at $p$ corresponding to the local Siegel-Weil sections associated to types constructed in Section \ref{subsec:Local SW section at p}. Then, we rely on known formulas obtained by Shimura in \cite{Shi97} and extended by Eischen in \cite{Eis15} for the local coefficients at places away from $p$. We later combine these results with the discussion above to $p$-adically interpolate the Eisenstein series $E_{f, U}$.

\subsection{Calculations of local Fourier coefficients} \label{subsec:Calculation of local Fourier coef}
In this section, for each place $v$ of $\KK^+$, we compute $c(\beta_v, h_v; f_v)$. It is more convenient to compute these coefficients for $h_v = 1$. One can use \cite[Lemma 9]{Eis15} to relate $c(\beta, h; f_U)$ to $c(\beta, 1; f_U)$ for arbitrary $h \in \GL_n(\AA_{\KK})$.

\subsubsection{Local coefficients at $p$.} \label{subsubsec:Local coeff at p}
Assume $v \mid p$ and identify $v$ with the unique place $w \mid v$ in $\Sigma_p$. Let $f_v = f_w = f^{\Phi_w}$ be as in \eqref{eq:def f phi w}, for $\Phi_w = \Phi_w^{\tau_w \otimes \psi_w}$, see Remark \ref{rmk:mu tau nu tau and Phi tau}. Then, the local coefficient for $\beta_v = \beta_w$ is
\begin{align*}
    c(\beta_w, 1; f_w) 
        &= 
    \int_{M_n(\KK_w)}
        f_w\left(
            \begin{pmatrix}
                0 & -1 \\
                1 & 0
            \end{pmatrix}
            \begin{pmatrix}
                1 & N \\
                0 & 1
            \end{pmatrix}
        \right)
        e_p(-\tr \beta_w N)
    dN
        \\ &=
    \int_{\GL_n(\KK_w)}
        \chi_{w,1}^{-1}\chi_{w,2}(X)
        \absv{\det X}_w^{n+2s} 
        \\ &\hspace*{0.5cm} \times
        \int_{M_n(\KK_w)}
            \Phi_w
            \left(
                (0, X)
                \begin{pmatrix}
                    0 & -1 \\
                    1 & N
                \end{pmatrix}
            \right)
            e_p(-\tr \beta_w N)
        dN
    d^\times X
\end{align*}

From \eqref{eq:def Schwartz Phi w}, we have
\[
    \Phi_w
    \left(
        (0, X)
        \begin{pmatrix}
            0 & -1 \\
            1 & N
        \end{pmatrix}
    \right)
        =
    \frac{\dim \tau_w}{\vol(\GG_w)}
    \Phi_{1,w}(X)
    \Phi_{2,w}(XN)
\]
which is nonzero if and only if $X \in \GG_w$. It follows that $c(\beta_w, 1; f_w)$ is equal to
\begin{align*}
    &\frac{\dim \tau_w}{\vol(\GG_w)}
    \int_{\GG_w}
        \chi_{w,1}^{-1}\chi_{w,2}(X)
        \Phi_{1,w}(X)
        \int_{M_n(\KK_w)}
            \Phi_{2,w}(XN)
            e_p(-\tr \beta_w N)
        dN
    d^\times X \\
        &=
    \frac{\dim \tau_w}{\vol(\GG_w)}
    \int_{\GG_w}
        \chi_{w,1}^{-1}\chi_{w,2}(X)
        \Phi_{1,w}(X)
        \int_{M_n(\KK_w)}
            \Phi_{2,w}(N)
            e_p(\tr (-\beta_w X^{-1} N))
        dN
    d^\times X \\
        &=
    \frac{\dim \tau_w}{\vol(\GG_w)}
    \int_{\GG_w}
        \chi_{w,1}^{-1}\chi_{w,2}(X)
        \Phi_{1,w}(X)
        (\Phi_{2,w})^\wedge(-\beta_w X^{-1})
    d^\times X \\
        &=
    \frac{\dim \tau_w}{\vol(\GG_w)}
    \int_{\GG_w}
        \chi_{w,1}^{-1}\chi_{w,2}(X)
        \Phi_{1,w}(X)
        \nu_w^{\tau_w \otimes \psi_w}(-\beta_w X^{-1})
    d^\times X\,,
\end{align*}
using \eqref{eq:def Phi 2 w} in the last line and notation as in Remark \ref{rmk:mu tau nu tau and Phi tau} for $\nu_w = \nu_w^{\tau_w \otimes \psi_w}$.

Now, write
\[
    X = 
    \begin{pmatrix}
        A & B \\
        C & D
    \end{pmatrix}
\]
for some $B, \tp{C} \in M_{a_w \times b_w}(\ZZ_p)$, $A \in I_{a_w, r}^0$ and $D \in I_{b_w, r}^0$, so that the above equals
\begin{align*}
    \frac{\dim \tau_w}{\vol(\GG_w)}
    \int_{\GG_w}
        \chi_{w,1}^{-1}\chi_{w,2}(X)
        \mu_{a_w}(A) \mu_{b_w}(D)
        \nu_w(-\beta_w X^{-1})
    d^\times X\,,
\end{align*}
using \eqref{eq:def ext mu w} and \eqref{eq:def Phi 1 w}.

In particular, the above is zero unless $-\beta_w X^{-1} \in \X_w$. Thus, $\beta_w \in M_n(\OO_w)$ and we can write
\[
    \beta_w =
    \begin{pmatrix}
        \beta_1 & \beta_2 \\
        \beta_3 & \beta_4
    \end{pmatrix}
\]
with $\beta_1 \in M_{a_w}(\OO_w)$, $\beta_2, \tp{\beta}_3 \in M_{a_w \times b_w}(\OO_w)$, and $\beta_4 \in M_{b_w}(\OO_w)$. 

In particular,
\[
    -\beta_w X^{-1} 
        \equiv
    \begin{pmatrix}
        -\beta_1 A^{-1} & * \\
        * & -\beta_4 D^{-1}
    \end{pmatrix}
\]
modulo $\p_w^r$, where the precise description of the bottom-left and upper-right corners is irrelevant in what follows.

Using \eqref{eq:def nu w and phi nu w} and the definitions of both $\mu_{a_w}$ and $\mu_{b_w}$, we obtain
\begin{align*}
    c(\beta_w, 1; f_w)
        &= 
    \frac{\dim \tau_w}{\vol(\GG_w)} 
    \chi_{w,1}^{-1}\chi_{w,2}(-\beta_w) 
    \chi_2^{-1}(-\beta_1)
    \chi_1(-\beta_4) \\
        &\times
    \int_{\GG_w}
        \chi_{w,1}(A)\chi_{w,2}^{-1}(D)
        \mu'_{a_w}(A) \mu'_{b_w}(D)
        \mu'_{a_w}(-\beta_1 A^{-1})
        \mu'_{b_w}(-\beta_4 D^{-1})
    d^\times X \,.
\end{align*}

Using orthogonality relations between matrix coefficients, as in the end of the proof of Theorem \ref{thm:main thm - comp of Zw above p}, it follows that
\begin{align*}
    c(\beta, 1; f)
        &= 
    \chi_{w,1}^{-1}\chi_{w,2}(-\beta_w) 
    \chi_2^{-1}(-\beta_1)
    \chi_1(-\beta_4)
    \mu'_{a_w}(-\beta_1)
    \mu'_{b_w}(-\beta_4)\,,
\end{align*}
and using \eqref{eq:def nu w and phi nu w} once more, we ultimately obtain
\begin{align*}
    c(\beta_w, 1; f_w)
        &=
    \nu_w(-\beta_w) \,.
\end{align*}

From now on, we write $\nu_w(\bullet; \tau_w, \psi_w)$ for $\nu_w(\bullet) = \nu_w^{\tau_w \otimes \psi_w}(\bullet)$.

\subsubsection{Local coefficients at $\infty$.} \label{subsubsec:Local coeff at oo}

Assume Hypothesis \ref{hyp:hyp integer k} and consider the Siegel-Weil section $f_{\infty, \gimel} = f_{\infty}(\bullet; \gimel, \chi_\infty, s)$ defined at the end of Section \ref{subsubsec:canonical automorphy factors for GU n n}. 

Let $g_0 \in U_4(\RR)$ be any element such that $g_0 \gimel = i 1_n$. Then, we have
\begin{equation}
    f_{\infty}(g; \gimel, \chi_\infty, s)
        =
    f_{\infty}(gg_0^{-1}; i1_n, \chi_\infty, s)
    f_{\infty}(g_0^{-1}; i1_n, \chi_\infty, s)^{-1}\,,
\end{equation}
where $f_{\infty, i1_n} = f_{\infty}(\bullet; i 1_n, \chi_\infty, s)$ is defined by replacing $\gimel$ with $i 1_n$ in \eqref{eq:def SW f oo gimel}. In particular, $f_{\infty, \gimel}$ and $f_{\infty, i1_n}$ only differ by nonzero constant.

\begin{remark}
    In what follows, we use $f_{\infty}(\bullet; i1_n, \chi_\infty, s)$ instead of $f_{\infty}(\bullet; \gimel, \chi_\infty, s)$ to state the results of Shimura directly. However, the Eisenstein series appearing in the previous section is still the one associated to $f_{\infty}(\bullet; \gimel, \chi_\infty, s)$. 
    
    As we are currently trying to $p$-adically interpolate its Fourier coefficients, this change is not an issue as the two sections are related the Fourier coefficients of each are related by a nonzero constant.
\end{remark}

Let $f_{\infty, U} = \prod_{\sigma \in \Sigma} f_\sigma$ be the restriction of $f_{\infty, i1_n}$ to $U_4(\RR) = \prod_{\sigma \in \Sigma} U(\RR)$. It follows from \cite[Equation (7.12)]{Shi83} (see \cite[Section 2.2.6]{Eis15} as well) that at $s = \frac{k-n}{2}$, the archimedean Fourier coefficients at $\beta = \beta_\sigma$ is
\[
    c(\beta_\sigma, 1; f_\sigma)
        =
    \left(
        2^{(n-1)n}
        (2\pi i)^{n k} 
        \pi^{\frac{n(n-1)}{2}}
        \prod \limits_{j=0}^{n-1}
            \Gamma(k - j)
    \right)^{-1}
    \sigma(\det \beta)^{k-n}
    e^{i \tr(\sigma(\beta))}\,.
\]

Let $\beta_\infty = (\beta_\sigma)_{\sigma \in \Sigma}$. The product $c(\beta_\infty, 1; f_{\infty, U}) = \prod_{\sigma \in \Sigma} c(\beta_\sigma, 1; f_\sigma)$ is equal to
\begin{equation} \label{eq:archimedean fourier coeff expression}
    \left(
        2^{(n-1)n}
        (-2\pi i)^{-n k} 
        \pi^{\frac{n(n-1)}{2}}
        \prod \limits_{j=0}^{n-1}
            \Gamma(k - j)
    \right)^{-[\KK^+ : \QQ]}
    \prod_{\sigma}
        \det(\beta_\sigma)^{k-n}
        e^{i \tr(\beta_\sigma)}
    \,.
\end{equation}

Using \cite[Lemma 9]{Eis15}, we see that given any $h_\infty = (h_\sigma)_{\sigma \in \Sigma} \in \GL_n(\AA_{\KK, \infty})$, if $k > n$, then $c(\beta_\infty, h_\infty, f_\infty) \neq 0$ if and only if $\det \beta_\infty \neq 0$. In particular, the Fourier coefficients are nonzero only if $\beta$ is non-degenerate.

\subsubsection{Local coefficients at places away from $p$ and $\infty$.} \label{subsubsec:Local coeff at places away from p and oo}
Assume $v$ is a finite place of $\KK^+$ away from $p$. Let $f_v$ be the local Siegel-Weil section at $v$ constructed in Section \ref{subsubsec:Local SW section at finite unramd places} and \ref{subsubsec:Local SW section at finite ramd places}, for $v$ unramified and ramified respectively. As explained in Section \ref{subsubsec:comp to other choices in the literature}, see \cite[Section 4.2.2]{EHLS} as well, we have
\[
    c(\beta_v, 1; f_v) = c(\beta_v, 1; f_v^\bb)\,,
\]
for some ideal $\bb$ of $\OO_{\KK^+}$ prime to $p$.

As explained in \cite{Eis15}, it follows from \cite[Proposition 19.2]{Shi97} that
\begin{align*}
    \prod_{v \nmid p \infty}
        c(\beta_v, 1; f_v^\bb)
            &=
    \nm^{\KK^+}_{\QQ}(\bb)^{-n^2}
    \prod_{i=0}^{n-1}
        L^p(
            2s + n - i, (\chi^+)^{-1} \eta^i
        )^{-1}
        \\ &\times
    \prod_{v \nmid p \infty}
        P_{\beta_v, \bb}(
            \chi^+(\varpi_v)^{-1}
            |\varpi_v|^{2s + n}
        )\,, 
\end{align*}
where
\begin{enumerate}
    \item $\chi^+$ is the restriction of the unitary Hecke character $\chi$ from $\AA_\KK$ to $\AA_{\KK^+}$;
    \item $\eta$ is the quadratic character of $\AA_{\KK^+}$ associated to the extension $\KK / \KK^+$;
    \item $\varpi_v$ is a uniformizer of $\OO_{\KK^+, v}$, viewed as an element of $\KK^\times$ prime to $p$. In what follows, we identify $\varpi_v$ with its image in $(\OO_\KK \otimes \ZZ_p)^\times$; 
    \item $P_{\beta_v, \bb}$ is a polynomial, depending on $\beta_v$ and $\bb$, with $\ZZ$-coefficients and constant term 1, which is identically 1 for all but finitely many $v$; and
    \item $L^p(r, \chi^+ \eta^r) = \prod_{v \nmid p \infty \mathrm{cond} \eta} d_v(s + \frac{1}{2}, \chi)$, where $d_v(s, \chi)$ is as in \eqref{eq:def d v s chi}.
\end{enumerate}

Furthermore, note that only 
\[
    \alpha(\beta; \chi, s)
        =
    \alpha_{\bb}(\beta; \chi, s)
        :=
    \prod_{v \nmid p \infty}
        P_{\beta_v, \bb}(
            \chi^+(\varpi_v)^{-1}
            |\varpi_v|^{2s + n}
        )
\]
depends on $\beta_v$ in the expression on the right-hand side above. For future reference, we set $\alpha(\beta; \chi) = \alpha_{\bb}(\beta; \chi) := \alpha(\beta; \chi, \frac{k-n}{2})$ for $k$ as in Hypothesis \ref{hyp:hyp integer k}.

As explained in \cite[Section 2.2.10]{Eis15}, $\alpha(\beta; \chi)$ is a (finite) $\ZZ$-linear combination of terms of the form 
\[
    \prod_{v \nmid p \infty}
        \chi_v(\varpi)^{-1}
        |\varpi|_v^k\,,
\]
where $\varpi$ is a $p$-integral element of the integer ring of $\KK^+$. Furthermore, using \eqref{eq:unitary Hecke char of type A0 k nu}, we have
\begin{equation}
    \prod_{v \nmid p \infty}
        \chi_v(\varpi)^{-1}
        |\varpi|_v^k
            =
    \chi_1\chi_2^{-1}(\varpi) 
    \prod_{\sigma \in \Sigma} 
        \sigma(\varpi)^{-k}\,,
\end{equation}
where $\chi_i = \bigotimes_{w \in \Sigma_p} \chi_{i, w}$ for $i$ = 1 and 2, see \cite[Equation (28)]{Eis15}. In particular, from the definition of $\chi_{w,1}$ and $\chi_{w,2}$ in Section \ref{subsubsec:Locally constant matrix coeff}, we have $\chi_p = \otimes_{w \mid p} \chi_w = \chi_1 \otimes \chi_2^{-1}$. We can thus rewrite $\alpha(\beta; \chi)$ as
\begin{equation} \label{eq:alpha beta chi in terms of chi p nm}
    \alpha(\beta; \chi)
            =
    \prod_{v \nmid p\infty}
        P_{\beta_v, \bb}
        (
            \chi_p(\varpi_v) 
            \nm^{\KK^+}_{\QQ}(\varpi_v)^{-k}
        )
\end{equation}

\subsubsection{Global Fourier coefficients.} \label{subsubsec:Global Fourier coeff}
Assume Hypothesis \ref{hyp:hyp integer k}. Using the same notation as in the previous sections, let
\begin{align*}
    D(n, \KK, \bb, p, k)
        &=
    C(n, \KK) \nm^{\KK^+}_{\QQ}(\bb)
        \\ &\times
    \prod_{\sigma \in \Sigma}
    \left(
        2^{(n-1)n}
        (-2\pi i)^{- n k} 
        \pi^{\frac{n(n-1)}{2}}
        \prod \limits_{j=0}^{n-1}
            \Gamma(k - j)
    \right)^{[\KK^+:\QQ]}
        \\ &\times
    \prod_{i=0}^{n-1}
        L^p(
            k - i, (\chi^+)^{-1} \eta^i
        )^{-1}\,.
\end{align*}

\begin{proposition} \label{prop:global Fourier coeff formula}
    Assume $k > n$ and let $f_{\chi, \psi}^\tau$ be the Siegel-Weil section $f_{\chi, \tau}^\tau(\bullet; \frac{k-n}{2})$ as in \eqref{eq:def f tau kappa chi psi}. For $\beta \in \Her_n(\KK)$, all the nonzero Fourier coefficients $c(\beta, 1; f_\chi^\tau)$ are given by
    \begin{equation} \label{eq:Fourier coeff of E tau chi psi}
        D(n, \KK, \bb, p, k)
        \alpha(\beta)
        \nu_p(-\beta_p; \tau, \psi)
        \prod_{\sigma \in \Sigma}
            (\det \beta_\sigma)^{k-n}
        e^{i \tr^{\KK^+}_{\QQ}(\beta)}\,,
    \end{equation}
    where 
    \[
        \nu_p(-\beta_p; \tau, \psi)
            :=
        \prod_{w \in \Sigma_p}
            \nu_w(-\beta_w; \tau_w, \psi_w)\,.
    \]

    Furthermore,
    \[
        \nu_w(\bullet; \tau_w, \psi_w) 
            = 
        \nu_w^{\tau_w \otimes \psi_w}
            =
        \chi_{w,1}^{-1}\chi_{w,2}\mu_w(\bullet; \tau_w, \psi_w)
    \]
     is as in \eqref{eq:def nu w and phi nu w}, where $\mu_w(\bullet; \tau_w, \psi_w)$ is the product of matrix coefficients constructed in \eqref{eq:def ext mu w} with respect to $\tau_w \otimes \psi_w$.
\end{proposition}

Let $E^\tau_{\chi, \psi}$ be the Eisenstein modular form in \eqref{eq:def Eis mod form E tau chi psi}. It follows from Proposition \ref{prop:global Fourier coeff formula} that the algebraic $q$-expansion of
\begin{equation} \label{eq:def G tau chi psi}
    G^\tau_{\chi, \psi} 
        := 
    D(n, \KK, \bb, p, k)^{-1} E^\tau_{\chi, \psi}
\end{equation}
at a cusp $L$ is
\begin{equation} \label{eq:fourier coeff of G tau chi psi q}
    G^\tau_{\chi, \psi}(q)
        =
    \sum_{\beta \in L}
    \left(
        \alpha(\beta, \chi)
        \nu_p(-\beta_p; \tau, \psi)
        \prod_{\sigma \in \Sigma}
            \det(\beta_\sigma)^{k-n}
    \right) q^\beta\,,
\end{equation}
for $k > n$, see \cite[Section 2.2.11]{Eis15}. In particular, the coefficients are algebraic.

\begin{remark} \label{rmk:Fourier coeff for psi locally constant}
    Recall that $\psi$ is a finite-order character of $L(\ZZ_p)$. From now on, we identify $\psi$ as a character of the center $Z_P$ of $L_P$, see Remark \ref{rmk:char of Z and P/SP} and the comments that follow. 

    Note that all of the above, especially the content of Section \ref{subsubsec:Local coeff at p}, remains valid if $\psi$ is only a locally constant function on $Z_P$, and not necessarily a character. See Remark \ref{rmk:psi w locally constant}.
\end{remark}

\subsubsection{$p$-adic shifts of Hecke characters.} \label{subsubsec:padic shifts of Hecke char}
In this section, we recall the notion of the ``$p$-adic shift'' of $\chi$, see \cite[Section 2.2.13]{Eis15} and \cite[Section 8.2]{EHLS}.

Assume the conductor of $\chi$ divides $p^m N_0$ for some $m \geq 0$ and integer $N_0$ prime to $p$. Let
\[
    U_{m, N_0}
        =
    (1 + N_0\OO \otimes \wh{\ZZ}^p)^\times
        \times
    (1 + p^m\OO \otimes \ZZ_p)
        \subset
    (\KK \otimes \wh{\ZZ})^\times\,,
\]
where $\OO = \OO_\KK$, and consider
\[
    X_p = X_{p, N_0}
        :=
    \varprojlim_m
        \KK^\times \backslash (\KK \otimes \wh{\ZZ})^\times / U_{m, N_0}\,,
\]
the ray class group of $\KK$ of conductor $p^\infty N_0$. The $p$-adic shift of $\chi$ will be a character of $X_p$. We often decompose a character $\alpha$ of $X_p$ as $\alpha = \bigotimes_{w} \alpha_w$, where the tensor product runs over all the finite places of $\KK$.

Now, assume as usual that $\chi$ satisfies Hypothesis \ref{hyp:hyp integer k} for some integer $k \geq 0$. Let $\chi_0 = \chi|\cdot|^{-k/2}_{\AA_\KK}$ and write $\chi_0 = \prod_{w} \chi_{0,w}$, as the product runs over all the places of $\KK$. Similarly, for any place $v$ of $\QQ$, let $\chi_{0,v} = \prod_{w \mid v} \chi_{0,w}$, so that
\[
    \chi_{0, \infty}(a)
        =
    \prod_{\sigma \in \Sigma}
        \sigma(a)^{-k - \nu_\sigma}
        \ol{\sigma}(a)^{\nu_\sigma}\,,
\]
for all $a \in \KK$, where $\ol{\sigma} = \sigma c$ and $\nu = (\nu_\sigma)_{\sigma}$ is as in \eqref{eq:unitary Hecke char of type A0 k nu} (this sequence of integers $\nu$ should not be confused with the locally constant function $\nu_p(\bullet; \tau, \psi)$). We say that the $\infty$-type $\chi_0$ is
\[
    \Psi_{k, \nu}
        =
    \prod_{\sigma \in \Sigma}
        \sigma^{-k}
        \left(
            \frac{\ol{\sigma}}{\sigma}
        \right)^{\nu_\sigma}\,,
\]
viewed as a function of $(\OO_\KK \otimes \ZZ_p)^\times$. Note that for all $\varpi \in \OO_{\KK^+}^\times$, we have
\begin{equation} \label{eq:Psi k nu on O KK +}
    \Psi_{k, \nu}(\varpi)
        =
    \nm^{\KK^+}_\QQ(\varpi)^{-k}\,.
\end{equation}

Let $\wt{\chi}_{0, \infty} : (\KK \otimes \ZZ_p)^\times \to \bQQ_p^\times$ be the unique $p$-adically continuous character such that
\[
    \wt{\chi}_{0, \infty}(a)
        =
    \incl_p \circ \chi_{0, \infty}(a)\,,
\]
for all $a \in \KK^\times$. In particular, $\wt{\chi}_{0, \infty}(a) \in \OO_{\CC_p}^\times$ for all $a \in (\OO_\KK \otimes \ZZ_p)^\times$.

The $p$-adic shift of $\chi$ is defined as the $p$-adic character $\wt{\chi}_0 : X_p \to \OO_{\CC_p}^\times$ for which
\[
    \wt{\chi}_0(a)
        =
    \wt{\chi}_{0, \infty}((a_w)_{w \mid p})
    \prod_{w \nmid \infty}
        \chi_{0,w}(a_w)\,,
\]
for all $a = (a_w)_{w \nmid \infty} \in X_p$.

We now express the Fourier coefficients of $G^\tau_{\chi, \psi}$ in terms of $\wt{\chi}_0$. In the process, we also use the definition of $\nu_w(\bullet; \tau, \psi)$ in \eqref{eq:def nu w and phi nu w} to express these coefficients in terms of
\[
    \mu_p(\bullet; \tau, \psi)
        :=
    \prod_{w \in \Sigma_p}
        \mu_w(\bullet; \tau_w, \psi_w)\,,
\]
where $\mu_w(\bullet; \tau_w, \psi_w) = \mu_w^{\tau_w \otimes \psi_w}$ is as in \eqref{eq:def ext mu w} and Remark \ref{rmk:mu tau nu tau and Phi tau}.

Firstly, observe that using \eqref{eq:alpha beta chi in terms of chi p nm} and \eqref{eq:Psi k nu on O KK +}, we have
\[
    \alpha(\beta, \chi)
        =
    \prod_{v \nmid p\infty}
        P_{\beta_v, \bb}
        (\wt{\chi}_{0, p}(\varpi_v))\,,
\]
where
\[
    \wt{\chi}_{0, p}   
        := 
    \prod_{w \mid p}
        \wt{\chi}_{0, w}\,.
\]

Similarly, we have
\begin{align*}
    \nu_p(-\beta; \tau, \psi)
    \nm(\det \beta)^k
        &=
    \chi_p(-\beta^{-1})
    \mu_p(-\beta; \tau, \psi)
    \nm(\det \beta)^k
        \\ &=
    \chi_p(-1)
    \wt{\chi}_{0,p}(\beta^{-1})
    \mu_p(-\beta; \tau, \psi)\,.
\end{align*}

Therefore, the $\beta$-th coefficient of the $q$-expansion of $G^{\tau}_{\chi, \psi}$ at a cusp $L$ can be rewritten as a finite $\ZZ$-linear combinations of terms of the form
\begin{equation} \label{eq:Fourier coeff in terms of wt chi and mu}
    \wt{\chi}_{0,p}(\varpi)
    \wt{\chi}_{0,p}(\beta^{-1})
    \mu_p(-\beta; \tau, \psi)
    \nm(\det \beta)^{-n}\,,
\end{equation}
where the linear combination is over a finite set (which depends on $\beta$ and $L$) of $p$-adic units $\varpi \in \KK^\times$.

Note that if $\psi$ and $\psi'$ are two finite-order characters of $Z_P$ such that $\psi \equiv \psi'$ modulo $p^r$, then the $\mu_p(\bullet; \tau, \psi) \equiv \mu_p(\bullet; \tau, \psi')$ modulo $p^r$.

\begin{remark}
    The Fourier coefficients in \eqref{eq:Fourier coeff in terms of wt chi and mu} can be compared to the ones of the Eisenstein series constructed in \cite[Theorem 2]{Eis14}. The main difference is the level at $p$ of the Eisenstein series considered.
\end{remark}

From \eqref{eq:Fourier coeff in terms of wt chi and mu} and the $q$-expansion principle, it is clear that $G^\tau_{\chi, \psi}$ is a modular form on $G_4$ over the $p$-adic ring $\OO_\pi$ introduced in Section \ref{subsec:lattices of Pord holo forms}. We identify it with its image in the space of $p$-adic modular forms.

\subsection{$p$-adic differential operators} \label{subsec:padic diff operators}
In this section, we discuss the $p$-adic differential operators constructed in \cite[Section 5]{EFMV18} and their relevant properties for our purpose. The goal is to obtain a family of $p$-adic modular forms related to $E_{\chi, \psi, \theta}^{\tau, \kappa}$, see \eqref{eq:def f tau kappa chi psi}, using these $p$-adic differential operators (depending on $p$-adic weights $\kappa$ and $\tau$) and $G^\tau_{\chi, \psi}$.

Let $R = \OO_\pi$ and let $K = K_4 \subset G_4(\AA_f)$ be any neat open compact subgroup. Consider the space
\[
    \VV := \VV(G_4, K^p; R)
\]
of (scalar-valued) $p$-adic modular forms on $G_4$ (with respect to the parabolic $P_4 \subset H_4$, see Section \ref{subsubsec:comp of lvl structures} and \eqref{eq:partition of n n for G4}), as in \eqref{eq:def of scalar valued padic mod forms}.

Now, let $\kappa$ be an $R$-valued dominant weight for $G_1$, as in Section \ref{subsubsec:alg weights}. Then, $\wt{\kappa} = (\kappa, \kappa^\flat)$ is a dominant weight of $G_3$. Let $\wt{\kappa}_p$ be the corresponding $p$-adic weight of $T_{H_3}(\ZZ_p) = T_{H_4}(\ZZ_p)$.

Assume that $(\kappa, \chi)$ is critical, as in Definition \ref{def:critical pairs}, and let $\rho = (\rho_\sigma)_{\sigma \in \Sigma}$ and $\rho^v = (\rho_{\sigma}^v)_{\sigma \in \Sigma}$ be as in \eqref{eq:def rho sigma and rho sigma v shifts}. As above, let $\wt{\rho} = (\rho, \rho^\flat)$, $\wt{\rho}^v = (\rho^v, \rho^{v, \flat})$, and denote the corresponding $p$-adic weights as $\wt{\rho}_p$ and $\wt{\rho}^v_p$.

\begin{proposition} \label{prop:existence of padic differential operators}
    Assume Conjecture \ref{conj:density conj scalar padic forms}. Keeping the same notation as above, there exists a $p$-adic differential operator
    \[
        \theta_\chi^d(\rho^v)
            :
        \VV_\chi(G_4, K^p; R) \to \VV(G_4, K^p; R)\,,
    \]
    compatible with change of level subgroups, such that 
    \begin{equation} \label{eq:comparison at ordinary CM points}
        \Omega_{\wh{\kappa}, r, J_0', h_0}
            \circ
        \res_{J_0', h_0}
            \circ
        \delta_\chi^d(\rho^v)(f) 
            = 
        \res_{p, J_0', h_0}
            \circ
        \theta_\chi^d(\rho^v)
            \circ 
        \Omega_{\wt{\kappa}, r, G_4, X_4}(f)\,
    \end{equation}
    for any $f \in M_{\chi}(G_4, K_{4,r}; R)$ and any ordinary CM pair $(J_0', h_0)$, using the notation from Section \ref{subsubsec:evaluation at ordinary points}. Here, $\delta_{\chi}^d(\rho^v)$ is as in Section \ref{subsec:Sgl Eis series as Coo mod forms}.
\end{proposition}

\begin{proof}
    The differential operators $\theta_\chi^d(\rho^v)$ are exactly the operators denoted $\Theta^{\wt{\kappa}}$ in \cite[Theorem 5.1.3]{EFMV18}. 
    
    However, one need to show that these extend to the space of $p$-adic modular form $\VV$ considered here, which is larger in general than the space ``$V^N$'' in \emph{ibid}. This follows immediately if we assume Conjecture \ref{conj:density conj scalar padic forms} (which replaces the use of  \cite[Theorem 2.6.1]{EFMV18}).

    Lastly, \eqref{eq:comparison at ordinary CM points} is exactly \cite[Theorem 8.1.1 (a)]{EHLS}.
\end{proof}

\begin{proposition} \label{prop:def theta and theta hol kappa chi}
    In the setting of Proposition \ref{prop:existence of padic differential operators}, we have
    \begin{enumerate}
        \item Fix any neat open compact subgroups $K_1 \subset G_1(\AA_f)$ and $K_2 \subset G_2(\AA_f)$ such that $K_1 \times K_2 \subset G_3(\AA_f) \cap K_4$. Then,
        \[
            \theta(\kappa, \chi) := \res_3 \circ \theta_\chi^d(\rho^v)
        \]
        defines a differential operator
        \[
            \VV_\chi(G_4, K_4^p; R)
                \to
            \VV_\kappa(
                G_1, K_1^p; R
            )
                \otimes
            \VV_{\kappa^\flat}(
                G_2, K_2^p; R
            )
                \otimes
            (\chi \circ \det)\,,
        \]
        where $\res_3$ is the pullback of the first embedding $\gamma_p \circ \iota_3$ in \eqref{eq:embeddings of Igusa towers}.

    \item There is a differential operator
    \[
        \theta^\hol(\kappa, \chi)
            :
        \VV_\chi(G_4, K^p; R)
            \to
        \VV(G_4, K^p; R)
    \]
    whose composition with $\res_3$ coincides with $D^{\hol}(\kappa, \chi)$ from Section \ref{subsec:Sgl Eis series as Coo mod forms}, via pullback to functions on $G_4(\AA)$, restrictions to functions on $G_3(\AA)$ and \eqref{eq:def p adification map of scalar valued forms} for $G_3$.

    \item For all $\kappa^\dagger \leq \kappa$, there exists an operator 
    \[
        \theta(\kappa, \kappa^\dagger)
            :
        \VV(G_4, K^p; R) \to \VV(G_4, K^p; R)
    \]
    such that
    \[
        \theta(\kappa, \chi)
            =
        \sum_{\kappa^\dagger \leq \kappa}
            \res_3
                \circ
            \theta(\kappa, \kappa^\dagger)
                \circ
            \theta^\hol(\kappa, \chi)\,.
    \]
    \end{enumerate}
\end{proposition}

\begin{proof}
    This is simply \cite[Proposition 8.1.1 (b), (d)]{EHLS} and \cite[Corollary 8.1.2]{EHLS} in our settings. The proof remains the same using the existence of the differential operators in Proposition \ref{prop:existence of padic differential operators}.
\end{proof}

For any dominant weight $\kappa$ as above, using these differential operators, we define
\[
    G^{\tau, \kappa}_{\chi, \psi}
        :=
    \theta(\kappa, \chi)
        G^\tau_{\chi, \psi}\,.
\]
and let $K_3 = K_{3,r} \subset G_3(\AA_f)$ be its level, see \eqref{eq:def G tau chi psi} and the comments below \eqref{eq:def Eis mod form E tau chi psi}.

Furthermore, let $\theta$ be any $P$-parallel weight. Then, following the same logic as in Section \ref{subsec:Sgl Eis series as Coo mod forms}, in what follows we set $\theta(\kappa, \theta, \chi) := \theta(\kappa + \theta, \chi)$ and
\[
    G^{\tau, \kappa}_{\chi, \psi, \theta}
        :=
    G^{\tau, \kappa + \theta}_{\chi, \psi} 
        =
    \theta(\kappa, \theta, \chi) G^\tau_{\chi, \psi}\,.
\]

The action of $\theta(\kappa, \chi)$ on $p$-adic $q$-expansion is described in \cite[Corollary 5.2.10]{EFMV18}. Their work considers $p$-adic modular forms in
\[
    \VV_{\infty, \infty}
        =
    \varprojlim_m
    \varinjlim_r
        \VV_{r,m}(\OO_{\CC_p})^{B_H^u(\ZZ_p)}\,,
\]
for the Borel $B_H$ associated to the trivial partition, see Remark \ref{rmk:trivial partition}, but their computations hold for all $f \in \VV(G_4, K^p; R)$ if one assumes Conjecture \ref{conj:density conj scalar padic forms} (which we do in this paper).

Namely, there exists a polynomial $\phi^\kappa$ (on $n\times n$-matrices) such that for each $\beta \in L$, the $\beta$-th coefficient of $G^{\tau, \kappa}_{\chi, \psi}$ is equal to $\phi^\kappa(\beta)$ times the $\beta$-th coefficient of $G^\tau_{\chi, \psi}$, see \cite[Theorem 5.1.3 (1)]{EFMV18} and \cite[Section 5.2.2]{EFMV18}. 

\begin{remark}
    In our notation, the polynomial in \cite[Corollary 5.2.10]{EFMV18} should be written $\phi_{\wt{\kappa}}$ for $\wt{\kappa} = (\kappa, \kappa^\flat)$. However, we only consider polynomials associated to such characters, hence we only emphasize their dependence on $\kappa$.
\end{remark}

As above, considering $\kappa$ as fixed and considering any $\kappa' = \kappa + \theta$ in the $P$-parallel lattice $[\kappa]$, we set $\phi^\kappa_\theta := \phi^{\kappa + \theta}$. Then, it follows from \cite[Remark 5.2.11]{EFMV18} that if $\theta$ and $\theta'$ are two $P$-parallel weights such that $\theta \equiv \theta'$ modulo $p^r(p-1)$, then $\phi^\kappa_{\theta} \equiv \phi^\kappa_{\theta'}$ modulo $p^{r+1}$.

Using the above and \eqref{eq:Fourier coeff in terms of wt chi and mu}, one therefore readily checks that the $\beta$-th coefficient of $G^{\tau, \kappa}_{\chi, \psi, \theta}$ satisfy the ``usual'' Kummer congruences (\cite[(4.0.8)]{Kat78}) as $(\wt{\chi}_0, \psi \cdot \theta)$ vary $p$-adically as characters of $X_p \times Z_P$. See \cite[Section 5]{EHLS} for further details.

We obtain the following as a consequence of \cite[Proposition (4.1.2)]{Kat78} and using the same logic as in the construction of the analogous Eisenstein measures of \cite{Kat78, Eis15, EFMV18}.

\begin{proposition} \label{prop:existence of Eis measure on V3}
    Assume conjecture \ref{conj:classicality scalar padic forms}. Fix a $p$-adic weight $\kappa$ of $T_{H_1}(\ZZ_p)$ and a $P$-nebentypus $\tau$ with central character $\w_\tau$. There is a $\VV_3(K_3^p; R)$-valued measure $d\Eis^{[\kappa, \tau]}$ on $X_p \times Z_P$ such that
    \begin{equation} \label{eq:functional equation of d Eis kappa tau}
        \int_{X_p \times Z_P}
            (
                \wt{\chi}_0, 
                (\w_\tau \psi) \cdot \rho_{\kappa, \theta}^v
            )
        d\Eis^{[\kappa, \tau]}
            =
        G^{\tau, \kappa}_{\chi, \psi, \theta}\,,
    \end{equation}
    for any $p$-adic shift $\wt{\chi}_0$ of a Hecke character $\chi$, as in Section \ref{subsubsec:padic shifts of Hecke char}, and for any arithmetic characters on $Z_P$ whose finite-order part is $\w_\tau \psi$ and algebraic part is $(\kappa_p + \theta_p)|_{Z_P}$ for some critical pair $(\kappa + \theta, \chi)$. Here, $\rho_{\kappa, \theta}^v$ is the ``shift'' associated to $\kappa+\theta$ and $\chi$ as in \eqref{eq:def rho sigma and rho sigma v shifts}. 
    
    Both sides of \eqref{eq:functional equation of d Eis kappa tau} are independent of the choice of decompositions $\kappa' = \kappa + \theta \in [\kappa]$ and $\w_{\tau'} = \w_\tau \psi$ for the central character of some $\tau' \in [\tau]$.
\end{proposition}

\begin{remark}
    When $P=B$ as in Remark \ref{rmk:trivial partition}, the above agrees with the measure in \cite[Theorem 8.2.2]{EHLS}.
\end{remark}

\subsubsection{Comparison to classical Eisenstein series.} \label{subsubsec:comp to classical Eis series}

We first compare $\theta(\kappa, \chi)$ to the $C^\infty$-differential operators from Section \ref{subsubsec:Coo diff operators}.

\begin{proposition} \label{prop:padic and Coo diff op on Pord cusp forms}
    Assume Conjecture \ref{conj:classicality scalar padic forms}. With the same setting as in Proposition \ref{prop:def theta and theta hol kappa chi}, let $\theta(\kappa, \chi)^\cusp$ denote the restriction of $\theta(\kappa, \chi)$ to $\VV_\chi^\cusp(G_4, K_4; R)$. Then, 
    \[
        e_\kappa^{\Pord} \circ \theta(\kappa, \chi)^\cusp 
            = 
        e_\kappa^{\Pord} \circ \delta_\chi^d(\rho^v)
    \]
    as operators
    \[
        \VV^\cusp_\chi(G_4, K_4; R)
            \to
        S_{\kappa}(G_1, K_1; R)
            \otimes
        S_{\kappa^\flat}(G_2, K_2; R)
            \otimes
        (\chi \circ \det)\,.
    \]

    Furthermore, for any cuspidal $F \in H^0_!(\Sh(V_4), \LL(\chi))$, we have
    \[
        e_\kappa^{\Pord} 
            \circ 
        \theta(\kappa, \chi)(F)
            =
        e_\kappa^{\Pord} 
            \circ 
        D^\hol(\kappa, \chi)(F)
    \]
\end{proposition}

\begin{proof}
    The first part is exactly \cite[Theorem 8.1.1 (c)]{EHLS} with the obvious modifications to our setting. For the second part, we follow the same logic as in the proof of \cite[Proposition 8.1.3]{EHLS}.

    All one needs is the decompositions from \eqref{eq:decomposition D kappa chi with D hol kappa dagger chi} and Proposition \ref{prop:def theta and theta hol kappa chi} (iii), the first part of the above and the fact that for any $\kappa^\dagger < \kappa$,
    \[
        e_\kappa^{\Pord}
            :=
        \varinjlim_N
            \left(
                \prod_{w \in \Sigma_p}
                \prod_{j=1}^{r_w}
                    u_{w, D_w(j), \kappa}
            \right)^{N!}
    \]
    converges absolutely to $0$ on $S_{\kappa^\dagger}(K_r; R)$. This last claims follows from the fact that for $\kappa^\dagger < \kappa$, we have
    \[
        u_{w, D_w(j), \kappa}
            =
        \kappa'(t_{w,D_w(j)})
        U_{w, D_w(j)}
            =
        (
            \kappa' 
                \cdot 
            (\kappa^{\dagger, \prime})^{-1}
        )(t_{w,D_w(j)})
        u_{w, D_w(j), \kappa^\dagger}
    \]
    and, using the definition of $\kappa'$ as in \eqref{definition u w j}, that
    \[
        \prod_{w \in \Sigma_p}
        \prod_{j=1}^{r_w}
            (
            \kappa' 
                \cdot 
            (\kappa^{\dagger, \prime})^{-1}
        )(t_{w,D_w(j)})
            =
        \prod_{w \in \Sigma_p}
        \prod_{j=1}^{r_w}
            (
            \kappa_p
                \cdot 
            \kappa^{\dagger, -1}_p
        )(t_{w,D_w(j)})
            =
        p^m
    \]
    for some strictly negative integer $m$. This is clear from the definition of each $t_{w, D_w(j)}$ and the relation \eqref{eq:relation kappa kappa p}.
\end{proof}

We now wish to apply the previous proposition to an ordinary cusp form closely related to $G_{\chi, \psi, \theta}^{\tau, \kappa}$, using the notation of Proposition \ref{prop:existence of Eis measure on V3}.

Let $\pi$ be a $P$-anti-ordinary anti-holomorphic cuspidal automorphic form of $P$-anti-WLT $(\kappa, K_r, \tau)$. Let $\m = \m_\pi$ be the non-Eisenstein maximal ideal of $\bTT^{\Pord}_{K^p, [\kappa, \tau], \OO_\pi}$ associated to $\pi$ as in Remark \ref{rmk:lambda pi flat and m pi flat}.

Assuming Conjecture \ref{conj:classicality vector padic forms}, it follows from Proposition \ref{prop:padic and Coo diff op on Pord cusp forms} and \eqref{eq:def G tau chi psi} that for $\kappa$ very regular, after localization at $\m$, both
\begin{equation} \label{eq:comp Pord padic and classical Eis series}
    e_\kappa^\Pord G^{\tau, \kappa}_{\chi, \psi}
        \;\;\;\text{and}\;\;\;
    D(n, \KK, \bb, p, k)^{-1}
    e_\kappa^\Pord D^\hol(\kappa, \chi) E^{\tau}_{\chi, \psi}\,
\end{equation}
lie in 
\[
    S_{\kappa, V}(K_{1,r}, \OO_{\CC_p})_\m \otimes S_{\kappa^\flat, -V}(K_{2,r}, \OO_{\CC_p})_\m \otimes (\chi \circ \det)\,,
\]
and are equal.

In particular, using \eqref{eq:def pairing kappa Kr} on $G_3$, we can view
\[
    \brkt{
        e_\kappa^\Pord G^{\tau, \kappa}_{\chi, \psi}
    }{
        \bullet
    }_{\wt{\kappa}, K_{3,r}}
\]
as an element in the $\OO_\pi$-dual of
\[
    \wh{S}_{\kappa, V}(K_{1,r}, \OO_\pi)_\m
        \otimes
    \wh{S}_{\kappa^\flat, -V}(K_{2,r}, \OO_\pi)_\m 
        \otimes 
    (\chi^{-1} \circ \det)
\]
and the above is closely related to the integral involved in the doubling method.

Observe that together with the tautological pairing $\MMM_\tau \otimes \MMM_{\tau^\flat} = \MMM_\tau \otimes \MMM_\tau^\vee \to \OO_\pi$, the above can in fact be identified as an element in the dual of
\begin{equation} \label{eq:Pord Eis measure in dual of Paord aholo modular form}
    \wh{S}_{\kappa, V}(K_{1,r}, [\tau]; \OO_\pi)_\m
        \otimes
    \wh{S}_{\kappa^\flat, -V}(K_{2,r}, [\tau^\flat]; \OO_\pi)_\m 
        \otimes 
    (\chi^{-1} \circ \det)\,.
\end{equation} 

Now, fix any $F \in I_\pi$ and $F^\flat \in I_{\pi^\flat}$, using the notation from Section \ref{subsubsec:I pi and test vectors}. Fix non-zero elements $\iota \in \hom_{L_P}(\tau, \pi_p^{(\Paord, r)})$ and $\iota^\flat \in \hom_{L_P}(\tau^\flat, \pi_p^{\flat, (\Paord, r)})$, i.e. a basis for each of these two 1-dimensional spaces. 

Fix a basis $\mathcal{B}_\tau = \{v_1, \ldots, v_r\}$ of $\tau$ (where $r = \dim \tau$) and a dual basis $\mathcal{B}_\tau^\flat = \{v_1^\flat, \ldots, v_r^\flat\}$ of $\tau^\flat$. For each $1 \leq i \leq r$, let $\varphi_i$ (resp. $\varphi_i^\flat$) be the (anti-holomorphic $P$-anti-ordinary) test vector of $\pi$ (resp. $\pi^\flat$) determined by $F$, $\iota$ (resp. $\iota^\flat$) and $v_i$ (resp. $v_i^\flat$), as in \eqref{eq:test vector for pi determined by I pi and iota} (resp. \eqref{eq:test vector for pi flat determined by I pi flat and iota flat}). As explained at the end of Section \ref{subsubsec:I pi and test vectors}, we have
\[
    \brkt{\varphi}{\varphi^\flat}_\pi
        =
    \brkt{\varphi_i}{\varphi^\flat_i}_\pi
\]
for all $1 \leq i \leq \dim \tau$, where $\varphi := \varphi_1$ and $\varphi^\flat := \varphi_1^\flat$.

It follows from \eqref{eq:Serre pairing aut forms}, using the identifications \eqref{eq:lattice wh I pi identification} and \eqref{eq:lattice I pi identification}, that pairing the element in the dual of \eqref{eq:Pord Eis measure in dual of Paord aholo modular form} corresponding to $e_\kappa^{\Pord} G_{\chi, \psi}^{\tau, \kappa}$ with $F \otimes F^\flat$ is equal to
\begin{align*}
    \frac{1}{\vol(I_{r, V}^0)\vol(I_{r, -V}^0)}
    \sum_{i=1}^{\dim \tau}
    \int_{[G_3]}
        &D(n, \KK, \bb, p, k)^{-1} 
        E_{\chi, \psi}^{\tau, \kappa, \hol}\left(g_1, g_2; s+\frac{1}{2}\right)
    \\ &\times\,
        \varphi_i(g_1)\varphi_i^\flat(g_2)
        \chi^{-1}(\det(g_2))
        ||\nu(g_2)||^{a(\kappa)}
    dg_1dg_2\,,
\end{align*}
where $s = \frac{k-n}{2}$ and $[G_3] = G_3(\QQ)Z_{G_3}(\RR) \backslash G_3(\AA)$. By definition, this is equal to
\[
    \frac{1}{\vol(I_{r, V}^0) \vol(I_{r, -V}^0) D(n, \KK, \bb, p, k)}
    \sum_{i=1}^{\dim \tau}
        I\left(
            \varphi_i,
            \varphi_i^\flat ||\nu(g_2)||^{a(\kappa)},
            f_{\chi, \psi}^{\tau, \kappa, \hol},
            s+\frac{1}{2}
        \right)\,,
\]
for $s = \frac{k-n}{2}$. Lastly, by Theorem \ref{thm:main formula global zeta integral}, it is equal to
\begin{equation} \label{eq:pairing eis measure with global test vectors}
    \frac{\brkt{\varphi}{\varphi^\flat}_\pi}{\vol(I_{r, V}^0 \cap I_{r, -V}^0)}
    I_p\left(s+\frac{1}{2}, \Pord, \pi, \chi\right)
    I_\infty\left(s+\frac{1}{2}; \pi, \chi\right)
    I_S
    L^S\left(s+\frac{1}{2}; \pi, \chi\right)\,,
\end{equation}
at $s = \frac{k-n}{2}$, where $I_p$, $I_\infty$ and $I_S$ are as in \eqref{eq:def I p}, \eqref{eq:def I infty} and \eqref{eq:def I S} respectively.

%%%%%%%%%%%%%%%%%%%%%%%%%%%%%%%%%%%%%%%%%%%%%%%%%%%%%%%%%%%%%%%%%%

\part{$p$-adic $L$-functions for $P$-ordinary families.}\label{part:padic Lfct on Pord families}

\section{Pairing, periods and main result.} \label{sec:pairing}

\subsection{Eisenstein measures and $p$-adic $L$-functions} \label{subsec:Eis measures and padic L fct}
In this section, we adapt the material of \cite[Section 7.4]{EHLS} to the $P$-ordinary setting. The goal is to obtain an analogue of \cite[Proposition 7.4.10]{EHLS} in the $P$-ordinary setting and interpret the Eisenstein measure $d\Eis^{[\kappa, \tau]}$ of Proposition \ref{prop:existence of Eis measure on V3} as an element of the Hecke algebra $\TT$ from Section \ref{sec:P-(anti-)ord Hida families}.

The idea is to view $d\Eis^{[\kappa, \tau]}$  as a collection of linear transformations on locally constant functions compatible with the projective limit structure of $\TT$ over Hecke algebras of finite level.

\subsubsection{Equivariance and the Garrett map} \label{subsubsec:equivariance and the Garrett map}
Throughout this section, we fix a neat open compact subgroup $K_1^p \subset G(\AA_f^p)$ and set $K_{1,r} := K_1^p I_{P,r}$ for all $r \gg 0$. We let $K_{2,r} := K_{1,r}^\flat$ and set $K_{3,r} := (K_{1,r} \times K_{2,r}) \cap G_3(\AA_f)$. We often write $K_r$ for $K_{3,r}$. Furthermore, we set
\[
    \VV_3 := \VV^{\Pord, \cusp}(G_3, K^p; \OO)\,.
\]

Consider the center $Z = Z_P$ of $L_P(\ZZ_P)$ and for each $r \geq 1$, let $Z^r = 1 + p^r Z$. In particular, $Z_{P,r} = Z / Z^r$ as in Section \ref{subsubsec:(Anti-)holomorphic Hecke alg}.

Let $\Lambda = \Lambda_\pi = \OO_\pi[[Z_P]]$ be the Iwasawa algebra in Section \ref{subsubsec:indep of weights} for $R = \OO_\pi$. Since the ring $\OO = \OO_\KK$ does not appear in this section, we set $\OO = \OO_\pi$ in what follows. 

As usual, one identifies $\Lambda$ as the algebra of distributions on $Z$ with $\OO$-coefficients, equipped with a canonical perfect pairing $\Lambda \otimes C(Z,\OO) \to \OO$, where $C(Z, \OO)$ denotes the module of continuous $\OO$-valued functions on $\TT$.

Let $\II_r \subset \Lambda$ be the augmentation ideal associated to $Z^r$, and set $\Lambda_r = \Lambda / \II_r$. Furthermore, define
\[
    C_r(Z, \OO)
        =
    C(Z/Z^r, \OO)
        :=
    \{
        \text{continuous $Z^r$-invariant functions on $Z$}
    \}\,,
\]
a free $\OO$-module of locally constant functions on $Z$. Let $\eta_r : C_r(Z, \OO) \hookrightarrow C_{r+1}(Z, \OO)$ be the natural inclusion.

The restriction of the perfect pairing above to $\Lambda \otimes C_r(Z, \OO) \to \OO$ factors through a perfect pairing
\[
    \Lambda_r \otimes C_r(Z, \OO) \to \OO\,,
\]
identifying $\Lambda_r$ with the algebra of distributions $\hom_\OO(C_r(Z, \OO), \OO)$.

Now, fix some critical pair $(\kappa, \chi)$. Set
\begin{equation} \label{eq:def phi phi chi wrt dEis}
    \phi 
        = 
    \phi_\chi
        :=
    e_P \circ \int_{Z_P}
        (\wt{\chi}_0, \bullet)
    d\Eis^{[\kappa, \tau]}
\end{equation}
as linear functional on $C(Z, \OO)$ valued in $\VV_3^{\Pord}$. 

Let $\rho = (\rho_\sigma)_\sigma$ and $\rho^v = (\rho^v_\sigma)_\sigma$ be as in \eqref{eq:def rho sigma and rho sigma v shifts}. We identify $\rho$ and $\rho^v$ with $p$-adic weights of $T_{H_1}(\ZZ_p)$, as in Section \ref{subsubsec:padic weights}. In fact, in this section, we are mostly concerned with the restriction of $\rho$ and $\rho^v$ to $Z$, which we still denote $\rho$ and $\rho^v$ respectively by abuse of notation.

For any $r \geq 0$, consider the subset $C_r(Z, \OO) \cdot \rho^v \subset C(Z, \OO)$. By \cite[Lemma 7.4.2]{EHLS}, the measure $\phi = \phi_\chi$ on $Z$ is equivalent to a collection $\phi_{\chi, \rho} = \phi_\rho = (\phi_{\rho, r})_{r \geq 0}$, such that
\[
    \phi_{\rho, r}
        \in 
    \hom_\Lambda(C_r(Z, \OO) \cdot \rho^v, \VV_3^{\Pord})
        \;\;\;\text{and}\;\;\;
    \eta_r^*(\phi_{\rho, r+1}) 
        = 
    \phi_{\rho, r}\,,
\]
where the equivalence is given by $\phi(\psi) = \phi_{r, \rho}(\psi \cdot \rho^v)$ for all $\psi \in C_r(Z, \OO)$. For $\chi$ fixed, $\rho$ and $\kappa$ determine one another, hence we sometimes write $\phi_{\chi, \rho}$ by $\phi_{\chi, \kappa}$.

Let $\II_{\rho, r} \subset \Lambda$ be the annihilator of $C_r(Z, \OO) \cdot \rho^v$ with respect to the pairing $\Lambda \otimes C(Z, \OO) \to \OO$, and let $\Lambda_{\rho, r} = \Lambda / \II_{\rho, r}$. By definition, $\Lambda_{\rho, r}$ is identified with $\hom_\OO(C_r(Z, \OO) \cdot \rho^v, \OO)$.

As explained at the end of Section \ref{subsubsec:comp to classical Eis series}, we see that for all $\kappa$ very regular and $\psi \in C_r(Z, \OO)$, we have
\[
    \phi_{\rho, r}(\psi)
        \in
    \hom_\OO(
        \wh{S}^{\Pord}_{\kappa, V}(
            K_{1,r}, [\tau]; \OO
        )_\m,
        S^{\Pord}_{\kappa^\flat, -V}(
            K_{2,r}, [\tau^\flat]; \OO
        )_\m
            \otimes
        (\chi \circ \det)
    )\,,
\]
where $\m = \m_\pi$ is as in Remark \ref{rmk:lambda pi flat and m pi flat}. 

In fact, it follows from the work of \cite{Gar84, GPSR87} on the Garrett map, see \cite[Theorem 9.1.3--Corollary 9.1.4]{EHLS}, that the measure $\phi_{\rho, r}$ satisfies a stronger equivariance property with respect to the appropriate Hecke algebra, namely
\[
    \phi_{\rho, r}(\psi)
        \in
    \hom_{\TT_{K_r, \kappa, [\tau], \OO}}(
        \wh{S}^{\Pord}_{\kappa, V}(
            K_{1,r}, [\tau]; \OO
        )_\m,
        S^{\Pord}_{\kappa^\flat, -V}(
            K_{2,r}, [\tau^\flat]; \OO
        )_\m
            \otimes
        (\chi \circ \det)
    )\,,
\]
where $\TT_{K_r, \kappa, [\tau], \OO}$ is as in Proposition \ref{prop:VCT for big Hecke alg}, for all $\kappa$ very regular and $\psi \in C_r(Z, \OO)$.

To lighten notation, we omit $\kappa$ and $[\tau]$ from our notation momentarily (as they do not vary in this section), and write
\[
    \wh{S}^{\Pord}_{r, V, \pi}
        :=
    \wh{S}^{\Pord}_{\kappa, V}(
        K_{1,r}, [\tau]; \OO
    )_\m
        \;\;\;\text{and}\;\;\;
    S^{\Pord}_{r, -V, \pi^\flat}
        :=
    S^{\Pord}_{\kappa^\flat, -V}(
        K_{2,r}, [\tau^\flat]; \OO
    )_\m\,,
\]
both modules over $\TT_r := \TT_{K_r, \kappa, [\tau], \OO}$ using Lemma \ref{lma:comp Pord and Paord hecke alg}.

By definition of the finite free $\OO$-modules $\wh{I}_\pi = I_{\pi^\flat}$, we have isomorphisms
\[
    \TT_r
        \otimes
    \wh{I}_\pi
        \xrightarrow{\sim}
    \wh{S}^{\Pord}_{r, V, \pi}
        \;\;\;\text{and}\;\;\;
    \wh{\TT}_r
        \otimes
    I_{\pi^\flat}
        \xrightarrow{\sim}
    S^{\Pord}_{r, -V, \pi^\flat}\,,
\]
see \eqref{eq:basis wh I pi iso} and (the $\OO$-dual) of \eqref{eq:basis wh I pi flat iso}, where $\wh{\TT}_r$ denotes the $\OO$-dual of $\TT_r$. Therefore, tensoring with $(\chi^{-1} \circ \det)$, we obtain
\begin{equation} \label{eq:iso wh S and S flat with TT}
    \hom_{\TT_r}(
        \wh{S}^{\Pord}_{\kappa, V, \pi},
        S^{\Pord}_{\kappa^\flat, -V, \pi^\flat}
            \otimes
        (\chi \circ \det)
    )
        \xlongrightarrow{\sim}
    \hom_{\TT_r}(
        \TT_r \otimes \wh{I}_\pi,
        \wh{\TT}_r \otimes I_{\pi^\flat}
    )\,,
\end{equation}
and setting $C_r = C_r(Z, \OO)$, we can then identify $\phi_{\rho, r}$ as an element of
\begin{align*}
    \hom_{\TT_r}(
        C_r \otimes_\Lambda \wh{S}^{\Pord}_{\kappa, V, \pi},
        S^{\Pord}_{\kappa^\flat, -V, \pi^\flat}
    )
        \xlongrightarrow{\sim}
    &\hom_{\TT_r}(
        C_r \otimes_\Lambda \TT_r \otimes_\OO \wh{I}_\pi,
        \wh{\TT}_r \otimes_\OO I_{\pi^\flat}
    ) \\
        =\,\,
    &\hom_{\TT_r}(
        C_r \otimes_\Lambda \TT_r,
        \wh{\TT}_r
    )
        \otimes_\OO
    \ehom_\OO(I_{\pi^\flat}) \\
        \xlongrightarrow{\sim}
    &\wh{\TT}_r
        \otimes
    \ehom_\OO(I_{\pi^\flat}) \,,
\end{align*}
and the next step is to study the compatibility on both sides as $r \gg 0$ varies.

To understand the left-hand side of the above ar $r$ varites, consider the inclusions $\eta_r : C_r \hookrightarrow C_{r+1}$ and $\iota_r : S_{r, V}^{\Pord} \hookrightarrow S_{r+1, V}^{\Pord}$, as well as the dual maps $\iota_r^*$ and $\eta_r^*$ respectively, for all $r \gg 0$. Then, as explained in \cite[Fact 7.4.7]{EHLS}, we have
\begin{equation} \label{eq:compatibility phi rho r+1 and rho r}
    (\eta_r^* \otimes \mathrm{id}_{r+1})(\phi_{\rho, r+1})
        %=
    %\mathbf{i}_r^*(\phi_{\rho, r})
        =
    \iota_r \circ \phi_{\rho, r} \circ (\mathrm{id}_{C_r} \otimes \iota_r^*)\,.
\end{equation}

Furthermore, it follows from our work in Sections \ref{subsubsec:unnorm and norm Serre duality}--\ref{subsubsec:integral struct on aholo mod forms} that the map $\iota_r^* : \wh{S}^{\Pord}_{\kappa, V}(K_{r+1}, \OO_\pi) \to \wh{S}^{\Pord}_{\kappa, V}(K_{r+1}, \OO_\pi)$ is given by the trace map
\[
    t_r(h)
        =
    \frac{\#(I_{P, r}^0/I_{P,r})}{\#(I_{P, r+1}^0/I_{P,r+1})}
    \sum_{\gamma \in K_{P, r} / K_{P, r+1}}
        \gamma \cdot h
    \,,
\]
for all $h \in \wh{S}^{\Pord}_{\kappa, V}(K_{r+1}, \OO_\pi)$. Comparing this with first commutative diagram in Proposition \ref{prop:minimality hypothesis}, we obtain the following result.

\begin{proposition} \label{prop:existence of L(phi chi kappa)}
    Let $(\kappa, \chi)$ be a critical pair such that $\kappa$ is a very regular weight and assume Conjectures \ref{conj:indep of weight of padic Hecke alg} and \ref{conj:VCT}. Let $\rho$ be the weight determined by $(\kappa, \chi)$ in \eqref{eq:def rho sigma and rho sigma v shifts}. With respect to the identification
    \[
        \hom_{\TT_r}(
            C_r \otimes_\Lambda \wh{S}^{\Pord}_{\kappa, V, \pi},
            S^{\Pord}_{\kappa^\flat, -V, \pi^\flat}
        )
            \xrightarrow{\sim}
        \wh{\TT}_r
            \otimes
        \ehom_\OO(I_{\pi^\flat})\,
    \]
    the identity \eqref{eq:compatibility phi rho r+1 and rho r}, the isomorphism $G_r : \wh{\TT}_r \xrightarrow{\sim} \TT_r$ provided by Hypothesis \ref{hyp:Gorenstein hyp}, the collection $\phi_{\chi, \kappa} = \phi_\rho = (\phi_{\rho, r})_r$ defines an element
    \[
        L(\phi_{\chi, \kappa})
            =
        L(\phi_{\rho})
            \in
        \varprojlim_r 
        \TT_r
            \otimes
        \ehom_\OO(I_{\pi^\flat})
            \xrightarrow{\sim}
        \TT
            \otimes
        \ehom_\OO(I_{\pi^\flat})\,.
    \]

    Moreover, if $\kappa'$ is another very regular weight in the same $P$-parallel lattice as $\kappa$, i.e. $[\kappa] = [\kappa']$, then $L(\phi_{\chi, \kappa}) = L(\phi_{\chi, \kappa'})$ as elements of 
    \[
        \TT 
            =
        \TT_{K^p, [\kappa, \tau], \OO}
            \xrightarrow{\sim}
        \varprojlim_r
            \TT_{K_r, \kappa, [\tau], \OO}
            \xrightarrow{\sim}
        \varprojlim_r
            \TT_{K_r, \kappa', [\tau], \OO}\,.
    \]

    Therefore, $(\phi_{\chi, \kappa, r})_r$ and $(\phi_{\chi, \kappa', r})_r$ define the same element $L(\phi_{\chi, [\kappa]}) \in \TT \otimes \ehom_\OO(I_{\pi^\flat})$. Conversely, any $L \in \TT \otimes \ehom_\OO(I_{\pi^\flat})$ is induced as above from a collection $(\phi_{\chi, \kappa, r})_r$ associated to some very regular $\kappa$.
\end{proposition}

\begin{remark}
    This is the $P$-ordinary analogue of \cite[Proposition 7.4.10]{EHLS} and is proved the exact same way, namely unfolding definitions using the identifications introduced in this section.
\end{remark}

Lastly, to consider the variation of $\chi$ as a character of $X_p$, consider the Iwasawa algebra $\Lambda_{X_p} = \ZZ_p[[X_p]]$. Then, it follows from \cite[Proposition 7.4.13]{EHLS} and Proposition \ref{prop:existence of L(phi chi kappa)} that the $\VV_3$-valued measure $d\Eis^{[\kappa, \tau]} = \phi = \phi_\bullet$ on $X_p \times Z_p$ corresponds to an element
\begin{equation} \label{eq:def L d Eis kappa tau}
    L(\Eis^{[\kappa, \tau]}) 
        \in 
    \Lambda_{X_p} \wh{\otimes} \TT \otimes \ehom_{\OO}(I_{\pi^\flat})\,.
\end{equation}

\subsubsection{Evaluation at classical points} \label{subsubsec:evaluation at classical points}
Let $\pi$ be a anti-holomorphic $P$-anti-ordinary automorphic representation $\pi$ for $G_1$ of $P$-anti-WLT $(\kappa, K_r, \tau)$. Let $\lambda_\pi : \TT \to \OO_\pi$ be its associated Hecke character, see Section \ref{subsubsec:lattices in pi flat}, and let $\m_\pi$ be the kernel of $\lambda_\pi$. Consider the set of classical points $\SSS(K^p, \pi)$ defined in \eqref{eq:def SSS Kp pi}. 

Note that for any $\OO$-algebra $R$,
\[
    \ehom_R(I_{\pi^\flat}) 
        =
    \hom(\wh{I}_{\pi}, I_{\pi^\flat})
        \cong
    \hom(\wh{I}_{\pi} \otimes \wh{I}_{\pi^\flat}, R)
\]
so given any test vectors $\varphi \in \wh{I}_\pi$ and $\varphi^\flat \in \wh{I}_{\pi^\flat}$ as in Section \ref{subsubsec:I pi and test vectors}, we can define
\[
    L(\Eis^{[\kappa, \tau]}; \varphi, \varphi^\flat)
        :=
    [L(\Eis^{[\kappa, \tau]}), \varphi \otimes \varphi^\flat]_{\mathrm{loc}}
        \in
    \Lambda_{X_p} \wh{\otimes} \TT\,,
\]
and
\[
    L(\Eis^{[\kappa, \tau]}, \chi, \kappa; \varphi, \varphi^\flat)
        :=
    [L(\phi_{\chi, [\kappa]}), \varphi \otimes \varphi^\flat]_{\mathrm{loc}}
        \in
    \TT\,,
\]
where $[\bullet, \bullet]_{\mathrm{loc}}$ is induced from the tautological pairing in both cases (abusing notation), and we recall that the relation between $\dEis^{[\kappa, \tau]}$ and $\phi_{\chi, [\kappa]}$ is given by \eqref{eq:def phi phi chi wrt dEis} and Proposition \ref{prop:existence of L(phi chi kappa)}.

Given $R$-valued character $\wh{\chi}_0 : X_p \to R$ and any classical $\pi' \in \SSS(K^p, \pi)$ of $P$-anti-WLT $(\kappa', K_{r'}, \tau')$ such that $(\kappa', \chi)$ is critical and $\lambda_{\pi'}$ is $R$-valued, the image of $L(\Eis^{[\kappa, \tau]}; \varphi, \varphi^\flat)$ under the homomorphism $\wh{\chi}_0 \otimes \lambda_{\pi'} : \Lambda_{X_p, R} \otimes \TT_{\pi, R} \to R$ induced by $(\wt{\chi}_0, \lambda_{\pi'})$ is equal to
\[
    \lambda_{\pi'}(
        L(\Eis^{[\kappa, \tau]}, \chi, \kappa; \varphi, \varphi^\flat)
    )
        \in 
    R
\]
and our computations at the end of Section \ref{subsubsec:comp to classical Eis series} show that the latter is equal to the expression in \eqref{eq:pairing eis measure with global test vectors}.

%%%%%%%%%%%%%%%%%%%%%%%%%%%%%%%%%%%%%%%%%%%%%%%%%%%%%%%%%%%%%%%%%%

\subsection{Normalized periods and congruence ideals.} \label{subsec:periods and cong ideals}
We are now ready to state our main theorem to summarize the construction of the $p$-adic $L$-function in \eqref{eq:def L d Eis kappa tau}. However, we first adjust the definitions of certain periods studied in \cite[Section 6.7]{EHLS} to generalize the theory to $P$-anti-ordinary representation.

Fix an anti-holomorphic $P$-anti-ordinary automorphic representation $\pi$ on $G = G_1$ with $P$-anti-WLT $(\kappa, K_r, \tau)$. In what follows, we use the notation of Sections \ref{subsec:lattices of Pord holo forms}-\ref{subsec:lattices of aholo Paord forms} freely.

Consider the orthogonal complement
\[
    \wh{S}_{\kappa, V}^{\Paord}(
        K_r, \tau; R
    )[\pi]^\perp
        \subset
    \wh{S}_{\kappa^\flat, -V}^{\Paord}(
        K^\flat_r, \tau^\flat; R
    )_{\pi^\flat}
\]
of $\wh{S}_{\kappa, V}^{\Paord}(K_r, \tau; R)[\pi]$ with respect to $\frac{1}{\vol(I^0_{V, r} \cap I^0_{-V, r})}\brktdotdot_{\kappa, \tau}^\Ser$, see Lemma \ref{lma:jpi flat iso} and Section \ref{subsubsec:the involution flat}

\begin{definition}
    The congruence ideal $C(\pi) \subset R$ associated to $\pi$ is the annihilator of
    \[
        \wh{S}_{\kappa^\flat, -V}^{\Paord}(
            K^\flat_r, \tau^\flat; R
        )_{\pi^\flat}
            /
        \left(
            \wh{S}_{\kappa, V}^{\Paord}(
                K_r, \tau; R
            )[\pi]^\perp
                +
            \wh{S}_{\kappa^\flat, -V}^{\Paord}(
                K^\flat_r, \tau^\flat; R
            )[\pi^\flat]
        \right)
    \]
\end{definition}

\begin{lemma}
    Let $R \subset \CC$ be a ring as in Proposition \ref{prop:jpi iso (II)}, then
    \begin{align*}
        L[\pi] &:=
            \frac{1}{\vol(I^0_{V, r} \cap I^0_{-V, r})}
            \brkt
            {
                \wh{S}_{\kappa, V}^{\Paord}
                (
                    K_r, \tau; R
                )[\pi]
            }
            {
                \wh{S}_{\kappa^\flat, -V}^{\Paord}
                (
                    K_r^\flat, \tau^\flat; R
                )[\pi^\flat]
            }_{\kappa, \tau}^{\Ser}\,, \text{ and }
        \\
        L_\pi  &:=
            \frac{1}{\vol(I^0_{V, r} \cap I^0_{-V, r})}
            \brkt
            {
                \wh{S}_{\kappa, V}^{\Paord}
                (
                    K_r, \tau; R
                )[\pi]
            }
            {
                \wh{S}_{\kappa^\flat, -V}^{\Paord}
                (
                    K_r^\flat, \tau^\flat; R
                )_{\pi^\flat}
            }_{\kappa, \tau}^{\Ser}
    \end{align*}
    are rank one $R$-submodules of $\CC$, generated by positive real numbers $Q[\pi]$ and $Q_\pi$, respectively. Any $c(\pi) \in R$ such that $c(\pi) Q_\pi = Q[\pi]$ generates the congruence ideal $C(\pi)$.
\end{lemma}

Obviously, $Q[\pi]$, $Q_\pi$ and $c(\pi)$ are only well-defined up to units in $R$. However, our $p$-adic $L$-function does not depend on those choices. Furthermore, given a Hecke character $\chi$, one has analogues $Q[\pi, \chi]$, $Q_{\pi, \chi}$, $C(\pi, \chi)$ and $c(\pi, \chi)$ upon twisting by $\chi^{-1} \circ \det$ as explained in \cite[Section 6.7.6]{EHLS}.

\begin{proposition}
    Given anti-holomorphic $P$-anti-ordinary test vectors $\varphi \in \wh{I}_\pi$ and $\varphi^\flat \in \wh{I}_{\pi^\flat}$ as in Section \ref{subsubsec:I pi and test vectors}, the period
    \[
        \Omega_{\pi, \chi}(\varphi, \varphi^\flat)
            =
        \frac{
            \dim \tau \cdot
            \brkt{\varphi}{\varphi^\flat_\chi}_\chi}
        {
            \vol(I^0_{r,V} \cap I^0_{r, -V}) \cdot
            Q[\pi, \chi]
        }
    \]
    is independent of $r$ and is $p$-integral. It is a $p$-adic unit for an appropriate choice of $\varphi$ and $\varphi^\flat$.
\end{proposition}

\begin{proof}
    The independence on $r$ follows from the properties of the Serre pairing and $\varphi^\flat$ under the trace map as $r$ increases. 
    
    Furthermore, the fact that it is $p$-integral (resp. a $p$-adic unit) follows from the fact that the factor $\dim \tau$ in the above expression cancels with the factor $\dim \tau$ in the definition of $Q[\pi, \chi]$.
\end{proof}

\subsection{Statement of the main theorem} \label{subsec:statement of main thm}
Our main theorem is simply a summary of the properties of the $p$-adic $L$-function constructed in \eqref{eq:def L d Eis kappa tau} and incorporate the periods introduced above.

In the following statement, we refer to the Conjectures \ref{conj:density conj scalar padic forms}, \ref{conj:classicality scalar padic forms}, \ref{conj:classicality vector padic forms}, \ref{conj:indep of weight of padic Hecke alg} and \ref{conj:VCT} as the ``standard conjectures of $P$-ordinary Hida theory''.

\begin{theorem} \label{thm:main thm}
    Let $\pi$ be an anti-holomorphic, $P$-anti-ordinary cuspidal automorphic form $G_1(\AA)$ whose $P$-anti-WLT is $(\kappa, K_r, \tau)$, where $\tau$ is the SZ-type of $\pi$. Assume that the standard conjectures of $P$-ordinary Hida theory hold. Assume that $\pi$ satisfy Hypothesis \ref{hyp:Qw equals Pw}, Hypothesis \ref{hyp:mult one hyp (I)}, Hypothesis \ref{hyp:Gorenstein hyp}, and Proposition-Hypothesis \ref{prop:minimality hypothesis}.
    
    Let $\w_\tau$ denote the central character of $\tau$. Let $\m_\pi$ denote the maximal ideal of the $P$-ordinary Hecke algebra $\bTT^{\Paord}_{K_r, \kappa, \tau}$ corresponding to $\pi$ and let $\TT_\pi$ be the localization of $\bTT^{\Paord}_{K_r, \kappa, \tau}$ at $\m_\pi$.

    Given test vectors $\varphi \in \wh{I}_\pi$, $\varphi^\flat \in \wh{I}_{\pi^\flat}$ as in Section \ref{subsubsec:I pi and test vectors}, there exists a unique element
    \[
        L(\Eis^{[\kappa, \tau]}, \Pord; \varphi \otimes \varphi^\flat) 
            \in 
        \Lambda_{X_p, R} \wh{\otimes} \TT_{\pi}
    \]
    satisfying the following property :

    Let $\chi = ||\cdot||^{\frac{n-k}{2}}\chi_u : X_p \to R^\times$ be the $p$-adic shift of a Hecke character as in Section \ref{subsubsec:padic shifts of Hecke char}. Let $\pi' \in \SSS(K^p, \pi)$ be a classical point of the $P$-ordinary Hida family $\TT_\pi$.

    Then, $L(\Eis^{[\kappa, \tau]}, \Pord; \varphi \otimes \varphi^\flat)$ is mapped under the character $\chi \otimes \lambda_{\pi'}$ to
    \begin{align*}
        c(\pi',\chi)
        \Omega_{\pi', \chi}&(\varphi, \varphi^\flat)
        L_p
        \left(
            \frac{k-n+1}{2}, 
            \Pord,
            \pi',
            \chi_u
        \right) 
            \\ \times\,\,\,
        &L_\infty\left(
            \frac{k-n+1}{2};
            \chi_u, \kappa'
        \right)
        I_S
        \frac{
            L^S(
                \frac{k-n+1}{2}, 
                \pi', 
                \chi_u
            )
        }{
            P_{\pi', \chi}
        }\,,
    \end{align*}
    where $P_{\pi', \chi} = Q_{\pi', \chi}^{-1}$.
\end{theorem}

%%%%%%%%%%%%%%%%%%%%%%%%%%%%%%%%%%%%%%%%%%%%%%%%%%%%%%%%%%%%%%%%%%

%%%%%%%%%%%%%%%%
% References
%%%%%%%%%%%%%%%%
\bibliography{references.bib}

\newcommand{\etalchar}[1]{$^{#1}$}
\providecommand{\bysame}{\leavevmode\hbox to3em{\hrulefill}\thinspace}
\providecommand{\MR}{\relax\ifhmode\unskip\space\fi MR }
% \MRhref is called by the amsart/book/proc definition of \MR.
\providecommand{\MRhref}[2]{%
  \href{http://www.ams.org/mathscinet-getitem?mr=#1}{#2}
}
\providecommand{\href}[2]{#2}
\begin{thebibliography}{EFMV18}

\bibitem[BC09]{BC09}
J.~Bellaiche and G.~Chenevier, \emph{Families of galois representations and
  selmer groups}, Astérisque, vol. 324, Soc. Math. France, Paris, 2009.

\bibitem[BHR94]{BHR94}
D.~Blasius, M.~Harris, and D.~Ramakrishnan, \emph{Coherent cohomology, limits
  of discrete series, and {G}alois conjugation}, Duke Math. J. \textbf{73}
  (1994), no.~3, 647–685.

\bibitem[BK93]{BusKut93}
C.~J. Bushnell and P.~C. Kutzko, \emph{The admissible dual of {GL}(\textit{N})
  via compact open subgroups}, Annals of Mathematics Studies, vol. 129,
  Princeton University Press, 1993.

\bibitem[BK98]{BusKut98}
\bysame, \emph{Smooth representations of reductive \textit{p}-adic groups :
  {S}tructure theory via types}, Proceedings of the London Mathematical Society
  \textbf{77} (1998), no.~3, 582--634.

\bibitem[BK99]{BusKut99}
\bysame, \emph{Semisimple types in {GL}(\textit{n})}, Compositio Mathematica
  \textbf{119} (1999), 57--106.

\bibitem[Cas95]{Cas95}
W.~Casselman, \emph{Introduction to the {T}heory of {A}dmissible
  {R}epresentations of \textit{p}-adic {R}eductive {G}roups}, Unpublished
  manuscript.
  \url{https://personal.math.ubc.ca/~cass/research/pdf/p-adic-book.pdf}, 1995,
  80 pages.

\bibitem[CEF{\etalchar{+}}16]{CEFMV}
A.~Caraiani, E.~Eischen, J.~Fintzen, E.~Mantovan, and I.~Varma, \emph{p-adic
  q-expansion {P}rinciples on {U}nitary {S}himura {V}arieties}, Directions in
  Number Theory, Association for Women in Mathematics Series, vol.~3, Springer
  International Publishing, Cham, 2016, pp.~197--243.

\bibitem[Coa89]{Coa89}
J.~Coates, \emph{On \textit{p}-adic \textit{L}-functions attached to motives
  over \textbf{Q}. ii}, Bol. Soc. Brasil. Mat. (N.S.) \textbf{20} (1989),
  no.~1, 101--112.

\bibitem[EFMV18]{EFMV18}
E.~Eischen, J.~Fintzen, E.~Mantovan, and I.~Varma, \emph{Differential operators
  and families of automorphic forms on unitary groups of arbitrary signature},
  Doc. Math. \textbf{23} (2018), 445--495.

\bibitem[EHLS20]{EHLS}
E.~Eischen, M.~Harris, J.-S. Li, and C.~Skinner, \emph{p-adic {L}-functions for
  unitary groups}, Forum of Mathematics, Pi \textbf{8} (2020), e9.

\bibitem[Eis12]{Eis12}
E.~Eischen, \emph{\textit{p}-adic differential operators on automorphic forms
  on unitary groups}, Ann. Inst. Fourier (Grenoble) \textbf{62} (2012), no.~1,
  177--243.

\bibitem[Eis14]{Eis14}
\bysame, \emph{A \textit{p}-adic {E}isenstein measure for vector-weight
  automorphic forms}, Algebra Number Theory \textbf{8} (2014), no.~10,
  2433--2469.

\bibitem[Eis15]{Eis15}
\bysame, \emph{A \textit{p}-adic {E}isenstein measure for unitary groups}, J.
  Reine Angew. Math. \textbf{699} (2015), 111--142.

\bibitem[EL20]{EisLiu20}
E.~Eischen and Z.~Liu, \emph{Archimedean {Z}eta {I}ntegrals for {U}nitary
  {G}roups.}, arXiv:2006.04302 [math.NT], June 2020, Preprint available at
  \href{https://arxiv.org/abs/2006.04302}{arxiv:2006.04302}.

\bibitem[Gar84]{Gar84}
P.~B. Garrett, \emph{Pullbacks of {E}isenstein series; applications},
  {A}utomorphic {F}orms of {S}everal {V}ariables ({K}atata, 1983), Progress in
  Mathematics, vol.~46, Birkhäuser, Boston, MA, 1984, pp.~114--137.

\bibitem[GPSR87]{GPSR87}
S.~Gelbart, I.~Piatetski-Shapiro, and S.~Rallis, \emph{Explicit {C}onstructions
  of {A}utomorphic {L}-functions}, Lecture Notes in Mathematics, vol. 1254,
  Springer, Berlin, 1987.

\bibitem[Har86]{Har86}
M.~Harris, \emph{Arithmetic vector bundles and automorphic forms on {S}himura
  varieties. {II}}, Compositio Math. \textbf{60} (1986), no.~3, 323--378.

\bibitem[Har90]{Har90}
\bysame, \emph{Automorphic forms of $\overline{\partial}$-cohomology type as
  coherent cohomology classes}, J. Differential Geom. \textbf{32} (1990),
  no.~1, 1--63.

\bibitem[Har97]{Har97}
\bysame, \emph{L-functions and periods of polarized regular motives}, J. Reine
  Angew. Math. \textbf{483} (1997), 75–161.

\bibitem[Har08]{Har08}
\bysame, \emph{A simple proof of rationality of {S}iegel–{W}eil {E}isenstein
  series}, Eisenstein Series and Applications, Progress in Mathematics, vol.
  258, Birkhäuser, Boston, MA, 2008, p.~149–185.

\bibitem[Hid98]{Hid98}
H.~Hida, \emph{Automorphic induction and {L}eopoldt type conjectures for
  {GL}(\textit{n})}, Asian J. Math. \textbf{2} (1998), no.~4, 667--710, Mikio
  Sato: a great Japanese mathematician of the twentieth century.

\bibitem[Hid04]{Hid04}
\bysame, \emph{p-adic automorphic forms on {S}himura varieties}, Springer
  Monographs in Mathematics, Springer, New York, NY, 2004.

\bibitem[HLLM23]{HLLM23}
B.~L. Hung, D.~Le, B.~Levin, and S.~Morra, \emph{Local models for {G}alois
  deformation rings and applications}, Inventiones mathematicae \textbf{231}
  (2023), 1277--1488.

\bibitem[HLS06]{HLS}
M.~Harris, J.S. Li, and C.~Skinner, \emph{\textit{p}-adic \textit{L}-functions
  for unitary {S}himura varieties, {I}: {C}onstruction of the {E}isenstein
  {M}easure}, Doc. Math. \textbf{\textbf{Extra Vol.}} (2006), 393--464,
  (electronic).

\bibitem[Jac79]{Jac79}
H.~Jacquet, \emph{{P}rincipal \textit{L}-functions of the linear group}, in
  Automorphic {F}orms, {R}epresentations and {L}-functions (Proc. Sympos. Pure
  Math., Oregon State Univ., Corvallis, Ore., 1977), Part 2, Proceedings of
  Symposia in Applied Mathematics, XXXIII, (American Mathematical Society,
  Providence, RI), 1979, pp.~63--86.

\bibitem[Jan03]{Jan03}
J.~C. Jantzen, \emph{Representations of {A}lgebraic {G}roups}, second ed.,
  Mathematical surveys and monographs, vol. 107, American Mathematical Society,
  Providence, RI, 2003.

\bibitem[Kat78]{Kat78}
N.~M. Katz, \emph{\textit{p}-adic \textit{L}-functions for {CM} fields},
  Invent. Math. \textbf{49} (1978), no.~3, 199--297.

\bibitem[Kot92]{Kot92}
R.~E. Kottwitz, \emph{Points on {S}ome {S}himura {V}arieties {O}ver {F}inite
  {F}ields}, J. Amer. Math. Soc. \textbf{5} (1992), no.~2, 373--444.

\bibitem[Lan12]{Lan12}
K.-W. Lan, \emph{Comparison between analytic and algebraic constructions of
  toroidal compactifications of {PEL}-type {S}himura varieties}, Journal für
  die reine und angewandte Mathematik (Crelles Journal) \textbf{664} (2012),
  163--228.

\bibitem[Lan13]{Lan13}
\bysame, \emph{Arithmetic {C}ompactifications of {PEL}-type {S}himura
  {V}arieties}, London Mathematical Society Monographs, vol.~36, Princeton
  University Press, Princeton, 2013.

\bibitem[Lan16]{Lan16}
\bysame, \emph{Higher {K}oecher’s principle}, Math. Res. Lett. \textbf{23}
  (2016), no.~1, 163--199.

\bibitem[Li92]{Li92}
J.-S. Li, \emph{Nonvanishing theorems for the cohomology of certain arithmetic
  quotients}, J. Reine Angew. Math. \textbf{428} (1992), 177--217.

\bibitem[LR20]{LiuRos20}
Z.~Liu and G.~Rosso, \emph{Non-cuspidal {H}ida theory for {S}iegel modular
  forms and trivial zeros of \textit{p}-adic \textit{L}-functions}, Math. Ann.
  \textbf{378} (2020), 153--231.

\bibitem[Pil12]{Pil12}
V.~Pilloni, \emph{Sur la théorie de {H}ida pour le groupe {GSp}$_{2g}$}, Bull.
  Soc. Math. France \textbf{140} (2012), no.~3, 335--400.

\bibitem[Ren10]{Ren10}
D.~Renard, \emph{Repr{é}sentations des groupes r{é}ductifs p-adiques},
  Collection SMF. Cours sp{é}cialis{é}s, vol.~17, Soci{é}t{é}
  {M}ath{é}matique de France, Paris, France, 2010.

\bibitem[Shi83]{Shi83}
G.~Shimura, \emph{On {E}isenstein {S}eries}, Duke Math. J. \textbf{50} (1983),
  no.~2, 417--476.

\bibitem[Shi97]{Shi97}
\bysame, \emph{Euler {P}roducts and {E}isenstein {S}eries}, CBMS Regional
  Conference Series in Mathematics, vol.~93, Published for the Conference Board
  of the Mathematical Sciences, Washington, DC, 1997.

\bibitem[SU02]{SU02}
D.~Skinner and E.~Urban, \emph{Sur les déformations \textit{p}-adiques des
  formes de {S}aito-{K}urokawa}, C. R. Math. Acad. Sci. Paris \textbf{335}
  (2002), no.~7, 581--586.

\bibitem[SZ99]{SchZin99}
P.~Schneider and E.~W. Zink, \emph{\textit{K}-types for the tempered components
  of a \textit{p}-adic general linear group}, J. Reine Angew. Math.
  \textbf{517} (1999), 161--208.

\end{thebibliography}
\bibliographystyle{amsalpha}

\end{document}